\documentclass[12pt]{article}
\usepackage{array,epsfig,fancyheadings,rotating}
\usepackage[]{hyperref}
\usepackage[sectionbib]{natbib}
\usepackage{amsmath}
\usepackage{amssymb}
\usepackage{amsfonts}
\usepackage{multirow}
\usepackage{amsthm}
\usepackage{color}
\usepackage[utf8]{inputenc}
\usepackage[english]{babel}
 \usepackage{graphicx}
\usepackage{url}
\usepackage[toc,page]{appendix}
\usepackage{tabularx}
\usepackage{subcaption}
\usepackage{setspace}
\usepackage{authblk}
\usepackage{xr}
\usepackage{sectsty, secdot}
\sectionfont{\fontsize{12}{14pt plus.8pt minus .6pt}\selectfont}

\subsectionfont{\fontsize{12}{14pt plus.8pt minus .6pt}\selectfont}

\newtheorem{theorem}{Theorem}

\newtheorem{condition}{Condition}

\newtheorem{lemma}[theorem]{Lemma}

\newtheorem{proposition}[theorem]{Proposition}
\newtheorem{remark}{Remark}

\usepackage{setspace}
\usepackage{footmisc}

\newcommand{\RNum}[1]{\uppercase\expandafter{\romannumeral #1\relax}}
\begin{document}

\renewcommand{\baselinestretch}{2}

\lhead[\footnotesize\thepage\fancyplain{}\leftmark]{}\rhead[]{\fancyplain{}\footnotesize\thepage}

\markright{ \hbox{\footnotesize\rm Statistica Sinica
}\hfill\\[-13pt]
\hbox{\footnotesize\rm
}\hfill }

\markboth{\hfill{\footnotesize\rm Yinqiu He, Tiefeng Jiang, Jiyang Wen and Gongjun Xu} \hfill}
{\hfill {\footnotesize\rm LRT} \hfill}


\renewcommand{\thefootnote}{}
$\ $\par

\fontsize{12}{14pt plus.8pt minus .6pt}\selectfont \vspace{0.8pc}
\centerline{\large\bf Likelihood Ratio Test in Multivariate Linear Regression: }
\vspace{2pt} \centerline{\large\bf  from Low to High Dimension}
\vspace{.4cm} \centerline{Yinqiu He$^{1}$, Tiefeng Jiang$^2$, Jiyang Wen$^{3}$, Gongjun Xu$^{1}$
}
\centerline{\it
$^1$Department of Statistics, University of Michigan} \vspace{-.35cm}  \centerline{\it
$^2$School of Statistics, University of Minnesota} \vspace{-.35cm}  \centerline{\it
$^3$Department of Biostatistics, Johns Hopkins Bloomberg School of Public Health}
\vspace{.55cm} \fontsize{9}{11.5pt plus.8pt minus
.6pt}\selectfont





\begin{quotation}
\noindent {\it Abstract:}
Multivariate linear regressions are widely used statistical tools in  many     applications to model the associations between multiple related responses and a set of predictors. 
To infer such associations, it is often of interest to test the structure of the regression coefficients matrix, and the likelihood ratio test (LRT) is one of the most popular approaches in practice.
Despite its popularity, it is known that the classical  $\chi^2$ approximations for LRTs often fail in high-dimensional settings,  where the  dimensions of responses and predictors $(m,p)$ are allowed to grow with the sample size $n$. Though various corrected LRTs and other test  statistics have been proposed in the literature,  the  important question of when the classic LRT starts to fail  is less studied; an answer to this would provide insights for practitioners, especially when analyzing data with $m/n$ and $p/n$ small but not negligible. Moreover, the power performance of the LRT in high-dimensional data analysis remains underexplored.
To address these issues, the first part of this work gives the asymptotic boundary where the classical LRT fails and develops the corrected limiting distribution of the LRT for a general asymptotic regime. 
The second part of this work further studies  the test power of the LRT in the high-dimensional settings. The result not only  advances  the current understanding of asymptotic behavior of the LRT under alternative hypothesis, but also motivates the development of a power-enhanced LRT. 
The third part of this work considers the   setting with $p>n$, where the LRT is not well-defined. We propose a two-step testing procedure by first performing dimension reduction and then applying the proposed LRT. Theoretical properties are developed to ensure the validity of the proposed method.  
 Numerical studies are also presented to  demonstrate  its good performance.

\vspace{9pt}

\par
\end{quotation}\par





\fontsize{12}{14pt plus.8pt minus .6pt}\selectfont

\setcounter{section}{0} 
\setcounter{equation}{0} 

\section{Introduction}

Multivariate linear regressions are   widely used in econometrics, financial engineering, 
psychometrics and many other areas of applications  to model the   relationships between multiple related responses and a set of predictors. Suppose we have $n$ observations of $m$-dimensional responses $\mathbf{y}_i=(y_{i,1},\ldots,y_{i,m})^{\intercal}$ and  $p$-dimensional predictors $\mathbf{x}_i=(x_{i,1},\ldots,x_{i,p})^{\intercal}$, for $i=1,\ldots,n$. Let $Y=(\mathbf{y}_1,\ldots,\mathbf{y}_n)^{\intercal}$ be the $n\times m$ response  matrix and $X=(\mathbf{x}_1,\ldots,\mathbf{x}_n)^{\intercal}$ be the $n\times p$   design matrix. The   multivariate linear regression model assumes  
$
	Y=X B+E,
$
where  $B$ is a $p\times m$ matrix of unknown regression parameters and   $E=(\boldsymbol{\epsilon}_1,\ldots,\boldsymbol{\epsilon}_n)^{\intercal}$ is an  $n\times m$ matrix of regression errors,   with $\boldsymbol{\epsilon}_i$'s    independently  sampled from an $m$-dimensional Gaussian distribution $\mathcal{N}(\mathbf{0},\Sigma)$.   


Under the multivariate linear regression model, we are interested in testing the null hypothesis $H_0: CB=\mathbf{0}_{r\times m}$, where $C$ is an  $r\times p$ matrix with rank $r\leq p$ and $\mathbf{0}_{r\times m}$ is an all-zero matrix of size $r\times m$.  
This is often called general linear hypothesis in multivariate analysis and   has been popularly used in   multivariate analysis of  variance \citep[see, e.g.,][]{muirhead2009aspects}. 
Different choices of the testing matrix $C$ are of interest in various applications. For instance, if $B$ is partitioned as $B^{\intercal}=[B_1^{\intercal},B_2^{\intercal}]$, where $B_1$ is an $r\times m$ matrix; then the null hypothesis of $B_1=\mathbf{0}_{r\times m}$ is equivalent to taking $C=[I_r,\mathbf{0}_{r\times(p-r)}]$, which can be used to test the significance of  the first $r$ predictors of $X$.
 Another example is to test the equivalence of the effects of a set of $r+1$ predictors (such as different levels of some categorical  variables), where $C=[I_r, \mathbf{0}_{r\times (p-r-1)}, -\mathbf{1}_r ]$ and $\mathbf{1}_r$ represents an all 1 vector of length $r$.

 
To test $H_0: CB=\mathbf{0}_{r\times m}$, a popularly used   approach in the literature  is the likelihood ratio test (LRT)   \citep{anderson1958introduction,muirhead2009aspects}. 
Specifically, when $n>m+p$, $\Sigma$ is positive definite and $X$ has rank $p$,  the LRT statistic is
$L_n={\det({ S}_E)^{n/2}}/\{\det({ S}_E+{ S}_X)^{n/2}\} $,
 where  
${ S}_E=Y^{\intercal}[I-X(X^{\intercal}X)^{-1}X^{\intercal}]Y$ 
and ${ S}_X = (C\hat B)^{\intercal}[C(X^{\intercal}X)^{-1}C^{\intercal}]^{-1}C\hat B$ are   the residual  sum of squares and   the regression sum of squares  matrices respectively, and
$\hat B = (X^{\intercal}X)^{-1}X^{\intercal}Y$ is the least squares estimator. 
Assuming $m$ and $p$ are fixed, it is well known that $-2\log L_n $ converges weakly to a $\chi^2$ distribution  as $n \to \infty$  under the null hypothesis  \citep{anderson1958introduction}.

However, in the high-dimensional settings where the  dimension parameters ($p, m, r$)  are allowed to   increase with $n$, 
the LRT suffers from several issues. 
First, under the null hypothesis, the limiting distribution of $-2\log L_n $ may not be a $\chi^2$  distribution any more. The failure of the $\chi^2$  approximations  of LRT distributions under high dimensions has been studied by researchers under various model settings. For instance, \cite{bai2009} examined two LRTs on testing   covariance matrices, showed that their $\chi^2$ approximations perform poorly, and proposed the corrected normal limiting distributions. \cite{jiang2013} and \cite{jiang2015likelihood} studied   classical LRTs on testing sample means and covariance matrices, and showed  that the $\chi^2$ approximations also fail as the dimensions increase.  Moreover, \cite{bai2013testing} considered the LRT on testing linear hypotheses in  high-dimensional multivariate linear regressions, demonstrated the failure of $\chi^2$ approximation and derived the corrected LRT.   Note that \cite{bai2013testing} only considered the high-dimensional settings where $m,r$ and $n-p$ are proportional to each other with $m\leq r$.           
Despite these existing works, it is still unclear under which asymptotic regimes  the $\chi^2$ approximation of LRT starts to fail.
An answer to this question would provide insights for practitioners especially when analyzing data with $m/n$ and $p/n$ small but not negligible.

The second issue of the LRT is concerned with its   power performance  under high-dimensional alternative hypotheses.
When $n>p+m$,   the statistic  $-2\log L_n  =  {n} \sum_{i=1}^{\min\{m,r\}}  \log (1+\lambda_i)$, where $\lambda$'s are the eigenvalues of ${S}_X^{1/2}{ S}_E^{-1}{S}_X^{1/2}$; 
therefore it is expected that the asymptotic power behavior of LRT   would depend  on an averaged  effect of all the eigenvalues. 
Such a  study of the eigenvalues of the random matrix ${S}_X^{1/2}{ S}_E^{-1}{S}_X^{1/2}$ under alternative hypotheses remains   underexplored in the literature. 


 
The third issue of the LRT arises when the dimension parameters $p$ and $m$ are large such that $n<p+m$. In this situation,  the LRT is not well defined due to the  singularity of the matrix $S_E$.  This excludes the  LRT from many high-dimensional applications with $p>n$ or $m>n$  \cite[e.g.,][]{donoho2000high,fan2014challenges}. When $m>n$,  the  linear hypothesis testing problem has been studied in depth for specific  submodels such as one-way MANOVA  \cite[][etc.]{srivastava2006multivariate,hu2017testing,zhou2017highgeneral,cai2014highmanova}. 
More generally, \cite{li2018highgeneral} recently proposed a modified LRT for the general linear hypothesis tests via spectral shrinkage. However, these  works assume that  $p$ is  fixed.

This paper aims to address  the above problems. 
First, under the null hypothesis, we derive the asymptotic boundary  when the $\chi^2$   approximation fails as the dimension parameters $(p,m,r)$ increase with the sample size $n$. Moreover, we develop the corrected limiting distribution of $ \log L_n$ for a general asymptotic regime of $(p,m,r,n)$. 
 Second, under alternative hypotheses, we characterize  the statistical power of     $\log L_n$ in the high-dimensional setting. 
 Using a technique of analyzing partial differential equations induced by the test statistic, we show that the LRT is powerful when the trace of the signal matrix  $(CB)\Sigma^{-1}(CB)^\intercal $ is large while it  would lose power under   a low-rank signal matrix. With unknown alternatives in practice, we propose   an enhanced likelihood ratio  test such that it is also powerful against low-rank alternative signal matrices.
The power-enhanced test statistic  combines the LRT statistic and the largest eigenvalue \citep{johnstone2008,johnstone2009} to further improve the test power against low-rank alternatives. 
Third, when $n<p$ and the LRT is not well defined, we propose to use a two-step testing procedure by first performing  dimension  reduction to covariates and responses and then using the proposed (enhanced) LRT. 
  To control the estimation error induced by the dimension reduction  in the first step, we employ a {\it repeated} data-splitting approach and show the asymptotic type \RNum{1} error is well controlled under the null hypothesis. Simulation results further confirm the good performance of the proposed approach.  
 

The rest of the paper is organized as follows.  
In Section \ref{sec:nlargerpm}, we examine when the classic LRT fails under the null hypothesis and propose a corrected limiting distribution of $\log L_n$. In Section \ref{sec:althyp}, we further analyze the power of $\log L_n$ and propose the power-enhanced test statistic. In Section \ref{sec:dimred}, we discuss the multi-split LRT procedure when $n<p$.  Simulation studies and a real dataset analysis on breast cancer are reported  in Sections \ref{sec:simulation} and \ref{sec:realdata} respectively. 


\section{When likelihood ratio test starts to fail?}\label{sec:nlargerpm}

  
 In traditional multivariate  regression analysis where the dimension parameters $(p,m,r)$ are considered as fixed numbers,   the $\chi^2$ approximation of the   LRT,
\begin{eqnarray}
	-2\log L_n\xrightarrow{D}   \chi^2_{mr},\quad \mbox{ as } n\to \infty, \label{eq:chisqappro}
\end{eqnarray} is used   for $H_0: CB=\mathbf{0}_{r\times m}$  \citep[][]{muirhead2009aspects,anderson1958introduction}, where $\xrightarrow{D}$ denotes the convergence in distribution.  
However, it has been noted that the  $\chi^2$ approximation of the distribution of the LRT often performs poorly in high-dimensional applications \citep[see, e.g.,][]{bai1996effect,jiang2012likelihood,bai2009,bai2013testing,jiang2013}. 

As the three dimension parameters $(m,p,r)$ are allowed to grow with $n$, it is of interest to examine  the phase transition boundary where   the $\chi^2$  approximation fails.    
This is described in the following theorem. 

\begin{theorem}\label{thm:approxbound}
Consider  $n> p+m$ and $p\geq r$.  Let $\chi^2_{mr}(\alpha)$ denote the upper $\alpha$-quantile of $\chi^2_{mr}$ distribution. \\
(i) When $mr \to \infty$ and  $\max\{p,m,r\}/n\to 0$ as $n\to \infty$, $ P\{{-2 \log L_n}>\chi^2_{mr}(\alpha) \} \to \alpha,$ for any significance level $\alpha$, if and only if  \begin{align}
 	\lim_{n\to \infty} {\sqrt{mr}(p+m/2-r/2)}n^{-1}=0. \label{eq:limitboundrequire}
 \end{align}  (ii) When $mr$ is finite, $ P\{{-2 \log L_n}>\chi^2_{mr}(\alpha) \} \to \alpha,$ if and only if $\lim_{n\to \infty} p/n =0$.	
\end{theorem}

 Theorem \ref{thm:approxbound}  gives the  sufficient and necessary condition on  $(m,p,r,n)$ such that the $\chi^2$ approximation \eqref{eq:chisqappro} fails. 
We note that although \eqref{eq:limitboundrequire} is obtained when $mr\to \infty$,  \eqref{eq:limitboundrequire} becomes $\lim_{n\to \infty} p/n=0$ with finite $m$ and $r$, in agreement with the conclusion when $mr$ is finite. 
To further  examine the implications of  \eqref{eq:limitboundrequire}, we consider two special cases. Specifically, let  $m=\lfloor n^{\eta} \rfloor$ and $p=\lfloor n^{\epsilon} \rfloor$ with $\eta$ and $\epsilon\in (0,1)$, where $\lfloor \cdot \rfloor$ denotes the floor of a number. When $r$ is fixed,  \eqref{eq:limitboundrequire} implies 
$
	{\sqrt{m}(p+m/2)}=o(n), 
$ that is, 
$
	\max\{\epsilon,\eta\}+\eta/2<1. 
$ When $r=p=\lfloor n^{\epsilon} \rfloor$, \eqref{eq:limitboundrequire} implies  
$
	\sqrt{mp}(p+m)=o(n),
$  that is,
$
	\max\{\epsilon, \eta \}+(\eta+\epsilon)/2 <1.
$ For these two cases, we correspondingly give two  $(\eta, \epsilon)$-regions in Figure \ref{fig:areaplot} satisfying the constraint \eqref{eq:limitboundrequire}. In these two    regions, when $\epsilon$ becomes close to 0, the largest $\eta$  approaches  $2/3$. This implies that when $p$ is small, the largest $m$ such that \eqref{eq:limitboundrequire} holds is of order $n^{2/3}$. It is the same for both fixed $r$ and $r=p$ cases as $p$ is small and $r\leq p$. In addition, when  $\eta$ goes to 0, the largest $\epsilon$ values under fixed $r$ and $r=p$ cases converge to $1$ and $2/3$ respectively. This indicates that when $m$ is small, the largest $p$ values satisfying \eqref{eq:limitboundrequire} are of order $n$ and $n^{2/3}$ for the two cases respectively.   Moreover,  when $m=p$, the largest orders of $m$ and $p$ for the two cases are $n^{2/3}$ and $n^{1/2}$  respectively. 






\begin{figure}[!htbp]
\centering
\includegraphics[width=0.45\linewidth]{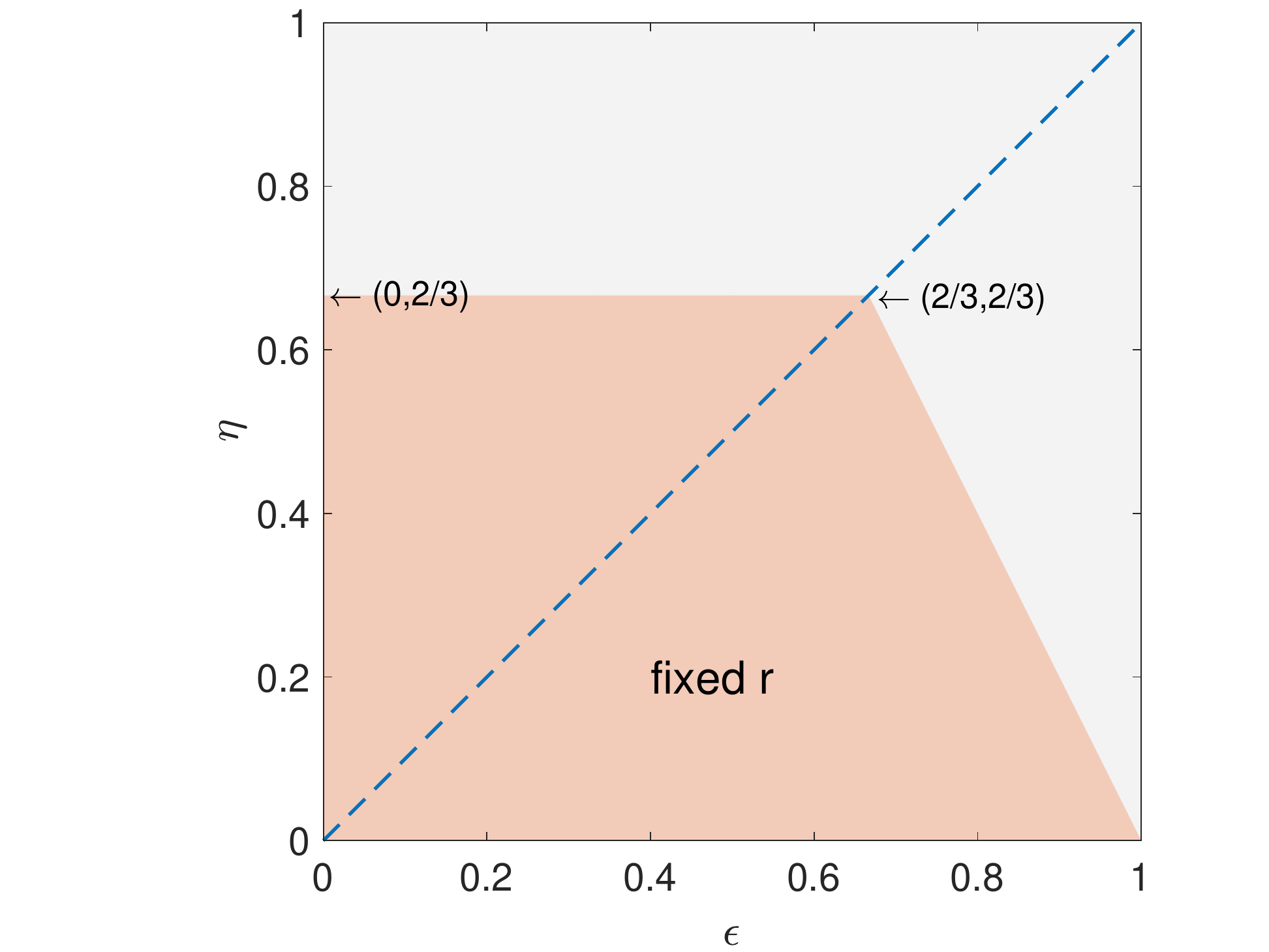} 
\includegraphics[width=0.45\linewidth]{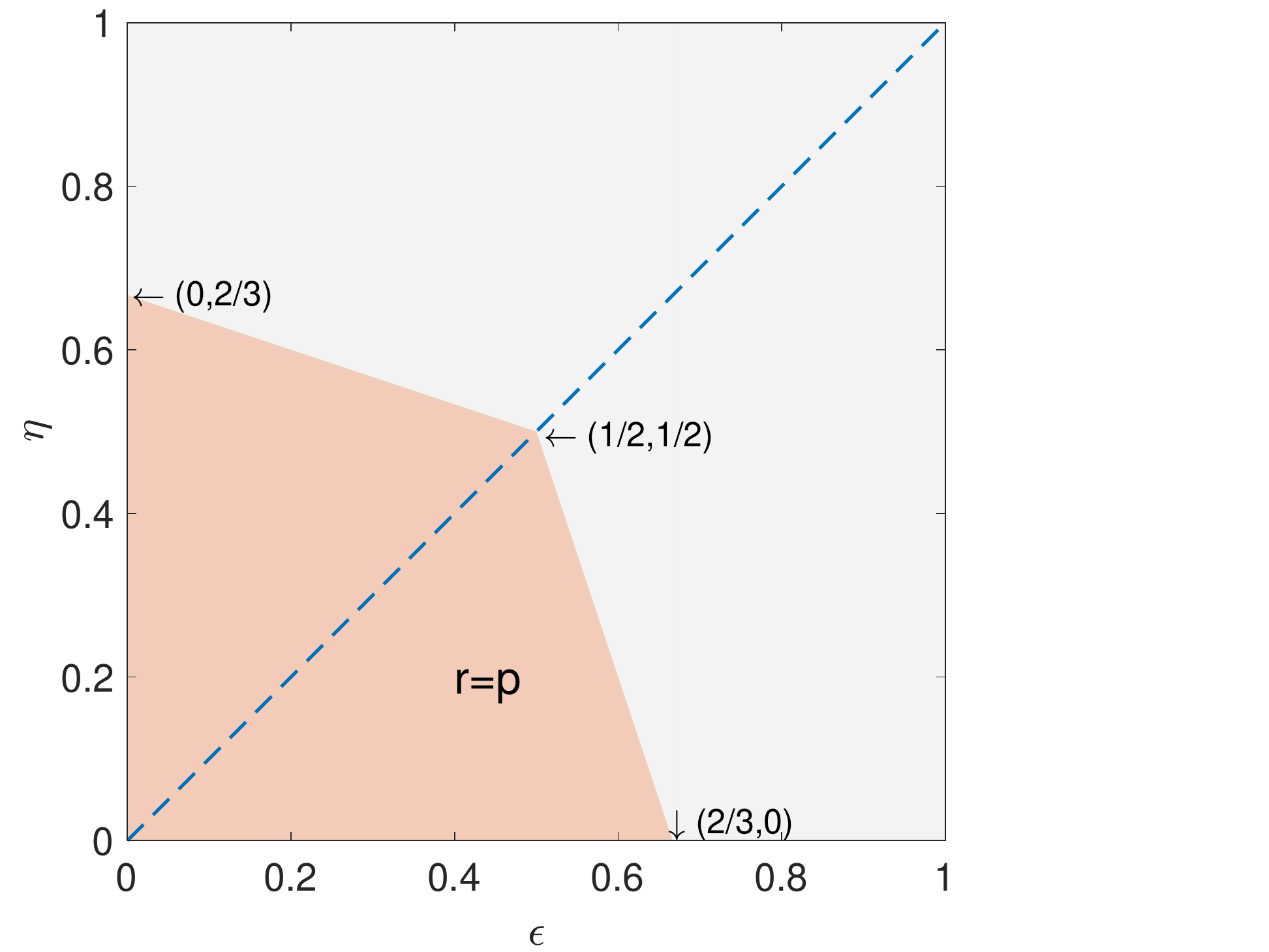}
\caption{ $\eta$ versus $\epsilon$ when $r$ is fixed and $r=p$}
\label{fig:areaplot}
\end{figure}

To illustrate this phase transition phenomenon, we present a simple simulation experiment. We set $\Sigma=I_m$, 
and estimate the type \RNum{1} errors of the $\chi^2$ approximation \eqref{eq:chisqappro}   with $10^4$ repetitions under the following four   cases:
$  (a) \mbox{~fixed~} m=r=2~\textrm{and} ~p=\lfloor n^{\eta} \rfloor ;  (b) \mbox{~fixed~} p=r=2 ~\textrm{and} ~m=\lfloor n^{\eta} \rfloor; (c) \mbox{~fixed~} m=2 ~\textrm{and} ~p=r=\lfloor n^{\eta}\rfloor;  (d) ~ p=m=r=\lfloor n^{\eta}\rfloor,
$
where $\eta\in\{1/24,\ldots, 23/24\}$. In Figure \ref{fig:nobartlettcorr}, we   plot the estimated type \RNum{1} errors against $\eta$ values for $n=100$ and $300$ respectively.  The plots show  consistent patterns with  the theoretical results. In particular, when $p=m=r=\lfloor n^{\eta} \rfloor$, the $\chi^2$ approximation  begins to fail for $\eta$ around  $1/2$. When $p$ and $r$ are fixed and $m=\lfloor n^{\eta} \rfloor$, or when $m$ is fixed and $p=r=\lfloor n^{\eta} \rfloor$, the $\chi^2$ approximation begins to fail for $\eta$ around  $2/3$. When $m$ and $r$ are fixed and $p=\lfloor n^{\eta}\rfloor$, the $\chi^2$ approximation begins to fail for  $\eta$ larger than the other three cases, which is consistent with the theoretical results.  



\begin{figure}[h!]
\centering
\begin{subfigure}{0.45\textwidth}
\centering
	\includegraphics[width=\textwidth]{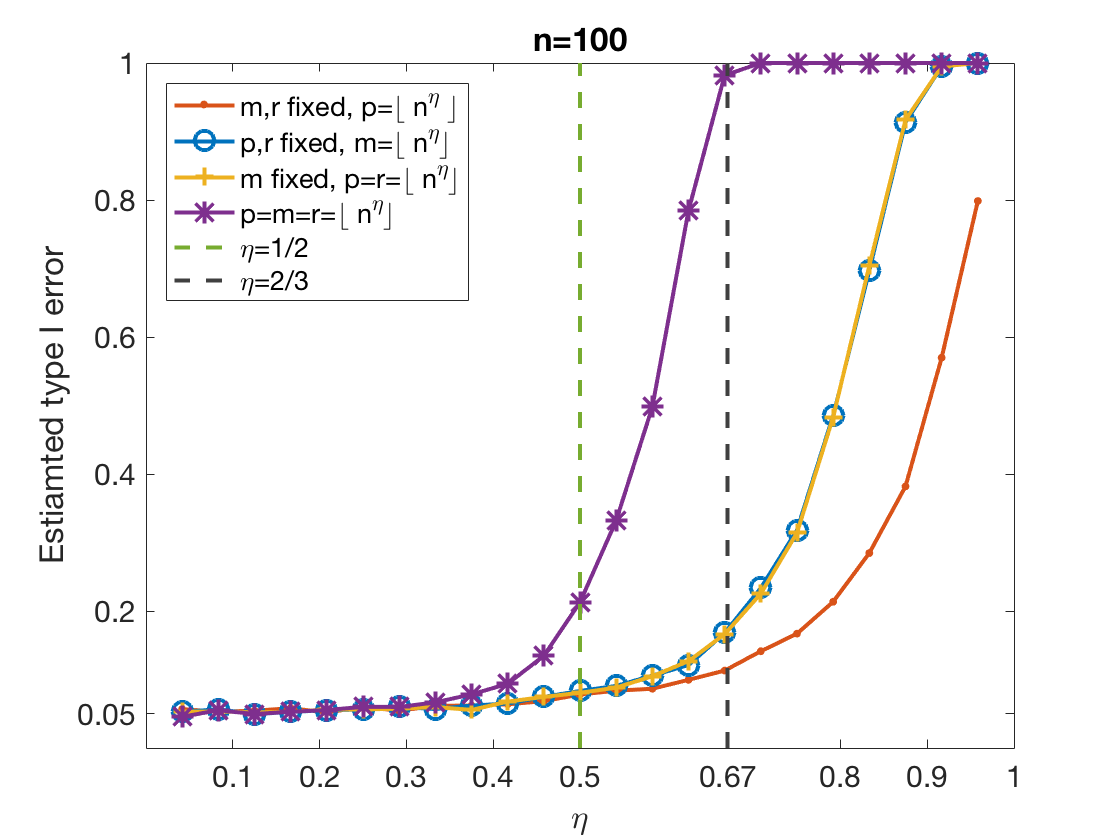}
	\end{subfigure}
\begin{subfigure}{0.45\textwidth}	
\centering
	\includegraphics[width=\textwidth]{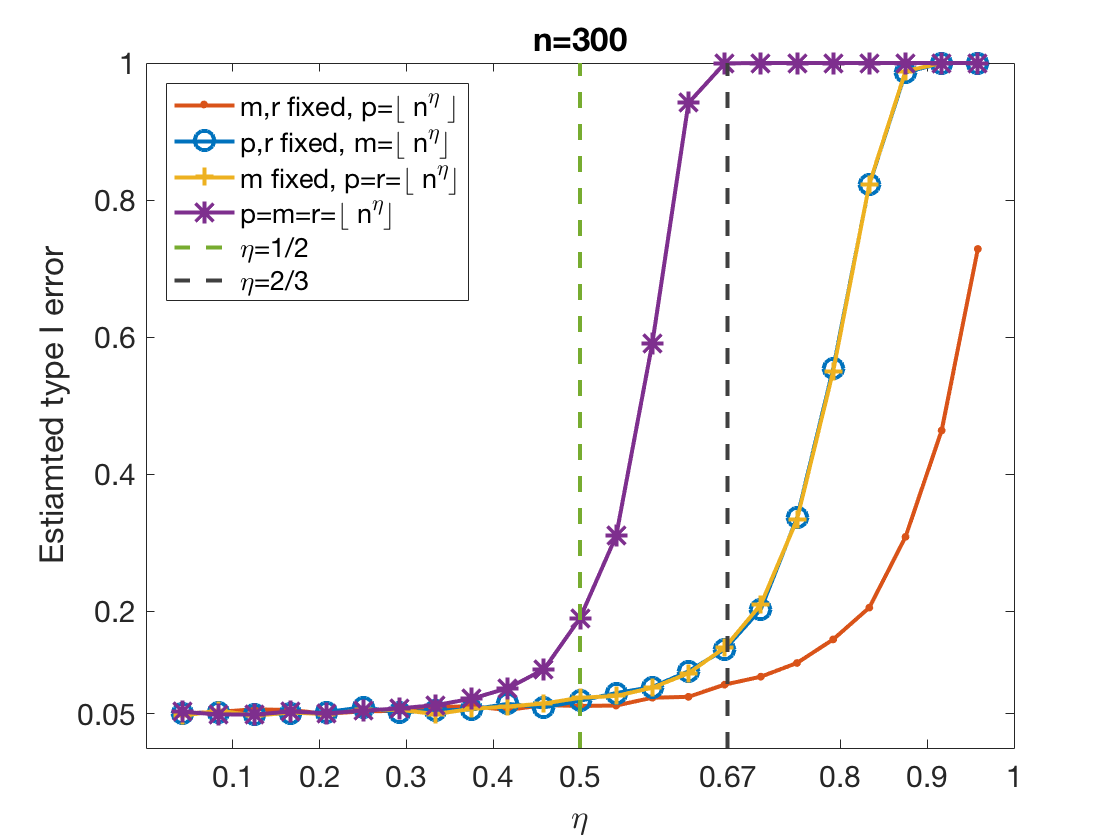}
\end{subfigure}
\caption{Estimated type \RNum{1} errors using $\chi^2$ approximation \eqref{eq:chisqappro}}
\label{fig:nobartlettcorr}
\end{figure}

It is worthy to mention that the sufficient and necessary  constraint  \eqref{eq:limitboundrequire} also characterizes the bias of the $\chi^2$ approximation. Specifically, under the conditions of Theorem \ref{thm:approxbound}, $\mathrm{E}(-2\log L_n-\chi^2_{mr})/\sqrt{\mathrm{var}(\chi^2_{mr})}=\sqrt{mr}(p+m/2-r/2+1/2)n^{-1}\{1+o(1)\}$. Thus when $(p,m,r)$ are large such that \eqref{eq:limitboundrequire} is violated and the $\chi^2$ approximation fails,  the bias of the $\chi^2$ approximation increases  as ${\sqrt{mr}(p+m/2-r/2+1/2)}n^{-1}$ increases. This can be seen in Figure \ref{fig:nobartlettcorr}  and is supported by simulations  in Section \ref{sec:simulation}.

In the classic regime with fixed $m$ and $p$,  researchers have also proposed  the Bartlett correction of the LRT      that $-2\rho\log L_n \xrightarrow{D}\chi_{mr}^2$, where   $\rho=1-(p-r/2+m/2+1/2)/n$.
  In particular, for any $z\in \mathbb{R}$, this corrected      approximation gets rid of the first order approximation error  $O(n^{-1})$; that is, for any $z$,  $P(-2\rho\log L_n <z)- P(\chi_{mr}^2<z) =O(n^{-2})$ when  $m$ and $p$ are fixed. Similarly to Theorem 1,      the   $\chi^2$  approximation with Bartlett correction  also fails as $m$ and $p$  increase with $n$.  The phase transition boundary is characterized in the following result. 
\begin{theorem}\label{prop:correctrho}
Consider  $n> p+m$ and $p\geq r$. \\
(i) When $mr \to \infty$ and  $\max\{p,m,r\}/n\to 0$ as $n\to \infty$, $ P\{{-2\rho \log L_n}>\chi^2_{mr}(\alpha) \} \to \alpha,$ for any significance level $\alpha$, if and only if $\lim_{n\to \infty} \sqrt{mr}(r^2+m^2)n^{-2}=0.$\\
(ii) When $mr$ is finite, $ P\{{-2\rho \log L_n}>\chi^2_{mr}(\alpha) \} \to \alpha,$ if and only if $n-p \to \infty$.
\end{theorem}
 

Theorem \ref{prop:correctrho} suggests that when $m$ and $r$ are fixed, the  corrected LRT approximation holds  when $n-p\to\infty$.
When $mr\to\infty$, the phase transition threshold in Theorem   \ref{prop:correctrho} only involves $m$ and $r$. In particular, when $r$ is fixed and $m=\lfloor n^\eta \rfloor$, or when $m$ is fixed  and $r=\lfloor n^{\eta}\rfloor$, the $\chi^2$  approximation with Bartlett correction fails when $\eta\geq 4/5$;   when $m=r=\lfloor n^{\eta}\rfloor$, the corrected approximation fails when $\eta\geq 2/3$. 


To illustrate the phenomenon, we also present a numerical experiment on the $\chi^2$  approximation with Bartlett correction in Figure  \ref{fig:withbartlettcorr} under the same set-up as in Figure \ref{fig:nobartlettcorr}.  
It shows that when $m$ and $r$ are fixed and $p=\lfloor n^{\eta}\rfloor$, the type \RNum{1} errors are well controlled for large $\eta$ approaching 1.  
Moreover, when $p$ and $r$ are fixed and $m=\lfloor n^{\eta} \rfloor$, or when $m$ is fixed and $p=r=\lfloor n^{\eta}\rfloor$, the  corrected $\chi^2$ approximation begins to fail around $\eta=4/5$. When $p=m=r=\lfloor n^{\eta}\rfloor$,  the  corrected $\chi^2$ approximation begins to fail around $\eta=2/3$. These  numerical results are also consistent with the theory.

\begin{figure}[!ht]
\centering
\begin{subfigure}{0.45\textwidth}	
\centering
	\includegraphics[width=\textwidth]{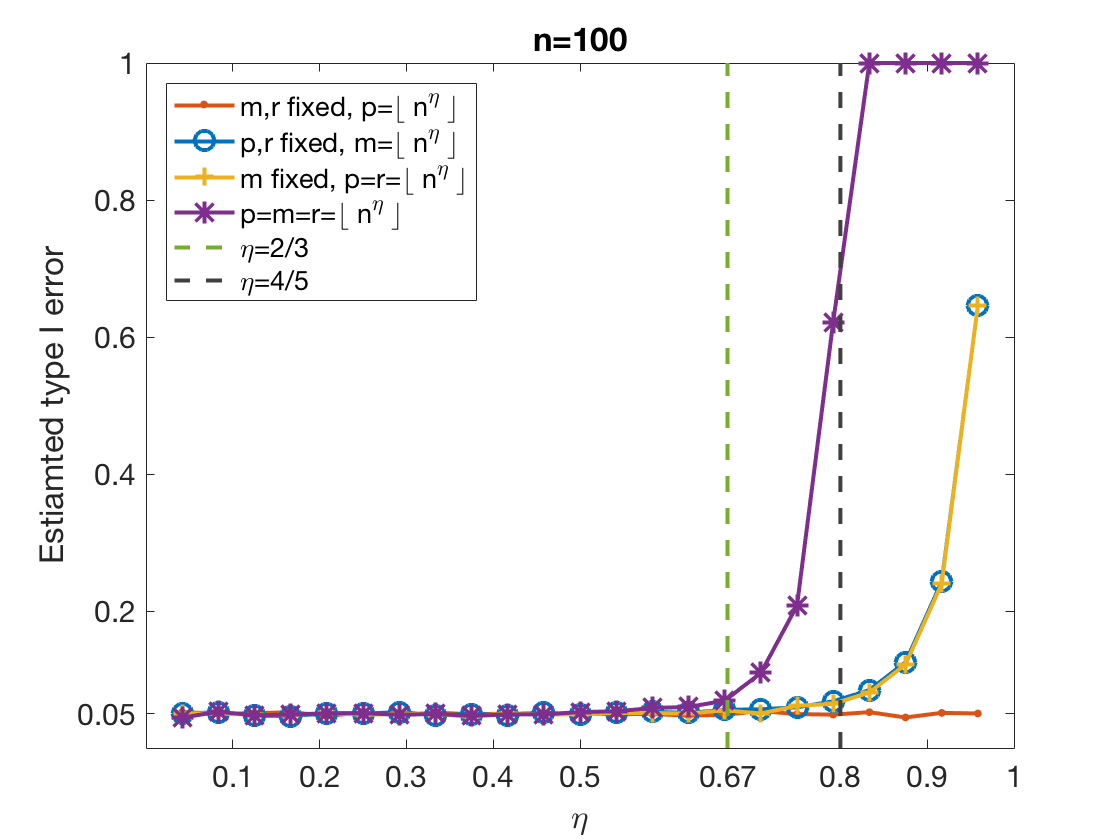}
	
\end{subfigure}
\begin{subfigure}{0.45\textwidth}	
\centering
	\includegraphics[width=\textwidth]{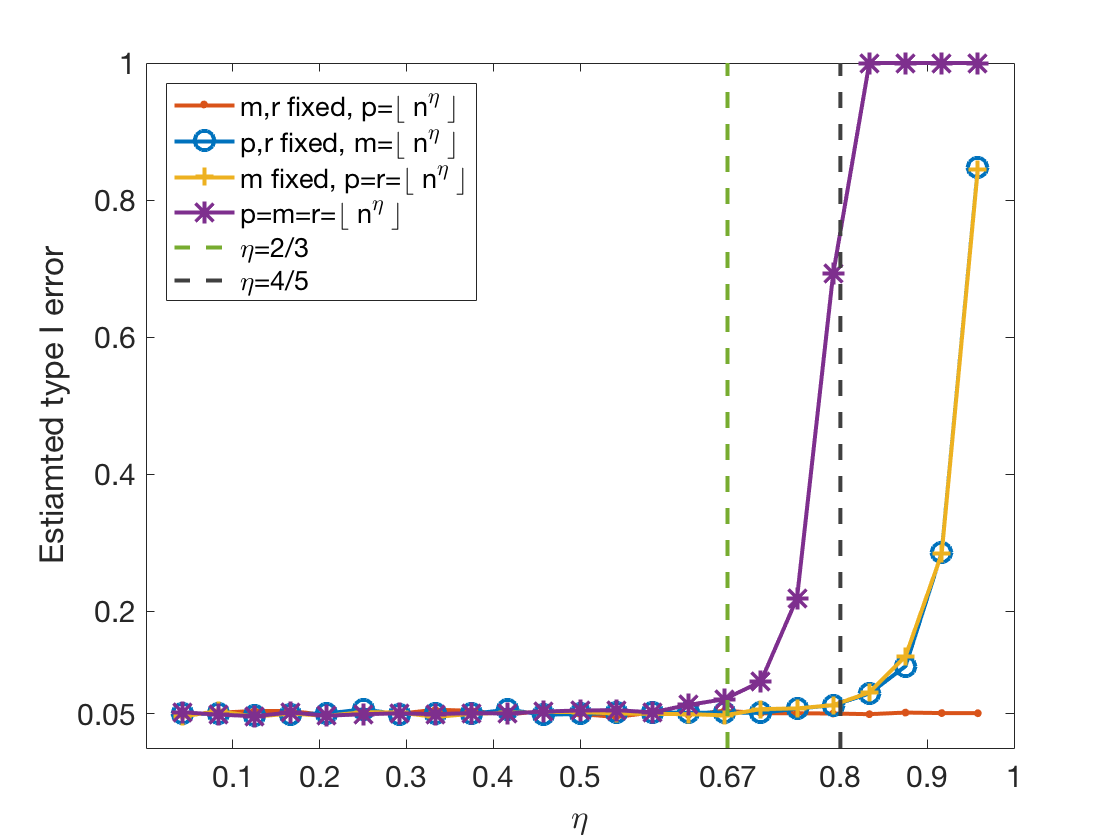}
	
\end{subfigure}
\caption{Estimated type \RNum{1} error using $\chi^2$ approximation with Bartlett correction}
\label{fig:withbartlettcorr}	
\end{figure}


More generally, to have a unified limiting distribution for analyzing high-dimensional data under a general asymptotic region of $(m,p,r,n)$,   we derive a corrected normal limiting distribution for the LRT  statistic. 



\begin{theorem} \label{thm:mainlimit}
	When $n> p+m$, $p\geq r$, $mr \to \infty$, and $n-p-\max\{m-r,0\}\to \infty$  as $n\to\infty$, the LRT  statistic $L_n$ has corrected form $T_1$  satisfying
	\begin{eqnarray}
		T_1:=\frac{-2\log L_n +\mu_n}{n\sigma_n}\xrightarrow{D} \mathcal{N}(0,1), \label{eq:maintestLn}
	\end{eqnarray}
where $\sigma^2_n= 2\log(n+r-p-m)(n-p) -2\log (n-p-m)(n+r-p) $, and
\begin{align*}
\mu_n
=~& {n(n-m-p-1/2)}
\log\frac{(n+r-p-m)(n-p)}{(n-p-m)(n+r-p)}+ {nr}\log\frac{(n+r-p-m)}{(n+r-p)}
\notag\\
~&+ {nm}\log\frac{(n-p)}{(n+r-p)}.
\end{align*}
\end{theorem}

 The theorem above covers the asymptotic regime where  $mr \to \infty$, $\max\{p,m,r\}/n \to 0$ and the constraint \eqref{eq:limitboundrequire} holds; under this region, we can show that 
 $\mu_n \rightarrow -mr$ and $(n\sigma_n)^2\rightarrow 2mr$, which are  consistent with the mean and variance of  $\chi^2_{mr}$ approximation. In addition, although Theorem  \ref{thm:mainlimit} requires $mr\to \infty$, the normal approximation \eqref{eq:maintestLn} could still perform well when $m$ or $r$ is small,  as long as $mr$ is large enough. The simulations in Section \ref{sec:simulation}      show that the performance of the $\chi^2$ and normal  approximations can be  similar in low dimensions. 
 
   
 
 Alternatively, under some high dimensional settings,  we can check that no $\chi^2$ or even noncentral $\chi^2$ distribution could match the asymptotic mean and variance of $-2\log L_n$ in Theorem \ref{thm:mainlimit}. Specifically, if the distribution of $-2\log L_n$ could be approximated by some $\chi^2$ distribution, then we should have $-(n\sigma_n)^2/\mu_n \to 2$, which is, however, not satisfied as $p/n,m/n$ and $r/n$ increase; or if the  distribution of $-2\log L_n$ could be approximated by some  noncentral $\chi^2$ distribution with degrees of freedom $k_n$, then we should have  $k_n=-2\mu_n-n^2\sigma_n^2/2$, which, however, can become negative as $p/n, m/n$ and $r/n$ increase.   Thus it implies that the $\chi^2$-type  approximation for $-2\log L_n$ could fail  fundamentally under high dimensions.


 \begin{remark}
 A similar result on the asymptotic normality of $\log L_n$ in Theorem \ref{thm:mainlimit} was proved in \cite{zheng2012central} and \cite{bai2013testing}.
 However, there are several differences between our result and theirs.
 First, our asymptotic regime is more general. Specifically,  \cite{zheng2012central} and  \cite{bai2013testing} requires that $m<r$,  $\min\{m,r\} \to \infty$,  and  $m/(n-p)$ converges to a constant in $(0,1)$, while we only need $mr\to \infty$  and $n-p-\max\{m-r,0\}\to \infty$.  Our analysis covers the case when $m/(n-p)\to 0$ and even when the limit does not exists.
 Second,  the proofs of \cite{zheng2012central} and \cite{bai2013testing} are based on the  random matrix theory, while  we prove Theorem \ref{thm:mainlimit}  by a moment generating function  technique motivated by \cite{jiang2013}. 
\end{remark}

\section{Power analysis and an enhanced likelihood ratio test}\label{sec:althyp}

Although the limiting behaviors of LRTs for high dimensional data have been recently explored under different testing problems, 
the power of LRTs is less studied and remains a challenging problem, as discussed  in \cite{jiang2013}.  In this section, we focus on the high-dimensional multivariate linear regression
and analyze  the  power of the LRT statistic. Moreover, based on the theoretical result, we   propose a power-enhanced testing approach to further improve the power of the LRT.

To examine the power of LRT statistic, we introduce the classic canonical form of the LRT problem, which writes  $H_0: CB=\mathbf{0}$ to an equivalent form as follows  \citep{muirhead2009aspects}. Specifically, consider the matrix decomposition  $X=O[I_p,\mathbf{0}_{p\times (n-p)}]^{\intercal}D$, where $O$ is an $n\times n$ orthogonal matrix and $D$ is a $p\times p$ nonsingular real matrix. Given $D$, we have similar decomposition  $CD^{-1}=\mathbb{E}[I_r, \mathbf{0}_{r\times (p-r)}]V$, where $\mathbb{E}$ is an $r\times r$ nonsingular matrix and $V$ is a $p\times p$ orthogonal matrix. This gives that $CB=CD^{-1}DB=\mathbb{E}[I_r,\mathbf{0}_{r\times (p-r)}]VDB$, and therefore $H_0: CB=\mathbf{0}_{r\times m}$  is equivalent to $M_1=\mathbf{0}_{r\times m}$,  where we define $M_1=[I_r,\mathbf{0}_{r\times (p-r)}]VDB=\mathbb{E}^{-1}CB$.

We next describe  the relationship between $M_1$ and the LRT statistic through a linear transformation of $Y$. Let $V_1$ denote the first $r$ rows of $V$.  Define $Y_1^*=[V_1, \mathbf{0}_{r\times (n-p)}]O^{\intercal}Y$ and $Y_2^*=[\mathbf{0}_{(n-p)\times p}, I_{n-p}]O^{\intercal}Y$. 
We then know that    ${Y_1^*}^{\intercal}Y_1^*=S_X$ and ${Y_2^*}^{\intercal}Y_2^*=S_E$.
We further define $\tilde{S}_X=\Sigma^{-1/2}S_X\Sigma^{-1/2}$, $\tilde{S}_E=\Sigma^{-1/2}S_E\Sigma^{-1/2}$, and $\Omega=\Sigma^{-1/2}M_1^{\intercal}M_1\Sigma^{-1/2}$. Then we can write the LRT statistic $-2\log L_n  =  {n} \sum_{i=1}^{\min\{m,r\}}  \log (1+\lambda_i)$, where $\lambda$'s are the eigenvalues of $\tilde{S}_E^{-1}\tilde{S}_X$. Given the fact that   $\mathrm{E}(\tilde{S}_E^{-1}\tilde{S}_X)=(rI_m+\Omega)/(n-p) 
 $ \citep{muirhead2009aspects},  it is then expected that the power of LRT   depends  on  an averaged  effect of all the eigenvalues of $\Omega$.

 We focus on the alternatives where the signal matrix $\Omega$ is of low rank and $(p,m,r)$ increase proportionally with $n$. In particular, we assume $\Omega$ has a fixed rank $m_0$, and write $\Omega=n\Delta$, where $\Delta$ has fixed nonzero  eigenvalues $\delta_1,\ldots,\delta_{m_0}$. Note that this is reasonable when the entries in $M_1\Sigma^{-1/2}$ are $O(1)$, as the entries in $\Omega$ could be $O(n)$ with $r$ proportional to $n$. The following theorem specifies how the    power of the LRT statistic $T_1$ depends on the eigenvalues of $\Omega$.


\begin{theorem}\label{thm:poweranalysis}
Consider the setting where $(p,m,r)$ increase proportionally with $n$, and  $p/n=\rho_p $, $m/n=\rho_m $, and $r/n= \rho_r$, where $\rho_p,\rho_m,\rho_r \in(0,1)$ are fixed constants and $\rho_p+\rho_m<1$.  Given $\Delta=\Omega/n$ with fixed nonzero eigenvalues $\delta_1,\ldots,\delta_{m_0}$, define $W_{\Delta}=\sum_{j=1}^{m_0} \log [1+\delta_j (1+\rho_r-\rho_p)^{-1}]$. There exists a constant ${A}_1>0$ such that 
$
	P( T_1>z_{\alpha})\to 1- {\Phi} (z_{\alpha}- A_1 W_{\Delta}), 
$ where $\Phi(\cdot)$  and  $z_{\alpha}$   denote the cumulative distribution function and the upper $\alpha$-quantile of $\mathcal{N}(0,1)$, respectively. 
\end{theorem}

Theorem \ref{thm:poweranalysis} establishes the relationship between the eigenvalues of $\Omega$  and the power of $T_1$ under the high-dimensional and low-rank signals. 
It implies that when $W_{\Delta}$ is large, $T_1$ has high power. Alternatively,   the LRT could be highly underpowered when $W_{\Delta}$ is small.
  As in real applications, the  truth is usually unknown, it is desired  to have a testing procedure with high statistical power against various alternatives. 
  
To enhance the power of LRT, we propose to combine it with the  Roy's   test  statistic that is
based on the largest eigenvalue of $S_E^{-1}S_X$   \citep{roy1953}. In particular, \cite{johnstone2008,johnstone2009}   extended Roy's test to  high dimensional  settings and proposed the largest eigenvalue test statistic 
$
	 T_2 = {[\log \{ \theta_{n,1}/ (1-\theta_{n,1})\} -\tilde{\mu}_n]}/{\tilde{\sigma}_n}, 
$ where $\theta_{n,1}=\lambda_{\max}\{(S_E+S_X)^{-1}S_E\}$ with $\lambda_{\max}(\cdot)$ denoting the largest eigenvalue, and $\tilde{\mu}_n=2\log \tan \{( {\phi+\gamma}/{2})\}$ and $\tilde{\sigma}_n^3=16(n-p+r-1)^{-2}\{\sin^2(\phi+\gamma)\sin \phi \sin \gamma \}^{-1}$ 
with $\sin^2(\gamma/2)=\{\min(m,r)-1/2\}/(n-p+r-1)$ and $\sin^2 (\phi/2)=\{\max(m,r)-1/2\}/(n-p+r-1)$.
Moreover, \cite{johnstone2008} proved that under the high-dimensional  null hypothesis,
$T_2\xrightarrow{D}\mathcal{TW},$ 
where $\mathcal{TW}$ denotes a Tracy-Widom distribution. Under the alternative hypothesis, \cite{dharmawansa2014local} further studied the spiked alternative with $\Omega=r UHU^{\intercal}$, where $U$ is an $m\times m_0$ matrix with orthonormal columns and fixed $m_0$, and $H=\mathrm{diag}(h_1,\ldots,h_{m_0})$ with $h_1>\ldots>h_{m_0}$. They showed that the phase transition threshold for $h$'s is a constant that depends on the  limit of $(p/n, m/n, r/n)$. Note that with fixed $r/n$, there exists a constant $c_2>0$ such that $\delta_1=c_2 h_1$.  This implies that when $\delta_1$ is a constant large enough, the power of $T_2$ can converge to 1, while the LRT statistic $T_1$ may only have   power smaller than 1 by Theorem \ref{thm:poweranalysis}.
On the other hand, when $\delta_1$ is below the phase transition threshold,     $T_1$ could be more powerful than $T_2$.


We therefore propose a combined test statistic $T_3=T_1+T_2 *I(T_2\geq F_n ),$
where $F_n$ is a positive constant. 
With  properly chosen $F_n$, the proposed test statistic $T_3$ could enhance the power of $T_1$ under alternative hypotheses while $T_3 \xrightarrow{D} \mathcal{N}(0,1)$ under $H_0$. 
Specifically,   under the null hypothesis, the type \RNum{1} error rate of  $T_3$ is   controlled if  $P\{T_2\geq F_n\}\to 0$. 
 On the other hand, under  alternative hypotheses, we have $P(T_3 >z_{\alpha})\geq P(T_1>z_{\alpha})$ as $T_2*I\{T_2>F_n\}\geq 0$ for $F_n>0$.   This guarantees that the power of $T_3$ is at least  as large as that of the LRT statistic $T_1$.  Moreover, consider the case when  $W_{\Delta}$ is relatively small but $\delta_1$ is significantly above the phase transition threshold, where $T_2$ is more powerful than $T_1$; then if   $F_n$   does not grow too quickly,   $T_3$ would also be powerful. Thus we can choose $F_n$ to be a slow varying function and the combined test statistic $T_3$ could improve the  power of $T_1$  with little size distortion. 
Through extensive simulation studies, we find $F(n)=\max\{\log \log n,2\}$  exhibits good performance; please see Section \ref{sec:simulation}.

\section{Likelihood ratio test when $p>n$} \label{sec:dimred}

When the number of predictors is large  such that $p>n$, $S_E$ becomes singular, and the test statistics $T_1$, $T_2$ and $T_3$ can not be directly applied. To deal with this issue, we propose a multiple data splitting procedure,  which  repeatedly   splits the data into two random subsets. We use the first subset   to  perform dimension reduction and obtain a manageable size of predictors, and then  apply  the proposed LRT on the second subset. The test statistics from different data splittings  are   aggregated together to provide the final test statistic.  The random splits of data ensure correct size control of the test type \RNum{1} error. Similar ideas are   used in   other high-dimensional problems \citep[etc.]{meinshausen2009p,berk2013}.
 We next describe the proposed procedure.

Consider the setting when $p>n$ and $m<n$. Denote $B=[\mathbf{b}_1,\ldots, \mathbf{b}_p]^{\intercal}$ and  $\mathcal{M}_{*}=\{ k : \mathbf{b}_k \neq \mathbf{0}, 1 \leq k \leq p \}$.  We assume the ``sparsity" structure that the responses   only depend on a subset of the predictors (or transformed predictors) such that   $n>m+|\mathcal{M}_{*}|$. Let $X_{\mathcal{M}_*}$ be the $n\times |\mathcal{M}_*|$  submatrix of $X$ with columns indexed by $\mathcal{M}_{*}$  and $B_{\mathcal{M}_{*}}$ be the $ |\mathcal{M}_*|\times m$ submatrix of $B$ with rows indexed by $\mathcal{M}_{*}$. The underlying model then satisfies
$
	Y=X_{\mathcal{M}_*}B_{\mathcal{M}_{*}}+E.
$   
Under this model,  for any subset $\mathcal{M}\subseteq \{1,\ldots,p\}$ such that $\mathcal{M}\supseteq \mathcal{M}_*$ and $n>m+|\mathcal{M}|$,  testing $CB=\mathbf{0}$ is equivalent to   $C_{\mathcal{M}}B_{\mathcal{M}}=\mathbf{0}$ and the LRT  is then applicable. Here $C_{\mathcal{M}}$ denotes the $r\times |\mathcal{M}|$  submatrix of $C$ with columns indexed by $\mathcal{M}$, and $B_{\mathcal{M}}$ denotes the $|\mathcal{M}|\times m$ submatrix of $B$ with rows indexed by $\mathcal{M}$.

To obtain such a set $\mathcal{M}\supseteq \mathcal{M}_*$, we propose a screening method for the multivariate linear regression. The seminal work of \cite{fan2008sure} first introduced a sure independence screening procedure  that can significantly reduce the number of predictors while preserving the true linear model with an overwhelming probability. This procedure  has been further  extended in various settings  \cite[e.g.,][]{fan2010sureindepglmnp,wang2016high,barut2016cond}. However, many of the existing works target at the settings with univariate response variable.


 

To utilize the joint  information from multivariate response variables, we propose a screening method that selects the columns of $X$ by their  canonical correlations  with $Y$. The canonical correlation  is a widely used dimension reduction criterion inferring information from cross-covariance matrices in multivariate analysis \citep{muirhead2009aspects}.    Specifically, for each column vector    $\mathbf{x}^j=(x_{1,j},\ldots, x_{n,j})^{\intercal}$,    $j=1,\ldots, p$, we first compute its canonical correlation with $Y$ denoted by 
\begin{eqnarray*}
	\omega_j=\max_{\mathbf{a}\in \mathbb{R}^m}\frac{\mathbf{a}^{\intercal}(Y-\mathbf{1}_n\bar{Y})^{\intercal}  (\mathbf{x}^j-\bar{x}^j\mathbf{1}_n)}{ \sqrt{\{\mathbf{a}^{\intercal}(Y-\mathbf{1}_n\bar{Y})^{\intercal}(Y-\mathbf{1}_n\bar{Y})\mathbf{a} \} \times \{(\mathbf{x}^j-\bar{x}^j\mathbf{1}_n)^{\intercal}(\mathbf{x}^j-\bar{x}^j\mathbf{1}_n) \} }  },
\end{eqnarray*} where $\bar{x}^j=\sum_{i=1}^n x_{i,j}/n$,  $\bar{Y}$ is the row mean vector of $Y$, and $\mathbf{1}_n$ is an all 1 column vector of length $n$.  Then for $0<\delta<1$, we select  $\lfloor\delta p\rfloor$ columns of $X$ with the highest canonical  correlations with $Y$, and define the selected column set as
$\mathcal{M}_{\delta}= \{j: |\omega_j|$   is  among  the largest   $\lfloor\delta p \rfloor$  of all,  $1\leq j\leq p  \}$.
In practice, we  choose an integer $\lfloor\delta p\rfloor$ such that $n_T>\lfloor\delta p\rfloor+m$ to apply the LRT. On the other hand, we keep $\lfloor\delta p\rfloor$ large to increase  the probability of     $\mathcal{M}_{\delta}\supseteq{\mathcal{M}_{*}}$.The following theoretical result provides the desired screening property that $P(\mathcal{M}_*\subseteq \mathcal{M}_{\delta} ) \to 1$ for properly chosen $\delta$.

\begin{theorem}\label{thm:selectionaccuracy}
Under Conditions 1-3 given in Supplementary Material Section \ref{sec:screenthmcondition},  for some constant $c_0>0$,
$P(\mathcal{M}_*\subseteq \mathcal{M}_{\delta} )=1-O[\exp\{-c_0n^{1-\iota}/\log n \}],$
where the constant $\iota<1$ is defined in Conditions 3. 	
\end{theorem}


\begin{remark}

 When testing the coefficients of the first $r$ predictors of $X$, 
such as $C=[I_r,\mathbf{0}_{r\times (p-r)}]$, we can keep the first $r$ predictors, denoted by $X_1$, in the model, while  screen the remaining predictors, denoted by $X_2$. In particular, we can   apply the screening procedure on  the residuals $\tilde{R}$, from the regression of $Y$ on $X_1$,  and $X_2$. More generally  when  $C$ is a matrix of rank $r$, we can use this conditional screening procedure through a linear transformation of data. In particular, given the singular value decomposition $C=UVD^{\intercal}$, we  can transform $X$ and $B$ into  $\tilde{X}=XD$ and $\tilde{B}=D^{\intercal}B$. Then testing $H_0: CB=\mathbf{0}_{r\times m}$ is equivalent to testing $H_0: [I_r,\mathbf{0}_{r\times (p-r)}]\tilde{B}=\mathbf{0}_{r\times m}$ under the model of the transformed data  $Y=\tilde{X}\tilde{B}+E$. The theoretical result similar to Theorem \ref{thm:selectionaccuracy} can be obtained under properly adjusted assumptions. 
\end{remark}

\begin{remark} \label{rm:screenother}
 The proposed procedure uses the canonical correlation, which is an extension of the marginal correlation in \cite{fan2008sure}. 
The computation of the canonical correlations is  fast and pre-implemented in many  softwares. Moreover, the proposed method aggregates the joint information of the response variables, and  thus could be better than simply applying the marginal screening with respect to each response variable. 
On the other hand, it was pointed out in the literature that the correlation-based method might have potential issues when predictors are highly correlated \citep{dutta2017note}. 
To study the effect of highly correlated predictors, we performed a preliminary simulation study  in Supplementary Material Section \ref{sec:comparelasso} and  compared our method with the method of using lasso with cross-validation to select predictors, which is expected to account for the dependence in the predictors while not for the   dependence in the responses.  
Under the considered settings with correlated predictors, our method performs better than the lasso.
The comparison results also show that both over and under selection of predictors can cause substantial loss of test power. To further enhance test power,  we may extend existing high-dimensional screening methods such as \cite{wang2016high} to the multivariate regression setting, so that the dependences in both predictors and responses are taken into account, and we leave this interesting topic for future study. 
\end{remark}

Given a proper screening approach, we propose a data splitting      procedure to apply the LRT.  We  randomly split $n$ observations into two independent sets -- the screening data $\{X_S,Y_S \}$  of size  $n_S$  and the testing data $\{X_T,Y_T\}$  of size $n_T$.  We use  $\{X_S,Y_S \}$ to  select  $\mathcal{M}$  and  apply the proposed LRT  to    $\{X_T,Y_T\}$ with the selected predictors in  $\mathcal{M}$.     
 Data splitting avoids the   influence of the screening step on the inference step   and provides valid inference  as  widely recognized in the  literature  \citep{berk2013,taylor2015statistical}.
 We also  demonstrate that the type I error rate can not be controlled  without splitting the data  by the simulation studies in Section \ref{sec:simulation}.   


It is known that the testing result from  a single random split is sensitive to the arbitrary split choice and would  be difficult to reproduce the result \citep{meinshausen2009p,meinshausen2010stability}. 
Therefore we further  propose to use multiple splits and aggregate the obtained results.  Note that computing test statistics by splitting data can be viewed as a resampling method.  Since resampling methods usually do not perform well for approximating statistics depending on the eigenvalues of high-dimensional random matrices  \citep{karoui2016bootstrap}, and the test statistics computed after splitting data are correlated, it is challenging to combine the    statistics in a valid and efficient method.

In this paper, we adopt the general $p$-value  combination method proposed by \cite{meinshausen2009p}. Specifically, we randomly split data $J$ times, and compute all the $J$ $p$-values for different splits. For each $j=1,\ldots,J$, we compute $p$-value $p^{(j)}$ with data splitting. Then for $\gamma\in (0,1)$, define $Q(\gamma)=\min\{ 1, q_{\gamma}( \{p^{(j)}/\gamma; j=1,\ldots, J\}) \},
 $
where $q_{\gamma}$ denotes the empirical $\gamma$ quantile function. As a proper selection of $\gamma$ may be difficult, we use the adaptive version below. Let $\gamma_{\min}\in (0,1)$ be a lower bound for $\gamma$, and define the adjusted $p$-value $p_t$ as $p_t=\min \{1, (1-\log \gamma_{\min} )\inf_{\gamma \in (\gamma_{\min},1)} Q(\gamma)\}. $
The extra correction factor $1-\log \gamma_{\min}$ ensures the type \RNum{1} error controlled despite the adaptive search for the best quantile.  For the adaptive  multi-split adjusted $p$-value $p_t$,   the null hypothesis is rejected  when  $p_t<\alpha$, where $\alpha$ is the prespecified threshold. Following the proof of Theorem 3.2 in \cite{meinshausen2009p}, we have the proposition below.


\begin{proposition} \label{prop:sizecontrol}
Under $H_0$,  for any $J$ random sample splits, if  Theorem \ref{thm:selectionaccuracy} holds for each split, then 
$
	\limsup_{n\to \infty} P(p_t \leq \alpha ) \leq \alpha.
$
\end{proposition}

Proposition \ref{prop:sizecontrol}   shows that
the multi-split and aggregation procedure can control the type \RNum{1} error.  
To apply the multi-split procedure, we need to choose  two parameters $J$ and $\gamma_{\min}$. In practice, we choose $J$ that is slightly large and of the same order of $n$. 
  We next discuss the choice of $\gamma_{\min}$. 
   To improve the test power, we want to choose $\gamma_{\min}$ such that  $\limsup_{n\to \infty} P(p_t \leq \alpha )$ in Proposition \ref{prop:sizecontrol} is maximized to be close to $\alpha$ under $H_0$.   By the proof of Proposition \ref{prop:sizecontrol},   it suffices to  make  $\mbox{argmax}_{\gamma\in(0,1)} P\{ Q(\gamma)\leq \alpha \}\in (\gamma_{\min},1)$, as   the adaptive search of $\gamma$ in $p_t$ is adjusted by the correction factor $1-\log \gamma_{\min}$. 
Note that $\{Q(\gamma)\leq \alpha \}$ is equivalent to $\{\psi(\alpha \gamma)\geq \gamma \}$ with $\psi(u)=J^{-1}\sum_{j=1}^J \mathrm{1}\{p^{(j)} \leq u\}$.
It is then equivalent to find the $\gamma$ value such that $P\{\psi(\alpha \gamma)\geq \gamma \}$ is the closest to the upper bound  $\alpha$. To evaluate this,
we consider two extreme cases with a given $J$. When $p^{(j)}$'s are highly dependent, $P\{ \psi(\alpha \gamma)\geq \gamma\} \simeq P\{p^{(1)} \leq \alpha \gamma \}=\alpha \gamma$, which approaches $\alpha$ when $\gamma$ is close to 1.  When $p^{(j)}$'s are nearly independent, $J\gamma \leq 1$ and  $\alpha \gamma$ is small, 
$P\{ \psi(\alpha \gamma)\geq \gamma\}\simeq P\{\min_{j}p^{(j)}\leq \alpha \gamma\} \simeq 1-(1-\alpha\gamma)^J \simeq J\alpha \gamma$; then $J\alpha \gamma \to \alpha$ if $\gamma \to J^{-1}$. When the dependence among $p^{(j)}$'s is between these two extreme cases, 
 we expect that the maximum of  $P\{ \psi(\alpha \gamma)\geq \gamma \}$ would be achieved at some $\gamma \in [J^{-1},1)$. Since the true correlation is unknown in practice, 
we recommend to take $\gamma_{\min}$ slightly smaller than $J^{-1}$ in the simulations so that the candidate $\gamma$ range contains $[J^{-1},1)$.
We further conduct a simulation study to illustrate how the value of $P\{ \psi(\alpha \gamma) \geq \gamma \}$ depends on the correlations of the $p$-values. The result is in the Supplementary Material Section \ref{sec:simugammavalue}, which is consistent with the theoretical analysis here.

We give a summary of the whole testing procedure when  $p$ is large.  

\noindent \textbf{Procedure} \ \ 
For $j=1,\ldots, J$, 
\begin{enumerate}
\item Randomly split data into screening dataset $\{X_S, Y_S \}$ and testing dataset $\{X_T, Y_T\}$.



          \item On $\{X_S, Y_S\}$: compute the canonical correlations between $Y_{S}$ and each column of $X_S$, then select the columns whose corresponding  correlations are the largest $\lfloor \delta p \rfloor$ ones. The selected column indexes  form a set $\mathcal{S_C} \subseteq  \{1,\ldots, p\}$. 
          \item On $\{X_T, Y_T\}$: choose the  columns of $X_T$ indexed by $\mathcal{S_C}$ to obtain $X_{\mathcal{\mathcal{S}_C}}$. Use $\{X_{\mathcal{S}_C}, Y_T\}$  to compute the  test statistic $T_{3}$   and  obtain the the $p$-value $p^{(j)}$.
     \end{enumerate}
After obtaining the set of $p$-values, $\{p^{(j)}: j=1,\ldots,J \}$, we compute the adjusted $p$-value $p_t$. Reject the null hypothesis if $p_t\leq \alpha$.  


\begin{remark}\label{rm:pcaony}
	When  the   dimension of response $Y$ is   large ($m>n$), we also need to reduce the dimension  of response vectors to apply the LRT. We can use a principal component analysis  (PCA) or factor analysis method to perform the dimension reduction.  
	In the simulation studies, we  select the first $m_0$ principal components of $Y_S$ as the columns of a matrix $\hat{W}$, where  $m_0$   satisfies $m_0+p<n_T$ and  can be  chosen by parallel analysis  \citep{buja1992remarks,dobriban2017deterministic}.  Then we transform the responses $Y_T$ in testing data to obtain $\hat{\mu}_T=Y_T\hat{W}$, which only has $m_0$ columns. We then use the transformed data $\{X_T, \hat{\mu}_T \}$ to examine   $CB\hat{W}=\mathbf{0}$. By the independence between the screening and testing datasets, the test is   valid. 
Under the sparse model setting,   the signal matrix $X_{\mathcal{M}_*}B_{\mathcal{M}_{*}}$ has a low rank decomposition  and we expect that the dimension reduction procedure  still maintains high power. 
This is verified by  simulation studies in Section \ref{sec:simulation}, which  show that applying dimension reduction on responses may even boost power for certain sparse models. 
Alternatively, other dimension reduction techniques can be applied \citep[e.g.,][]{yuan2007dimension,ma2013sparsepca}.
When both $m$ and $p$ are large,   we can simultaneously apply dimension reduction on $Y$   and   $X$ to reduce both $m$ and $p$. 
\end{remark}


\section{Simulations} \label{sec:simulation}

In this section, we report some simulation studies to evaluate the  theoretical results and proposed methods in this paper for $n>p+m$ and $n<p+m$ respectively. 

\subsection{When $n>p+m$}

\noindent When $n>p+m$, we conduct simulations under null and alternative hypotheses respectively to examine the type \RNum{1} error and power of our proposed test statistics.



In the first setting, we sample the test statistics by simulating data following the canonical form introduced in Section \ref{sec:althyp}. To be specific, we generate random matrices $Y_1^*$ of size $r\times m$ and $Y_2^*$ of size $(n-p)\times m$, where the rows of $Y_1^*$ and $Y_2^*$ are independent $m$-variate Gaussian with covariance $I_m$, and $E(Y_1^*) = M_1$ and $E(Y_2^*) = \mathbf{0}$. Under the canonical form, we know $H_0$ is equivalent to  $M_1=\mathbf{0}$ as discussed. In the following, each simulation is based on 10,000 replications with significance level $0.05$. 



Under the null hypothesis, we compare the traditional $\chi^2$  approximation \eqref{eq:chisqappro} with the  normal approximations for $T_1$  in \eqref{eq:maintestLn} and $T_3$. In particular, we study  how the dimension parameters $(p,m,r)$ influence these approximations by varying only one parameter each time. Figure \ref{fig:type1errort1norm} gives estimated type \RNum{1} errors as $p$ increases.   It shows that as   $p$  becomes larger, the $\chi^2$ approximation \eqref{eq:chisqappro} performs poorly, while the normal approximations for $T_1$ and $T_3$ still control the type \RNum{1} error  well.   
Other simulation  results with varying $m$ or $r$  are given in the supplement material Section \ref{sec:additiontypei}  and similar patterns are observed.

\begin{figure}[!htbp]
\centering
\begin{subfigure}{\textwidth}
\centering
\includegraphics[width=0.44\textwidth]{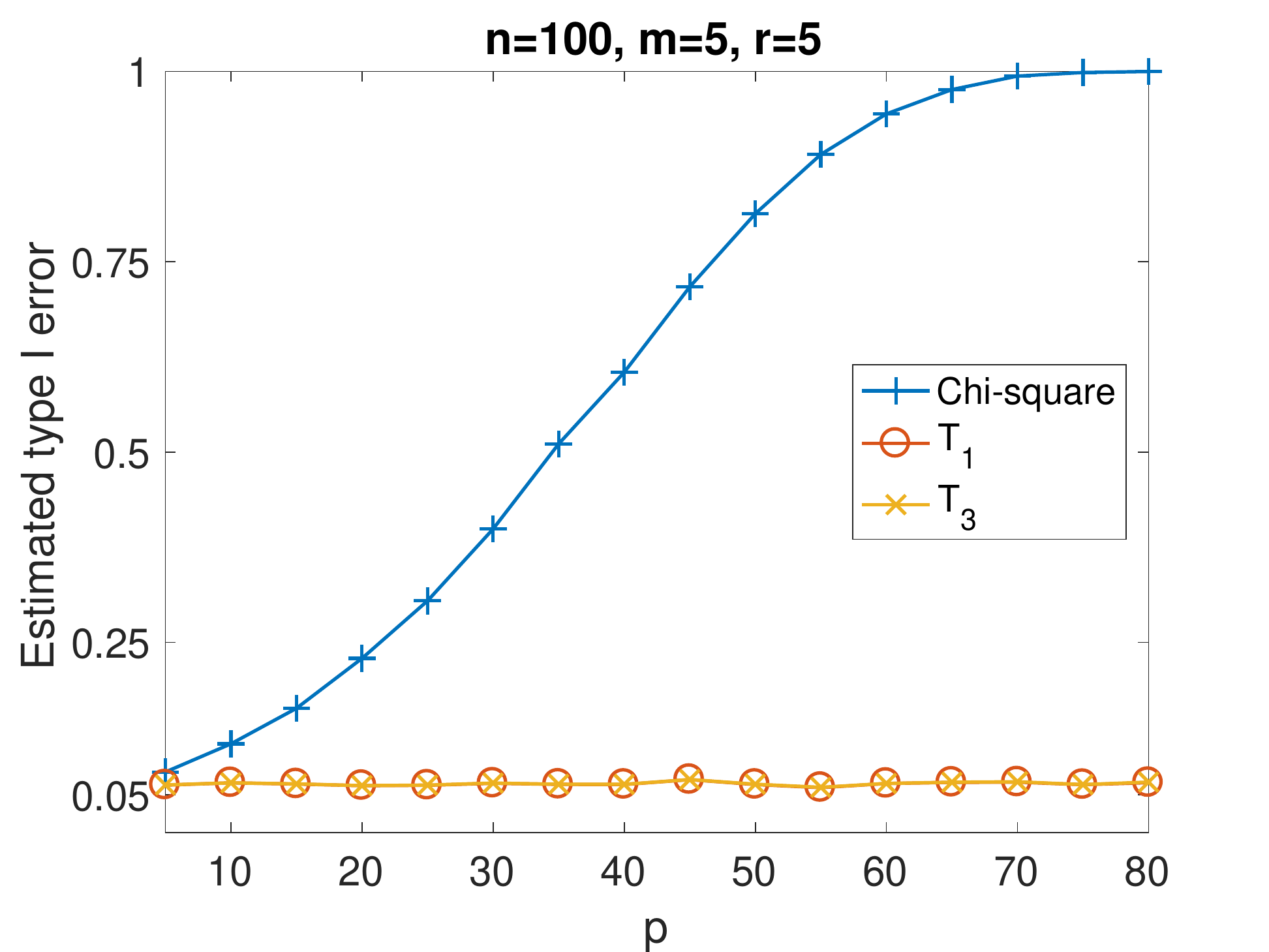}
\includegraphics[width=0.44\textwidth]{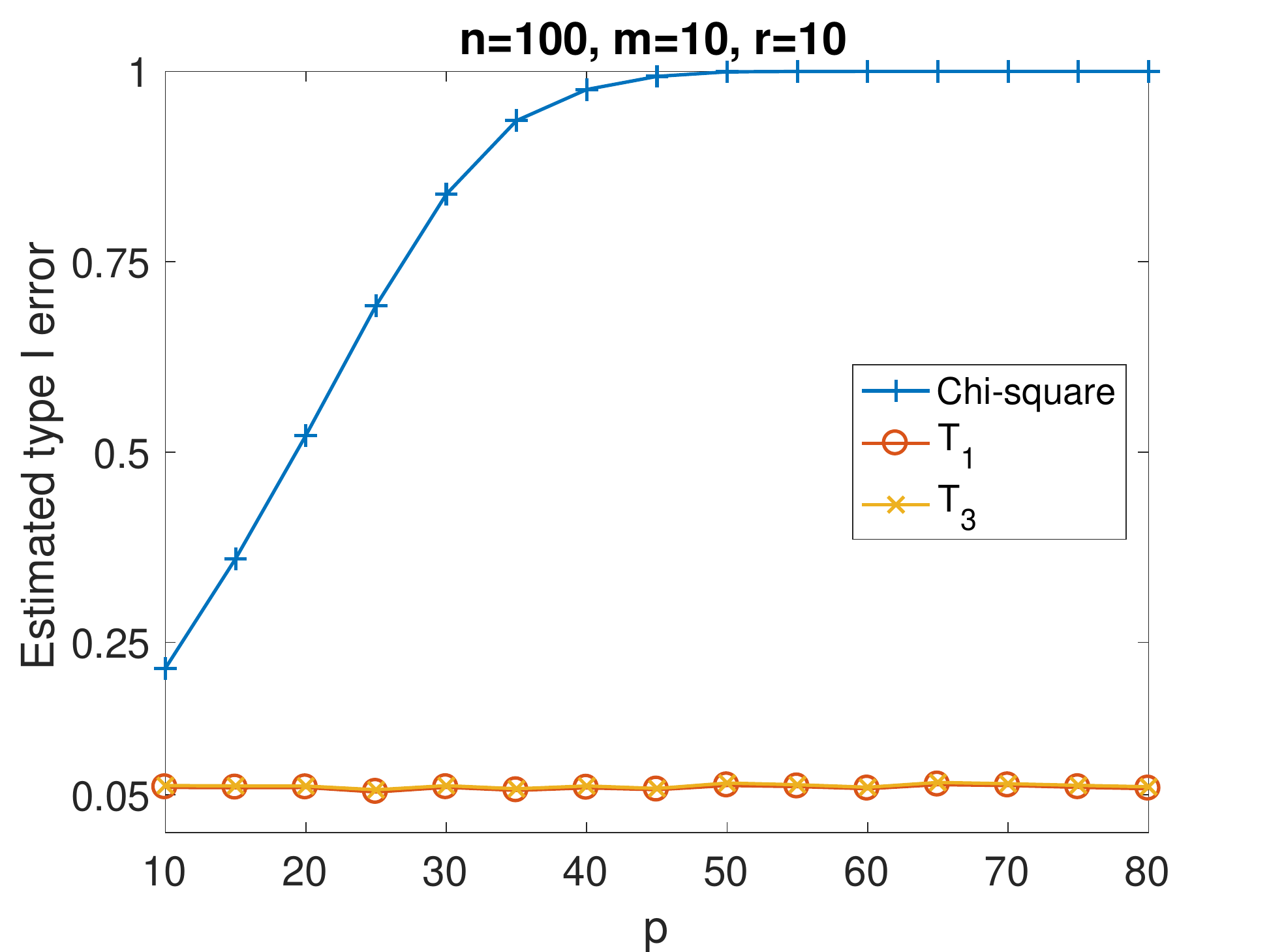}	
\end{subfigure}
\caption{Estimated type \RNum{1} error  versus $p$}
\label{fig:type1errort1norm}
\end{figure}

Under alternative hypotheses, we compare the power  of the test statistics $T_1$, $T_2$   and $T_3$ and  show the power improvement   of $T_3$ over $T_1$ and $T_2$. Specifically, under the canonical form, we simulate data with $M_1=\mathrm{diag}(\delta_1,\ldots, \delta_{r_k},0,\ldots,0)$, that is, a diagonal matrix with $r_k$ nonzero  elements.   It follows that $\Omega= \mathrm{diag}(\delta_1^2,\ldots, \delta_{r_k}^2, 0, \ldots, 0)$ has rank $r_k$. Under this setup, we test four different cases:  $(a)~  r_k=1; ~(b)~  r_k=2 ~\mathrm{and}~ \delta_1=\delta_2; ~(c)~ r_k=2 ~\mathrm{and}~  \delta_1=10\delta_2; ~(d)~  r_k=3~\mathrm{and}~ \delta_1=\delta_2=\delta_3;$ all under $n=100, m=20, p=50$ and $r=30$.  For each case, we plot estimated powers versus $\mathrm{tr}(\Omega)/m$ in Figure \ref{fig:powerthree}. 
Results in Figure \ref{fig:powerthree} show that when the rank of $\Omega$, $r_k$,  is small or the significant entries in $\Omega$ has low rank, $T_2$ is more powerful than $T_1$; and when $r_k$ or the rank of significant entries in $\Omega$ increases, $T_1$ becomes more powerful.   Moreover, in both sparse and non-sparse cases, the combined statistic $T_3$ has power close to the better one among $T_1$ and  $T_2$,  with the type \RNum{1} error well controlled. These patterns are consistent with  our theoretical power analysis in Section \ref{sec:althyp}.

\begin{figure}[h!]
\centering
\begin{subfigure}{\textwidth}
\centering
\includegraphics[width=0.44\textwidth]{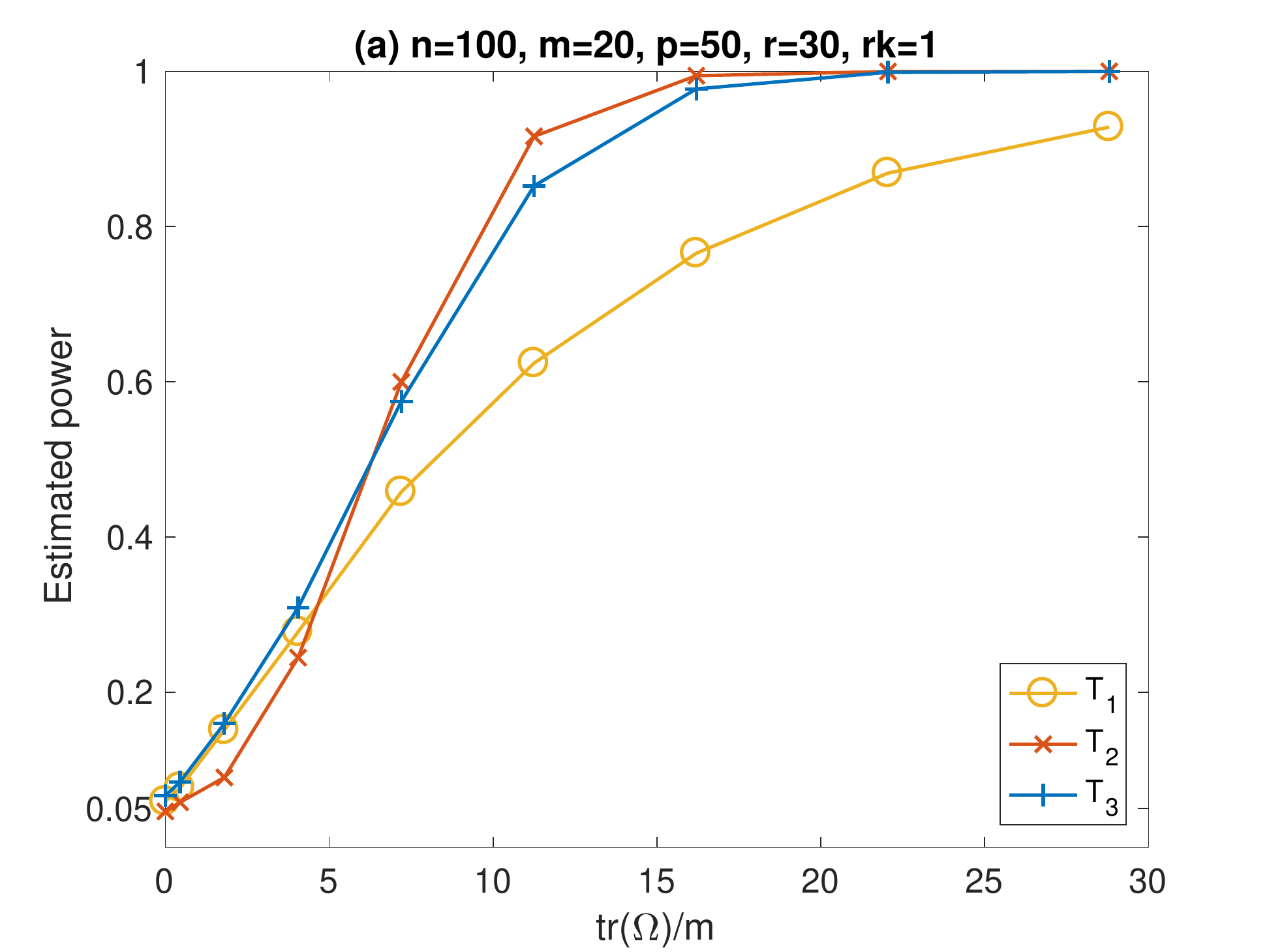}	
\includegraphics[width=0.44\textwidth]{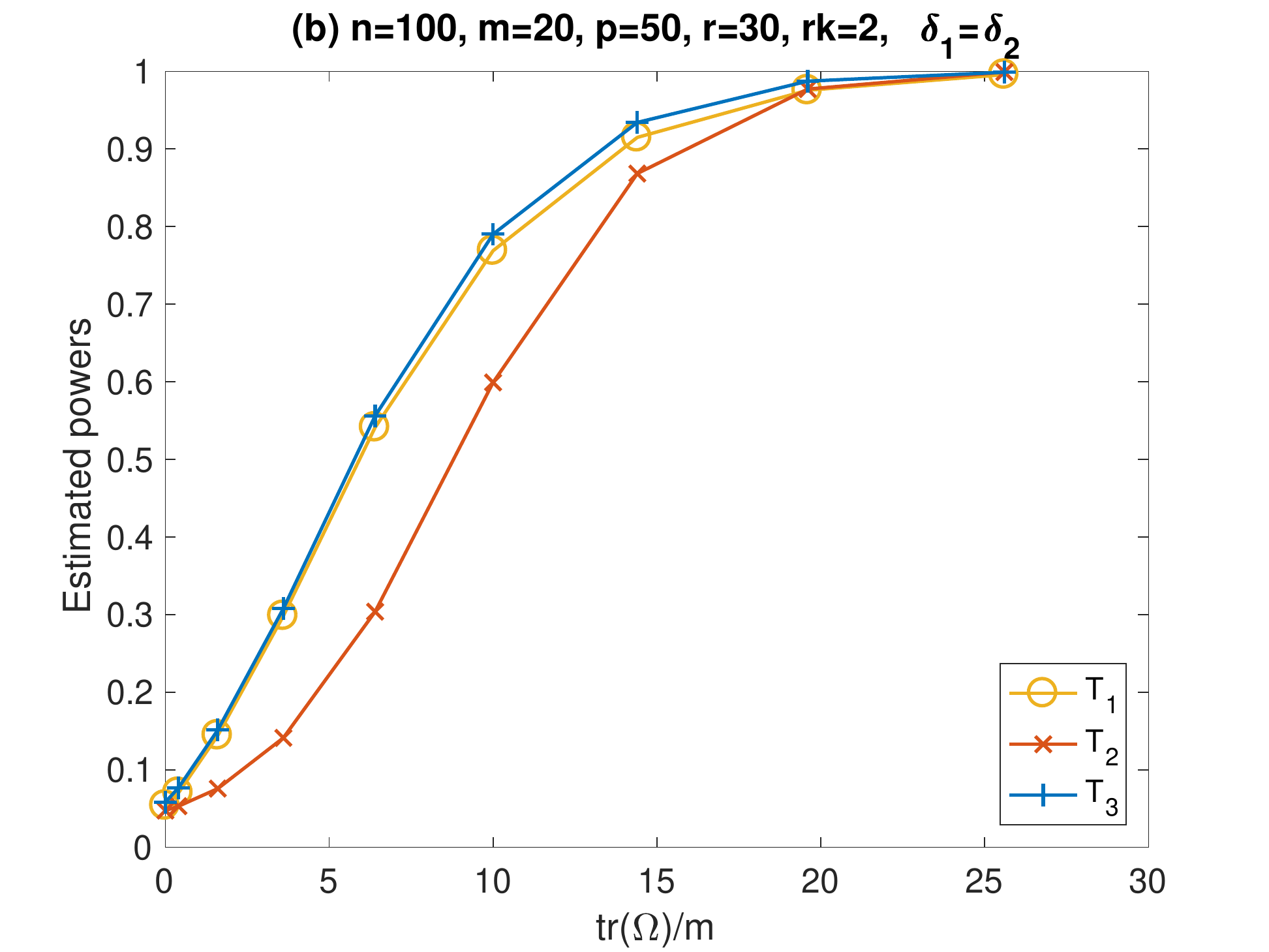}	
\end{subfigure}
\begin{subfigure}{\textwidth}
\centering
	\includegraphics[width=0.44\textwidth]{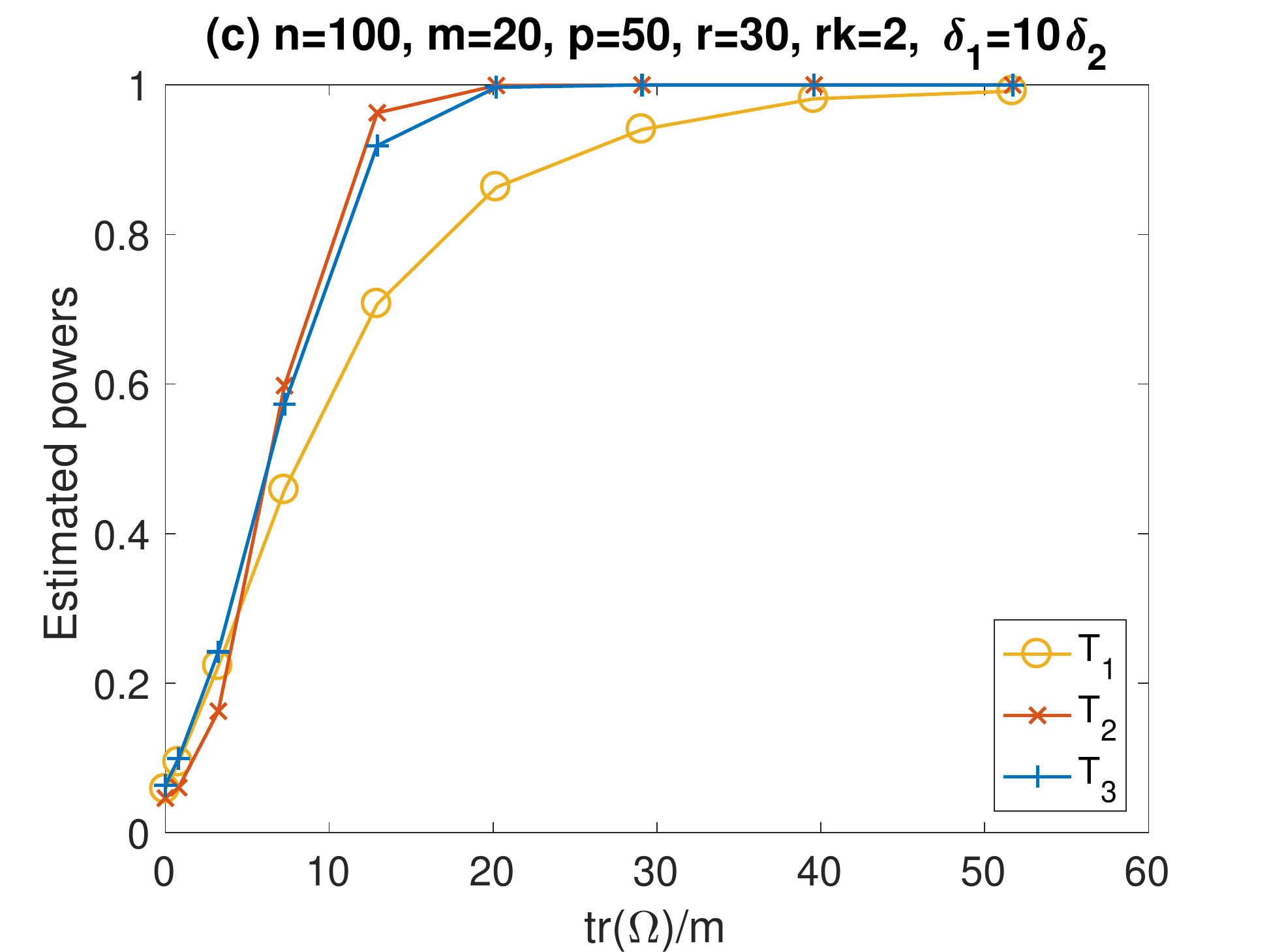}
	\includegraphics[width=0.44\textwidth]{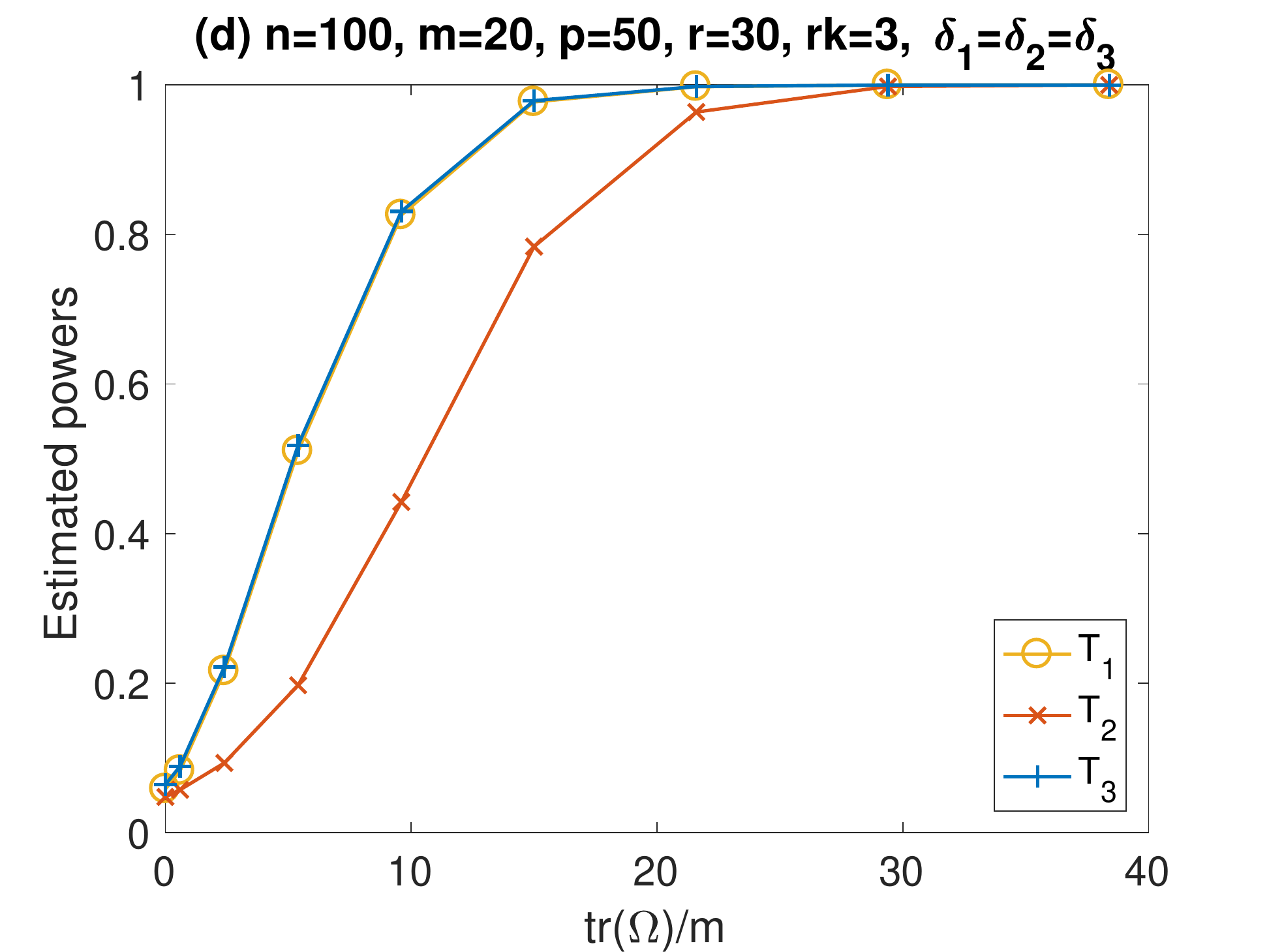}	
\end{subfigure}
\caption{Powers of $T_1$, $T_2$ and $T_3$ versus $\mathrm{tr}(\Omega)/m$}
\label{fig:powerthree}	
\end{figure}

In addition, we conduct simulations when $X$ and $Y$ are generated following $Y=XB+E$, where the rows of $E$ follow multivariate Gaussian distributions. The results are given in Supplementary Material Section \ref{sec:additionaltersimu}, and  show that $T_3$ is powerful under both dense and sparse $B$ cases. Moreover, we also conduct similar studies when $X$ and $Y$ take discrete values or the statistical error   follows a heavy-tail $t$ distribution. The results are provided in the Supplementary Material Section  \ref{sec:simulotherdist}. We observe similar patterns to the normal cases in Section \ref{sec:additionaltersimu}, which suggests  that the proposed test statistic is robust to the normal assumption of the statistical error.

\subsection{When $n<p+m$} \label{sec:simunsmallpm}


\noindent This section  studies the simulation settings  when $n<p+m$  and   evaluates the performance of our proposed procedure  in Section \ref{sec:dimred}. 
Specifically,  we take $C=[I_r,\mathbf{0}_{r\times (p-r)}]$  and $B$ to be a $p\times m$ diagonal matrix with  $\sigma_s$ in the first $r_k$ diagonal entries, where $\sigma_s$ represents the signal size that varies in simulations.  The rows of $X$ and $E$ are independent multivariate Gaussian with covariance matrices $\Sigma_x=(\rho^{|i-j|})_{p\times p}$ and $\Sigma=(\rho^{|i-j|})_{m\times m}$ respectively. We set $n=100$, $p=120$ and $r=120$, and test the cases when $m\in\{20, 120\}, r_k \in \{5, 10\}$ and $ \rho \in \{ 0.3, 0.7\}$.  
We conduct each simulation  with 200 replications, and split the data into screening and testing datasets with ratio 3:7 (the performances of 2:8 and 4:6 are similar in our  simulations).  
  Figure \ref{fig:powerestpowerscreen} reports   the simulation results  when $r_k=10$ while the other results are  presented in the Supplementary Material Section \ref{sec:simulationnsmallpmnormal}. In Figure \ref{fig:powerestpowerscreen}, ``screening" represents   the proposed screening procedure on $X$ (with $20\%$ features  selected);  ``PCA" represents   the principle component analysis on $Y$ as in Remark \ref{rm:pcaony}; $J$ represents the number of splits and $J=0$ represents testing on the same data without splitting.
\begin{figure}[!h]
\centering
\begin{subfigure}{\textwidth}
\centering	\includegraphics[width=0.44\textwidth]{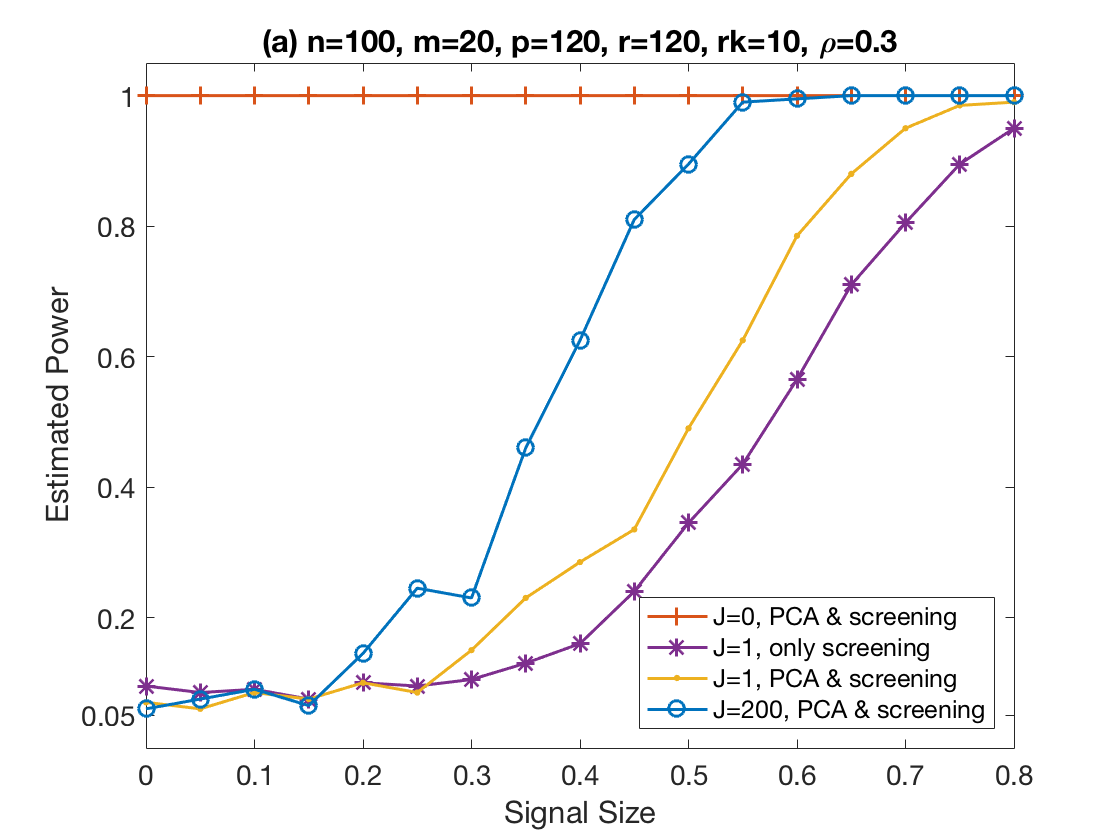}
	\includegraphics[width=0.44\textwidth]{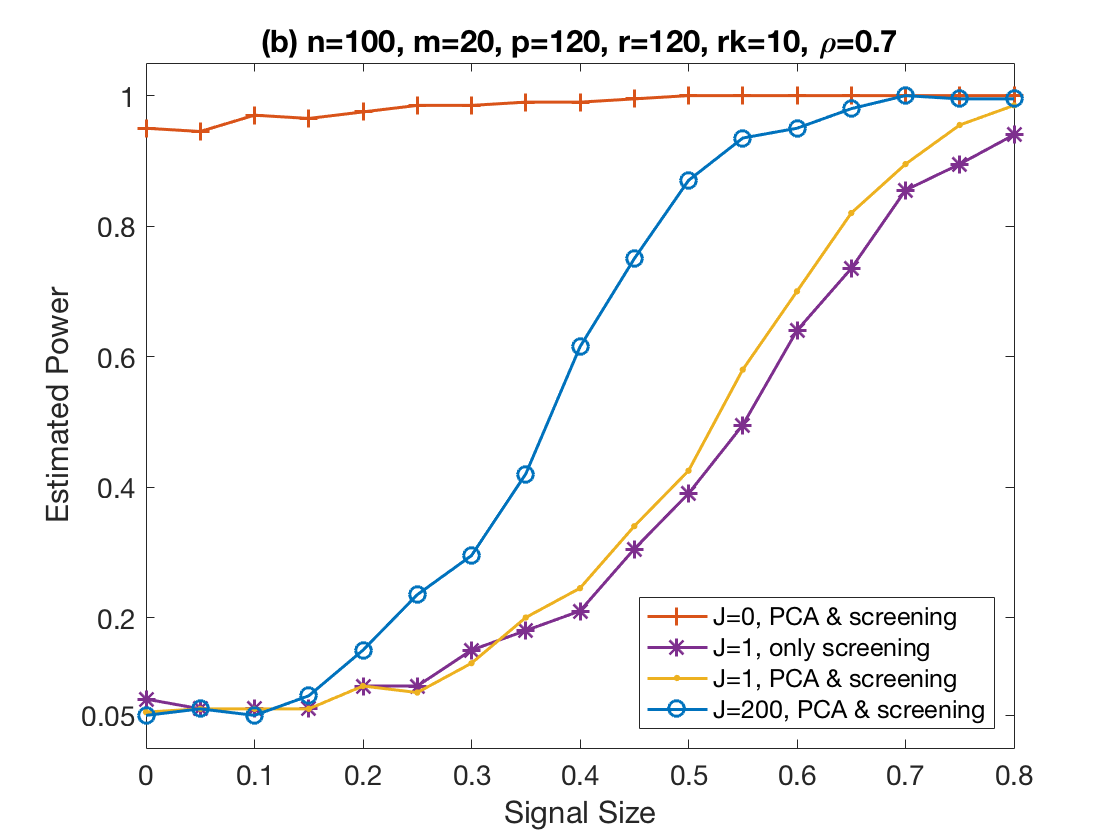}
\end{subfigure}
\begin{subfigure}{\textwidth}
\centering	\includegraphics[width=0.44\textwidth]{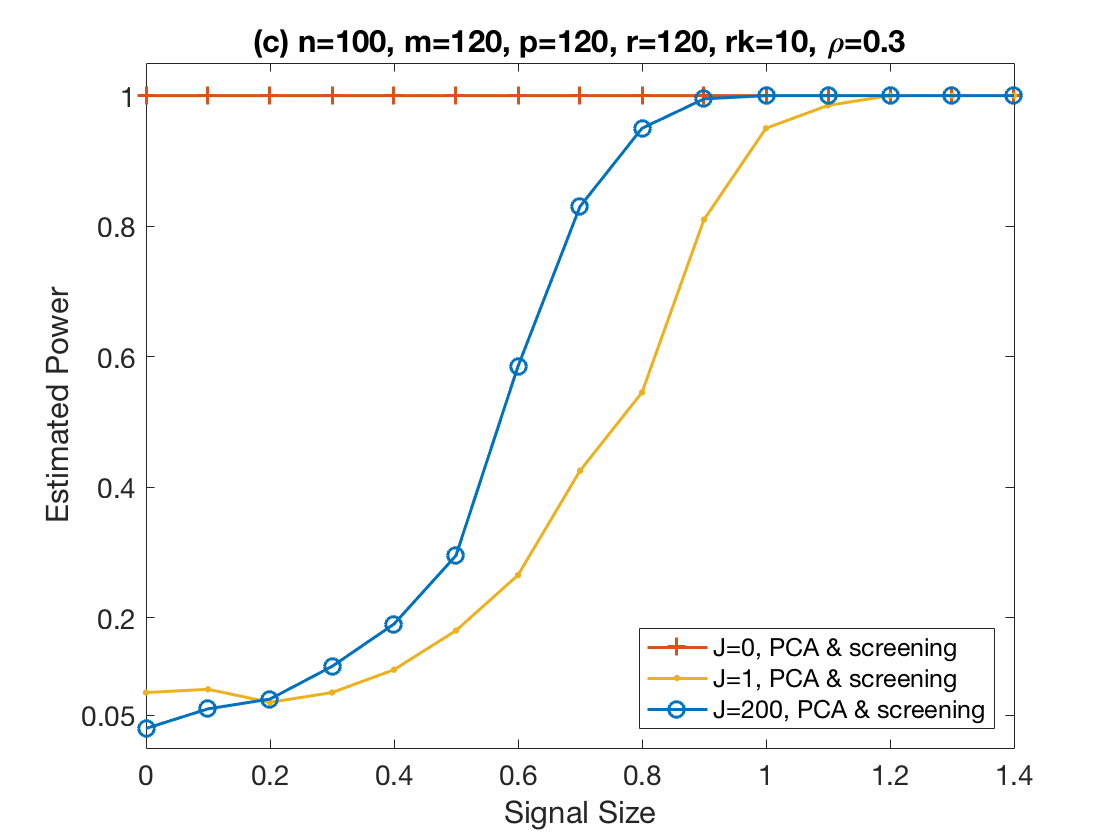}
	\includegraphics[width=0.44\textwidth]{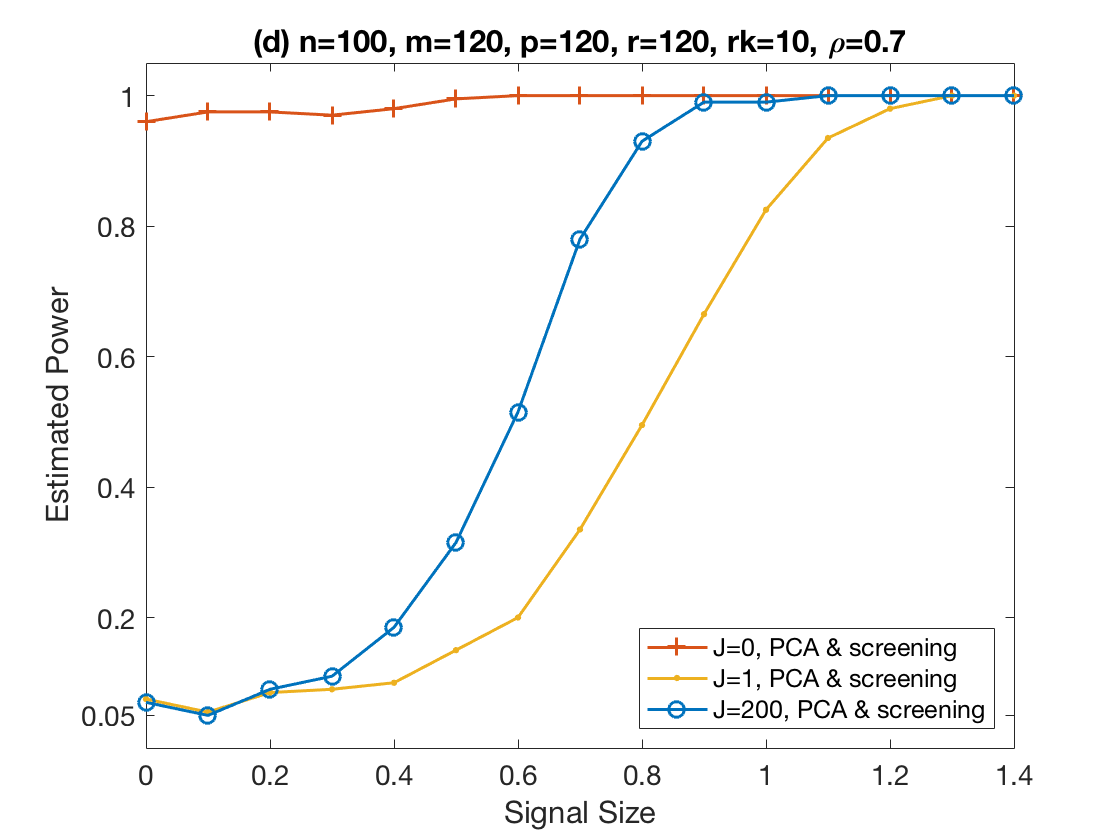}
\end{subfigure}
\caption{Estimated powers versus signal sizes when $n<m+p$}
\label{fig:powerestpowerscreen}
\end{figure}

Figure \ref{fig:powerestpowerscreen} shows that when we do not split the data $(J=0)$, the type \RNum{1} errors can not be controlled under all cases. If we split the data once $(J=1)$, the type \RNum{1} errors become  closer to the significance level, but can still be unstable. If we use the multi-split method with 200 splits $(J=200)$, the type \RNum{1} errors become well controlled. The results imply the necessity  of  data splitting    for    the proposed two-stage  testing procedure, and show that the multiple splits help us to obtain  stable results. In addition, in the four cases, the multi-split method $(J=200)$ also achieves higher power than the single split $(J=1)$ as the signal size increases.  Moreover, for cases (a) and (b) in  Figure \ref{fig:powerestpowerscreen} with the single split of data $(J=1)$, we also compare the test power  when only screening on $X$ with the test power when performing dimension reduction for both $X$ and $Y$. 
The results are given by the curves  ``$J=1$, only screening" and  ``$J=1$, PCA \& screening", respectively.  We observe that     the test power is slightly  enhanced by performing dimension reduction for both $X$ and $Y$. 

In addition, we also conduct similar studies when $X$ and $Y$ take discrete values or the statistical error   follows a heavy-tail $t$ distribution  in the Supplementary Material Section \ref{sec:robtwostepdist}.  We observe similar patterns to those in Figure \ref{fig:powerestpowerscreen}, which suggests  that the proposed method is   robust to the normal assumption of the statistical error.

\section{Real Data Analysis}\label{sec:realdata}
We demonstrate the application of our proposed method by analyzing   a breast cancer dataset from \cite{chin2006genomic}, which was also studied by \cite{chen2013reduced} and \cite{molstad2016indirect}. The dataset is available in the R package \texttt{PMA}, and  consists of measured gene expression profiles (GEPs) and
DNA copy-number variations (CVNs) for $n=89$ subjects.   Prior studies have demonstrated a link between DNA copy-number variations and
cancer risk \citep[see, e.g.,][]{peng2010regularized}. It is of interest to further examine the relationship between CNVs and  GEPs, where multivariate regression methods are useful. 


In this study, we examine the three  chromosomes 8, 17, and 22 and test whether they are associated (i.e., $C=I_{p}$). {We report the regression results of CNVs on GEPs in this section, and provide the regression results  of GEPs on CNVs in the Supplementary Material Section \ref{sec:supprealdata}, where similar patterns are observed.   
Here, the $m$-variate response is the  CNVs data and the $p$-variate predictor is the GEPs data, where the  dimension parameters   are $(p,m)=(673, 138), (1161, 87), (516, 18)$ for the three chromosomes  correspondingly.  As  the parameters $p$ and  $m$ are either comparable to or larger than the sample size $n=89$,   we   apply the proposed  testing procedure in Section  \ref{sec:dimred}. In particular, we choose the screening data size $n_S=26$  and the testing data size $n_T=63$, where $n_S:n_T$ is approximately 3:7.    We reduce the  dimension of response CNVs  data matrix by parallel analysis  and select the columns of GEPs data matrix by the screening method in Section  \ref{sec:dimred}.
To include as much information of predictors as possible, we    select the number of predictors between 40-50 when screening.    
 For each chromosome, we split the data $J=2000$ times and obtain the corresponding  $p$-values, $p^{(j)}$ for $j=1,\ldots, J$, from the limiting distribution of the test statistic $T_3$. We then compute the final $p$-value $p_t$  and reject the null hypothesis if $p_t<\alpha$. 

We summarize the testing  results in Table \ref{testre1}.   The column   ``$p_0$" indicates the number of selected predictors, and  the columns  ``$k_1\to k_2$" under ``Chromosome pair" means that we use GEPs from $k_1$th chromosome to predict CNVs from $k_2$th chromosome. For each setting, the symbols ``\textsf{x}"  and ``\checkmark" indicate that we reject and accept the null hypothesis respectively.
The test results show that the null hypothesis gets rejected when CNVs and GEPs are from the same chromosome, which makes biological sense. 
On the other hand,   if we use GEPs from the 8th chromosome to predict CNVs from the 17th chromosome or GEPs from the 17th chromosome to predict CNVs from the 22nd chromosome,  the null hypotheses are rejected; if we use GEPs from the 8th chromosome to predict CNVs from the 22nd chromosome, the null hypothesis is accepted. These conclusions indicate different relationships between CNVs and GEPs of different chromosomes, which might deserve further investigation by scientists.  

\begin{table}
\centering
\begin{tabular}{c|ccc|ccc}
\hline \hline
\multicolumn{1}{c|}{} & 
\multicolumn{6}{c}{ Chromosome pair} \\
$p_0$  & $8\to 8$ & $17\to 17$ & $22 \to  22$ & $8\to 17$ & $17 \to 22$ & $8\to 22$ \\
\hline
40   & \textsf{x} &\textsf{x} & \textsf{x}  & \textsf{x} &\textsf{x} &  \checkmark  \\
45   & \textsf{x} & \textsf{x} & \textsf{x} & \textsf{x} & \textsf{x} & \checkmark \\
50   & \textsf{x} &\textsf{x} & \textsf{x} & \textsf{x}& \textsf{x} & \checkmark \\
\hline
\end{tabular}
\caption{Decision results} \label{testre1}
\end{table}

To further illustrate the test results,  we report the boxplots of  $\{p^{(j)}: j=1,\ldots, J \}$  with respect to different chromosome pairs 
 when $p_0=45$ in Figure \ref{fig:histpval}.
 We find that the medians of  $p$-values obtained from the regressions of $8\to 17$, $17\to 22$ and the same chromosome pairs are smaller than $0.05$, which are consistent with the rejection decisions in Table \ref{testre1}.
Moreover, for $8\to 22$, the majority of the  $p$-values are larger than $0.05$.  It is thus consistent with the decision that we  accept the null hypothesis when using GEPs from the 8th chromosome to predict CNVs from the 22nd chromosome.

\begin{figure}[!h]
\centering
 \includegraphics[width=0.9\textwidth]{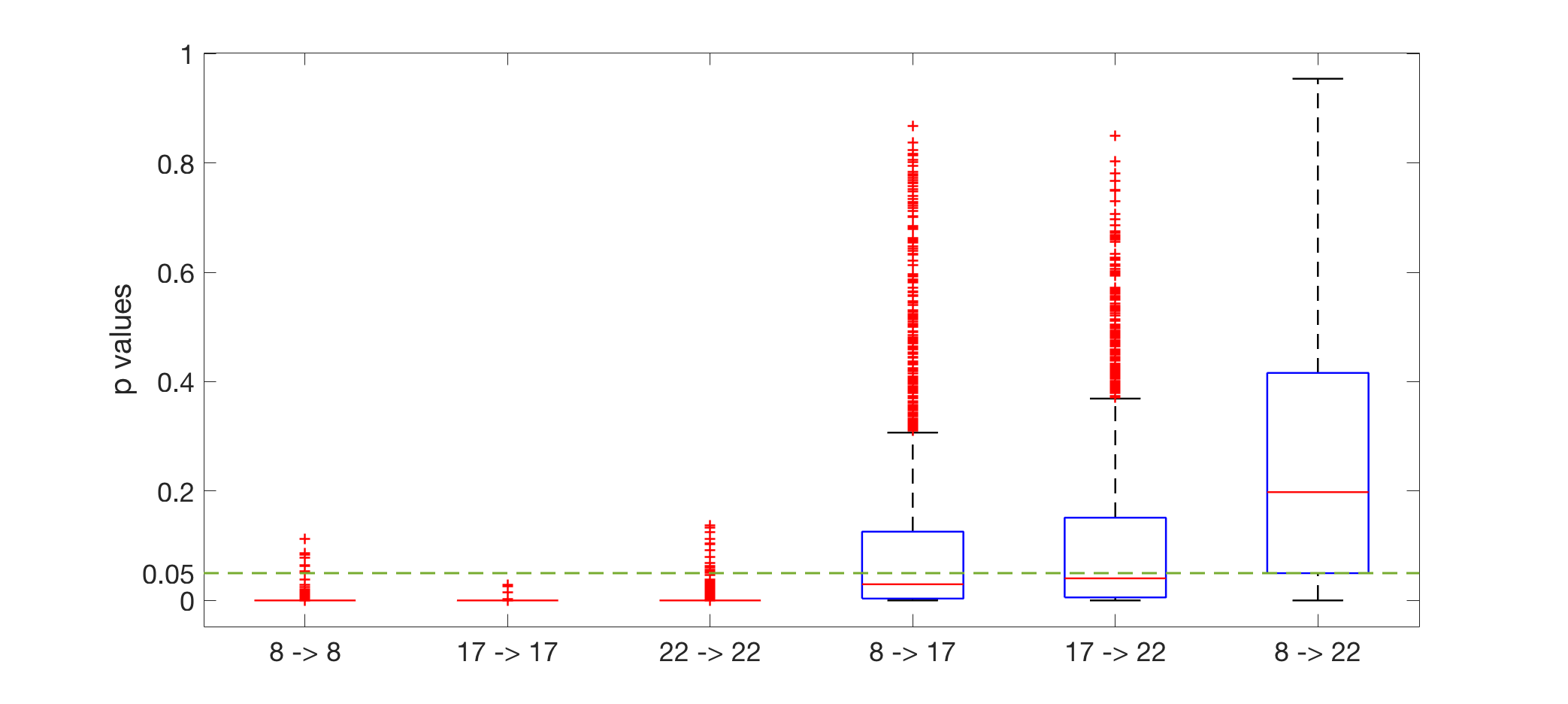}
 \caption{Boxplot of $p$ values for regressions on different chromosome pairs}
 \label{fig:histpval}
\end{figure}

\section{Conclusions and Discussions}

In this paper, we study the LRT for  $H_0: CB=\mathbf{0}_{r\times m}$ in high-dimensional multivariate linear regression, where $p$ and $m$ are allowed to increase with $n$. Under the null hypothesis,  we derive the asymptotic boundary where the classical $\chi^2$ approximation fails, and further develop the corrected limiting distribution of $\log L_n$ for a general asymptotic regime of $(p,m,r,n)$. Under alternative hypotheses, we characterize the statistical power of $\log L_n$ in the high-dimensional setting, and propose a power-enhanced test statistic. In addition, when $n<p+m$ and then LRT is not well-defined, we propose to use a two-step testing procedure   with repeated data-splitting.   


The study on LRT of multivariate linear regression can also be extended to vector non-parametric regression models. In specific, for $k=1,\ldots,m,$ suppose the $k$-th response variable depends on the  $p$-dimensional predictor vector $\mathbf{x}$ through the regression equation
$
	y_k=\mathbb{M}_k(\mathbf{x})+e_k,  
$  where the $\mathbb{M}_k$'s are unknown smooth functions and $e_k$'s are the error terms. We begin with the case when the predictor is univariate. Then we can model $\mathbb{M}_k(x)$ by using regression splines: 
$
	\mathbb{M}_k(x)= \sum_{j=1}^{M} b_{k,j} \phi_{j} (x),
$ where $\Phi=(\phi_{j}: k=1,\ldots,M)^{\intercal}$ are some basis functions.  Write $\mathbf{y}=(y_1,\ldots, y_m)^{\intercal}$, $\mathbf{e}= (e_1,\ldots,e_m)^{\intercal}$ and  $B=(b_{k,j})_{M\times m}$, then $\mathbf{y}=B^{\intercal} \Phi +\mathbf{e}$, which is in the form of the multivariate linear regression. To test the coefficients  $B$, we can apply the method in this paper.  More generally, when the predictors are multivariate,    additive models  \citep{hastie1986generalized} are commonly used to finesse the ``curse of dimensionality". The multivariate functions $\mathbb{M}_k$'s are written as
$
	\mathbb{M}_k(\mathbf{x})=\mathbb{M}_{k,1}(x_1)+\ldots+\mathbb{M}_{k,p}(x_p),  
$ for $k=1,\ldots,m$, and here $\mathbb{M}_{k,1}(\cdot),\ldots, \mathbb{M}_{k,p}(\cdot)$ are univariate functions. Suppose $\Phi_1,\ldots,\Phi_p$ are the basis functions for $\mathbb{M}_{k,1}(\cdot),\ldots, \mathbb{M}_{k,p}(\cdot)$ respectively. Then $\mathbf{y}=\tilde{B}^{	\intercal}\tilde{\Phi}+\mathbf{e}$, where $\tilde{B}=(B_1^{\intercal},\ldots,B_p^{\intercal})^{\intercal}$ and $\tilde{\Phi}=(\Phi^{\intercal}_1,\ldots,\Phi_p^{\intercal})^{\intercal}$. Therefore we can apply the proposed LRT method to test the structure of the coefficients matrix $\tilde B$. 

This work establishes the theoretical results under the assumption that the error terms $E$'s follow Gaussian distributions, but we expect that the conclusion holds over a larger range of distributions. Numerically, we conduct simulations when the error terms follow discrete distributions or heavy-tail $t$ distributions, which are provided in the Supplementary Material. The simulation results show similar patterns to the Gaussian cases and imply that the theoretical results may be still valid.
Theoretically, \cite{bai2013testing}  showed that the linear spectral of $F$-matrix $S_1S_2^{-1}$ also had asymptotic normal distribution, without specifying the distributions of entries of $S_1$ and $S_2$ to be normal. But they assumed  that entries of $S_1$ and $S_2$ are i.i.d., which is usually not satisfied   in   general multivariate regression analysis. Recently, \cite{li2018highgeneral} proposed a modified LRT via non-linear spectral shrinkage, and established its asymptotic normality without normal assumption on $E$ when $m$ is proportional to $n$. However, they assumed that $p$, the number of predictors, is fixed.  Thus the  asymptotic distribution of $\log L_n$ for general  high-dimensional  non-Gaussian cases remains an open question in the literature.

\section*{Supplementary Material}
The online Supplementary Material includes proofs  and additional simulations.

\paragraph{Acknowledgment}
The authors thank the co-Editor Dr.~Hans-Georg M\"{u}ller, an Associate Editor and two anonymous referees  for their constructive comments. The authors also thank Prof. Xuming He for helpful discussions.
This research is partially supported by National Science Foundation grants DMS-1406279, DMS-1712717, SES-1659328 and SES-1846747. 


\bibhang=1.7pc
\bibsep=2pt
\fontsize{9}{14pt plus.8pt minus .6pt}\selectfont
\renewcommand\bibname{\large \bf References}
\expandafter\ifx\csname
natexlab\endcsname\relax\def\natexlab#1{#1}\fi
\expandafter\ifx\csname url\endcsname\relax
  \def\url#1{\texttt{#1}}\fi
\expandafter\ifx\csname urlprefix\endcsname\relax\def\urlprefix{URL}\fi

\bibliographystyle{apalike}
\bibliography{minimalBib}

\newpage
\begin{center}
\textbf{\Large{Supplement to ``Likelihood Ratio Test in Multivariate Linear Regression: from Low to High Dimension"}}	
\end{center}

\vspace{2em}

\normalsize

We give proofs of main results and  additional simulations in the Supplementary Material. Specifically, in Section \ref{sec:approxbound}--\ref{sec:proofselection}, we prove  Theorems \ref{thm:approxbound}--\ref{thm:selectionaccuracy}, respectively. 
We present the proof of Proposition   \ref{prop:sizecontrol} in   Section \ref{sec:proofpropsize} and   provide additional simulations in Section \ref{sec:suppsimu}.

\appendix




\section{Theorem \ref{thm:approxbound}} \label{sec:approxbound}

Theorem \ref{thm:approxbound} has two parts of conclusions, with $mr\to \infty$ and $mr$ is finite respectively. We next prove the two parts  in Sections  \ref{sec:proofthmapproxbd} and \ref{sec:fixedmrlrt} respectively. A lemma  used in Section \ref{sec:proofthmapproxbd} is given and proved in Section \ref{sec:lmproofapprox}. 

\subsection{Proof of the part for $mr\to \infty$ in Theorem \ref{thm:approxbound}} \label{sec:proofthmapproxbd}

In this section, we consider $mr\to \infty$ and $\max\{p,m,r\}/n\to 0$. 
We prove the conclusion for $mr\to \infty$ in Theorem \ref{thm:approxbound} based on the result of Theorem \ref{thm:mainlimit}, which is proved in Section  \ref{sec:proofthmmainlimit}.


When $(p,m,r)$ are all fixed, we know that $-2\log L_n\xrightarrow{D}   \chi^2_{mr}$ as $n\to \infty$. Note that  $\mathrm{E}(\chi^2_{mr})=mr$,  $\mathrm{var}(\chi^2_{mr})=2mr$, and when $mr\to \infty$, $(\chi_{mr}^2-mr)/\sqrt{2mr } \xrightarrow{D} \mathcal{N}(0,1)$. 
It follows that $P(\chi^2_{mr}>\sqrt{2mr}z_{\alpha}+mr)\to \alpha$ and
\begin{eqnarray}
	\chi^2_{mr}(\alpha)=\sqrt{2mr}\times \{ z_{\alpha}+o(1)\}+mr, \label{eq:chisqquanapprox}
\end{eqnarray}
 where $z_{\alpha}$ denotes the upper $\alpha$-quantile of $\mathcal{N}(0,1)$.

We define the asymptotic regime $\mathcal{R}_A=\{ (p,m,r,n): n>p+m,~ p\geq r,~ mr\to \infty, ~\mbox{and}~ \max\{ p,m,r\}/n \to 0 ~ \mbox{\ as\ }~n\to \infty \}$.  Under the asymptotic regime $\mathcal{R}_A$,  Theorem  \ref{thm:mainlimit} shows that $(-2\log L_n+\mu_n)/(n\sigma_n) \xrightarrow{D} \mathcal{N}(0,1)$. Note that
\begin{eqnarray*}
   P \{-2\log L_n > \chi^2_{mr}(\alpha)\} 
	= P\Big\{ \frac{-2\log L_n+\mu_n}{n\sigma_n}>\frac{ \chi^2_{mr}(\alpha)+\mu_n }{n\sigma_n }  \Big\}. \notag 
\end{eqnarray*} 
Thus when $n\to \infty$, $P \{-2\log L_n > \chi^2_{mr}(\alpha)\} \to \alpha$ is equivalent to 
\begin{eqnarray}
	\frac{\chi^2_{mr}(\alpha)+\mu_n }{n\sigma_n } \to z_\alpha, \quad \mbox{as}\ n\to  \infty. \label{eq:quantileratio}
\end{eqnarray}
When $mr\to \infty$, by \eqref{eq:chisqquanapprox}, we know \eqref{eq:quantileratio} is equivalent to
\begin{eqnarray}
	\frac{\sqrt{2mr}\times \{z_{\alpha} +o(1)\}+mr +\mu_n}{n\sigma_n} \to z_{\alpha}, \quad \mbox{as}\ n\to  \infty.  \label{eq:quantilequiv}
\end{eqnarray} \eqref{eq:quantilequiv} holds for any significance level $\alpha$ if and only if $n\sigma_n=\sqrt{2mr}\{1+o(1)\}$ and $(\mu_n+mr)/\sqrt{2mr}=o(1)$. 
 
Next we will prove that under $\mathcal{R}_A$, $n\sigma_n=\sqrt{2mr}\{ 1+ o(1)\}$ in the first step, derive the form of $\mu_n$ in the second step, and obtain the conclusion in the third step.

\paragraph{Step 1.}

Note that 
\begin{eqnarray*}
	\sigma_n^2 = 2\log\frac{(n+r-p-m)(n-p)}{(n-p-m)(n+r-p)}. 
\end{eqnarray*}
By the Taylor expansion, $\log(1-a)=-a-a^2/2-a^3/3+O(a^4)$ for $a=o(1)$.
Under $\mathcal{R}_A$, we know that $p/n, m/n, r/n \to 0$ and $r/(n-p-m)\to 0$. Then we have
\begin{align}
	&~\log \frac{n+r-p-m}{n-p-m} \notag \\ 	
	=&~ \log \Big(1+ \frac{r}{n-p-m}  \Big) \notag \\
	=&~\frac{r}{n-p-m} -\frac{1}{2} \frac{r^2}{(n-p-m)^2}+\frac{1}{3} \frac{r^3}{(n-p-m)^3}+O \Big(\frac{r^4}{n^4} \Big), \label{eq:logterm1} 
\end{align} and similarly,
\begin{eqnarray}
	 -\log \frac{n-p+r}{n-p} &=&-\log \Big( 1+\frac{r}{n-p} \Big) \notag \\
	&=& -\frac{r}{n-p}+\frac{1}{2}\frac{r^2}{(n-p)^2}-\frac{1}{3} \frac{r^3}{(n-p)^3}+O\Big(\frac{r^4}{n^4}\Big).\label{eq:logterm2} 
\end{eqnarray} Since for any numbers $a$ and $b$, $a^2-b^2=(a-b)(a+b)$ and  $a^3-b^3=(a-b)(a^2+b^2+ab)$,   we then  know
\begin{align}
	&~\eqref{eq:logterm1}+\eqref{eq:logterm2} \notag \\
	 =&~ \frac{r}{n-p-m}-\frac{r}{n-p} -\frac{1}{2}\Big\{  \frac{r^2}{(n-p-m)^2} -\frac{r^2}{(n-p)^2} \Big\} \notag \\
	&~+\frac{1}{3}\Big\{\frac{r^3}{(n-p-m)^3}- \frac{r^3}{(n-p)^3}  \Big\}+ O\Big(\frac{r^4}{n^4}  \Big)  \notag \\
	=&~  \frac{rm}{(n-p-m)(n-p)} - \frac{1}{2} r^2\times \frac{m(2n-2p-m)}{(n-p-m)^2(n-p)^2}+ O\Big\{\frac{r^3(m+r)}{n^4}  \Big\}. \label{eq:twologterm} 
\end{align}	

 We next examine the first two terms in \eqref{eq:twologterm}. 
Note that for $a=o(1)$ and $b=o(1)$, $1/(1-a)=1+a+O(a^2)$ and $1/\{(1-a)(1-b)\}=1+a+b+O(a^2+b^2)$. Then for the first term in \eqref{eq:twologterm}, we have
\begin{eqnarray}
	\frac{rm}{(n-p-m)(n-p)}	&=& \frac{rm}{n^2\{1-(p+m)/n\}(1-p/n)} \label{eq:rmtermapprox}  \\
	&=& \frac{rm}{n^2} \Big\{1+\frac{2p+m}{n}+O\Big(\frac{p^2+m^2}{n^2}\Big) \Big\}. \notag
\end{eqnarray}
In addition, note that for $a=o(1)$ and $b=o(1)$, $1/\{(1-a)^2(1-b)^2\}=1+2a+2b+O(a^2+b^2)$. Then for the second term in \eqref{eq:twologterm}, we have
\begin{eqnarray}
	&&-\frac{1}{2}r^2\times \frac{m(2n-2p-m)}{(n-p-m)^2(n-p)^2}\label{eq:secodrmterm} \\
	 &=& -nmr^2\Big\{1-\frac{p+m/2}{n}\Big\}\frac{1}{n^4}\Big\{1+\frac{2(p+m)}{n}+\frac{2p}{n}+ O\Big( \frac{p^2+m^2}{n^2} \Big)  \Big\} \notag \\
	 &=&-\frac{mr^2}{n^3}\Big\{1+\frac{3p+3m/2}{n}  +O\Big(\frac{p^2+m^2}{n^2} \Big)\Big\}.  \notag
\end{eqnarray} 

Combining \eqref{eq:rmtermapprox} and \eqref{eq:secodrmterm}, we obtain
\begin{eqnarray}
 \eqref{eq:twologterm} &=&\eqref{eq:rmtermapprox}+\eqref{eq:secodrmterm}+O\Big\{\frac{r^3(m+r)}{n^4}  \Big\}   \notag \\
	&=& \frac{rm}{n^2}+\frac{rm}{n^2}\Big\{\frac{2p+m-r}{n}  \Big\} +O\Big\{\frac{mr(m^2+r^2+p^2)}{n^4} \Big\}. \label{eq:summthirdtermin}
\end{eqnarray} 
We then know that $\sigma_n^2 =2\times  \eqref{eq:twologterm}= (2mr/n^2) \times \{1+o(1)\},$ and thus $n\sigma_n=\sqrt{2mr}\{1+o(1)\}$. 

\paragraph{Step 2.} In this step, we  derive the form of  $\mu_n$. Under the asymptotic region $\mathcal{R}_A$, we know that by Lemma \ref{lm:meanvalformula} and Taylor expansion,
\begin{eqnarray*}
	\mu_n &=& -mr \Big\{ 1+\frac{p-r}{n} +O\Big(\frac{p^2+r^2}{n^2} \Big)\Big\} \notag \\
	&& -\frac{1}{2} \times \frac{mr(m+r)}{n} \Big\{ 1+O\Big(\frac{p+r}{n}\Big) \Big\} + o(1) mr \times \frac{p+m+r}{n} \notag \\ 
	&=& -mr -mr \frac{p+m/2-r/2}{n}+ o(1) mr \times \frac{p+m+r}{n}. 
\end{eqnarray*}

\paragraph{Step 3.}
As discussed, under $\mathcal{R}_A$,  \eqref{eq:quantilequiv} holds for any level $\alpha$, if and only if $n\sigma_n=\sqrt{2mr}\{ 1+ o(1)\}$ and $(\mu_n+mr)/\sqrt{2mr}=o(1)$. In the first step, we have shown that $n\sigma_n=\sqrt{2mr}\{1+o(1)\}$ under $\mathcal{R}_A$. In the second step, we obtain the form of $\mu_n$.  Thus we have
\begin{eqnarray*}
	\frac{\mu_n+mr}{\sqrt{2mr}}=-\frac{\sqrt{mr}}{\sqrt{2}}\Big( \frac{p+m/2-r/2}{n}\Big) \times \{1+o(1)\}, 
\end{eqnarray*} which converges to 0, if and only if $\lim_{n\to \infty} {\sqrt{mr}(p+m/2-r/2)}n^{-1}=0.$


\subsection{Proof of the part for finite $mr$ in Theorem \ref{thm:approxbound}} \label{sec:fixedmrlrt}

By \cite{muirhead2009aspects}, the characteristic function of $-2\log L_n$ is $ \phi_1 (t)	= \mathrm{E}\{\exp(-2it\log L_n)\}$ and  
\begin{align}
  \log \phi_1 (t)	= -\frac{mr}{2}\log (1-2it) + \sum_{l=1}^{\infty} \varsigma_{l}\{(1-2it)^{-l}-1\}, \label{eq:charexp1} 
\end{align} where 
\begin{align*}
	\varsigma_l=\frac{(-1)^{l+1}}{l(l+1)}\Big[ \sum_{k=1}^{m}\Big\{\frac{ \mathbb{B}_{l+1} \{(1-k-p)/2\} }{(n/2)^l} - \frac{\mathbb{B}_{l+1} \{(1-k+r-p)/2\}}{(n/2)^l} \Big\}\Big],
\end{align*}
and 
$ \mathbb{B}_{l+1} (\cdot)$ is the Bernoulli polynomials which takes the form $ \mathbb{B}_{l+1}(z)=\sum_{v=0}^{l+1}c_v z^v$. We next estimate the order of $\varsigma_l$ with respect to $n$. We note that for any $z_1$ and $z_2$, 
\begin{align}
	&\mathbb{B}_{l+1} (z_1) - 	\mathbb{B}_{l+1} (z_2) \label{eq:bzlexpansion} \\
=& \sum_{v=0}^{l+1}c_v (z_1^v-z_2^v) = \sum_{v=1}^{l+1} c_v \sum_{w=1}^{v}\binom{v}{w}(z_1-z_2)^{w}z_2^{v-w} \notag \\
=& (z_1-z_2) \sum_{v=1}^{l+1} \sum_{w=1}^{v}c_v \binom{v}{w}(z_1-z_2)^{w-1}z_2^{v-w}.  \notag 
\end{align}
Let $z_1=(1-k-p)/2$ and $z_2=(1-k+r-p)/2$. Then we have $z_1-z_2=(-r)/2$. When $m$ and $r$ are finite, the order of $\varsigma_l$ with respect to $n$ is $O\{(p/n)^l\}$.
When $p/n \to 0$, by the expansion \eqref{eq:charexp1}, we have $\phi_1(t)=(1-2it)^{-mr/2}\{1+o(1)\}.$ Then $-2\log L_n \xrightarrow{D} \chi^2_{mr}$ as $n\to \infty$. When $p/n$ is bounded from 0 below, \eqref{eq:charexp1} does not converge to $-2^{-1}mr\log(1-2it)$ generally for all $t$. Then the approximation $-2\log L_n \xrightarrow{D} \chi^2_{mr}$ fails.     


%

\subsection{Lemma used in Section  \ref{sec:proofthmapproxbd}}  \label{sec:lmproofapprox}
\begin{lemma} \label{lm:meanvalformula}
	Under the asymptotic regime $\mathcal{R}_A$, \begin{eqnarray*}
\mu_n &=&-\frac{nmr}{n+r-p}-\frac{1}{2}\frac{nmr(m+r)}{(n+r-p)^2} -\frac{nmr(m^2/3+mr/2+r^2/3)}{(n+r-p)^3} \notag \\
	&&+ O(1)\frac{mr(m^3+r^3)}{n^3}+O(mr/n).		
\end{eqnarray*}
\end{lemma}
\begin{proof}
By the definition of $\mu_n$ in Theorem \ref{thm:mainlimit}, 
\begin{eqnarray*}
	\mu_n &=& {n(n-m-p-1/2)}
\log\frac{(n+r-p-m)(n-p)}{(n-p-m)(n+r-p)}+ {nr}\log\frac{(n+r-p-m)}{(n+r-p)}
\notag\\
&&+ {nm}\log\frac{(n-p)}{(n+r-p)}. \notag 
\end{eqnarray*}
Note that 
\begin{eqnarray*}
	&& \log\frac{(n+r-p-m)(n-p)}{(n-p-m)(n+r-p)} \notag \\
	&=&\log \frac{n+r-p-m}{n+r-p} + \log \frac{n-p}{n+r-p} - \log \frac{n-p-m}{n+r-p} \notag \\
	&=& \log \Big(1-\frac{m}{n+r-p} \Big)+\log \Big( 1-\frac{r}{n+r-p} \Big)-\log \Big(1-\frac{m+r}{n+r-p} \Big).
\end{eqnarray*}
It follows that
\begin{eqnarray}
	\mu_n &=& n(n-m-p+r-1/2) \log \Big(1-\frac{m}{n+r-p} \Big)\label{eq:mun1} \\
	&&+n(n-p-1/2) \log \Big( 1-\frac{r}{n+r-p} \Big) \label{eq:mun2} \\
	&& - n(n-m-p-1/2) \log  \Big(1-\frac{m+r}{n+r-p} \Big), \label{eq:mun3}
\end{eqnarray}
which gives $\mu_n=\eqref{eq:mun1}+\eqref{eq:mun2}+\eqref{eq:mun3}$. We next analyze $\eqref{eq:mun1}/n$, $\eqref{eq:mun2}/n$ and $\eqref{eq:mun3}/n$ respectively. 

By the Taylor expansion, we  have
\begin{eqnarray*}
	\log \Big(1-\frac{m}{n+r-p} \Big) &=& -\sum_{k=1}^{\infty} \frac{1}{k}\Big( \frac{m}{n+r-p} \Big)^k,  \notag \\
	\log \Big( 1-\frac{r}{n+r-p} \Big)  &=& -\sum_{k=1}^{\infty} \frac{1}{k}\Big( \frac{r}{n+r-p} \Big)^k, \notag \\
	\log  \Big(1-\frac{m+r}{n+r-p} \Big) &=& -\sum_{k=1}^{\infty} \frac{1}{k}\Big( \frac{m+r}{n+r-p} \Big)^k.  
\end{eqnarray*}Then
\begin{eqnarray}
\eqref{eq:mun1}/n	&=&- (n+r-p)\sum_{k=1}^{\infty} \frac{1}{k}\Big( \frac{m}{n+r-p} \Big)^k + (m+1/2)\sum_{k=1}^{\infty} \frac{1}{k}\Big( \frac{m}{n+r-p} \Big)^k \notag \\
	&=&- \sum_{k=1}^{\infty} \frac{1}{k} \frac{m^k}{ (n+r-p)^{k-1}} + \sum_{k=1}^{\infty} \frac{1}{k}\frac{m^{k+1}}{(n+r-p)^{k}}+\sum_{k=1}^{\infty} \frac{1}{k} \frac{m^k/2}{(n+r-p)^k},\notag 
\end{eqnarray}
\begin{eqnarray}
\eqref{eq:mun2}/n	&=& -(n+r-p) \sum_{k=1}^{\infty} \frac{1}{k}\Big( \frac{r}{n+r-p} \Big)^k  +(r+1/2)\sum_{k=1}^{\infty} \frac{1}{k}\Big( \frac{r}{n+r-p} \Big)^k \notag \\
	&=& - \sum_{k=1}^{\infty} \frac{1}{k} \frac{r^k}{(n+r-p)^{k-1}}+ \sum_{k=1}^{\infty} \frac{1}{k} \frac{r^{k+1}}{(n+r-p)^k} + \sum_{k=1}^{\infty} \frac{1}{k} \frac{r^{k}/2}{(n+r-p)^k},  \notag  
\end{eqnarray}
and
\begin{eqnarray}
\eqref{eq:mun3}/n&=& (n+r-p) \sum_{k=1}^{\infty} \frac{1}{k}\Big( \frac{m+r}{n+r-p} \Big)^k  - (m+r+1/2)\sum_{k=1}^{\infty} \frac{1}{k}\Big( \frac{m+r}{n+r-p} \Big)^k  \notag \\
 &=&  \sum_{k=1}^{\infty} \frac{1}{k} \frac{(m+r)^k}{(n+r-p)^{k-1}} - \sum_{k=1}^{\infty} \frac{1}{k} \frac{(m+r)^{k+1}}{(n+r-p)^k} - \sum_{k=1}^{\infty} \frac{1}{k} \frac{(m+r)^{k}/2}{(n+r-p)^k}. \notag
\end{eqnarray}
It follows that $\{ \eqref{eq:mun1}+\eqref{eq:mun2}+\eqref{eq:mun3}\}/n=\eqref{eq:threetermsum1}+ \eqref{eq:threetermsum2}$, where
\begin{eqnarray}
&&\sum_{k=1}^{\infty} \frac{1}{k} \frac{(m+r)^k-m^k-r^k}{(n+r-p)^{k-1}} - \sum_{k=1}^{\infty} \frac{1}{k} \frac{(m+r)^{k+1}-m^{k+1}-r^{k+1}}{(n+r-p)^{k} }, \label{eq:threetermsum1} \\
	&& -\frac{1}{2}\sum_{k=1}^{\infty}\frac{1}{k}\frac{(m+r)^k-m^k-r^k}{(n+r-p)^k}. \label{eq:threetermsum2} 
\end{eqnarray} 
As $(m+r)^1-m^1-r^1=0$, we know 
\begin{eqnarray}
\eqref{eq:threetermsum1}&=& \sum_{k=2}^{\infty} \frac{1}{k} \frac{(m+r)^k-m^k-r^k}{(n+r-p)^{k-1}} - \sum_{k=1}^{\infty} \frac{1}{k} \frac{(m+r)^{k+1}-m^{k+1}-r^{k+1}}{(n+r-p)^{k} } \notag \\
&=& \sum_{k=1}^{\infty} \Big(\frac{1}{k+1}-\frac{1}{k}\Big) \frac{(m+r)^{k+1}-m^{k+1}-r^{k+1}}{(n+r-p)^{k}}  \notag \\
&=& - \frac{1}{2} \frac{2mr}{n+r-p} \label{eq:taylor1} \\
&& -\frac{1}{6} \frac{(m+r)^3-m^3-r^3}{(n+r-p)^2}\label{eq:taylor2}  \\
&& -\frac{1}{12} \frac{(m+r)^4-m^4-r^4}{(n+r-p)^3} \label{eq:taylor3} \\
&&- \sum_{k=4}^{\infty} \frac{1}{k(k+1)} \frac{(m+r)^{k+1}-m^{k+1}-r^{k+1}}{(n+r-p)^{k}}, \label{eq:taylorremain} 
\end{eqnarray} which gives $\eqref{eq:threetermsum1}=\eqref{eq:taylor1}+\eqref{eq:taylor2}+\eqref{eq:taylor3}+\eqref{eq:taylorremain}$.  
We have
$ n\times \eqref{eq:taylor1} =- {nmr}{(n+r-p)^{-1}} $, 
$n\times \eqref{eq:taylor2} = -{2}^{-1}{nmr(m+r)}{(n+r-p)^{-2}} $, and 
$	n\times \eqref{eq:taylor3}= -{nmr(m^2/3+mr/2+r^2/3)}\times {(n+r-p)^{-3}}$. In addition,
\begin{eqnarray*}
	|\eqref{eq:taylorremain}|&=&\sum_{k=4}^{\infty} \frac{1}{k(k+1)} \frac{\sum_{q=1}^k \binom{k+1}{q}m^q r^{k+1-q}}{(n+r-p)^{k}} \notag \\
	&\leq & \frac{mr}{n+r-p} \sum_{k=4}^{\infty} \frac{1}{k} \times \frac{2^{k+1} (\max\{m,r\})^{k-1}}{(n+r-p)^{k-1}} \notag \\
	&=&  \frac{mr}{n+r-p}  O\Big\{ \Big( \frac{\max\{m,r\}}{n+r-p}\Big)^3 \Big\} = O(1)\frac{mr(m^3+r^3)}{n^4} ,
\end{eqnarray*} where in the last two equations, we use the property of Taylor expansion and  the condition that $\max\{p,m,r\}=o(n)$. Therefore, 
$
	n\times \eqref{eq:taylorremain} =mr \times O \{ {(m^3+r^3)}/{n}^3 \}.
$ Moreover,
\begin{eqnarray*}
\eqref{eq:threetermsum2} &=&\sum_{k=1}^{\infty}\frac{1}{k}\frac{(m+r)^k-m^k-r^k}{(n+r-p)^k} \notag  \\
&=&\sum_{k=2}^{\infty} \frac{1}{k}\frac{ \sum_{q=1}^{k-1}\binom{k}{q} m^q r^{k-q} }{(n+r-p)^{k}} \notag \\
&\leq & \frac{mr}{(n+r-p)^{2}} \sum_{k=2}^{\infty} \frac{1}{k} \times  \frac{ 2^k (\max\{m,r\})^{k-2} }{(n+r-p)^{k-2}} \notag \\
&=& O({mr}/{n^2}),
\end{eqnarray*} where in the last equation we use the fact that
\begin{eqnarray*}
	 \sum_{k=2}^{\infty} \frac{1}{k} \times  \frac{ 2^k (\max\{m,r\})^{k-2} }{(n+r-p)^{k-2}}  	&\leq & 2 + \sum_{k=3}^{\infty} \frac{1}{k-2} \times  \frac{ 2^k (\max\{m,r\})^{k-2} }{(n+r-p)^{k-2}}  \notag \\
	&=& 2 + 4 \sum_{k=1}^{\infty} \frac{1}{k} \times  \frac{ 2^k (\max\{m,r\})^{k} }{(n+r-p)^{k}} \notag \\
	&=&2+4\log [1-  \{2\max\{m,r\}/(n+r-p)\} ].
\end{eqnarray*} 
In summary,
\begin{eqnarray*}
	\mu_n &=& \eqref{eq:mun1}+\eqref{eq:mun2}+\eqref{eq:mun3}\notag \\
	&=& n \times \{ \eqref{eq:threetermsum1}+ \eqref{eq:threetermsum2} \} \notag \\
	&=& -\frac{nmr}{n+r-p}-\frac{1}{2}\frac{nmr(m+r)}{(n+r-p)^2} -\frac{nmr(m^2/3+mr/2+r^2/3)}{(n+r-p)^3} \notag \\
	&&+ O(1)\frac{mr(m^3+r^3)}{n^3}+O(mr/n).
\end{eqnarray*}
\end{proof}

\section{Theorem \ref{prop:correctrho}} \label{sec:correctrho}
Similarly to Section \ref{sec:approxbound}, we prove Theorem \ref{prop:correctrho} when $mr\to \infty$ and $mr$ is finite in Sections \ref{sec:pfthmcorrlrtinf} and \ref{sec:fixedmrcorrectlrt} respectively.  
\subsection{Proof of the part for $mr\to \infty$ in Theorem \ref{prop:correctrho}} \label{sec:pfthmcorrlrtinf}

   When $(p,m,r)$ are all fixed, by Bartlett correction, we know that with $\rho=1-(p-r/2+m/2+1/2)/n$, $-2\rho \log L_n\xrightarrow{D}   \chi^2_{mr}$ as $n\to \infty$. Note that under $\mathcal{R}_A=\{ (p,m,r,n): n>p+m,~ p\geq r,~ mr\to \infty, ~\mbox{and}~ \max\{ p,m,r\}/n \to 0 ~ \mbox{\ as\ }~n\to \infty \}$, $\rho=1+o(1)$. Then similarly to the proof of Theorem \ref{thm:approxbound} in Section \ref{sec:proofthmapproxbd}, we know that under $\mathcal{R}_A$, $P\{{-2\rho \log L_n}>\chi^2_{mr}(\alpha) \} \to \alpha$ holds for any given significance level $\alpha$
if and only if $n\sigma_n=\sqrt{2mr}\{ 1+ o(1)\}$ and $ (\mu_n+mr/\rho)/\sqrt{2mr}=o(1)$.

Following the same argument as in Section \ref{sec:proofthmapproxbd}, we know that under $\mathcal{R}_A$, $n\sigma_n=\sqrt{2mr}\{ 1+ o(1)\}$.  
In addition, by the Taylor expansion, 
\begin{eqnarray*}
	\frac{mr}{\rho}&=& \frac{mr}{1-(p+m/2-r/2+1/2)/n} \notag \\
	  &=& \frac{nmr}{n-p+r-(m+r)/2} +mrO\Big(\frac{1}{n} \Big)  \notag \\
  &=& \frac{nmr}{n-p+r} \sum_{k=0}^{\infty} \Big\{ \frac{m+r}{2(n-p+r)} \Big\}^k +mrO\Big(\frac{1}{n} \Big)  \notag \\
  &=& \frac{nmr}{n-p+r} + \frac{nmr(m+r)}{2(n+r-p)^2} +\frac{nmr(m+r)^2}{4(n-p+r)^3} \notag \\
  && + mr\times  O\Big(\frac{m^3+r^3}{n^3}\Big) +mrO\Big(\frac{1}{n} \Big),  \notag 
\end{eqnarray*}
where in the last equation, we use the fact that  $\sum_{k=3}^{\infty} [(m+r)/\{2(n-p+r) \}]^k=O\{(m^3+r^3 )/n^3\}$ as $\max\{p,m,r\}=o(n)$. It follows that under $\mathcal{R}_A$, by Lemma \ref{lm:meanvalformula},
\begin{eqnarray}
	&&  \frac{\mu_n-(-mr/\rho)}{\sqrt{2mr}} \label{eq:secordermenadiff} \\
	&=& \frac{1}{\sqrt{2mr}}\Big\{  -\frac{nmr}{n+r-p}-\frac{nmr(m+r)}{2(n+r-p)^2} -\frac{nmr(m^2/3+mr/2+r^2/3)}{(n+r-p)^3}\notag \\
	&&  \quad \quad \quad \quad +\frac{nmr}{n-p+r} + \frac{nmr(m+r)}{2(n+r-p)^2} +\frac{nmr(m+r)^2}{4(n-p+r)^3} \Big\} \notag \\
	&& + \sqrt{mr} \times O\Big(\frac{m^3+r^3}{n^3} \Big) + O(\sqrt{mr}/n) \notag \\
	&=& -\frac{\sqrt{mr}}{\sqrt{2}} \frac{1}{12} \frac{n(m^2+r^2)}{(n+r-p)^3} + \sqrt{mr} \times O\Big(\frac{m^3+r^3}{n^3}\Big) + O(\sqrt{mr}/n) \notag \\
	&=& - \frac{\sqrt{mr}(m^2+r^2)}{12\sqrt{2}n^2} + o(1)\times \frac{\sqrt{mr}(m^2+r^2)}{n^2} + O(\sqrt{mr}/n), \notag
\end{eqnarray} where in the last equation, we use the fact that $\max\{p,m,r\}=o(n)$.  We thus know that $\eqref{eq:secordermenadiff}=0$ if and only if $\sqrt{mr}(m^2+r^2)/n^2 \to 0$.

\subsection{Proof of the part for finite $mr$ in Theorem \ref{prop:correctrho}} \label{sec:fixedmrcorrectlrt}

By \cite{muirhead2009aspects},  for the LRT with  Bartlett correction, the characteristic function of $-2\rho \log L_n$ is $\phi_2(t)=\mathrm{E}\{\exp(-2it\rho\log L_n)\}$. Moreover, we have
\begin{align*}
	\log \phi_2(t)=-\frac{mr}{2}\log (1-2it)+\sum_{l=1}^{\infty} \tilde{\varsigma}_{l}\{(1-2it)^{-l}-1\},
\end{align*} where 
\begin{align*}
	\tilde{\varsigma}_l=\frac{(-1)^{l+1}}{l(l+1)}\Big[ \sum_{k=1}^{m}\Big\{\frac{ \mathbb{B}_{l+1} (\tilde{z}_{k,1}) }{(\rho n/2)^l} - \frac{\mathbb{B}_{l+1} (\tilde{z}_{k,2})}{(\rho n/2)^l} \Big\}\Big],
\end{align*}
 $\tilde{z}_{k,1}=(1-\rho)n/2+(1-k-p)/2$ and $\tilde{z}_{k,2}=(1-\rho)n/2+(1-k+r-p)/2$. 
Since $\rho=1-(p-r/2+m/2+1/2)/n$, 
\begin{align*}
	\tilde{z}_{k,1}=&(p-r/2+m/2+1/2)/2+(1-k-p)/2=(3-r+m)/4, \notag \\
	\tilde{z}_{k,2}=& (p-r/2+m/2+1/2)/2+ (1-k+r-p)/2=(3 +r+m)/4. 
\end{align*}In addition, $\rho n= n-(p-r/2+m/2+1/2)$. Therefore, by the expansion in \eqref{eq:bzlexpansion}, when $m$ and $r$ are fixed and $n-p\to \infty$, we have $\log \phi_2(t)=-2^{-1}mr\log(1-2it)+O\{(n-p)^{-1}\}$ and $\phi_2(t)=(1-2it)^{-mr/2}[1+O\{(n-p)^{-1}\} ].$ It follows that when $m$ and $r$ are fixed and $n-p\to \infty$, $-2\rho\log L_n  \xrightarrow{D} \chi^2_{mr}$. On the other hand, when $n-p$ is fixed, by the expansion in \eqref{eq:bzlexpansion},  we know $\tilde{\varsigma}_{l}$ is of constant order in $n$, and thus $\sum_{l=1}^{\infty} \tilde{\varsigma}_{l}\{(1-2it)^{-l}-1\}$ is not ignorable generally for all $t$. We then know the approximation $-2\rho\log L_n  \xrightarrow{D} \chi^2_{mr}$ fails. 



\section{Theorem \ref{thm:mainlimit}} \label{sec:proofthmmainlimit}

In this section, we give the proof of Theorem \ref{thm:mainlimit}, where    the main proof is in Section \ref{sec:mainproofthm1} and some lemmas used are provided and proved in Section \ref{sec:lemmasthm1}. 

\subsection{Proof of Theorem \ref{thm:mainlimit}}\label{sec:mainproofthm1}

\begin{proof}
To prove the central limit theorem that
$
	H_n:= \{-2\log L_n+\mu_n \}/{(n\sigma_n)} \xrightarrow{D} \mathcal{N}(0,1),
$  it is sufficient to show
\begin{eqnarray}
	E\exp\left\{\frac{\log L_n-\mu_n/2}{n\sigma_n/2}s\right\}\to \exp\{s^2/2\}, \label{eq:proofgoal}
\end{eqnarray}
as $n\to\infty$ and $|s|<1,$ where $\sigma_n^2$ and $\mu_n$ are defined  in Theorem \ref{thm:mainlimit}. Equivalently, it suffices to show that for any subsequence $\{n_k\}$, there is a further subsequence $\{n_{k_j}\}$ such that $H_{n_{k_j}}$ converges to $\mathcal{N}(0,1)$ in distribution as $j \to \infty$. In the following, the further subsequence is selected in a way such that the subsequential limits of some bounded quantities (to be specified in the proof below) exist, which is guaranteed by Bolzano-Weierstrass Theorem. Therefore, we only need to verify the theorems by assuming that the limits for these bounded quantities  exist. In the following, we give the proof by discussing two settings $r\geq m$ and $m \geq r$ separately.   



\paragraph {Case 1. When $r\geq m$ and  $r\to\infty$.} 
By Lemma \ref{prop:betaapprox}, under the null hypothesis, the distribution of $L_n$ can be reexpressed as the distribution of  a product of independent beta random variables.
Let $h=2s/(n\sigma_n)$, by Lemma \ref{lm:momentln}, then under the null hypothesis, $L_n$'s $h$th moment can be written as
\begin{eqnarray}
	E\exp\left\{\frac{\log L_n}{n\sigma_n/2}s\right\}=E (L_n^h )= \frac{\Gamma_m\{\frac{1}{2}n(1+h) -\frac{1}{2}p\}\Gamma_m\{\frac{1}{2} (n+r-p)\}}
{\Gamma_m\{\frac{1}{2}(n-p)\}\Gamma_m\{\frac{1}{2}n(1+h) +\frac{1}{2}(r-p)\}}, \label{eq:generatingfunctionexp}
\end{eqnarray}
where $\Gamma_m(a)$, $a\in \mathbb C$ and $\mbox{Re}(a)>(m-1)/2$, is the multivariate Gamma function defined to be
\begin{eqnarray}
	\Gamma_m(a) = \int_{A>0}e^{-\mathrm{tr} (A)}\det A^{a-(m+1)/2} (dA). \label{eq:multigammadef}
\end{eqnarray}
The above integration is taken over the space of positive definite $m\times m$ matrices, i.e., $\{A_{m\times m}: A\succ 0\}$; and $\mathrm{tr}(A)$ is the trace of $A$. 
Note that when $m=1$, $\Gamma_m(a)$ becomes the usual definition of Gamma function. 
By Lemma \ref{lm:lmgamma}, $\Gamma_m(a) $ can be written as a product of ordinary Gamma functions as
$$\Gamma_m(a) = \pi^{m(m-1)/4}\prod_{j=1}^m \Gamma \{a-(j-1)/2\}.$$ 

Note that $n>m+p$ and $r \geq 1$. Thus the limits of $m/(n+r-p)$ and $m/(n-p)$ are in $[0,1]$ for all $n$. Applying the subsequence argument above, for any subsequence $\{n_{k}\}$,  we take a further subsequence $n_{k_j}$ such that $m_{k_j}/(n_{k_j}+r_{k_j}-p_{k_j})$ and $m_{k_j}/(n_{k_j}-p_{k_j})$ converge to some constants in $[0,1]$. Thus without loss of generality, we consider the cases when $m/(n+r-p)$ and $m/(n-p)$ converge to some constants in $[0,1]$. Next we give the proof by discussing  different cases below. 


\paragraph{Case 1.1} 
If $m/(n+r-p)\to\gamma >0$, this implies that $m\rightarrow \infty$ as $n\to \infty$. And as $r \geq m$ and $n>p+m$, we know $ {m}/{(n+r-p)} \leq 1/2$, then $\gamma \in (0,1/2]$. Since $1 \geq {m}/{(n-p)} \geq {m}/{(n+r-p)}$, then $m/(n-p)  \to \gamma' \in (0,1]$. 

If $\gamma'\in (0,1)$, $nh\times [-\log \{1-m/(n-p)\}]^{1/2}=O(1)$, which satisfies the assumption of Lemma 5.4 in  \cite{jiang2013}. If $\gamma'=1$, as 
\begin{eqnarray}
\sigma_n^2=2\log \Big(1- \frac{m}{n+r-p} \Big)- 2\log \Big(1- \frac{m}{n-p}\Big),	\label{eq:sigmaordercase1}
\end{eqnarray} and $m/(n+r-p)\to \gamma \in (0,1/2]$, we know $\sigma_n^2$ has leading order $\log\{1-m/(n-p)\}$.  
Then as $nh\sigma_n=O(1)$ by definition,  we also know
$
 nh \times [-\log \{1 -{m}/{(n-p)}\}]^{1/2} = O(1), 
$ which satisfies the assumption of Lemma 5.4 in  \cite{jiang2013}. Following the  lemma, we have 
\begin{align}
\log \frac{\Gamma_m\{\frac{1}{2}n(1+h) -\frac{1}{2}p\}}
{\Gamma_m\{\frac{1}{2}(n-p)\}} =~& \log \frac{\Gamma_m\{\frac{1}{2}(n-p) +\frac{1}{2}nh\}}
{\Gamma_m\{\frac{1}{2}(n-p)\}}\nonumber \\
= &-\Big\{\frac{n^2h^2}{4}+\frac{nh}{2}\Big(n-m-p-\frac{1}{2}\Big)\Big\}\log\Big(1-\frac{m}{n-p}\Big)\nonumber\\
~&+\frac{mnh}{2}\{\log(n-p)-\log 2e\}+o(1), \label{5.1}
\end{align}
and similarly, we can obtain
\begin{align}
\log \frac{\Gamma_m\{\frac{1}{2} (n+r-p)\}}{\Gamma_m\{\frac{1}{2}n(1+h) +\frac{1}{2}(r-p)\}} = & \log \frac{\Gamma_m\{\frac{1}{2} (n+r-p)\}}{\Gamma_m\{\frac{1}{2}(n+r-p) +\frac{1}{2}nh \}} \nonumber\\
=   & \Big\{\frac{n^2h^2}{4}+\frac{nh}{2}\Big(n+r-m-p-\frac{1}{2}\Big)\Big\}\log\Big(1-\frac{m}{n+r-p}\Big)\nonumber \\
~&-  \frac{mnh}{2}\{\log(n+r-p)-\log 2e\}+o(1). \label{5.2}
\end{align}
Combining \eqref{eq:generatingfunctionexp}, \eqref{5.1} and \eqref{5.2}, we have 
\begin{align}
\log E\exp\left\{\frac{\log L_n}{n\sigma_n/2}s\right\}
=~& \frac{n^2h^2}{4}\log\frac{(n+r-p-m)(n-p)}{(n-p-m)(n+r-p)}
+ \frac{h\mu_n}{2}  +o(1)\nonumber\\
=~& \frac{s^2}{2}+ \frac{h\mu_n}{2} +o(1),\nonumber
\end{align}
where 
\begin{align*}\mu_n
=~& {n(n-m-p-1/2)}
\log\frac{(n+r-p-m)(n-p)}{(n-p-m)(n+r-p)}+ {nr}\log\frac{(n+r-p-m)}{(n+r-p)}\\
~&+  {nm}\log\frac{(n-p)}{(n+r-p)}.
\end{align*}
Therefore, $\log E\exp\left\{\frac{\log L_n -\mu_n/2}{n\sigma_n/2}s\right\}=s^2/2+o(1)$ is proved.

\paragraph{Case 1.2} We discuss the case when $m/(n+r-p)\to 0$ and $m/(n-p)\to 0$ below. 
By  Lemma \ref{prop:approxgamma}, we know that when  $n-p\rightarrow \infty$ and  $r \rightarrow \infty$,
\begin{align}
& \log \frac{\Gamma_m\{\frac{1}{2}n(1+h)-\frac{1}{2}p\}}{\Gamma_m\{\frac{1}{2}(n-p) \}} \notag \\
=  &- \Big\{2m+\Big(n-p-m-\frac{1}{2}\Big)\log\Big(1-\frac{m}{n-p}\Big)\Big\} \frac{nh}{2}\nonumber\\
 & -\Big\{\frac{m}{n-p}+\log\Big(1-\frac{m}{n-p}\Big)\Big\} \frac{n^2h^2}{4}\nonumber \\
 & +m\Big\{ \frac{(n-p+nh)}{2}\log \frac{(n-p+nh)}{2}-\frac{(n-p)}{2}\log \frac{(n-p)}{2} \Big\} +o(1), \label{eq:approximation1}
\end{align}
and
\begin{align}
&~\log \frac{\Gamma_m\{\frac{1}{2} (n+r-p)\}}{\Gamma_m\{\frac{1}{2}n(1+h) +\frac{1}{2}(r-p)\}}\nonumber \\
=~& 
\Big\{2m+\Big(n+r-p-m-\frac{1}{2}\Big)\log\Big(1-\frac{m}{n+r-p}\Big)\Big\} \frac{nh}{2}\nonumber\\
 & -m\Big\{ \frac{(n+r-p+nh)}{2}\log \frac{(n+r-p+nh)}{2}-\frac{(n+r-p)}{2}\log \frac{(n+r-p)}{2} \Big\}\nonumber \\
 & +\Big\{ \frac{m}{n+r-p}+\log\Big(1-\frac{m}{n+r-p}\Big)\Big\} \frac{n^2h^2}{4}+o(1). \label{5.3}
\end{align}
By Taylor expansion of the $\log$ function,  we have
\begin{eqnarray}
\sigma_n^2 &=& 2\log \left(1-\frac{m}{n+r-p}\right)
-2\log \left(1-\frac{m}{n-p}\right) \notag \\
&=&\frac{2mr}{(n-p)(n+r-p)}\{1+o(1)\}, \label{5.4}
\end{eqnarray}
where  the second order terms of Taylor expansion of the $\log$  functions is ignorable as $m=o(n-p)$. 
Also, as $r\to \infty$, 
\begin{eqnarray}
	h =\frac{s}{n\sigma_n/2} = \frac{s\sqrt{2(n-p)(n+r-p)}}{n\sqrt{mr}}\{1+o(1)\} \to 0. \label{eq:happrox0}
\end{eqnarray}
Therefore, combining \eqref{eq:generatingfunctionexp},  \eqref{eq:approximation1} and  \eqref{5.3}, we obtain
\begin{eqnarray}
&&  \log E\exp\left\{\frac{\log L_n}{n\sigma_n/2}s\right\} \label{5.5} \\
&=& \frac{n^2h^2}{4}\log\frac{(n+r-p-m)(n-p)}{(n-p-m)(n+r-p)} 
+\frac{n^2h^2}{4}\left(\frac{m}{n+r-p} -\frac{m}{n-p}\right)\nonumber\\
&& + \frac{nh}{2} (n-m-p-1/2)\log\frac{(n+r-p-m)(n-p)}{(n-p-m)(n+r-p)}\nonumber\\
&& + \frac{nh}{2}  r\log\frac{(n+r-p-m)}{(n+r-p)}
+ \frac{nh}{2} m \log\frac{n-p+nh}{n+r-p+nh}\nonumber\\
&& + \frac{m(n+r-p)}{2} \log\frac{n+r-p}{n+r-p+nh}\nonumber\\
&& + \frac{m(n-p)}{2} \log\frac{n-p+nh}{n-p} +o(1). \notag 
\end{eqnarray}

We then analyze the terms in \eqref{5.5} separately. By \eqref{eq:happrox0},
\begin{eqnarray}
&&\frac{n^2h^2}{4}\left(\frac{m}{n+r-p} -\frac{m}{n-p}\right) \notag \\
&= &-\frac{s^2(n-p)(n+r-p)}{2mr}\times\frac{mr}{(n-p)(n+r-p)}\{1+o(1)\}\notag \\
& =&-\frac{s^2}{2}+o(1). \label{5.6}
\end{eqnarray}
In addition, as $nh/(n-p) \to 0$ and $nh/(n+r-p) \to 0$,
we have
\begin{eqnarray}
	&& \frac{m(n+r-p)}{2} \log\frac{n+r-p}{n+r-p+nh} \label{eq:rm1norm} \\
	&=&- \frac{m(n+r-p)}{2} \left\{\frac{nh}{n+r-p}-\frac{n^2h^2}{2(n+r-p)^2}+R_{n,1}\right\}, \notag
\end{eqnarray} and
\begin{eqnarray}
	&& \frac{m(n-p)}{2} \log\frac{n-p+nh}{n-p} \label{eq:rm2norm} \\
	&=& \frac{m(n-p)}{2} \Big\{ \frac{nh}{n-p} -\frac{n^2h^2}{2(n-p)^2}+R_{n,2} \Big\}, \notag 
\end{eqnarray}
where the remainder terms
\begin{eqnarray}
	R_{n,1}=\sum_{k=3}^{\infty} \frac{1}{k} (-1)^{k+1} \frac{(nh)^k}{(n+r-p)^k},\quad R_{n,2}=\sum_{k=3}^{\infty} \frac{1}{k} (-1)^{k+1} \frac{(nh)^k}{(n-p)^k}. \label{eq:remainnormaldiff}
\end{eqnarray} 
Then we have
\begin{eqnarray}
	&& \eqref{eq:rm1norm}+\eqref{eq:rm2norm} \notag \\
	&=&\frac{mn^2h^2}{4(n+r-p)}-\frac{mn^2h^2}{4(n-p)}-\frac{m(n+r-p)}{2}R_{n,1}+\frac{m(n-p)}{2}R_{n,2} \notag \\
	&=& -\frac{s^2}{2} +o(1), \label{eq:c1314sum}
\end{eqnarray} where in the last equation, we use \eqref{5.6} and Lemma \ref{lm:remainnormaldiff}. 
Furthermore, by $nh/(n+r-p) \to 0$ and \eqref{eq:happrox0}, we have 
\begin{align}
&~\frac{nh}{2} m \log\frac{n-p+nh}{n+r-p+nh} \notag \\
=&~ \frac{nh}{2} m \log\left\{\frac{n-p}{n+r-p}+\frac{nhr}{(n+r-p)(n+r-p+nh)}\right\}\nonumber \\
=&~ \frac{nh}{2} m \log\frac{n-p}{n+r-p}+\frac{nmh}{2} \frac{n+r-p}{n-p} \frac{nhr}{(n+r-p)(n+r-p+nh)}+o(1)\nonumber\\
=&~ \frac{nh}{2} m \log\frac{n-p}{n+r-p}+s^2+o(1). \label{5.8}
\end{align}
Combining \eqref{5.5}, \eqref{5.6}, \eqref{eq:c1314sum} and  \eqref{5.8},  we obtain $\log E\exp\left\{\frac{\log L_n -\mu_n/2}{n\sigma_n/2}s\right\}
=s^2/2+o(1).$

\paragraph{Case 1.3} 
When $m/(n+r-p)\to 0$ and $m/(n-p)\to\gamma\in(0,1],$  we know \eqref{5.1} still holds following similar analysis to Case 1.1. And  \eqref{5.3} also holds following similar analysis to Case 1.2. To establish \eqref{eq:proofgoal}, we next show that under this case, the difference between the result of \eqref{5.2} and \eqref{5.3} is ignorable. 
\begin{align}
~& \eqref{5.3}-\eqref{5.2} \notag \\
=~&2m \frac{nh}{2}+\frac{mnh}{2}\left \{ \log(n+r-p)-\log2e \right\}  +\frac{n^2h^2}{4}\times \frac{m}{n+r-p}\nonumber \\
~&- \frac{m(n+r-p)}{2} \log \frac{n+r-p+nh}{n+r-p} -\frac{mnh}{2} \log \frac{n+r-p+nh}{2}+o(1)\nonumber \\
=~&  mnh+\frac{mnh}{2}\left \{ \log \Big(\frac{n+r-p}{2}\Big)-1 \right\}+\frac{n^2h^2}{4}\times \frac{m}{n+r-p} \nonumber \\
~& -\frac{m(n+r-p)}{2} \log \left(1+\frac{nh}{n+r-p}  \right) - \frac{mnh}{2} \log \left( \frac{n+r-p}{2} +\frac{nh}{2}\right)+o(1). \label{5.9}
\end{align} We then analyze the terms in \eqref{5.9} separately.

Since $m/(n-p)\rightarrow \gamma\in (0,1]$, similarly to \eqref{eq:sigmaordercase1}, we know that $nh=2s/\sigma_n=O(s)$. As $m/(n+r-p)\to 0$, it follows that $n^2h^2m/(n+r-p)\to 0$.   Applying Taylor expansion,  we then have 
\begin{eqnarray}
& & \frac{mnh}{2} \log \left( \frac{n+r-p}{2} +\frac{nh}{2}  \right) \notag \\
&= &\frac{mnh}{2}  \left\{ \log \left(\frac{n+r-p}{2} \right)   + O\Big( \frac{nh}{n+r-p}  \Big)  \right\}\nonumber \\
&= &\frac{mnh}{2} \log \left( \frac{n+r-p}{2} \right)  + O\Big(\frac{mn^2h^2}{n+r-p} \Big)\nonumber \\
&= &\frac{mnh}{2}\log \left( \frac{n+r-p}{2} \right) + o(1). \label{5.10}
\end{eqnarray}
Similarly, by $nh=O(s)$, $m/(n+r-p)\to 0$, and Taylor  expansion, we have
\begin{align}
	~& \frac{m(n+r-p)}{2} \log \left(1+\frac{nh}{n+r-p}  \right) \notag \\
	 ~=&\frac{m(n+r-p)}{2} \left\{ \frac{nh}{n+r-p} + O\Big(\frac{n^2h^2}{(n+r-p)^2} \Big) \right\}\nonumber\\ ~=&\frac{mnh}{2} + o(1). \label{5.11}
\end{align}
In summary, combining  \eqref{5.10} and \eqref{5.11}, we have
$
	\eqref{5.9} =\eqref{5.3}-\eqref{5.2}=o(1). $
Then by the results in Case 1.1, we  get the same conclusion as in Case 1.1. 

\paragraph{Case 2. When $m>r$, $m\to\infty$.}
According to Lemma  \ref{prop:betaapprox}, we can make the following substitution 
$$m\to r,\quad r\to m,\quad n-p\to n+r-p-m. $$
 Then the substituted mean and variance are
\begin{align*}\mu_n
=~& {n(n-p-m-1/2)}
\log\frac{(n-p)(n-p+r-m)}{(n-p-m)(n+r-p)}+ {nm}\log\frac{(n-p)}{(n+r-p)}\\
~&+  {nr}\log\frac{(n-p+r-m)}{(n+r-p)},
\end{align*}
and $$\sigma_n^2 = 2\log\frac{(n-p)(n-p+r-m)}{(n-p-m)(n+r-p)},$$
which take the same forms as those in the  setting when $r\geq m$. And the theorem can be proved following similar analysis when $m\to \infty$, $n-p+r-m\to \infty$. 
\end{proof}

\subsection{Lemmas in the proof of Theorem \ref{thm:mainlimit}} \label{sec:lemmasthm1}

\begin{lemma}[Corollary 10.5.2 in \cite{muirhead2009aspects}] \label{lm:momentln}
Under the null hypothesis, $L_n$'s $h$-th moment can be written as
\begin{eqnarray*}
	\mathrm{E} (L_n^h)= \frac{\Gamma_m\{\frac{1}{2}n(1+h) -\frac{1}{2}p\}\Gamma_m\{\frac{1}{2} (n+r-p)\}}
{\Gamma_m\{\frac{1}{2}(n-p)\}\Gamma_m\{\frac{1}{2}n(1+h) +\frac{1}{2}(r-p)\}}.
\end{eqnarray*}	
\end{lemma}
 
\begin{lemma}[Theorem 10.5.3 in \cite{muirhead2009aspects}] \label{prop:betaapprox}
Under the null hypothesis, when $n-p\geq m$ and $r\geq m$,
$\frac{2}{n}\log L_n$ has the same distribution as $\sum_{i=1}^m\log V_i$, where $V_i$'s are independent random variables and $V_i\sim \mbox{beta}(\frac{1}{2}(n-p-i+1), \frac{1}{2}r)$;
 when $n-p\geq m\geq r$,
$\frac{2}{n}\log L_n$ has the same distribution as $\sum_{i=1}^r\log V_i$, where $V_i$'s are independent  and $V_i\sim \mbox{beta}(\frac{1}{2}(n+r-p-m-i+1), \frac{1}{2}m).$
\end{lemma}

\begin{lemma}[Theorem 2.1.12 in \cite{muirhead2009aspects}] \label{lm:lmgamma}
The multivariate Gamma function defined in \eqref{eq:multigammadef} can be written as
$$\Gamma_m(a) = \pi^{m(m-1)/4}\prod_{j=1}^m \Gamma(a-(j-1)/2).$$ 	
\end{lemma}

\begin{lemma} \label{lm:firstlemma}
Consider $m$ is fixed and $a \rightarrow \infty$. We have
\begin{eqnarray}
	\frac{1}{a-1} \sum_{i=1}^m \frac{i-1}{a-i} &=&\Big\{\frac{1}{2}\Big(\sigma_a^2-\frac{m}{(a-1)^2} \Big)\Big\}\{1+O(1/a)\}, \label{eq:sigmaestimate}\\
	\sum_{i=1}^m \{\log(a-1)-\log(a-i)\} &=& - \mu_a +O\Big(\frac{m^2}{a^2}\Big), \label{eq:muestimate}
\end{eqnarray}  where $\mu_a=-(m-a+3/2)\log\{1-m/(a-1)\}+(a-1)m/a$ and $\sigma_a^2=-2[m/(a-1)+\log\{1-m/(a-1)\}]$.
\end{lemma}
\begin{proof}
We first prove \eqref{eq:sigmaestimate}. As $m$ is fixed and $a\to \infty$, we have
\begin{eqnarray*}
	\sigma_a^2= -2\Big[ \frac{m}{a-1}+\log\Big(1-\frac{m}{a-1} \Big) \Big] = \Big( \frac{m}{a-1}\Big)^2 \{1+O(m/a)\},
\end{eqnarray*} 
and
 \begin{eqnarray*}
 \frac{1}{a-1} \sum_{i=1}^m \frac{i-1}{a-i}= \frac{1}{a-1} \sum_{i=1}^m \frac{i-1}{a-1} + \epsilon_a  = \frac{m(m-1)}{2(a-1)^2} + \epsilon_a,
 \end{eqnarray*} where $|\epsilon_a|\leq 2(a-1)^{-3}\sum_{i=1}^m(i-1)^2 \leq 3(m/a)^3$.  Therefore,
\begin{eqnarray*}
  \frac{1}{a-1} \sum_{i=1}^m \frac{i-1}{a-i} &=& \frac{m(m-1)}{2(a-1)^2} \Big\{1+O\Big(\frac{m}{a}\Big)\Big\} \notag \\
  &=& \Big[\frac{1}{2}\Big\{\sigma_a^2-\frac{m}{(a-1)^2} \Big\}\Big]\{1+O(1/a)\},
\end{eqnarray*} where in the last equation, we use the fact that $O(m/a)=O(1/a)$ as $m$ is fixed. Then \eqref{eq:sigmaestimate} is proved. 

We then prove \eqref{eq:muestimate}. Recall Stirling formula, \citep[see, e.g., p. 368][]{gamelin2003complex}
\begin{eqnarray*}
	\log \Gamma(x)=(x-1/2)\log x- x + \log \sqrt{2\pi}+\frac{1}{12x}+O(x^{-3})
\end{eqnarray*} as $x\rightarrow \infty$.  Therefore,
\begin{eqnarray*}
	&& \log \Gamma(a-1)-\log \Gamma(a-m-1) \notag \\
	&=& (a-3/2) \log (a-1)-(a-m-3/2)\log(a-m-1)-m \notag \\
	&& +\frac{1}{12} \Big(\frac{1}{a-1}-\frac{1}{a-m-1}\Big)+O(a^{-3}) \notag \\
	&=& (a-3/2)\log(a-1)-(a-m-3/2)\log(a-m-1)-m+O(ma^{-2}).
\end{eqnarray*} Since for integers $k \geq 1$, $\Gamma(k)=(k-1)!=\Pi_{i=1}^{k-1}i$. Then we have
\begin{eqnarray*}
	&& \sum_{i=1}^m \{\log(a-1)-\log(a-i)\} \notag \\
	&=& m \log (a-1)- \{\log \Gamma(a-1) - \log \Gamma(a-m-1) \}+ \log \Big(1-\frac{m}{a-1}\Big) \notag \\
	&=& m \log (a-1) - (a-3/2)\log(a-1)+(a-m-3/2)\log(a-m-1)\notag \\
	&& +m+ \log\Big(1-\frac{m}{a-1}\Big) +O(ma^{-2}) \notag \\
	&=& -(m-a+3/2) \log \Big(1-\frac{m}{a-1}\Big)+m-\frac{m}{a-1}+O\Big(\frac{m^2}{a^2}\Big) \notag \\
	&=& -(m-a+3/2) \log \Big(1-\frac{m}{a-1}\Big)+\frac{a-1}{a}m+O\Big(\frac{m^2}{a^2}\Big)\notag \\
	&=& - \mu_a +O\Big(\frac{m^2}{a^2}\Big),
\end{eqnarray*} where in the last two equations, we use the fact that $\frac{a-2}{a-1}m=\frac{a-1}{a}m+ O(ma^{-2})$.
\end{proof}

\begin{lemma} \label{lm:lemmaongexp}
Consider $m$ is fixed and $a \rightarrow \infty$. Define
\begin{eqnarray*}
	g_i(x)=\Big( \frac{a-i}{2}+x\Big)\log\Big( \frac{a-i}{2}+x\Big)-\Big( \frac{a-1}{2}+x\Big) \log \Big( \frac{a-1}{2}+x\Big)
\end{eqnarray*} 	for $1\leq i\leq m$ and $x>-(a-m)/2$. Let $\mu_a$ and $\sigma_a$ be as in Lemma \ref{lm:firstlemma}. If  $t=o(a)$ and $mt^2/a^2=o(1)$, we have that as $a\rightarrow \infty$,
\begin{eqnarray*}
	\sum_{i=1}^m \{g_i(t)-g_i(0)\}=\mu_a t + \frac{\sigma_a^2}{2} t^2 +o(1).
\end{eqnarray*}
\end{lemma}

\begin{proof}
We know for $1\leq i \leq m$,
\begin{eqnarray*}
  g_i'(x) &=& \log ( \frac{a-i}{2} +x)- \log ( \frac{a-1}{2} +x), \notag \\
    g_i^{''}(x)&=& \frac{1}{\frac{a-i}{2}+x}-\frac{1}{\frac{a-1}{2}+x}=\frac{ \frac{i-1}{2} }{(\frac{a-i}{2}+x)(\frac{a-1}{2}+x)},  \notag \\
  g_i^{(3)} (x) &=& -\frac{1}{(\frac{a-i}{2}+x)^2}+\frac{1}{(\frac{a-1}{2}+x)^2} \notag \\
  &=& - \frac{\frac{i-1}{2}\cdot \frac{2a-i-1}{2}+(i-1)x}{(\frac{a-i}{2}+x)^2(\frac{a-1}{2}+x)^2}. 
\end{eqnarray*}   By Taylor expansion,
\begin{eqnarray*}
	g_i(t)-g_i(0) &=& g'_i(0)t+\frac{t^2}{2}g_i^{''} (0) + \frac{t^3}{6}g_i^{(3)} (\xi_i) \notag \\
	&=& \{\log(a-i)-\log(a-1)\}t+ \frac{i-1}{(a-1)(a-i)}t^2 +\frac{t^3}{6}g_i^{(3)} (\xi_i).
\end{eqnarray*}

For $1\leq i\leq m$, fixed $m$ and $0\leq \xi_i \leq t=o(a)$, we have 
$
	\sup_{|\xi_i|\leq |t|, 1\leq i \leq m}|g_i^{(3)} (\xi_i)| \leq ca^{-3},
$ where $c$ denotes an universal constant. Therefore, as $t=o(a)$,
$
	|t^3g_i^{(3)} (\xi_i)|\leq c t^3 a^{-3}=o(1). 
$
In addition, by Lemma \ref{lm:firstlemma}, and the fact that $mt^2/(a-1)^2= o(1)$,  we have as $a \to \infty$,
\begin{eqnarray*}
	\sum_{i=1}^m \{g_i(t)-g_i(0)\}&=& \mu_a t + \Big[\frac{1}{2}\Big(\sigma_a^2-\frac{m}{(a-1)^2} \Big)\Big] t^2 +o(1) \notag \\
	&=& \mu_a t +  \frac{\sigma_a^2}{2}t^2+o(1).
\end{eqnarray*}
\end{proof}

\begin{lemma} \label{prop:approxgamma}
Consider $n-p \to \infty$, $r\to \infty$,  $m/(n-p)\to 0$ and $m/(n-p+r)\to 0$.  For $t=nh/2$, $a=n-p+r$ or $a=n-p$, we have
	\begin{eqnarray*}
		\log  \frac{\Gamma_m(\frac{a-1}{2}+t)}{\Gamma_m(\frac{a-1}{2})}=\upsilon_a t +\vartheta_a t^2+\gamma_a(t)+o(1),
	\end{eqnarray*} where
\begin{align*}
		& \upsilon_a=-[2m+(a-m-3/2)\log \{1-m/(a-1)\}];\ \vartheta_a=-[m/(a-1)+\log\{1-m/(a-1)\}]; \notag \\
		& \gamma_a(t)=m\Big\{\Big(\frac{a-1}{2}+t\Big)\log  \Big(\frac{a-1}{2}+t\Big)-\frac{a-1}{2} \log  \Big(\frac{a-1}{2}\Big) \Big\}.  
	\end{align*}
\end{lemma}
\begin{proof}
By Lemma \ref{lm:lmgamma}, we know
\begin{eqnarray}
	\log  \frac{\Gamma_m(\frac{a-1}{2}+t)}{\Gamma_m(\frac{a-1}{2})}= \sum_{i=1}^m \log \frac{\Gamma ( \frac{a-i}{2} +t )}{ \Gamma(\frac{a-i}{2} )}.\label{eq:loggammaexp}
\end{eqnarray} To prove the lemma, we expand each summed term in \eqref{eq:loggammaexp}, $\log\{ \Gamma(\frac{a-i}{2}+t) /\Gamma(\frac{a-i}{2})\}$,  by Lemma A.1. in \cite{jiang2015likelihood}. To apply the lemma, we first need to check the condition that for each $1\leq i\leq m$, $t\in [-\delta (a-i)/2, \delta (a-i)/2]$ for any given $\delta \in (0,1)$.

Recall that we previously define $nh=2s/\sigma_n$ in Section \ref{sec:mainproofthm1}. Then $t=nh/2=s \sigma_n^{-1}$. 
Note that when $m/(n-p)$ and $ m/(n-p+r)\rightarrow 0$,
\begin{eqnarray*}
\sigma_n^2 &=& \frac{1}{2}\log \left(1-\frac{m}{n+r-p}\right)
-\frac{1}{2}\log \left(1-\frac{m}{n-p}\right) \notag \\
&=&\frac{mr}{2(n-p)(n+r-p)}\{1+o(1)\}. 
\end{eqnarray*} Thus we have 
\begin{eqnarray}
	t= O(s)  \sqrt{ \frac{(n-p)(n-p+r)}{mr} }. \label{eq:tordernpr}
\end{eqnarray}
For $a=n-p+r$ or $a=n-p$, and $1\leq i\leq m$, by \eqref{eq:tordernpr}, we then have
\begin{eqnarray*}
	\frac{t}{a-i} &\leq & \frac{t}{n-p-m} =O(s)\sqrt{ \frac{(n-p)(n-p+r)}{mr (n-p-m)^2} } \notag \\
	&=&O(s) \sqrt{ \frac{1+r/(n-p)}{mr \{1-m/(n-p)\}^2} } \notag \\
	&=&O(s) \sqrt{  \Big\{\frac{1}{mr}+\frac{1}{m(n-p)}\Big\}\{1+o(1)\} } =o(1),
\end{eqnarray*} where the last two equations follow from the condition that $m/(n-p) \to 0, r \to \infty$ and $n-p \to \infty$. Then we know that for each $1\leq i\leq m$, $t\in [-\delta (a-i)/2, \delta (a-i)/2]$ for any given $\delta \in (0,1)$. 


Therefore, the condition  of Lemma A.1. in \cite{jiang2015likelihood} is satisfied. By that lemma,  we know when $a\rightarrow \infty$, for uniformly $1\leq i\leq m$, 
\begin{eqnarray*}
	\log \frac{\Gamma ( \frac{a-i}{2} +t )}{ \Gamma (\frac{a-i}{2} )}&=& \Big(\frac{a-i}{2} +t\Big)\log \Big(\frac{a-i}{2}+t \Big)-\frac{a-i}{2} \log \frac{a-i}{2} \notag \\
	&& -t - \frac{t}{a-i} +O\Big(\frac{t^2}{a^2} \Big).
\end{eqnarray*} 
Write $
	\frac{t}{a-i}=\frac{t}{a}+\frac{t}{a}\times  \frac{i}{a-i}. 
$ Then similarly to Lemma \ref{lm:firstlemma}, we have
\begin{eqnarray}
	\sum_{i=1}^m \frac{t}{a-i} &=& \frac{mt}{a} +  \frac{tm(m+1)}{2a(a-1)}  + O\Big\{ \frac{t}{a}\times \Big (\frac{m}{a} \Big)^3\Big\}.  \label{eq:sumexpansion}
\end{eqnarray}
For $a=n-p$, by  \eqref{eq:tordernpr}, $m/(n-p)\to 0$ and $m\leq r$,
\begin{eqnarray*}
	\frac{tm(m+1)}{a(a-1)} &=& O(s) \sqrt{\frac{(n-p)(n-p+r)}{mr} } \frac{m^2}{(n-p)^2} \notag \\
	&=& O(s) \sqrt{ \frac{m(n-p+r)}{r(n-p)} } \frac{m}{n-p} \notag \\
	&=& O(s)\sqrt{ \frac{m}{ \min \{n-p,r \}} }  \frac{m}{n-p}=o(1).
\end{eqnarray*} For $a=n-p+r$,  similar conclusion,  $tm(m+1)/\{a(a-1)\}=o(1)$, holds by substituting $n-p$ with $n-p+r$. In addition, for $a=n-p$ or $a=n-p+r$, by \eqref{eq:tordernpr},
\begin{eqnarray}
	\frac{t}{a} \leq \frac{t}{n-p}&=& O(s) \sqrt{\frac{  \max\{n-p,r \} } {m r(n-p) }  } \notag \\
	&=& O(s)  \sqrt{ \frac{1}{ m \times \min \{n-p,r \}} } = o(1).  \label{eq:ratiotnminp}
\end{eqnarray} Then based on \eqref{eq:sumexpansion} and   \eqref{eq:ratiotnminp}, we obtain
\begin{eqnarray*}
	\sum_{i=1}^m \Big\{ -t -\frac{t}{a-i}+O(t^2/a^2) \Big\}= -mt- \frac{mt}{a}+ o(1). 
\end{eqnarray*}  

  Therefore, from \eqref{eq:loggammaexp}, we have 
\begin{eqnarray}
	& & \log  \frac{\Gamma_m(\frac{a-1}{2}+t)}{\Gamma_m(\frac{a-1}{2})} \label{eq:sumtermverify} \\
	&=& - \frac{(a+1)mt}{a} + \sum_{i=1}^m \Big\{\Big(\frac{a-i}{2} +t\Big)\log \Big(\frac{a-i}{2}+t\Big)-\frac{a-i}{2} \log \frac{a-i}{2}  \Big\} + o(1). \notag
\end{eqnarray} 
For $1\leq i \leq m$, define the function
\begin{eqnarray*}
	g_i(x)=\Big(\frac{a-i}{2} +x \Big)\log \Big(\frac{a-i}{2}+x \Big)- \Big(\frac{a-1}{2} +x \Big)\log \Big(\frac{a-1}{2}+x \Big),
\end{eqnarray*} and $x>-(a-m)/2$. 
We then know that the summation term ``$\sum$"  in \eqref{eq:sumtermverify} equals to
\begin{align}
	 m\Big[ \Big(\frac{a-1}{2}+t \Big)\log \Big(\frac{a-1}{2}+t\Big) - \frac{a-1}{2} \log \frac{a-1}{2} \Big] +\sum_{i=1}^m \{g_i(t)-g_i(0)\}.  \label{eq:44sumterm}
\end{align} 

We then examine the function $\sum_{i=1}^m \{g_i(t)-g_i(0)\}$ in \eqref{eq:44sumterm}. Note that by \eqref{eq:ratiotnminp}, we know $t=o(a)$,  $mt^2/a^2=o(1)$ and $mt/a=O(1)$ as $m<n-p$ and $m\leq r$. Thus the conditions of Lemma \ref{lm:lemmaongexp} and  Lemma A.3. in \cite{jiang2015likelihood} are satisfied when $m$ is fixed and $m\to \infty$ respectively. When $m$ is fixed, we apply Lemma \ref{lm:lemmaongexp}; when $m\to \infty$, we apply Lemma A.3. in \cite{jiang2015likelihood}. Then we obtain
\begin{eqnarray*}
	\sum_{i=1}^m \{g_i(t)-g_i(0)\}= \mu_a t + \frac{1}{2} \sigma_a^2 t^2 +o(1),
\end{eqnarray*} where 
\begin{eqnarray*}
	\mu_a &=&(m-a+3/2) \log \Big(1-\frac{m}{a-1}\Big)-m\frac{a-1}{a}, \notag \\
	\sigma_a^2 &=& -2\Big[ \frac{m}{a-1}+\log\Big(1-\frac{m}{a-1} \Big) \Big]. 
\end{eqnarray*}
  Therefore, the proposition can be proved by noticing
\begin{eqnarray*}
   && \upsilon_a = - \frac{(a+1)m}{a}+ \mu_a ;\quad	\vartheta_a=\sigma_a^2/2; \notag \\
   && \gamma_a(t)=m\Big[ \Big(\frac{a-1}{2}+t\Big)\log \Big(\frac{a-1}{2}+t\Big) - \frac{a-1}{2} \log \frac{a-1}{2} \Big].
\end{eqnarray*}
\end{proof}


\begin{lemma} \label{lm:remainnormaldiff}
Under Case 1 in Section \ref{sec:mainproofthm1}, $R_{n,1}$ and $R_{n,2}$ defined in \eqref{eq:remainnormaldiff} satisfy
\begin{eqnarray*}
	-\frac{m(n+r-p)}{2}R_{n,1}+\frac{m(n-p)}{2}R_{n,2}=o(1).
\end{eqnarray*}
\end{lemma}

\begin{proof} Note that
\begin{eqnarray}
	&& -\frac{m(n+r-p)}{2}R_{n,1}+\frac{m(n-p)}{2}R_{n,2} \notag \\
	&=& \frac{m}{2}\Big[ \sum_{k=3}^{\infty} \frac{1}{k} (-nh)^k \Big\{\frac{1}{(n+r-p)^{k-1}}-\frac{1}{(n-p)^{k-1}} \Big\}\Big] \notag \\
	&=& \frac{m}{2} \Big\{ \sum_{k=3}^{\infty} \frac{1}{k} (-nh)^k\frac{-\sum_{q=1}^{k-1}\binom{k-1}{q}r^q(n-p)^{k-1-q}}{(n+r-p)^{k-1}(n-p)^{k-1}}\Big\} \notag \\
	&=& \frac{mnh}{2} \sum_{k=3}^{\infty} \frac{1}{k} \Big(\frac{-nh}{n+r-p}\Big)^{k-1}\sum_{q=1}^{k-1}\binom{k-1}{q}\Big(\frac{r}{n-p}\Big)^q. \label{eq:rnormremain1}\end{eqnarray}

If $r/(n-p)=1$, 
\begin{eqnarray*}
	| \eqref{eq:rnormremain1}| \leq {mnh} \sum_{k=3}^{\infty} \Big( \frac{2nh}{n+r-p} \Big)^{k-1}=O\Big\{ mnh \frac{n^2h^2}{(n+r-p)^2}  \Big\},
\end{eqnarray*} where
\begin{eqnarray*}
\frac{mnh\times n^2h^2}{(n+r-p)^2}&=&O\Big\{\frac{m\sqrt{(n-p)(n+r-p)}}{\sqrt{mr}}\times \frac{(n-p)(n+r-p)}{mr(n+r-p)^2}\Big\} \notag \\
&=& O\Big\{  \frac{m\sqrt{r^2}}{\sqrt{mr}} \times \frac{r^2}{mr \times r^2}  \Big\}=o(1),
\end{eqnarray*} as $r\to \infty$.

If $r/(n-p)>1$, as $\{nh/(n+r-p)\}\times \{r/(n-p)\}=O\{\sqrt{r} /\sqrt{m(n-p)(n+r-p)}\}=o(1)$,
\begin{eqnarray*}
	| \eqref{eq:rnormremain1}| \leq mnh \sum_{k=3}^{\infty} \Big( \frac{2nh}{n+r-p} \times \frac{r}{n-p} \Big)^{k-1} =O\Big\{ mnh \Big(\frac{2nh}{n+r-p}\Big)^2 \Big(\frac{r}{n-p} \Big)^2\Big\},
\end{eqnarray*}	where
\begin{eqnarray*}
	&& mnh \Big(\frac{nh}{n+r-p}\Big)^2 \Big(\frac{r}{n-p} \Big)^2 \notag \\
	&=& O\Big\{\frac{m\sqrt{(n-p)(n+r-p)}}{\sqrt{mr}}\times \frac{(n-p)(n+r-p)}{mr (n+r-p)^2}\times \frac{r^2}{(n-p)^2}  \Big\} \notag \\
	&=& O\Big\{  \frac{\sqrt{(n-p)(n+r-p)}}{\sqrt{mr}} \frac{r}{(n-p)(n+r-p)}  \Big\} \notag \\
	&=& O\Big\{ \frac{r}{\sqrt{mr(n+r-p)(n-p)}} \Big\} =o(1),
\end{eqnarray*} as $n+r-p\geq r$ and $n-p\to \infty$.

If $r/(n-p)<1$,
\begin{eqnarray*}
	| \eqref{eq:rnormremain1}| \leq  mnh \sum_{k=3}^{\infty} \Big( \frac{nh}{n+r-p} \Big)^{k-1} \frac{r}{(n-p)}=O\Big\{mnh \frac{(nh)^2}{(n+r-p)^2} \times \frac{r}{(n-p)} \Big\},
\end{eqnarray*} where
\begin{eqnarray*}
	&& mnh \frac{(nh)^2}{(n+r-p)^2} \times \frac{r}{(n-p)} \notag \\
	&=& O\Big\{\frac{m\sqrt{(n-p)(n+r-p)}}{\sqrt{mr}}\times \frac{(n-p)(n+r-p)}{mr (n+r-p)^2}\times \frac{r}{(n-p)}  \Big\} \notag \\
	&=& O\Big\{ \frac{\sqrt{n-p}}{\sqrt{mr(n+r-p)}} \Big\}=o(1).
\end{eqnarray*}
\end{proof}


\section{Theorem \ref{thm:poweranalysis}} \label{sec:proofpower}

We give the main proof of Theorem \ref{thm:poweranalysis} in Section \ref{sec:mainproofthmpower}, where we use some concepts of hypergeometric function, which is introduced in Section  \ref{sec:reviewhyper}, and the lemmas we use   are given and proved in Section \ref{sec:lmmasthmpower}. 

\subsection{Proof of Theorem \ref{thm:poweranalysis}}\label{sec:mainproofthmpower}
As $p/n 
= \rho_p$, $r/n = \rho_r$, $m/n= \rho_m$  with $\rho_p, \rho_r, \rho_m \in (0,1)$ and $\rho_p+\rho_m<1$,  we know that $\sigma_n^2$ in Theorem \ref{thm:mainlimit} satisfies
\begin{eqnarray*}
\sigma_n^2  = 2\log \Big\{ \Big(1-\frac{\rho_m}{1 +\rho_r-\rho_p } \Big) \Big(1-\frac{\rho_m}{1-\rho_p}  \Big)^{-1}   \Big\}, 
\end{eqnarray*} which is a positive constant, and we write the constant as $\sigma^2$. Then $T_1=\{-2\log L_n+\mu_n\}/(n\sigma)$, 
 and we examine the moment generating function $\mathrm{E}\{2s\log L_n/(n\sigma)\}$. Let $h=2s/(n\sigma)$. By Lemma \ref{lm:momentaltln}, we have
\begin{eqnarray}
&&\mathrm{E}\{2s\log L_n/(n\sigma)\} =\mathrm{E}\{\exp(h\log L_n)\}\notag \\
&=&\mathrm{E} L_n^h =  \mathrm{E}_0L_n^{h} \times 
{_1F_1}\Big( \frac{nh}{2}; \frac{1}{2}(n+r-p)+ \frac{nh}{2}; -\frac{1}{2}\Omega\Big), \label{eq:expmomgenln}
\end{eqnarray} where $\mathrm{E}_0 L_n^h$ is the moment generating function of $\log L_n$ under $H_0$, and ${_1F_1}$ is the hypergeometric function, which depends on $\Omega$ only through its eigenvalues symmetrically. 

As ${_1F_1}$ only depends on $\Omega$ via its eigenvalues symmetrically, without loss of generality, we consider the alternative with $\Omega = \mathrm{diag}(w_1,\cdots,w_m)$ and $w_1\geq\cdots\geq w_m\geq 0$. Let $\xi_a=nh/2$, $\xi_b=(n+r-p+nh)/2$ and $Q=-\Omega/2=\mathrm{diag}(-w_1/2,\cdots,-w_m/2)$, then we write $ {_1F_1}( {nh}/{2}, (n+r-p+nh)/2; -\Omega/2)$ as ${_1F_1}( \xi_a; \xi_b; Q) $. Note that we assume that $\Omega$ has fixed rank $m_0$ in Theorem \ref{thm:poweranalysis}, then $\omega_1 \geq \ldots \geq \omega_{m_0}>0$ are $m_0$ nonzero eigenvalues of $\Omega$.  Further define $\tilde{Q}=\mathrm{diag}(-\omega_1/2,\ldots,-\omega_{m_0}/2)$. By Lemma \ref{lm:hypergeomfunc}, we know ${_1F_1}( \xi_a; \xi_b; Q) ={_1F_1}( \xi_a; \xi_b; \tilde{Q})$. Then to evaluate ${_1F_1}( \xi_a; \xi_b; Q)$  when $Q$ has fixed rank, without loss of generality, we consider ${_1F_1}( \xi_a; \xi_b; \tilde{Q})$.

Let $W =\log {_1F_1}( \xi_a; \xi_b; \tilde{Q})$ and $\tilde{Q}=-n\tilde{\Delta}/2$ with $\tilde{\Delta}=\mathrm{diag}(\delta_1,\ldots,\delta_{m_0})$. From Lemma \ref{lm:diffres}, we know that $W(\tilde{\Delta})$ is the unique solution of each of the $m_0$ partial differential equations 
\begin{align}
	\Big[ \frac{1}{2}(n+r-p-m_0+1)+\frac{nh}{2}+\frac{1}{2}n\delta_j+\frac{1}{2} \sum_{i\neq j}^{m_0} \frac{\delta_j}{\delta_j-\delta_i}\Big]\frac{\partial W}{\partial \delta_j} \notag \\
	+ \delta_j \Big[ \frac{\partial^2 W}{\partial \delta_j^2} + \left(\frac{\partial W}{\partial \delta_j} \right)^2  \Big]-\frac{1}{2}\sum_{i\neq j}^{m_0} \frac{\delta_i}{\delta_j-\delta_i}\frac{\partial W}{\partial \delta_i}=-\frac{nh}{2}\times n, \label{eq:partial1sol}
\end{align} for $j=1,\ldots,m_0$, subject to the conditions that  $W(\tilde{\Delta})$ is $(a)$  a symmetric function of $\delta_1, \ldots, \delta_{m_0} $, and $(b)$ analytic at $\tilde{\Delta}=\mathbf{0}_{m_0\times m_0}$ with $W(\mathbf{0}_{m_0\times m_0})=0$. 
As $r/n=\rho_r$, $p/n=\rho_p$, $m_0$ is a fixed number and $nh=2s/\sigma$,  we can write \eqref{eq:partial1sol} into
\begin{align}
	\Big[ \frac{1}{2}(1+\rho_r-\rho_p)n+\frac{1}{2}(2s/\sigma-m_0+1)+\frac{1}{2}n\delta_j+\frac{1}{2} \sum_{i\neq j}^{m_0} \frac{\delta_j}{\delta_j-\delta_i}\Big]\frac{\partial W}{\partial \delta_j} \notag \\
	+ \delta_j \Big[ \frac{\partial^2 W}{\partial \delta_j^2} + \left(\frac{\partial W}{\partial \delta_j} \right)^2  \Big]-\frac{1}{2}\sum_{i\neq j}^{m_0} \frac{\delta_i}{\delta_j-\delta_i}\frac{\partial W}{\partial \delta_i}=-\frac{s}{\sigma}\times n. \label{eq:partialeqmatch}
\end{align}  Similarly to Theorem 10.5.6 in \cite{muirhead2009aspects}, we write $W(\tilde{\Delta})=P_0(\tilde{\Delta})+P_1(\tilde{\Delta})/n+\ldots$. Note that $nh=2s/\sigma$. Matching $n$ on both sides of \eqref{eq:partialeqmatch}, we obtain
\begin{align*}
	\Big[\frac{1}{2}(1+\rho_r-\rho_p)n+\frac{1}{2}n\delta_j \Big]\frac{\partial P_0}{\partial \delta_j} = - \frac{sn}{\sigma}.
\end{align*} Solving this subject to conditions $(a)$ and $(b)$, we obtain
\begin{align*}
	P_0(\tilde{\Delta})=-\frac{2s}{\sigma}\sum_{j=1}^{m_0} \log \Big( 1 + \frac{\delta_j}{1+\rho_r-\rho_p} \Big).
\end{align*} Then we have $W(\tilde{\Delta})=P_0(\tilde{\Delta}) +O(n^{-1})$.
From \eqref{eq:expmomgenln}, we know
\begin{eqnarray}
	\mathrm{E}L_n^h=\mathrm{E}_0L_n^h \times e^{\log{_1F_1}}
	=\mathrm{E}_0 e^{\frac{s}{n\sigma/2}\log\left( L_n\right)+ W}. \label{eq:altermomgen}  
\end{eqnarray}  
Write $W_{\Delta}=\sum_{j=1}^{m_0} \log[ 1+\delta_j( 1+\rho_r-\rho_p)^{-1}]$ and $A_1=2/\sigma$.  
\eqref{eq:proofgoal} and \eqref{eq:altermomgen}  show that  $\{\log L_n-\mu_n/2\}/(n\sigma_n/2) \xrightarrow{D} \mathcal{N}(- {A}_1W_{\Delta},1)$, and thus $\{-2\log L_n+\mu_n\}/(n\sigma) \xrightarrow{D} \mathcal{N}({A}_1W_{\Delta},1)$.  
Then the power $P(T_1>z_{\alpha})\to \bar{\Phi}(z_{\alpha}-{A}_1 W_{\Delta})$. 


\subsection{Brief review of hypergeometric function}  \label{sec:reviewhyper}

We rephrase some related definitions and results about hypergeometric function, where the details can be found in  Chapter 7 in \cite{muirhead2009aspects}. 

Let $k$ be a positive integer; a partition $\kappa$ of $k$ is written as $\kappa=(k_1,k_2,\ldots)$, where $\sum_{t} k_t=k$ and $k_1\geq k_2 \geq  \ldots$ are  non-negative integers.  In addition, let $M$ be an $m\times m$ symmetric  matrix with eigenvalues $l_1,\ldots,\l_m$, and let $\kappa=(k_1,k_2\ldots)$ be a partition of $k$ into no more than $m$ nonzero parts. We write the zonal polynomial of $M$ corresponding to $\kappa$  as $C_{\kappa}(M)$. Then by the definition, we know the hypergeometric function ${_1F_1}( \xi_a; \xi_b; Q)$ satisfies 
\begin{align}
	{_1F_1}( \xi_a; \xi_b; Q) =\sum_{k=0}^{\infty} \sum_{\kappa: k} \frac{(\xi_a)_{\kappa}}{(\xi_b)_{\kappa} }\frac{C_{\kappa}(Q)}{k!}, \label{eq:hyperlinearcomb}
\end{align} where $\sum_{\kappa: k} $ represents the summation over  the partitions $\kappa=(k_1,\ldots,k_m)$, $k_1\geq \ldots \geq k_m \geq 0$, of $k$,    $C_{\kappa}(Q)$ is the zonal polynomial of $Q$ corresponding to $\kappa$, and the generalized hypergeometric coefficient $(\xi)_{\kappa}$ is given by
$
 	(\xi)_{\kappa} = \prod_{i=1}^t ( \xi-(i-1)/2)_{k_i}
$ with $(a)_{k_i}=a(a+1)\ldots(a+k_i-1)$ and $(a)_0=1$.

We then characterize the zonal polynomials  $C_{\kappa}(M)$. For given partition $\kappa=(k_1,k_2,\ldots)$ of $k$, define the monomial symmetric functions $N_{\kappa}(M)= \sum_{\{i_1,\ldots,i_t \}} l_{i_1}^{k_1}\ldots l_{i_t}^{k_t}$, where $t$ is the number of nonzero parts in the partition $\kappa$, and the summation is over the distinct permutations $(i_1,\ldots, i_t)$ of $t$ different integers from $1,\ldots,m$. For another partition $\lambda=(\lambda_1,\lambda_2, \ldots)$, we write $\kappa>\lambda$ if $k_i>\lambda_i$ for the first index $i$ for which the parts in $\kappa$ and $\lambda$ are unequal. Then we have $C_{\kappa}(M)=\sum_{\lambda \leq \kappa} c_{\kappa,\lambda} N_{\lambda}(M)$, where $c_{\kappa,\lambda}$ are constants. 

\subsection{Lemmas in the proof of Theorem \ref{thm:poweranalysis}} \label{sec:lmmasthmpower}

\begin{lemma}\label{lm:momentaltln}
$
\mathrm{E} L_n^h =  \mathrm{E}_0L_n^{h} \times 
{_1F_1}( {nh}/{2}; (n+r-p+nh)/2; -\Omega /2)
$.
\end{lemma}
\begin{proof}
The result follows from Theorem 10.5.1 in \cite{muirhead2009aspects}. 	
\end{proof}
\begin{lemma} \label{lm:monomialfunc}
Suppose  matrix $M$ of size $m\times m$  has $m$ eigenvalues $l_1,\ldots, l_m$, but only has $m_0$ positive eigenvalues $l_1,\ldots, l_{m_0}$ and $\tilde{M}=\mathrm{diag}(l_1,\ldots,l_{m_0})$. Then for given partition $\kappa$, the zonal polynomial functions satisfy $N_{\kappa}(M)=N_{\kappa}(\tilde{M})$. 
\end{lemma}
\begin{proof}
By  the definition of monomial function $N_{\lambda}(M)$, we note that 
$
 \sum_{\{i_1,\ldots,i_t \}} l_{i_1}^{k_1}\ldots l_{i_t}^{k_t}	= \sum_{\{\tilde{i}_1,\ldots,\tilde{i}_t \}}  l_{\tilde{i}_1}^{k_1}\ldots l_{\tilde{i}_t}^{{k}_t},
$ where $\sum_{\{\tilde{i}_1,\ldots,\tilde{i}_t \}}$ represents the summation over the distinct permutations $(\tilde{i}_1,\ldots,\tilde{i}_t)$ of $t$ different integers from $1,\ldots, m_0$. It follows that  $N_{\lambda}(M)=N_{\lambda}(\tilde{M})$, where $\tilde{M}=\mathrm{diag}(l_1,\ldots, l_{m_0})$.   
\end{proof}

\begin{lemma}\label{lm:hypergeomfunc}
Suppose $Q$ has fixed rank $m_0$, then ${_1F_1}( \xi_a; \xi_b; Q) ={_1F_1}( \xi_a; \xi_b; \tilde{Q})$.	
\end{lemma}
\begin{proof}
 As $Q$ has rank $m_0$, it only has $m_0$ nonzero eigenvalues. 
 To prove the lemma, we  note that the  hypergeometric function can be expressed as the linear combination of the zonal polynomials of a matrix.  We then state two properties of the zonal polynomial functions $C_{\kappa}(Q)$. First,  by Corollary 7.2.4 in \cite{muirhead2009aspects}, we know that when $\kappa$ is a partition of $k$ into more than $m_0$ nonzero parts, $C_{\kappa}(Q)=0$. Second, when $\kappa$ is a partition of $k$ into fewer than $m_0$ nonzero parts, $C_{\kappa}(Q)=C_{\kappa}(\tilde{Q})$. To see this, we note that $C_{\kappa}(Q)=\sum_{\lambda\leq \kappa} c_{\lambda, \kappa} N_{\lambda}(Q)$ and the constants $c_{\kappa,\lambda}$ do not depend on the eigenvalues of $Q$. Then by Lemma \ref{lm:monomialfunc}, we know that  $C_{\kappa}(M)=C_{\kappa}(\tilde{M})$.  Finally, by the definition in \eqref{eq:hyperlinearcomb}, we have ${_1F_1}( \xi_a; \xi_b; Q) ={_1F_1}( \xi_a; \xi_b; \tilde{Q})$. \end{proof}

\begin{lemma} \label{lm:diffres}
$W =\log {_1F_1}( \xi_a; \xi_b; \tilde{Q})$ with $\tilde{Q}=-n\tilde{\Delta}/2$ discussed in Section \ref{sec:mainproofthmpower} is the unique solution of each of the $m_0$ partial differential equations
\begin{align*}
	\Big[ \frac{1}{2}(n+r-p-m_0+1)+\frac{nh}{2}+\frac{1}{2}n\delta_j+\frac{1}{2} \sum_{i\neq j}^{m_0} \frac{\delta_j}{\delta_j-\delta_i}\Big]\frac{\partial W}{\partial \delta_j} \notag \\
	+ \delta_j \Big[ \frac{\partial^2 W}{\partial \delta_j^2} + \left(\frac{\partial W}{\partial \delta_j} \right)^2  \Big]-\frac{1}{2}\sum_{i\neq j}^{m_0} \frac{\delta_i}{\delta_j-\delta_i}\frac{\partial W}{\partial \delta_i}=-\frac{nh}{2}\times n, 
\end{align*}	for $j=1,\ldots,m_0$,  subject to the conditions that  $W(\tilde{\Delta})$ is $(a)$  a symmetric function of $\delta_1, \ldots, \delta_{m_0} $, and $(b)$ analytic at $\tilde{\Delta}=\mathbf{0}_{m_0\times m_0}$ with $W(\mathbf{0}_{m_0\times m_0})=0$.
\end{lemma}
\begin{proof}
As $m_0$ is fixed, the result follows from	Theorem 7.5.6 in  \cite{muirhead2009aspects} by changing of variables. 
\end{proof}

\section{Theorem \ref{thm:selectionaccuracy}} \label{sec:proofselection}

We give the conditions of Theorem \ref{thm:selectionaccuracy} in Section \ref{sec:screenthmcondition}, and the main proof Theorem \ref{thm:selectionaccuracy}  is given in Section  \ref{sec:proofthmselect}, while the lemmas we use in the proof are given and proved in Section \ref{sec:lemmaproofselect}.

\subsection{Conditions of Theorem \ref{thm:selectionaccuracy}} \label{sec:screenthmcondition}

To derive Theorem \ref{thm:selectionaccuracy}, we  need some regularity conditions. We use $\lambda_{\max}(\cdot)$ and $\lambda_{\min}(\cdot)$ to denote the largest and smallest eigenvalues of a matrix respectively; $\mathrm{diag}(\cdot)$ denotes the vector of diagonal elements of a matrix;  $\max \mathrm{diag}(\cdot)$ and $\min \mathrm{diag}(\cdot)$ represent the maximum and minimum value of the diagonal elements of a matrix respectively; $\|\cdot\|$ denotes the $\ell_2$-norm of a vector; and $e_i=(0,\ldots,0,1,0,\ldots,0)^{\intercal}$ denotes the  indicator vector with 1 on the $i$th entry.  

\begin{condition} \label{cond:concentrateprob}
	The rows of $X$ and $E$ independently follow multivariate Gaussian distribution with covariance matrices $\Sigma_x$ and $\Sigma$ respectively. There exist nonnegative constants $t$ and $\tau$ and positive constants $(c_1,c_2,c_3,c_4,c_5)$
	 such that $\lambda_{\max}(\Sigma_x)\leq c_1n^{\tau}, \lambda_{\min}(\Sigma)\geq c_2n^{-t},$ $\min \mathrm{diag}(\Sigma_x) \geq c_3$, $\max \mathrm{diag}(\Sigma)\leq c_4$ and $\max \mathrm{diag}(B^{\intercal}\Sigma_x B) \leq c_5$.

\end{condition}
\begin{condition}  \label{cond:minimumbound}
For some constants $\kappa$, $u$, $c_6>0$ and $c_7>0$, and fixed $i\in \mathcal{M}_*$, there exists $\mathbf{a}_{0,i}\in \mathbb{R}^m$ with $\|\mathbf{a}_{0,i}\|=1$ such that $\max\{ \|\Sigma_x^{1/2}B\mathbf{a}_{0,i}\|, \|\Sigma^{1/2}\mathbf{a}_{0,i} \| \}\leq c_6 n^{u}$ and $|e_i^{\intercal}\Sigma_x B\mathbf{a}_{0,i}|\sigma_{x,i}^{-1}\geq c_7n^{-\kappa}$, where $\sigma_{x,i}^2$ is the $i$-th diagonal element of $\Sigma_x$.  
\end{condition}

\begin{condition} \label{cond:dimensionrelation}
Assume $m=O(n^{s})$ with $0 \leq s <1$;    $\iota+\tau<1$, where $\iota=2\kappa+2u+t+s $; $p>c_9n$ for some constant $c_9>1$;  $\log p=O(n^{\pi})$ for some constant $\pi \in (0,1-2\kappa-2u-t-s)$;  and 
$\delta n^{1-\iota-\tau}\to \infty$ as $n \to \infty$.
\end{condition}


\begin{remark}
In Condition \ref{cond:concentrateprob}, we assume that $X$ and $E$ follow the Gaussian distribution for the ease of theoretical developments. We allow the eigenvalues of $\Sigma_x$ and $\Sigma$ to diverge or degenerate as $n$ grows, which is similarly assumed in \cite{fan2008sure} and \cite{wang2016high} etc. in studying the linear regression with univariate response. The boundedness of the diagonal elements of $ \Sigma$ and $B^{\intercal}\Sigma_x B$ is satisfied when the variances of  response variables are $O(1)$. 
Condition \ref{cond:minimumbound} implies  that there exists a combination of the response variables whose absolute covariance with the $i$-th predictor is sufficiently large. In particular,  suppose for each $i\in \mathcal{M}_*$, there exists $k_i \in\{1,\ldots, m \}$ such that $\mathrm{cov}( {x}_{1,i} ,{y}_{1,k_i} ) \sigma_{x,i}^{-1} \geq c_7n^{-\kappa}$. Then Condition \ref{cond:minimumbound} is satisfied under Condition \ref{cond:concentrateprob}.  
Condition \ref{cond:dimensionrelation} allows the number of predictors $p$ grow exponentially with $n$. The requirement $2u+2\kappa+\tau+t+s<1$ is satisfied when the eigenvalues of $\Sigma_x$, $B^{\intercal}\Sigma_x B$ and $\Sigma$ do not diverge or degenerate too fast with $n$, and the covariance between $x_{1,i}$ and $\mathbf{y}_1^{\intercal}\mathbf{a}_{0,i}$ is sufficiently large. 
\end{remark}

\subsection{Proof of Theorem \ref{thm:selectionaccuracy}} \label{sec:proofthmselect}

Before proceeding to the proof, we define some notations and provide some preliminary results. Note that by the form of $\omega_j$,  we could assume $\mathrm{E}(X)=\mathbf{0}$ with loss of generality. Let $Z=X\Sigma_x^{-1/2}.$ We know that  the entries in $Z$  are i.i.d. $\mathcal{N}(0,1)$ by Condition \ref{cond:concentrateprob}, and then with probability 1, the $n\times p$ matrix $Z$ has full rank $n$.   Let  $\mu_1^{1/2},\ldots,\mu_{n}^{1/2}$ be the $n$ singular values of $Z$. Then $Z^{\intercal}Z$ has the eigendecomposition
\begin{eqnarray}
	Z^{\intercal}Z=U^{\intercal}\mathrm{diag}(\mu_1,\ldots,\mu_{n},0,\ldots,0)U, \label{eq:pztransposepz}
\end{eqnarray} where $U$ belongs to the orthogonal group $\mathcal{O}(p)$. 
We write $U^{\intercal}=(\mathbf{u}_1,\ldots, \mathbf{u}_p)$. It follows that the Moore-Penrose generalized inverse of \eqref{eq:pztransposepz}  is
\begin{eqnarray*}
	(Z^{\intercal}Z )^{+}=\sum_{i=1}^{n} \frac{1}{\mu_i} \mathbf{u}_i \mathbf{u}_i^{\intercal}.
\end{eqnarray*} 
Moreover, we have the decomposition 
\begin{eqnarray}
	S:=(Z^{\intercal}Z )^{+}Z^{\intercal}Z=U^{\intercal}\mathrm{diag}(I_{n},0)U=\tilde{U}^{\intercal}\tilde{U},\label{eq:smatdecomposition}
\end{eqnarray} where $\tilde{U}=(I_{n}, \mathbf{0})_{n\times p} U$ and $(I_{n}, \mathbf{0})_{n\times p}$ represents an $n\times p$ matrix with first $n$ columns being $I_{n}$ and 0 in the remaining columns. 
Since $X=Z\Sigma_x^{1/2}$, by \eqref{eq:pztransposepz}, we know that
\begin{eqnarray}
	X^{\intercal}X=\Sigma_x^{1/2}\tilde{U}^{\intercal} \mathrm{diag}(\mu_1,\ldots, \mu_{n})\tilde{U}\Sigma_x^{1/2}.\label{eq:xexpansion}
\end{eqnarray} 
In addition, define $P=I_n-\mathbf{1}_n\mathbf{1}_n^{\intercal}/n$.  We can then write $\omega_i$ equivalently as
\begin{eqnarray*}
\omega_i = \max_{\mathbf{a}: \|\mathbf{a}\|=1} \frac{ \mathbf{a}^{\intercal}Y^{\intercal}P^{\intercal}P\mathbf{x}^i}{ \sqrt{(\mathbf{a}^{\intercal} Y^{\intercal}P^{\intercal}PY\mathbf{a})\{(\mathbf{x}^i)^{\intercal}P^{\intercal}P\mathbf{x}^i\}} }.
\end{eqnarray*} 
By the property of $\omega_i$, we assume without loss of generality that $X$ and $E$ have mean zero.
Furthermore, suppose $\mathrm{diag}(\Sigma_x)=\mathrm{diag}(\sigma_{x,1},\ldots, \sigma_{x,p})$ and let
\begin{eqnarray}
	\zeta_i=\max_{\mathbf{a}: \|\mathbf{a}\|=1} \frac{\mathbf{a}^{\intercal}Y^{\intercal}P^{\intercal}P\mathbf{x}^i}{\sigma_{x,i}\sqrt{n\times \mathbf{a}^{\intercal} Y^{\intercal}P^{\intercal}PY\mathbf{a}}}.  \label{eq:zetaidefinition}
\end{eqnarray}
 Then by Lemma \ref{lm:approximatezetalm}, we know $\omega_i=\zeta_i\{1+o(1)\}$ 
with probability $1-O\{\exp(-c_0n/\log n)\}$ for some constant $c_0>0$. 
As $Y=XB+E$, we have
\begin{eqnarray}
\zeta_i = \max_{\mathbf{a}: \|\mathbf{a}\|=1} \frac{ \mathbf{a}^{\intercal}B^{\intercal}X^{\intercal}P^{\intercal}P\mathbf{x}^i+\mathbf{a}^{\intercal}E^{\intercal}P^{\intercal}P\mathbf{x}^i}{\sigma_{x,i} \sqrt{n\times \mathbf{a}^{\intercal} (Y^{\intercal}P^{\intercal}PY) \mathbf{a}} }= \xi_i + \eta_i, \label{eq:omegaitwoparts}
\end{eqnarray} where
\begin{eqnarray*}
\xi_i = \max_{\mathbf{a}: \|\mathbf{a}\|=1}  \frac{\mathbf{a}^{\intercal}B^{\intercal}X^{\intercal}P^{\intercal}P\mathbf{x}^i}{ \sigma_{x,i}\sqrt{n\mathbf{a}^{\intercal} (Y^{\intercal}P^{\intercal}PY) \mathbf{a}}  }, \quad \eta_i=  \max_{\mathbf{a}: \|\mathbf{a}\|=1}  \frac{ \mathbf{a}^{\intercal}E^{\intercal}P^{\intercal}P\mathbf{x}^i}{ \sigma_{x,i}\sqrt{n \mathbf{a}^{\intercal} (Y^{\intercal}P^{\intercal}PY) \mathbf{a}}  }. 
\end{eqnarray*} Moreover, we write $B=[\boldsymbol{\beta}_1,\ldots,\boldsymbol{\beta}_m]$, where $\boldsymbol{\beta}_j$ represents the $j$-th column of $B$. We then study $\xi_i$ and $\eta_i$ separately. 

\paragraph{Step 1:} We first examine $\boldsymbol{\xi}=(\xi_1,\ldots, \xi_p)^{\intercal}$.

\paragraph{Step 1.1} (bounding $\| \boldsymbol{\xi}\|$ from above) For $i=1,\ldots,p$,
\begin{eqnarray*}
|\xi_i| \leq \{n\lambda_{\min} (Y^{\intercal}P^{\intercal}PY)\}^{-1/2}\sigma_{x,i}^{-1} \|B^{\intercal}X^{\intercal}P^{\intercal}P\mathbf{x}^i\|, 
\end{eqnarray*} where $\|\cdot \|$ represents the $\ell_2$-norm of a vector. Then we know
\begin{eqnarray}
	\| \boldsymbol{\xi}\|^2=\sum_{i=1}^p \xi_i^2\leq \{n\lambda_{\min} (Y^{\intercal}P^{\intercal}PY)\}^{-1}\sum_{i=1}^p \sigma_{x,i}^{-2} \|B^{\intercal}X^{\intercal}P^{\intercal}P\mathbf{x}^i\|^2 \label{eq:l2normeachtermbound}
\end{eqnarray} By Lemma \ref{lm:ytransyminimumeigbound}, we know that there exist constants $c_1$ and $c_0$ such that  $\lambda_{\min} (Y^{\intercal}P^{\intercal}PY)\geq c_1n^{1-t}$ with probability $1-O\{\exp(-c_0n)\}$.  To bound $\| \boldsymbol{\xi}\|$ from above, we then examine $\sum_{i=1}^p \sigma_{x,i}^{-2}\|B^{\intercal}X^{\intercal}P^{\intercal}P\mathbf{x}^i\|^2$. Since  $\min_{1\leq i\leq p} \sigma_{x,i}^2 \geq c_3$ by Condition \ref{cond:concentrateprob}, 
\begin{eqnarray}
	\sum_{i=1}^p \sigma_{x,i}^{-2} \|B^{\intercal}X^{\intercal}P^{\intercal}P\mathbf{x}^i\|^2
	&\leq & c_3^{-1} \sum_{i=1}^p  \sum_{k=1}^m (\boldsymbol{\beta}_k^{\intercal}X^{\intercal}P^{\intercal}P\mathbf{x}^i)^2 \notag \\
	&=& c_3^{-1}\sum_{k=1}^m \boldsymbol{\beta}_k^{\intercal} X^{\intercal} P^{\intercal}P\sum_{i=1}^p \mathbf{x}^i(\mathbf{x}^i)^{\intercal}P^{\intercal}P X\boldsymbol{\beta}_k.  \label{eq:step1upperbound0} 
\end{eqnarray} 
As $\sum_{i=1}^p \mathbf{x}^i(\mathbf{x}^i)^{\intercal}=XX^{\intercal}$ and $P^{\intercal}P=I_n-\mathbf{1}_n\mathbf{1}_n^{\intercal}/n$, we have
\begin{eqnarray}
	\eqref{eq:step1upperbound0}&=& c_3^{-1}\sum_{k=1}^m \|\boldsymbol{\beta}_k^{\intercal} X^{\intercal} (I_n-\mathbf{1}_n\mathbf{1}_n^{\intercal}/n)X\|^2 \notag \\
	&\leq & 2c_3^{-1}\times (A_{\boldsymbol{\xi},1}+A_{\boldsymbol{\xi},2}), \label{eq:bxibound12}
\end{eqnarray} where $A_{\boldsymbol{\xi},1}=\sum_{k=1}^m \|\boldsymbol{\beta}_k^{\intercal} X^{\intercal} X \|^2$ and $A_{\boldsymbol{\xi},2}=\sum_{k=1}^m \|\boldsymbol{\beta}_k^{\intercal} X^{\intercal} (\mathbf{1}_n\mathbf{1}_n^{\intercal}/n)X \|^2$.

We next examine $A_{\boldsymbol{\xi},1}$ and $A_{\boldsymbol{\xi},2}$ separately. By \eqref{eq:xexpansion},
\begin{eqnarray}
	A_{\boldsymbol{\xi},1}&=& \sum_{k=1}^m  \boldsymbol{\beta}_k^{\intercal}\Sigma_x^{1/2}\tilde{U}^{\intercal} \mathrm{diag}(\mu_1,\ldots, \mu_{n})\tilde{U}\Sigma_x^{1/2}\Sigma_x^{1/2}\tilde{U}^{\intercal} \mathrm{diag}(\mu_1,\ldots, \mu_{n})\tilde{U}\Sigma_x^{1/2}\boldsymbol{\beta}_k \notag \\
	&\leq &  p^2 \{\lambda_{\max}(p^{-1}ZZ^{\intercal})\}^2 \lambda_{\max}(\Sigma_x) \sum_{k=1}^m\boldsymbol{\beta}_k^{\intercal} \Sigma_x^{1/2}\tilde{U}^{\intercal} \tilde{U}\Sigma_x^{1/2}\boldsymbol{\beta}_k, \label{eq:l2normupperbound1}
\end{eqnarray} where in the last inequality, we use the fact that  $\tilde{U}\Sigma_x \tilde{U}^{\intercal}\preceq \lambda_{\max}(\Sigma_x)I_{n}$, and $\mathrm{diag}(\mu_1,\ldots, \mu_{n}) \preceq p\lambda_{\max}(p^{-1}ZZ^{\intercal})I_{n}$, as $\mu_1^{1/2},\ldots,\mu_{n}^{1/2}$ are the singular values of $Z$.
We then bound \eqref{eq:l2normupperbound1} from above by examining $\boldsymbol{\beta}_k^{\intercal}\Sigma_x^{1/2}\tilde{U}^{\intercal} \tilde{U} \Sigma_x^{1/2}\boldsymbol{\beta}_k$.
For fixed $k=1,\ldots, m$, let $Q\in \mathcal{O}(p)$ such that $\Sigma_x^{1/2}\boldsymbol{\beta}_k=\|\Sigma_x^{1/2}\boldsymbol{\beta}_k \|_2Q e_1$. By \eqref{eq:smatdecomposition} and  Lemma \ref{lm:invariantdist}, we know
\begin{eqnarray}
	\boldsymbol{\beta}_k^{\intercal}\Sigma_x^{1/2}\tilde{U}^{\intercal} \tilde{U} \Sigma_x^{1/2}\boldsymbol{\beta}_k =\| \Sigma_x^{1/2} \boldsymbol{\beta}_k\|^2\langle Q^{\intercal}SQ e_1,e_1 \rangle \overset{(d)}{=}\| \Sigma_x^{1/2} \boldsymbol{\beta}_k\|^2\langle S e_1,e_1 \rangle. \label{eq:equivalencedistribution}
\end{eqnarray}  By Condition \ref{cond:concentrateprob}, $\| \Sigma_x^{1/2} \boldsymbol{\beta}_k\|^2=\boldsymbol{\beta}_k^{\intercal}\Sigma_x \boldsymbol{\beta}_k\leq c_5$ for some constant $c_5>0$. 
Then by  \eqref{eq:equivalencedistribution} and Lemma \ref{lm:grasmantailbd}, we know for some positive constants $c_0$ and $c_1$,
\begin{eqnarray}
	P(\boldsymbol{\beta}_k^{\intercal}\Sigma_x^{1/2}\tilde{U}^{\intercal} \tilde{U} \Sigma_x^{1/2}\boldsymbol{\beta}_k>c_1n/p ) \leq O\{ \exp(-c_0n) \}. \label{eq:quadraticformtailbound}
\end{eqnarray} Combining  \eqref{eq:l2normupperbound1}, Lemma \ref{lm:eigenvalueboundonz}, Condition \ref{cond:concentrateprob} and \eqref{eq:quadraticformtailbound}, we then know for some positive constants $c_0$ and $c$, with probability $1-O\{ m\exp(-c_0n) \}$, $A_{\boldsymbol{\xi},1} \leq cmp^2 n^{\tau} n/p=cm pn^{1+\tau}$.

For $A_{\boldsymbol{\xi},2}$, note that
\begin{eqnarray}
	A_{\boldsymbol{\xi},2}= \sum_{k=1}^m \boldsymbol{\beta}_k^{\intercal} \Sigma_x^{1/2} Z^{\intercal} (\mathbf{1}_n\mathbf{1}_n^{\intercal}/n)Z\Sigma_x Z^{\intercal}(\mathbf{1}_n\mathbf{1}_n^{\intercal}/n)Z\Sigma_x^{1/2}\boldsymbol{\beta}_k.  \label{eq:boundbxi2sum}
\end{eqnarray}
Similarly, considering fixed $k=1,\ldots,m$, we let $Q\in \mathcal{O}(p)$ such that $\Sigma_x^{1/2}\boldsymbol{\beta}_k=\|\Sigma_x^{1/2}\boldsymbol{\beta}_k \|Q e_1$. Then
\begin{eqnarray}
	&&\boldsymbol{\beta}_k^{\intercal} \Sigma_x^{1/2} Z^{\intercal} (\mathbf{1}_n\mathbf{1}_n^{\intercal}/n)Z\Sigma_x Z^{\intercal}(\mathbf{1}_n\mathbf{1}_n^{\intercal}/n)Z\Sigma_x^{1/2}\boldsymbol{\beta}_k \notag \\
	&=&\|\Sigma_x^{1/2}\boldsymbol{\beta}_k \|^2 e_1^{\intercal}Q^{\intercal} Z^{\intercal} (\mathbf{1}_n\mathbf{1}_n^{\intercal}/n)Z\Sigma_x Z^{\intercal}(\mathbf{1}_n\mathbf{1}_n^{\intercal}/n)ZQe_1 \notag \\
	&\leq & \lambda_{\max}(\Sigma_x) \|\Sigma_x^{1/2}\boldsymbol{\beta}_k \|^2 e_1^{\intercal}Q^{\intercal} Z^{\intercal} (\mathbf{1}_n\mathbf{1}_n^{\intercal}/n)ZQQ^{\intercal}Z^{\intercal}(\mathbf{1}_n\mathbf{1}_n^{\intercal}/n)ZQe_1 \notag \\
	&\overset{(d)}{=}& \lambda_{\max}(\Sigma_x) \|\Sigma_x^{1/2}\boldsymbol{\beta}_k \|^2 \|Z^{\intercal}(\mathbf{1}_n\mathbf{1}_n^{\intercal}/n)Ze_1 \|^2,  \label{eq:equivadistaxi2}
\end{eqnarray} where in the last equality, we use the fact that $ZQ\overset{(d)}{=}Z$. Since the entries in $Z$ are i.i.d. $\mathcal{N}(0,1)$, we have $L=(L_1,\ldots, L_p)^{\intercal}=Z^{\intercal}\mathbf{1}_n/\sqrt{n}\sim \mathcal{N}(\mathbf{0}_{p\times 1}, I_p)$ with  $L_1=\mathbf{1}_n^{\intercal}Z e_1/\sqrt{n}$. It follows that
$
	\|Z^{\intercal}(\mathbf{1}_n\mathbf{1}_n^{\intercal}/n)Ze_1 \|^2= \sum_{j=1}^p L_1^2 L_j^2.
$ Since $L_1^2 \sim \chi^2_1$, there exist constants $c_0$ and $c_1$ such that $P(|L_1^2-1|>c_1n)\leq O\{\exp(-c_0 n)\}$. Moreover, note that $L_j^2$'s are i.i.d. $\chi_1^2$-distributed random variables. By Lemma \ref{lm:chisqtailprob}, for some positive constants $c_0$ and $c_1$, when $p\geq n$,
\begin{eqnarray}
	P\Big(\sum_{j=2}^p L_j^2/(p-1)>1 +c_1 \Big) \leq O\{\exp(-c_0p)\} \leq O\{\exp(-c_0n)\}. \label{eq:chisquaredtail}
\end{eqnarray} 
Thus there exist constants $c$ and $c_0$ such that with probability $1-O\{\exp(-c_0n)\}$,
\begin{eqnarray*}
	\|Z^{\intercal}(\mathbf{1}_n\mathbf{1}_n^{\intercal}/n)Ze_1 \|^2= \sum_{j=1}^p L_1^2 L_j^2 =L_1^4+L_1^2\sum_{j=2}^p L_j^2 \leq cpn.
\end{eqnarray*} By Condition \ref{cond:concentrateprob}, $ \lambda_{\max}(\Sigma_x)\leq c_1 n^{\tau}$ and $ \|\Sigma_x^{1/2}\boldsymbol{\beta}_k \|^2\leq c_2$ for some constant $c_1$ and $c_2$. From \eqref{eq:boundbxi2sum} and \eqref{eq:equivadistaxi2}, we know that $A_{\boldsymbol{\xi},2}\leq c m pn^{\tau+1}$ with probability $1-O\{m\exp(-c_0n)\}$.


In summary, we obtain that for some constants $c$ and $c_0$,  $A_{\boldsymbol{\xi},1}$ and $A_{\boldsymbol{\xi},2} \leq c m pn^{\tau+1}$ with probability $1-O\{\exp(-c_0n)\}$. Then by  \eqref{eq:l2normeachtermbound}, \eqref{eq:bxibound12} and Lemma \ref{lm:ytransyminimumeigbound}, we have for some positive constants $c_1$ and $c_0$,
\begin{eqnarray}
  P\{ \|\boldsymbol{\xi}\|^2> c_1 n^{-(1+1-t)} p m n^{1+\tau} \}
	\leq  O\{m\exp(-c_0n)\}=O\{\exp(-c_0n)\}, \label{eq:xiupperbound}
\end{eqnarray} where the last equality is from Condition \ref{cond:dimensionrelation}.

\paragraph{Step 1.2} (bounding $|\xi_i|$ for $i \in \mathcal{M}_*$ from below) Without loss of generality, we consider $B\neq \mathbf{0}_{p\times m}$.  
For fixed $i\in \mathcal{M}_*$,
\begin{eqnarray}
\xi_i &=& \max_{\mathbf{a}: \| \mathbf{a}\|_2=1}  \frac{ \mathbf{a}^{\intercal}B^{\intercal}X^{\intercal}P^{\intercal}P\mathbf{x}^i}{\sigma_{x,i} \sqrt{ n\times \mathbf{a}^{\intercal} (Y^{\intercal}P^{\intercal}PY) \mathbf{a}}  } \notag \\
&\geq & \{n\times \mathbf{a}_{0,i}^{\intercal}Y^{\intercal}P^{\intercal}PY\mathbf{a}_{0,i}\}^{-1/2}{\sigma_{x,i}^{-1}|\mathbf{a}_{0,i}^{\intercal}B^{\intercal}X^{\intercal}P^{\intercal}PX e_i|},  \label{eq:xilowerbound1}
\end{eqnarray} where $\mathbf{a}_{0,i}$ in the last inequality is specified in Condition \ref{cond:minimumbound}.  To bound $|\xi_i|$ from below, we then examine \eqref{eq:xilowerbound1}. By Lemma \ref{lm:ytransyminimumeigbound}, there exist constants $c_0$ and $c_2$ such that with probability $1-O\{\exp(-c_0n) \}$, 
\begin{eqnarray}
	\mathbf{a}_{0,i}^{\intercal}Y^{\intercal}P^{\intercal}PY\mathbf{a}_{0,i} \leq c_2 n^{2u+1}. \label{eq:denominatorboundn2u}
\end{eqnarray}
Moreover, as $P^{\intercal}P=I_n-\mathbf{1}_n\mathbf{1}_n^{\intercal}/n$, 
$
	\sigma_{x,i}^{-1} \mathbf{a}_{0,i}^{\intercal}B^{\intercal}X^{\intercal}P^{\intercal}PXe_i 
	= \tilde{A}_{\xi,i,1}-\tilde{A}_{\xi,i,2},
$ where 
$
	\tilde{A}_{\xi,i,1}= \sigma_{x,i}^{-1}\mathbf{a}_{0,i}^{\intercal}B^{\intercal}X^{\intercal}Xe_i 
$ and $\tilde{A}_{\xi,i,2}=\sigma_{x,i}^{-1}\mathbf{a}_{0,i}^{\intercal}B^{\intercal}X^{\intercal}(\mathbf{1}_n \mathbf{1}_n^{\intercal}/n) X e_i$.

We first consider $\tilde{A}_{\xi,i,1}$. 
From \eqref{eq:xexpansion}, 
\begin{eqnarray}
	\sigma_{x,i}^{-1}\mathbf{a}_{0,i}^{\intercal}B^{\intercal}X^{\intercal}Xe_i=\sigma_{x,i}^{-1}\mathbf{a}_{0,i}^{\intercal}B^{\intercal} \Sigma_x^{1/2}\tilde{U}^{\intercal}\mathrm{diag}(\mu_1,\ldots, \mu_{n}) \tilde{U} \Sigma_x^{1/2} e_i. \label{eq:etapart1equivform}
\end{eqnarray}
Note that for fixed $i=1,\ldots,p$, $\|\Sigma^{1/2}_x e_i\|^2\sigma_{x,i}^{-2} =1$. 
Then there exists $\tilde{Q}\in \mathcal{O}(p)$ such that $\Sigma_x^{1/2}e_i\sigma_{x,i}^{-1}=\tilde{Q}e_1$, and 
\begin{eqnarray}
	&&\Sigma_x^{1/2}B\mathbf{a}_{0,i}-\langle  \Sigma_x^{1/2}B \mathbf{a}_{0,i}, \Sigma_x^{1/2}e_{i} \sigma_{x,i}^{-1} \rangle \Sigma_x^{1/2}e_i\sigma_{x,i}^{-1} \notag \\
	&=& (\|\Sigma_x^{1/2}B\mathbf{a}_{0,i}\|^2- \langle  \Sigma_x^{1/2}B \mathbf{a}_{0,i}, \Sigma_x^{1/2}e_{i} \sigma_{x,i}^{-1} \rangle^2 )^{1/2}\tilde{Q}e_2. \label{eq:differnqtilde}
\end{eqnarray} 
By Condition \ref{cond:minimumbound}, there exists  constant $c$ such that $\|\Sigma_x^{1/2}B\mathbf{a}_{0,i}\|\leq cn^{u}$. Thus
\begin{eqnarray}
	\Sigma_x^{1/2}B\mathbf{a}_{0,i}=\langle  \Sigma_x^{1/2}B\mathbf{a}_{0,i} , \Sigma_x^{1/2}e_{i} \sigma_{x,i}^{-1} \rangle \tilde{Q}e_1+O(n^{u})\tilde{Q}e_2. \label{eq:sigxba0twoparts}
\end{eqnarray} 
Let $T_{\eta,1}=\tilde{U}^{\intercal}\mathrm{diag}(\mu_1,\ldots,\mu_{n}) \tilde{U}\tilde{Q}e_1 $. As $\tilde{Q}e_1 =\Sigma_x^{1/2} e_i \sigma_{x,i}^{-1}$, it follows that
\begin{eqnarray}
	\eqref{eq:etapart1equivform}=\langle  \Sigma_x^{1/2}B\mathbf{a}_{0,i} , \Sigma_x^{1/2}e_{i} \sigma_{x,i}^{-1} \rangle e_1^{\intercal}\tilde{Q}^{\intercal} T_{\eta,1}+O(n^u)e_2^{\intercal}\tilde{Q}^{\intercal}T_{\eta,1}. \label{eq:septermxi}
\end{eqnarray} Since the uniform distribution on the orthogonal group $\mathcal{O}(p)$ is invariant under itself, $\tilde{U}\tilde{Q}\overset{(d)}{=}\tilde{U}$. Then as $(\mu_1,\ldots,\mu_{n})^{\intercal}$ is independent of $\tilde{U}$ by Lemma \ref{lm:invariantdist}, we know that $\tilde{Q}^{\intercal}T_{\eta,1}\overset{(d)}{=}\mathbf{R}$, where $\mathbf{R}=(R_1,\ldots, R_p)^{\intercal}=\tilde{U}^{\intercal}\mathrm{diag}(\mu_1,\ldots,\mu_{n}) \tilde{U}e_1$. By  \eqref{eq:septermxi}, we then have 
\begin{eqnarray}
	\eqref{eq:etapart1equivform}\overset{(d)}{=} \xi_{i,1}+\xi_{i,2}, \label{eq:xitwotermssum}
\end{eqnarray}
where  $\xi_{i,1}=\langle  \Sigma_x^{1/2}B\mathbf{a}_{0,i}, \Sigma_x^{1/2}e_{i}\sigma_{x,i}^{-1} \rangle R_1$ and $\xi_{i,2}=O(n^u) R_2$.

We next examine $\xi_{i,1}$ and $\xi_{i,2}$ separately. For $\xi_{i,1}$,  as $\mu_1,\ldots,\mu_{n} \geq p\lambda_{\min}(p^{-1}ZZ^{\intercal})$, and by \eqref{eq:smatdecomposition}, we have
\begin{eqnarray*}
	R_1 \geq pe_1^{\intercal} \tilde{U}^{\intercal} \lambda_{\min}(p^{-1}ZZ^{\intercal})I_n \tilde{U}e_1=p\lambda_{\min}(p^{-1}ZZ^{\intercal})\langle Se_1,e_1 \rangle.
\end{eqnarray*} Thus, by Condition \ref{cond:concentrateprob},  Lemmas \ref{lm:grasmantailbd} and  \ref{lm:eigenvalueboundonz}, and Bonferroni inequality,  we have for some positive constants $c_1$ and $c_0$,
\begin{eqnarray}
	P(R_1<c_1 p\times n/p) \leq O\{\exp(-c_0n)\}. \label{eq:r1boundupper}
\end{eqnarray} This, along with Condition \ref{cond:minimumbound}, show that for some positive constants $c_1$ and $c_0$,
\begin{eqnarray}
	P(|\xi_{i,1}|<c_1n^{1-\kappa})\leq O\{ \exp(-c_0n)\}. \label{eq:xifirsttermlowerbound}
\end{eqnarray} 
We then consider $\xi_{i,2}=O(n^u)R_2$. By Lemma \ref{lm:rtildeinvariant}, we know that there exist positive constants $c_1$ and $c_0$  such that 
$P(|R_2|>c_1n^{1/2}|W_1|)\leq O\{\exp(-c_0n\}$, where $W_1$ is an independent $\mathcal{N}(0,1)$-distributed random variable. It follows that for some positive constants $c_1$ and $c_0$, we have
\begin{eqnarray}
	P(|\xi_{i,2}|>c_1 n^{u+1/2}|W_1|) \leq O\{\exp(-c_0n) \}. \label{eq:xisecondboundbyw}
\end{eqnarray} For some constant $c_2>0$, let $x_n=\sqrt{2c_2}n^{1-\kappa-u}/\sqrt{\log n}$. Then by the classical Gaussian tail bound, we have
\begin{eqnarray*}
	P(n^{1/2}|W|>x_n)\leq \sqrt{2/\pi} \frac{ \exp\{-c_2n^{1-2\kappa-2u}/\log n \} }{ \sqrt{2c_2}n^{1/2-\kappa-u}/\sqrt{\log n} } \leq O\{ \exp(-c_2 n^{1-2\kappa-2u}/\log n   )\},
\end{eqnarray*} which, along with inequality \eqref{eq:xisecondboundbyw}, show that for some positive constants $c_1$ and $c_0$,
\begin{eqnarray}
	P(|\xi_{i,2}|>c_1n^ux_n) \leq O[\exp\{-c_0n^{1-2\kappa-2u}/\log n \}], \label{eq:xisecondtermlowerbound}
\end{eqnarray}  where $n^ux_n=\sqrt{2c_2}n^{1-\kappa}/\sqrt{\log n}$. 

We then consider $\tilde{A}_{\xi,i,2}$. Similarly, we take $\tilde{Q}\in \mathcal{O}(p)$ satisfying $\Sigma_x^{1/2}e_i\sigma_{x,i}^{-1}=\tilde{Q}e_1$ and \eqref{eq:differnqtilde}. As   $Z\tilde{Q}\overset{(d)}{=}Z$ and by \eqref{eq:sigxba0twoparts}, similarly to  \eqref{eq:xitwotermssum}, we have
\begin{eqnarray*}
	\tilde{A}_{\xi,i,2}= \mathbf{a}_{0,i}^{\intercal}B^{\intercal}\Sigma_x^{1/2} Z^{\intercal}(\mathbf{1}_n \mathbf{1}_n^{\intercal}/n) Z \tilde{Q} e_1 
	\overset{(d)}{=} \tilde{\xi}_{i,1}+\tilde{\xi}_{i,2},
\end{eqnarray*} where $ \tilde{\xi}_{i,1}=\langle  \Sigma_x^{1/2}B\mathbf{a}_{0,i}, \Sigma_x^{1/2}e_{i} \sigma_{x,i}^{-1} \rangle e_1^{\intercal} Z^{\intercal}\mathbf{1}_n \mathbf{1}_n^{\intercal}Ze_1/n$ and $\tilde{\xi}_{i,2}=O(n^u)e_2^{\intercal}Z^{\intercal}\mathbf{1}_n \mathbf{1}_n^{\intercal}Ze_1/n$. Note that  $\mathbf{1}_n^{\intercal}Ze_1/\sqrt{n}\sim \mathcal{N}(0,1)$ and $\mathbf{1}_n^{\intercal}Ze_2/\sqrt{n}\sim \mathcal{N}(0,1)$ independently. Then for some positive constants $c_1$ and $c_0$, 
\begin{eqnarray*}
	P(|e_1^{\intercal} Z^{\intercal}\mathbf{1}_n \mathbf{1}_n^{\intercal}Ze_1/n|>c_1 n^{1-\kappa-u}/\log n)&\leq &O[\exp\{ -c_0n^{1-\kappa-u}/\log n\}], \notag \\
	P(|e_2^{\intercal} Z^{\intercal}\mathbf{1}_n \mathbf{1}_n^{\intercal}Ze_1/n|>c_1 n^{1-\kappa-u}/\log n)&\leq &O[\exp\{ -c_0n^{1-\kappa-u}/\log n\}]. 
\end{eqnarray*} These, combined with \eqref{eq:r1boundupper}, show that there exist some constants $c_1$ and $c_0$ such that 
\begin{eqnarray}
	P(|\tilde{\xi}_{i,1}|> c_1 |\xi_{i,1}|n^{-u-\kappa}/\log n) &\leq & P(|\tilde{\xi}_{i,1}|> c_1\langle  \Sigma_x^{1/2}B\mathbf{a}_{0,i}, \Sigma_x^{1/2}e_{i}\sigma_{x,i}^{-1} \rangle n^{1-\kappa-u}/\log n ) \notag \\
	&& +P(R_1<c_1n) \notag \\
	&\leq & O\{\exp(-c_0n^{1-\kappa-u}/\log n)\}, \notag \\
	P(|\tilde{\xi}_{i,2}|>c_1 n^{1-\kappa}/\log n) &\leq & O\{\exp(-c_0n^{1-\kappa-u}/\log n)\}.\label{eq:tildeepsupperbound}
\end{eqnarray}


In summary, by Bonferroni's inequality, combining \eqref{eq:xilowerbound1}, \eqref{eq:denominatorboundn2u},  \eqref{eq:xitwotermssum},    \eqref{eq:xifirsttermlowerbound},    \eqref{eq:xisecondtermlowerbound} and \eqref{eq:tildeepsupperbound},  we have for some positive constants $c_1$ and $c_0$,
\begin{eqnarray}
	P(|\xi_i|<c_1(n \times n^{2u+1})^{-1/2}n^{1-\kappa}) \leq O[\exp\{-c_0n^{1-2\kappa-2u}/\log n \}], \quad i\in \mathcal{M}_* . \label{eq:xilowerboundres}
\end{eqnarray}


\paragraph{Step 2} We next examine $\boldsymbol{\eta}=(\eta_1,\ldots,\eta_p)^{\intercal}$  defined in \eqref{eq:omegaitwoparts}. 
\paragraph{Step 2.1} (bounding $\|\boldsymbol{\eta} \|_2$ from above) By Condition \ref{cond:concentrateprob},
\begin{eqnarray}
	\eta_i = \max_{\mathbf{a}:\|\mathbf{a} \|=1}\frac{ \mathbf{a}^{\intercal}E^{\intercal} P^{\intercal}P \mathbf{x}^i  }{\sigma_{x,i} \sqrt{n\times\mathbf{a}^{\intercal} Y^{\intercal}P^{\intercal}PY\mathbf{a} }} 
	\leq \{n \lambda_{\min}(Y^{\intercal}P^{\intercal}PY)\}^{-1/2} c_3^{-1} \|E^{\intercal}P^{\intercal}P\mathbf{x}^i\|. \label{eq:etaiupperbound}
\end{eqnarray}
Let $\boldsymbol{\epsilon}^j$ denote the $j$-th column of $E$, then $E=[\boldsymbol{\epsilon}^1,\ldots, \boldsymbol{\epsilon}^m]$. As $P^{\intercal}P=I_n-\mathbf{1}_n\mathbf{1}_n^{\intercal}/n$, we have $\|E^{\intercal}P^{\intercal}P\mathbf{x}^i\|^2=\sum_{j=1}^m \{(\boldsymbol{\epsilon}^j)^{\intercal}\mathbf{x}^i-(\boldsymbol{\epsilon}^j)^{\intercal}\mathbf{1}_n\mathbf{1}_n^{\intercal}\mathbf{x}^i/n\}^2$. Note that $\sum_{i=1}^p  \mathbf{x}^i (\mathbf{x}^i)^{\intercal}=XX^{\intercal}$. Then by \eqref{eq:etaiupperbound},
\begin{eqnarray}
	\sum_{i=1}^p \eta_i^2 &\leq & c_3^{-2} \{n\lambda_{\min}(Y^{\intercal} P^{\intercal}P Y)\}^{-1} \sum_{i=1}^p   \sum_{j=1}^m 2\times [ \{(\boldsymbol{\epsilon}^j)^{\intercal}\mathbf{x}^i\}^2+\{ (\boldsymbol{\epsilon}^j)^{\intercal}\mathbf{1}_n\mathbf{1}_n^{\intercal}\mathbf{x}^i/n\}^2 ] \notag \\
	&= & 2c_3^{-2} \{n\lambda_{\min}(Y^{\intercal} P^{\intercal}P Y)\}^{-1} \notag \\
	&& \times \sum_{j=1}^m \{(\boldsymbol{\epsilon}^j)^{\intercal}\sum_{i=1}^p  \mathbf{x}^i (\mathbf{x}^i)^{\intercal}\boldsymbol{\epsilon}^j+  (\boldsymbol{\epsilon}^j)^{\intercal} \mathbf{1}_n \mathbf{1}_n^{\intercal} \sum_{i=1}^p  \mathbf{x}^i (\mathbf{x}^i)^{\intercal} \mathbf{1}_n \mathbf{1}_n^{\intercal}\boldsymbol{\epsilon}^j/n^2 \} \notag \\
	&= & 2 c_3^{-2}  \{n\lambda_{\min}(Y^{\intercal} P^{\intercal}P Y)\}^{-1} \sum_{j=1}^m \{(\boldsymbol{\epsilon}^j)^{\intercal}XX^{\intercal}\boldsymbol{\epsilon}^j+  (\boldsymbol{\epsilon}^j)^{\intercal} \mathbf{1}_n \mathbf{1}_n^{\intercal} XX^{\intercal} \mathbf{1}_n \mathbf{1}_n^{\intercal}\boldsymbol{\epsilon}^j/n^2 \} \notag \\
	&\leq & 2 c_3^{-2} \{n\lambda_{\min}(Y^{\intercal} P^{\intercal}P Y)\}^{-1} \lambda_{\max}(\Sigma_x) \sum_{j=1}^m (A_{\eta,j,1}+ A_{\eta,j,2}),\label{eq:etaboundbyepsilonx} 
\end{eqnarray} where $A_{\eta,j,1}=(\boldsymbol{\epsilon}^j)^{\intercal}ZZ^{\intercal}\boldsymbol{\epsilon}^j$ and $A_{\eta,j,2}=(\boldsymbol{\epsilon}^j)^{\intercal} \mathbf{1}_n \mathbf{1}_n^{\intercal} ZZ^{\intercal} \mathbf{1}_n \mathbf{1}_n^{\intercal}\boldsymbol{\epsilon}^j/n^2$.

Note that $A_{\eta,j,1}\leq p\lambda_{\max}(p^{-1}ZZ^{\intercal}) \|\boldsymbol{\epsilon}^j\|^2$. Suppose $\mathrm{diag}(\Sigma)=(\sigma_{\epsilon,1}^2,\ldots, \sigma_{\epsilon,m}^2)^{\intercal}$.   
Then by Condition  \ref{cond:concentrateprob} and Lemma \ref{lm:chisqtailprob}, we know for some positive constants $c$ and $c_0$,
\begin{eqnarray}
	P(\|\boldsymbol{\epsilon}^j\|_2^2>cn\sigma_{\epsilon,j}^2/\log n)\leq \exp(-c_0n/\log n). \label{eq:eachepsilonbound}
\end{eqnarray}  In addition, 
$
	A_{\eta,j,2} \leq p\lambda_{\max}(p^{-1}ZZ^{\intercal}) \times (\mathbf{1}^{\intercal}_n \boldsymbol{\epsilon}^j)^2/n.
$ Similarly to \eqref{eq:eachepsilonbound}, by Condition \ref{cond:concentrateprob} and the tail bound of the Chi-squared distribution, there exist some positive constants $c$  and $c_0$,
\begin{eqnarray}
	P\{(\mathbf{1}^{\intercal}_n \boldsymbol{\epsilon}^j)^2/n  > cn\sigma_{\epsilon,j}^2/\log n\}\leq O\{\exp(-c_0n/\log n)\}. \label{eq:eachepsilonbound11}
\end{eqnarray}  Combining \eqref{eq:eachepsilonbound} and \eqref{eq:eachepsilonbound11}, we know that for some constants $c_1,c_2$ and $c_0$, with probability $1-O\{m \exp(-c_0n/\log n)\}$, 
\begin{eqnarray}
	A_{\eta,j,1}+A_{\eta,j,2}\leq c_1pn\sum_{j=1}^m\sigma^2_{\epsilon,j}/\log n \leq c_2pnm/\log n, \label{eq:aasumbound}
\end{eqnarray} where the last inequality is from $\mathrm{diag}(\Sigma)\leq c_4$ for some constant $c_4>0$ by Condition \ref{cond:concentrateprob}.

Combining \eqref{eq:etaboundbyepsilonx}, \eqref{eq:aasumbound},  Lemma \ref{lm:ytransyminimumeigbound} and Conditions   \ref{cond:concentrateprob} and \ref{cond:dimensionrelation}, we know for some positive constants $c_1$, $c_2$ and $c_0$,
\begin{eqnarray}
	P(\| \boldsymbol{\eta} \|^2  >c_1\{ n\times n^{1-t}\}^{-1}pn^{1+\tau} m/\log n )  &\leq &O\{m\exp(-c_2n/\log n)\} \notag\\
	&=&O\{\exp(-c_0n/\log n)\}. \label{eq:etaupperboundres}
\end{eqnarray}


\paragraph{Step 2.2} (bounding  
  $|\eta_i|$ from above)  From Step 2.1,  we know
 \begin{eqnarray}
 	\eta_i^2\leq  \{n \lambda_{\min}(Y^{\intercal}P^{\intercal}PY)\}^{-1} \sigma_{x,i}^{-2} \sum_{j=1}^m (\boldsymbol{\epsilon}_j^{\intercal}P^{\intercal}P\mathbf{x}^i)^2. \label{eq:eta2upperboundsing}
 \end{eqnarray} 
  Then conditioning on $X$, 
$\sigma_{x,i}^{-1} \boldsymbol{\epsilon}_j^{\intercal}P^{\intercal}P\mathbf{x}^i \sim \mathcal{N}(0,\sigma_{\epsilon,j}^2 (\mathbf{x}^i)^{\intercal}P^{\intercal}P\mathbf{x}^i \sigma_{x,i}^{-2})$. Let $\mathcal{E}_1$ be the event $\{\mathrm{var}(\sigma_{x,i}^{-1}\boldsymbol{\epsilon}_j^{\intercal}P^{\intercal}P\mathbf{x}^i |X) \leq c_1n \}$ for some constant $c_1>0$. Note that 
\begin{eqnarray*}
	\mathrm{var}(\boldsymbol{\epsilon}_{\epsilon,j}^{\intercal}P^{\intercal}P\mathbf{x}^i  \sigma_{x,i}^{-1}|X)=\sigma_{\epsilon,j}^2\{ (\mathbf{x}^i)^{\intercal}\mathbf{x}^i- (\mathbf{x}^i)^{\intercal}\mathbf{1}_n^{\intercal}\mathbf{1}_n\mathbf{x}^i/n\}  \sigma_{x,i}^{-2} \leq \sigma_{\epsilon,j}^2 (\mathbf{x}^i)^{\intercal}\mathbf{x}^i  \sigma_{x,i}^{-2}.
\end{eqnarray*} 
Using the same argument as in Step 1.1, we can show that, there exist some positive constants $c_1$ and $c_0$,
\begin{eqnarray}
	P(\mathcal{E}_1^c) \leq P\{\sigma_{\epsilon,j}^2 (\mathbf{x}^i)^{\intercal}\mathbf{x}^i \sigma_{x,i}^{-2} > c_1n\}\leq O\{\exp(-c_0n)\},\label{eq:eeventsetbound}
\end{eqnarray} where $\mathcal{E}_1^c$ represents the complement of the event $\mathcal{E}_1$. On the event $\mathcal{E}_1$, for any $a>0$, by Condition \ref{cond:concentrateprob},  we have
\begin{eqnarray}
	P(|\boldsymbol{\epsilon}_j^{\intercal}P^{\intercal}P\mathbf{x}^i| \sigma_{x,i}^{-1} >a|X,\mathcal{E}_1) \leq P\{\sqrt{c_1n}|W|>a \}, \label{eq:prodtloewrboundterm}
\end{eqnarray} where $W$ is an independent $\mathcal{N}(0,1)$-distributed random variable. Thus, combining \eqref{eq:eeventsetbound} and  \eqref{eq:prodtloewrboundterm}, we have
\begin{eqnarray}
	P( |\boldsymbol{\epsilon}_j^{\intercal}P^{\intercal}P\mathbf{x}^i| \sigma_{x,i}^{-1}>a)\leq O\{\exp(-c_0n)\}+P\{\sqrt{c_1n}|W|>a \}.  \label{eq:prodtlowerboundtwoterms}
\end{eqnarray} 
Let $x_n'=\sqrt{2c_0c_1} n^{1-\kappa-t/2-s/2-u}/\sqrt{\log n}$. Invoking the classical Gaussian tail bound, we have  
\begin{eqnarray*}
	P\{\sqrt{c_1n}|W|>x_n'\}=O\{\exp(-c_0n^{1-2\kappa-t-s-2u}/\log n) \}.
\end{eqnarray*}

By \eqref{eq:eta2upperboundsing} and Lemma \ref{lm:ytransyminimumeigbound}, we then have
\begin{eqnarray*}
	P(|\eta_i|> (n^{1+1-t})^{-1/2}x_n'\sqrt{m}) &\leq & \sum_{j=1}^m P(  |\boldsymbol{\epsilon}_j^{\intercal}P^{\intercal}P\mathbf{x}^i|>x_n' )\notag \\
	&&+ P\{\lambda_{\min}(Y^{\intercal}P^{\intercal}PY)<c_1n^{1-t}  \},
\end{eqnarray*} where $(n^{1+1-t})^{-1/2}x_n'\sqrt{m}= \sqrt{2c_0c_1m}n^{-\kappa-s/2-u}/\sqrt{\log n}\leq c_2 n^{-\kappa-u}/\sqrt{\log n}$ for some constant $c_2>0$ by Condition \ref{cond:dimensionrelation}.   
In summary, we have
\begin{eqnarray}
	&& P[\max_{1\leq i \leq p} |\eta_i|>c_2 n^{-\kappa-u}/\sqrt{\log n}] \notag \\
	&\leq & O[p\exp\{-c_2n^{1-2\kappa-t-s-2u}/\log n \}]  \notag \\
	&=& O[\exp\{-c_0n^{1-2\kappa-t-s-2u}/\log n \}],\label{eq:etalowerboundresfinal}
\end{eqnarray} where the last equality is   from Condition \ref{cond:dimensionrelation}.



\paragraph{Step 3.} We combine the results in Steps 1 and 2. By Bonferroni's inequlaity,  it follows from  \eqref{eq:xiupperbound}, \eqref{eq:xilowerboundres}, \eqref{eq:etaupperboundres} and  \eqref{eq:etalowerboundresfinal} that, for some positive constants $\tilde{c}_1,\tilde{c}_2$ and $\tilde{c}$,
\begin{eqnarray}
	&& P\{ \min_{i \in \mathcal{M}_*}|\zeta_i|<\tilde{c}_1(n\times n^{2u+1})^{-1/2}n^{1-\kappa} \ \mathrm{or}\ \| \boldsymbol{\zeta}\|^2> \tilde{c}_2 (n\times n^{1-t})^{-1} n^{1+\tau} p m \} \notag \\
	&\leq & O[ |\mathcal{M}_*| \exp\{-\tilde{c}n^{1-2\kappa-2u-t-s}/\log n \}].  \label{eq:zetafinalprobbound}
\end{eqnarray}
By Lemma \ref{lm:approximatezetalm} and \eqref{eq:zetafinalprobbound}, we know that there exist  some positive constants $c_1,c_2$ and $c$, 
\begin{eqnarray*}
	&& P\{ \min_{i \in \mathcal{M}_*}|\omega_i|<c_1(n\times n^{2u+1})^{-1/2}n^{1-\kappa} \ \mathrm{or}\ \| \boldsymbol{\omega}\|^2>c_2 (n\times n^{-t})^{-1} n^{1+\tau} p m \} \notag \\
	&\leq & O[ |\mathcal{M}_*| \exp\{-cn^{1-2\kappa-2u-t-s}/\log n \}], 
\end{eqnarray*} which is $O[ \exp\{-c_0n^{1-2\kappa-2u-t-s}/\log n \}] $ for some constant $c_0>0$ by Condition \ref{cond:dimensionrelation} and $|\mathcal{M}_*| \leq p$. 
This shows that with overIwhelming probability $1-O[  \exp\{-c_0n^{1-2\kappa-2u-t-s}/\log n \}]$, the magnitudes of $\omega_i, i \in \mathcal{M}_*$ are uniformly at least of order $c_1 (n\times n^{2u+1})^{-1/2}n^{1-\kappa}$, and  for some positive constants $c_1$ and $c_2$, 
\begin{eqnarray*}
	\#\{1\leq k \leq p: |\omega_k|\geq \min_{i\in \mathcal{M}_*}|\omega_i| \} \leq c_1 \frac{(n\times n^{2u+1})\times pmn^{1+\tau}  }{ n\times n^{1-t} \times (n^{1-\kappa})^2} \leq c_2 pn^{s+2u+2\kappa+\tau+t-1},  
\end{eqnarray*} where the last inequality is from Condition \ref{cond:dimensionrelation}. 
Thus, if the proportion $\delta$ of features selected  satisfies
$
	\delta n^{1-2\kappa-2u-\tau-t-s} \to \infty,
$
then $ \delta p \geq c_2 pn^{s+\tau+t+2\kappa+2u-1} $ when $\delta$ is sufficiently large, and we know with probability $1-O[\exp\{-c_0n^{1-2\kappa-2u-\tau-t-s}/\log n \}]$, $\mathcal{M}_*\subseteq \mathcal{M}_{\delta} $. 
%



\subsection{Lemmas in the proof of Theorem \ref{thm:selectionaccuracy}} \label{sec:lemmaproofselect}

\begin{lemma}[Lemma 3 in \cite{fan2008sure}]\label{lm:chisqtailprob}
Let $\vartheta_i$, $i=1,2,\ldots,n$ be i.i.d. $\chi_1^2$-distributed random variables. Then	for any $\epsilon>0$, we have
$
	P( n^{-1}\sum_{i=1}^n \vartheta_i >1+\epsilon )\leq \exp(-A_{\epsilon}n),
$ where $A_{\epsilon}=[\epsilon-\log(1+\epsilon) ]/2>0$; for any $\epsilon \in (0,1)$, 
$
P(n^{-1}\sum_{i=1}^n \vartheta_i <1-\epsilon )\leq \exp(-B_{\epsilon}n),
$where $B_{\epsilon}=[-\epsilon-\log(1-\epsilon) ]/2>0$.
\end{lemma}

\begin{lemma}[Lemma 1 in \cite{fan2008sure}]\label{lm:invariantdist}
For $U$ and $(\mu_1,\ldots,\mu_{n})^{\intercal}$ in \eqref{eq:pztransposepz}, and  $\tilde{O}$ uniformly distributed on the orthogonal group $\mathcal{O}(p)$, we know that	
\begin{eqnarray}
	(I_{n},\mathbf{0})_{n\times p} U \overset{(d)}{=} (I_{n},\mathbf{0})_{n\times p}\tilde{O}, \label{eq:similarlm1}
\end{eqnarray} and $(\mu_1,\ldots,\mu_{n})^{\intercal}$ is independent of $(I_{n},\mathbf{0})_{n\times p} U$.
\end{lemma}
\begin{proof}
As $\mu_1^{1/2},\ldots,\mu_{n}^{1/2}$ are $n$ singular values of $Z$, we know that $Z$ has the singular value decomposition $Z=V_1D_1U$, where $V_1\in \mathcal{O}(n)$,  $U\in \mathcal{O}(p)$ is given in \eqref{eq:pztransposepz}, and $D_1$ is an $n\times p$ diagonal matrix whose  diagonal elements are $\mu_1^{1/2},\ldots,\mu_{n}^{1/2}$. Since the entries in $Z$ are i.i.d. $\mathcal{N}(0,1)$, for any $\tilde{O} \in \mathcal{O}(p)$, $Z\tilde{O} \overset{(d)}{=} Z$. Thus, conditional on $V_1$ and  $(\mu_1,\ldots,\mu_{n})^{\intercal}$, the conditional distribution of $(I_{n},\mathbf{0})_{n\times p} U$ is invariant under $\mathcal{O}(p)$. This shows that  \eqref{eq:similarlm1} holds for $\tilde{O}$ uniformly distributed on the orthogonal group $\mathcal{O}(p)$, and $(\mu_1,\ldots,\mu_{n})^{\intercal}$ is independent of $(I_{n},\mathbf{0})_{n\times p} U$.
\end{proof}

\begin{lemma}[Lemma 4 in \cite{fan2008sure}]\label{lm:grasmantailbd}
	$S$ defined in \eqref{eq:smatdecomposition} is uniformly distributed on the Grassmann manifold $\mathcal{G}_{p,n}$. For any constant $c_0>0$, there are constants $c_1$ and $c_2$ with $0< c_1< 1<c_2$ such that 
\begin{eqnarray*}
	P(\langle Se_1,e_1 \rangle <c_1n/p\mathrm{\ or} >c_2n/p) \leq 4\exp(-c_0n). 
\end{eqnarray*} 
\end{lemma}


\begin{lemma} \label{lm:eigenvalueboundonz}
The matrix $Z$ is of size $n\times p$ and the matrix $\tilde{Z}$ is of size $n\times m$ with Condition \ref{cond:dimensionrelation} satisfied. The entries in $Z$ and $\tilde{Z}$ are i.i.d. $\mathcal{N}(0,1)$. For some constants $c_1,c_0>0$,
\begin{eqnarray}
	P\{ \lambda_{\max}(n^{-1}{Z}{Z}^{\intercal})>{c}_1 \mathrm{\ or \ } \lambda_{\min}(n^{-1}{Z}^{\intercal}{Z})<1/{c}_1   \}\leq \exp(-{c}_0n). \label{eq:originalconcent}
\end{eqnarray} There exist some constants $c_1>1$, $c>0$ and $c_0>0$, when $ n>c$,
	\begin{eqnarray}
P[ \lambda_{\max}\{n^{-1}(P\tilde{Z})^{\intercal}P\tilde{Z}\}>c_1 \mathrm{\ or \ } \lambda_{\min}\{n^{-1}(P\tilde{Z})^{\intercal}P\tilde{Z}\}<1/c_1   ]\leq \exp(-c_0n/\log n).  \label{eq:concentrationinequalztilde}
\end{eqnarray}
\end{lemma}

\begin{proof}
As the entries in $Z$ are i.i.d. $\mathcal{N}(0,1)$, by Appendix A.7 in \cite{fan2008sure},  we know that \eqref{eq:originalconcent} holds. For $\tilde{Z}$, since its  entries are also i.i.d. $\mathcal{N}(0,1)$ and $n>c_7m$ for some $c_7>1$, symmetrically, we know there exist constants $\tilde{c}_1>1$ and $\tilde{c}_0>0$ such that
\begin{eqnarray}
	P\{ \lambda_{\max}(n^{-1}\tilde{Z}^{\intercal}\tilde{Z})>\tilde{c}_1 \mathrm{\ or \ } \lambda_{\min}(n^{-1}\tilde{Z}^{\intercal}\tilde{Z})<1/\tilde{c}_1   \}\leq \exp(-\tilde{c}_0n). \label{eq:concentrationproperty}
\end{eqnarray} 
Since
$
	(P\tilde{Z})^{\intercal}P\tilde{Z}=\tilde{Z}^{\intercal}P\tilde{Z}=\tilde{Z}^{\intercal}\tilde{Z}-\tilde{Z}^{\intercal}\mathbf{1}_n \mathbf{1}_n^{\intercal} \tilde{Z}/n,
$
by Weyl's inequality, we have
\begin{eqnarray}
	 \lambda_{\max}\{(P\tilde{Z})^{\intercal}P\tilde{Z} \} &\leq &\lambda_{\max}(\tilde{Z}^{\intercal}\tilde{Z})+\lambda_{\max}(-\tilde{Z}^{\intercal}\mathbf{1}_n \mathbf{1}_n^{\intercal} \tilde{Z}/n), \notag \\
	\lambda_{\min}\{(P\tilde{Z})^{\intercal}P\tilde{Z} \} &\geq & \lambda_{\min}(\tilde{Z}^{\intercal}\tilde{Z})+\lambda_{\min}(-\tilde{Z}^{\intercal}\mathbf{1}_n \mathbf{1}_n^{\intercal} \tilde{Z}/n). \label{eq:weylinequalres}
\end{eqnarray} 
Let $A_Z=\tilde{Z}^{\intercal}\mathbf{1}_n \mathbf{1}_n^{\intercal} \tilde{Z}/n$. As $\mathrm{rank}(A_Z)=1$ and $\mathrm{tr}(A_Z)\geq 0$, we know $\lambda_{\max}(-A_Z)=-\lambda_{\min}(A_Z)=0$ and $\lambda_{\min}(-A_Z)=-\lambda_{\max}(A_Z)=-\mathrm{tr}(A_Z)=-\mathbf{1}_n^{\intercal} \tilde{Z}\tilde{Z}^{\intercal}\mathbf{1}_n/n$.



We then examine $\mathbf{1}_n^{\intercal}\tilde{Z}\tilde{Z}^{\intercal}\mathbf{1}_n/n$. For $z_{ij} {\sim} \mathcal{N}(0,1)$ independently,
\begin{eqnarray*}
	\mathbf{1}_n^{\intercal}\tilde{Z}\tilde{Z}^{\intercal}\mathbf{1}_n/n=\sum_{j=1}^m \Big(\sum_{i=1}^n z_{ij}/\sqrt{n}\Big)^2 \sim \chi^2_m. 
\end{eqnarray*} 
By Lemma \ref{lm:chisqtailprob}, we know for the random variable $\tilde{W}\sim \chi^2_m$ and any constant $c_2>0$, there exists constant $c_3>0$ such that
\begin{eqnarray}
	P\{ |\tilde{W}/m-1| > c_2 n/(m\log n) \} \leq \exp\{-c_3 m\times n/(m\log n)\}, \label{eq:chisqdev}
\end{eqnarray}
 This implies that with probability $1-O\{\exp(-c_3n/\log n)\}$, $\lambda_{\max}(A_Z/n)=\mathbf{1}_n^{\intercal}\tilde{Z}\tilde{Z}^{\intercal}\mathbf{1}_n/n^2 \leq c_2/\log n$ for some constant $c_2>0$, as $m=O(n^s)$ with $s\in[0,1)$. 

When $n$ is sufficiently large, there exists constant $c_1$ such that $1>c_1>\tilde{c}_1$ and  $1/c_1+c_2/\log n<1/\tilde{c}_1$. Thus by \eqref{eq:weylinequalres} and \eqref{eq:chisqdev}, we know there exists constant $c_0>0$ such that with probability $1-\exp(-c_0n/\log n)$, 
\begin{eqnarray*}
	\{ \lambda_{\min}(n^{-1}(P\tilde{Z})^{\intercal}P\tilde{Z})<1/c_1 \} &\subseteq &  \{ \lambda_{\min}(n^{-1}\tilde{Z}^{\intercal}\tilde{Z})<1/c_1+c_2/\log n \} \notag \\
	&\subseteq &  \{ \lambda_{\min}(n^{-1}\tilde{Z}^{\intercal}\tilde{Z})<1/\tilde{c}_1 \}, \notag \\
		\{ \lambda_{\max}(n^{-1}(P\tilde{Z})^{\intercal}P\tilde{Z})>c_1 \}  &\subseteq &  \{ \lambda_{\max}(n^{-1}\tilde{Z}^{\intercal}\tilde{Z})>\tilde{c}_1 \}.
\end{eqnarray*}
By \eqref{eq:concentrationproperty}, \eqref{eq:concentrationinequalztilde} is then proved.
	
\end{proof}

\begin{lemma} \label{lm:rtildeinvariant}
For $\mathbf{R}=(R_1,\ldots, R_p)^{\intercal}=\tilde{U}^{\intercal}\mathrm{diag}(\mu_1,\ldots,\mu_{n}) \tilde{U}e_1$, there exist positive constants $c_1$ and $c_0$  such that 
$P(|R_2|>c_1n^{1/2}|W_1|)\leq O\{\exp(-c_0n\}$, where $|W_1|$ is an independent $\mathcal{N}(0,1)$-distributed random variable. 
\end{lemma}
\begin{proof}
Let $\tilde{\mathbf{R}}=(R_2,\ldots, R_p)^{\intercal}$.  We first show that $\tilde{\mathbf{R}}$ is invariant under the orthogonal group $\mathcal{O}(p-1)$. 
For any $Q\in \mathcal{O}(p-1)$, let $\tilde{Q}=\mathrm{diag}(1,Q)\in \mathcal{O}(p)$. By Lemma \ref{lm:invariantdist}, we know that $\tilde{U}$ is independent of $\mathrm{diag}(\mu_1,\ldots,\mu_n)$ and $\tilde{Q}\tilde{U}\overset{(d)}{=}\tilde{U}$. Thus  
\begin{eqnarray*}
	\tilde{Q}^{\intercal}\mathbf{R}=\tilde{Q}^{\intercal} \tilde{U}^{\intercal} \mathrm{diag}(\mu_1,\ldots,\mu_n)\tilde{U}\tilde{Q}\tilde{Q}^{\intercal}e_1 \overset{(d)}{=}  \tilde{U}^{\intercal} \mathrm{diag}(\mu_1,\ldots,\mu_n)\tilde{U}e_1, 
\end{eqnarray*} where we use the fact that $\tilde{Q}^{\intercal}e_1=e_1$. This implies that $\tilde{\mathbf{R}}$ is invariant under the orthogonal group $\mathcal{O}(p-1)$. It follows that $\tilde{\mathbf{R}}\overset{(d)}{=}\|\tilde{\mathbf{R}}\|\mathbf{W}/  \| \mathbf{W}\|_2$, where $\mathbf{W}=(W_1,\ldots,W_{p-1})^{\intercal} \sim \mathcal{N}(0,I_{p-1})$, independent of $\|\tilde{ \mathbf{R}}\|$. 

In particular, we have $R_2 \overset{(d)}{=}\| \tilde{\mathbf{R}}\|W_1/  \| \mathbf{W}\|_2$. Note that $\| \tilde{\mathbf{R}}\|\leq \| \mathbf{R}\|$ and $\|\mathbf{R}\|^2=e_1^{\intercal}\tilde{U}^{\intercal}\mathrm{diag}(\mu_1^2,\ldots,\mu_n^2)\tilde{U}e_1$. Since $\mu_1,\ldots,\mu_n \leq p\lambda_{\max}(p^{-1}ZZ^{\intercal})$, $$\|\mathbf{R}\|^2 \leq  \{ \lambda_{\max}(p^{-1}ZZ^{\intercal}) \}^2 p^2  e_1^{\intercal}\tilde{U}^{\intercal}\tilde{U}e_1 {=}  \{ \lambda_{\max}(p^{-1}ZZ^{\intercal}) \}^2 p^2 \langle Se_1, e_1\rangle.$$ By Lemmas \ref{lm:invariantdist} and  \ref{lm:eigenvalueboundonz}, we then know for some positive  constants $c_1$ and $c_0$, $P(\|\mathbf{R}\|> c_1 \sqrt{pn})\leq O\{\exp(-c_0n)\}$.  Moreover, by Lemma \ref{lm:chisqtailprob}, we know for some constants $c_1$ and $c_0$, $P\{\|\mathbf{W}\|^2<c_1 (p-1)\}\leq \exp(-c_0n)$ when $p>n$. Thus we obtain that for some constants $c_1$ and $c_0$, $P(|R_2|>c_1 |W_1| n^{1/2})\leq O\{\exp(-c_0 n)\}$, where $W_1$ is an independent $\mathcal{N}(0,1)$-distributed random variable.    
\end{proof}

\begin{lemma} \label{lm:approximatezetalm}
For some constant $c_0>0$, $\omega_i=\zeta_i\{1+o(1)\}$ with probability $1-O\{\exp(-c_0n/\log n)\}$.	
\end{lemma}
\begin{proof}
 By the definitions of $\omega_i$ and $\zeta_i$ and $P^{\intercal}P=P$,  we know
 	\begin{eqnarray*}
	\omega_i-\zeta_i &\leq & \max_{\mathbf{a}}\frac{ \mathbf{a}^{\intercal}Y^{\intercal} P \mathbf{x}^i ( \sqrt{n}\sigma_{x,i} -\sqrt{ (\mathbf{x}^i)^{\intercal}P\mathbf{x}^i})}{\sigma_{x,i} \sqrt{ n  ((\mathbf{x}^i)^{\intercal}P\mathbf{x}^i) (\mathbf{a}^{\intercal} Y^{\intercal}PY\mathbf{a}) }}=\zeta_i \frac{ \sqrt{n}\sigma_{x,i} -\sqrt{(\mathbf{x}^i)^{\intercal}P\mathbf{x}^i}  }{\sqrt{(\mathbf{x}^i)^{\intercal}P\mathbf{x}^i} },\notag \\
	\zeta_i-\omega_i &\leq & \max_{\mathbf{a}}\frac{\mathbf{a}^{\intercal} Y^{\intercal}P^{\intercal}\mathbf{x}^i (\sqrt{(\mathbf{x}^i)^{\intercal}P\mathbf{x}^i}-\sqrt{n}\sigma_{x,i} )}{\sigma_{x,i} \sqrt{ n ((\mathbf{x}^i)^{\intercal}P\mathbf{x}^i) (\mathbf{a}^{\intercal} Y^{\intercal}PY\mathbf{a}) }} =\zeta_i \frac{ \sqrt{(\mathbf{x}^i)^{\intercal}P\mathbf{x}^i}-\sqrt{n} \sigma_{x,i}  }{\sqrt{(\mathbf{x}^i)^{\intercal}P\mathbf{x}^i} }.
\end{eqnarray*}
Thus
\begin{eqnarray*}
	|\omega_i-\zeta_i|\leq \zeta_i \Biggr|   \frac{ \sqrt{(\mathbf{x}^i)^{\intercal}P\mathbf{x}^i}-\sqrt{n} \sigma_{x,i} }{\sqrt{(\mathbf{x}^i)^{\intercal}P\mathbf{x}^i} }\Biggr|=\zeta_i \Big|   1-\sigma_{x,i} \sqrt{n/\{(\mathbf{x}^i)^{\intercal}P\mathbf{x}^i\} }\Big|.
\end{eqnarray*} 
Let $\bar{x}^i=\sum_{k=1}^n x_{k,i}/n$, which is the mean of the entries in $\mathbf{x}^i$. It follows that $(\mathbf{x}^i)^{\intercal}P\mathbf{x}^i=\sum_{k=1}^n (x_{k,i}-\bar{x}^i)^2=\sum_{k=1}^n x_{k,i}^2- n (\bar{x}^i)^2$. Then with $\tilde{c}_1=c_1/2$
\begin{eqnarray}
	&& P(|(\mathbf{x}^i)^{\intercal}P\mathbf{x}^i/(n\sigma_{x,i}^2)-1|>c_1/{\log n}) \label{eq:diffgausbdd} \\
	&\leq &  P\{ (\bar{x}^i)^2/\sigma_{x,i}^2 > \tilde{c}_1/{\log n} \}+ P\Big\{ \Big| \sum_{k=1}^n x_{k,i}^2/(n\sigma_{x,i}^2)-1 \Big|>\tilde{c}_1/{\log n}\Big\}. \notag
\end{eqnarray} By Condition \ref{cond:concentrateprob}, $(\sqrt{n}\bar{x}^i)^2/\sigma_{x,i}^2 \sim \chi^2_1$. Then by the tail of $\chi_1^2$ distribution, 
for some constant $c_0>0$, 
$
	P\{ (\bar{x}^i)^2/\sigma_{x,i}^2 > \tilde{c}_1/{\log n} \} \leq O\{\exp(-c_0 n/{\log n}) \}. 
$ In addition, $x_{k,i}^2/\sigma_{x,i}^2$, $k=1,\ldots,n$ are i.i.d. $\chi_1^2$-distributed random variables. By Lemma  \ref{lm:chisqtailprob}, there exists some constant $c_0>0$, 
\begin{eqnarray*}
	P\Big\{ \Big| \sum_{k=1}^n x_{k,i}^2/(n\sigma_{x,i}^2)-1 \Big|>\tilde{c}_1/{\log n}\Big\}\leq O\{\exp(c_0 n/{\log n} )\}.
\end{eqnarray*}
In summary, we know for any constant $c_1>0$, there exists constant $c_0>0$ such that $\eqref{eq:diffgausbdd} \leq O\{\exp(-c_0n/\log n)\}.$
Thus, for $i=1,\ldots,p$, $\omega_i=\zeta_i\{1+O(1/\sqrt{\log n})\}=\zeta_i(1+o(1))$ with probability $1-O\{p\exp(-c_0n/\log n)\}=1-O\{\exp(-c_0n/\log n)\}$, where the last equality is from Condition \ref{cond:dimensionrelation}.
\end{proof}

\begin{lemma} \label{lm:ytransyminimumeigbound}
Consider  $n>c$ for $c$ in Lemma \ref{lm:eigenvalueboundonz}. There exist constants $c_1,c_2$ and $c_0$, with probability $1-O\{\exp(-c_0n/\log n)\}$,
\begin{eqnarray*}
	\lambda_{\min}(Y^{\intercal}P^{\intercal}PY)\geq  c_1n^{1-t}, 
\end{eqnarray*}
and for $\mathbf{a}_{0,i}$ in Condition \ref{cond:minimumbound}, 
\begin{eqnarray}
	\mathbf{a}_{0,i}^{\intercal}Y^{\intercal}P^{\intercal}PY\mathbf{a}_{0,i} \leq c_2 n^{2u+1}. \label{eq:denominatorbound}
\end{eqnarray}
\end{lemma}

\begin{proof}
	Since $X$ and $E$ follow independent Gaussian distributions by Condition \ref{cond:concentrateprob},   the rows of $Y$ are  independent multivariate Gaussian   with mean zero and covariance $\Sigma_y=B^{\intercal}\Sigma_xB+\Sigma$. Define $ \tilde{Z}=Y\Sigma_{y}^{-1/2}.$ Then $\tilde{Z}$ is of size $n\times m$ and the entries in $\tilde{Z}$ are i.i.d. $\mathcal{N}(0,1)$. Thus the concentration inequality \eqref{eq:concentrationinequalztilde} in  Lemma \ref{lm:eigenvalueboundonz} holds.   It follows that there exist constants $c_1$ and $c_0$, with probability $1-O\{\exp(-c_0n)\}$,
\begin{align*}
	Y^{\intercal}P^{\intercal}PY =& \Sigma_y^{1/2}\tilde{Z}^{\intercal} P^{\intercal}P\tilde{Z}\Sigma_y^{1/2} \notag \\
	\succeq & n\Sigma_y^{1/2} \lambda_{\min}(n^{-1}\tilde{Z}^{\intercal} P^{\intercal}P\tilde{Z})I_m\Sigma_y^{1/2} \succeq  c_1n\Sigma_y.
\end{align*}  By Weyl's inequality and Condition \ref{cond:concentrateprob},  we then know 
\begin{eqnarray}
	\lambda_{\min}(Y^{\intercal}P^{\intercal}PY)\geq c_1 n\lambda_{\min}(\Sigma_y) \geq c_1 n\lambda_{\min}(\Sigma)\geq c_1 n^{1-t}. \label{eq:minimeigenytransy}
\end{eqnarray}

Similarly we know for some constant $c_2$, with probability $1-O\{\exp(-c_0n)\}$,
\begin{eqnarray*}
	\mathbf{a}_{0,i}^{\intercal}Y^{\intercal}P^{\intercal}PY\mathbf{a}_{0,i} &=& \mathbf{a}_{0,i}^{\intercal}\Sigma_y^{1/2}\tilde{Z}^{\intercal} P^{\intercal}P\tilde{Z}\Sigma_y^{1/2}\mathbf{a}_{0,i}  \notag \\
	&\leq & c_2n\mathbf{a}_{0,i}^{\intercal}\Sigma_y \mathbf{a}_{0,i}\notag \\
	&=&c_2n\mathbf{a}_{0,i}^{\intercal}( B^{\intercal}\Sigma_x B+\Sigma)\mathbf{a}_{0,i} \notag \\
	&\leq & c_2n^{2u+1},
\end{eqnarray*}  where the last inequality is from Condition \ref{cond:minimumbound}.
\end{proof}

\section{Proposition \ref{prop:sizecontrol} \cite[Theorem 3.2]{meinshausen2009p}} \label{sec:proofpropsize}
The proof of Proposition \ref{prop:sizecontrol} directly follows the proof in \cite{meinshausen2009p}. For $z\in (0,1)$, define
\begin{eqnarray}
	\psi(z)=\frac{1}{J}\sum_{j=1}^J 1\{p^{(j)}\leq z \}. \label{eq:defpiu}
\end{eqnarray} Note that $\{Q(\gamma)\leq \alpha\}$ and $\{\psi(\alpha \gamma) \geq \gamma \}$ are equivalent.
For a random variable $U$ taking values in $[0,1]$,
\begin{eqnarray*}
	\sup_{\gamma\in (\gamma_{\min},1)} \frac{1\{U\leq \alpha \gamma \} }{\gamma}=\left\{\begin{matrix}
0 & U \geq \alpha  \\ 
\alpha/U &  \alpha \gamma_{\min} \leq U <\alpha  \\ 
1/\gamma_{\min} & U<\alpha \gamma_{\min}. 
\end{matrix}\right.
\end{eqnarray*}
Thus when $U$ has a uniform distribution on $[0,1]$,
\begin{eqnarray*}
	\mathrm{E}\Big[ \sup_{\gamma \in (\gamma_{\min},1)} \frac{1\{ U\leq  \alpha \gamma \}}{\gamma}  \Big]=\int_{0}^{\alpha \gamma_{\min}}\gamma_{\min}^{-1} dx+\int_{\alpha \gamma_{\min}}^{\alpha} \alpha x^{-1} dx =\alpha (1-\log \gamma_{\min}).
\end{eqnarray*} Hence, define the event $\mathcal{B}_{(j)}$ as $\mathcal{M}_*\subseteq \mathcal{M}_{\delta} \ \mathrm{for\ the}\ j\mathrm{th\ split}$, then 
\begin{eqnarray*}
	\mathrm{E}\Big[ \sup_{\gamma \in (\gamma_{\min},1)} 1\{ p^{(j)}\leq \alpha \gamma \}/\gamma \Big] &\leq &  \mathrm{E}\Big\{ \mathrm{E}\Big[ \sup_{\gamma \in (\gamma_{\min},1)} 1\{ p^{(j)}\leq \alpha \gamma \}/\gamma  \Big| \mathcal{B}^{(j)} \Big]\Big\}+ \frac{1}{\gamma_{\min}} P\{\mathcal{B}^{(j)}\} \notag \\
	&\leq & \alpha (1-\log \gamma_{\min}) + O[\exp\{-c_0n^{1-\iota}/\log n \}],
\end{eqnarray*} where the constant $\iota$ is given in Theorem \ref{thm:selectionaccuracy}. 
Averaging over $J$ splits yields
\begin{eqnarray*}
	\mathrm{E}\Big[\sup_{\gamma\in (\gamma_{\min},1)} \frac{1}{\gamma}\frac{1}{J}\sum_{j=1}^J 1\{p^{(j)}/\gamma \leq \alpha \} \Big] \leq \alpha (1-\log \gamma_{\min})+ O[\exp\{-c_0n^{1-\iota}/\log n \}].
\end{eqnarray*}
From Markov inequality and \eqref{eq:defpiu},
$
	\mathrm{E}[\sup_{\gamma \in (\gamma_{\min},1)} 1\{\psi(\alpha \gamma) \geq \gamma \} ]\leq \alpha(1-\log \gamma_{\min})+ O[\exp\{-c_0n^{1-\iota}/\log n \}].
$
Since $\{Q(\gamma)\leq \alpha\}$ and $\{\psi (\alpha \gamma )\geq \gamma \}$ are equivalent,  it follows that
$
	P[\inf_{\gamma \in (\gamma_{\min},1)} Q(\gamma) \leq \alpha ] \leq \alpha(1-\log \gamma_{\min})+ O[\exp\{-c_0n^{1-\iota}/\log n \}],
$ which implies that
$
	P[\inf_{\gamma \in (\gamma_{\min},1)} Q(\gamma)(1-\log \gamma_{\min}) \leq \alpha ]\leq \alpha+ O[\exp\{-c_0n^{1-\iota}/\log n \}].
$ By definition of $p_t$,
$
	\limsup_{n\to \infty}  P[p_t \leq \alpha ] \leq \alpha
$ is obtained.


\section{Supplementary Simulations} \label{sec:suppsimu}
\subsection{Supplementary simulations when $n>p+m$}
\subsubsection{Estimated type I errors} \label{sec:additiontypei}
We provide additional simulations under $H_0$ following the same set-up as in Figure \ref{fig:type1errort1norm}. In Figure \ref{fig:typeierror}, we present the estimated type I errors of the $\chi^2$ approximation and the normal approximations of $T_1$ and $T_3$ with varying $m$ and $r$ respectively.  It exhibits similar pattern as in Figure \ref{fig:type1errort1norm}, which shows that as $(p,m,r)$ become larger, the $\chi^2$ approximation performs poorly, while the normal approximations for $T_1$ and $T_3$ still control the type I error well.  

\begin{figure}[!h]
\centering
\begin{subfigure}{\textwidth}
\centering
\includegraphics[width=0.45\textwidth]{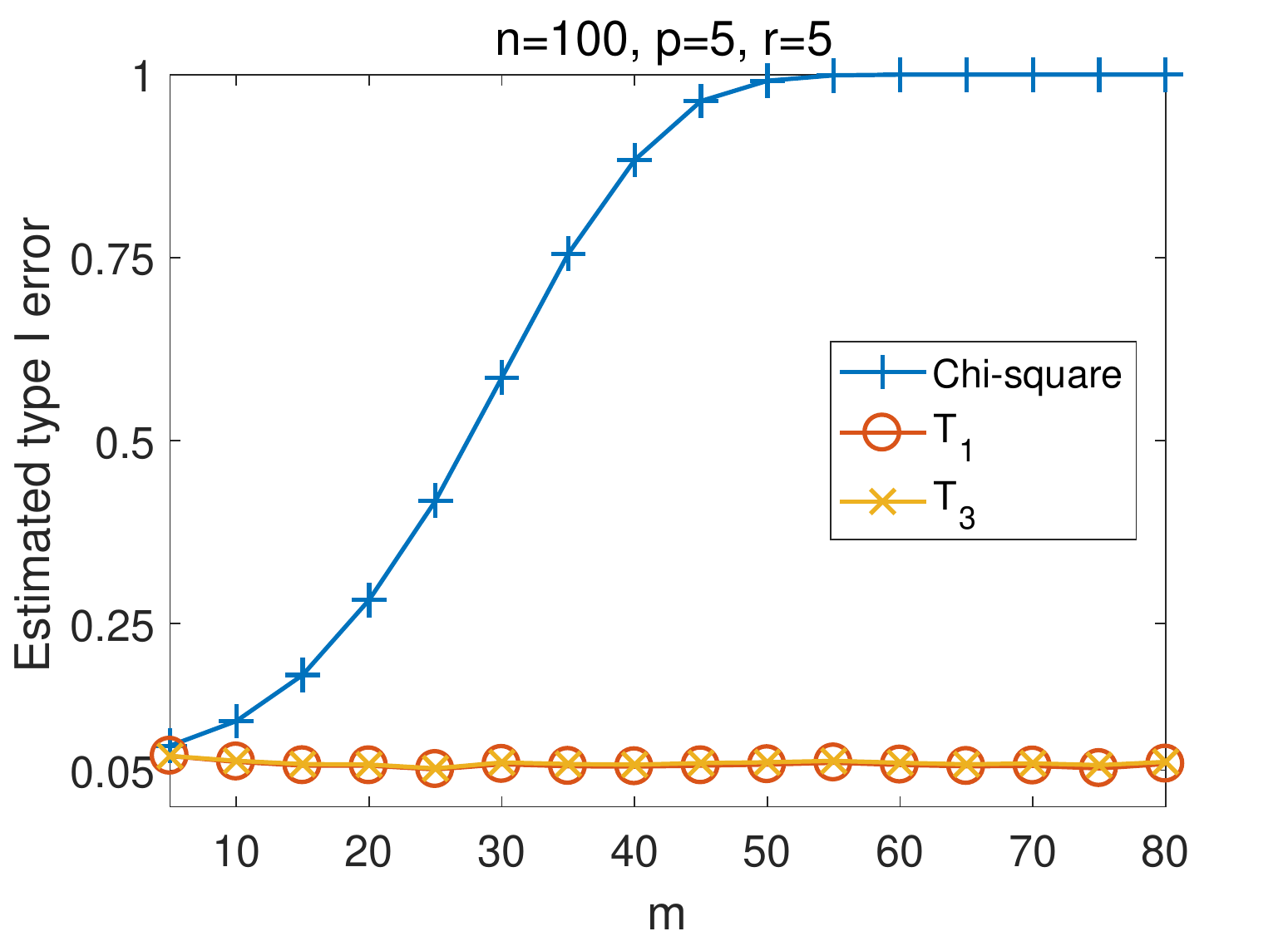}
\includegraphics[width=0.45\textwidth]{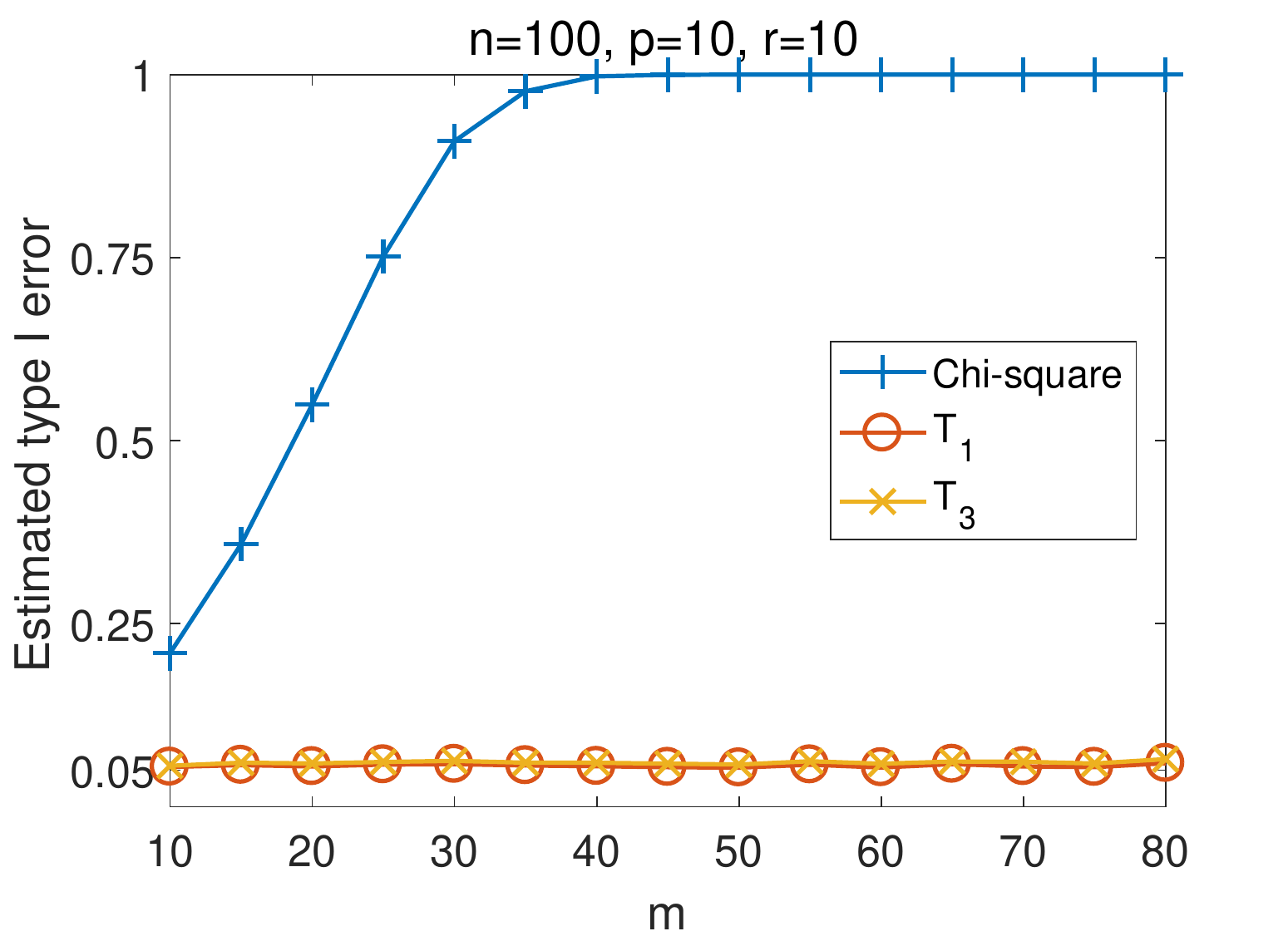}
\caption{Estimated type I error versus $m$}
\end{subfigure}
\begin{subfigure}{\textwidth}
\centering
\includegraphics[width=0.45\textwidth]{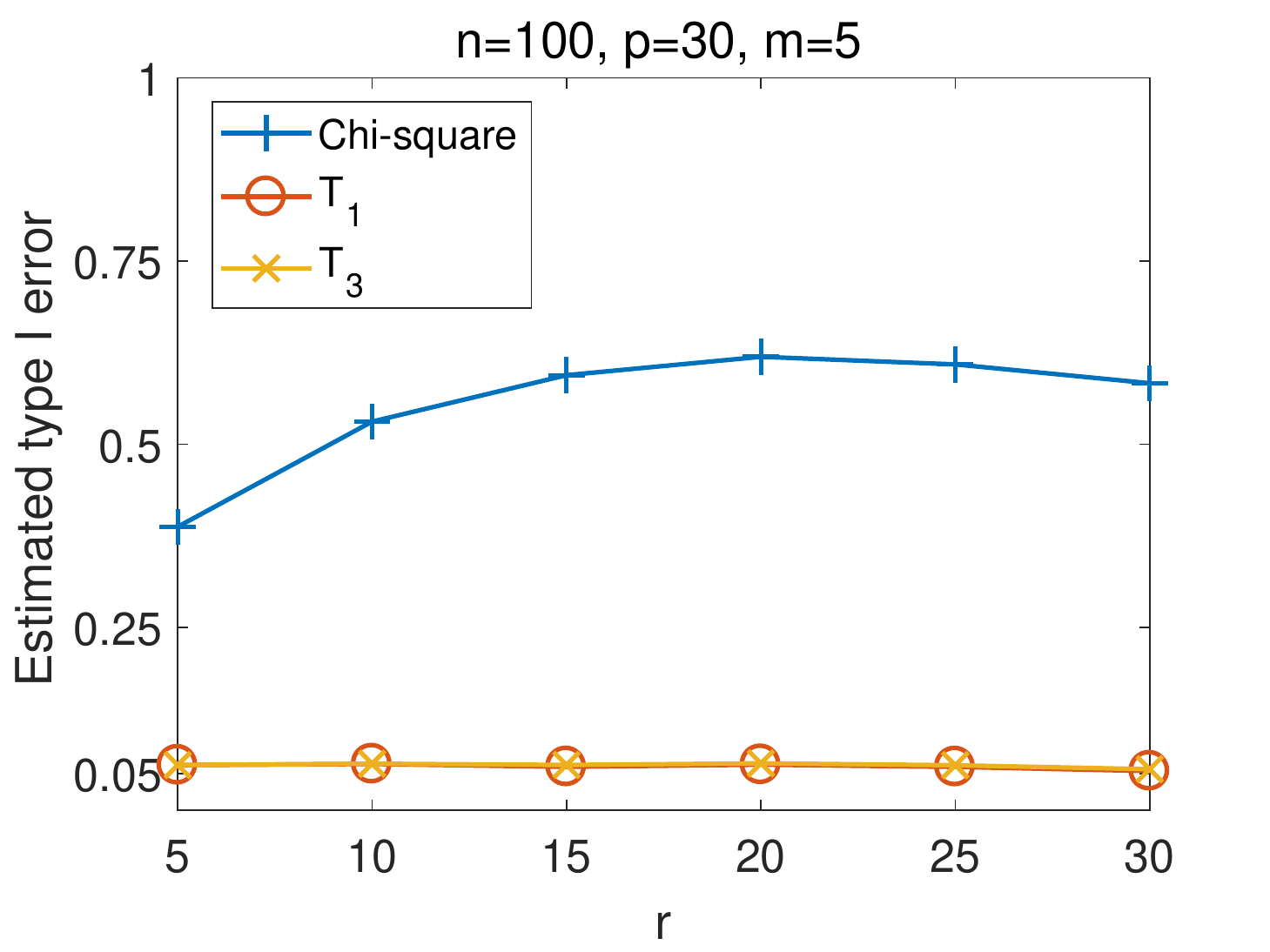}
\includegraphics[width=0.45\textwidth]{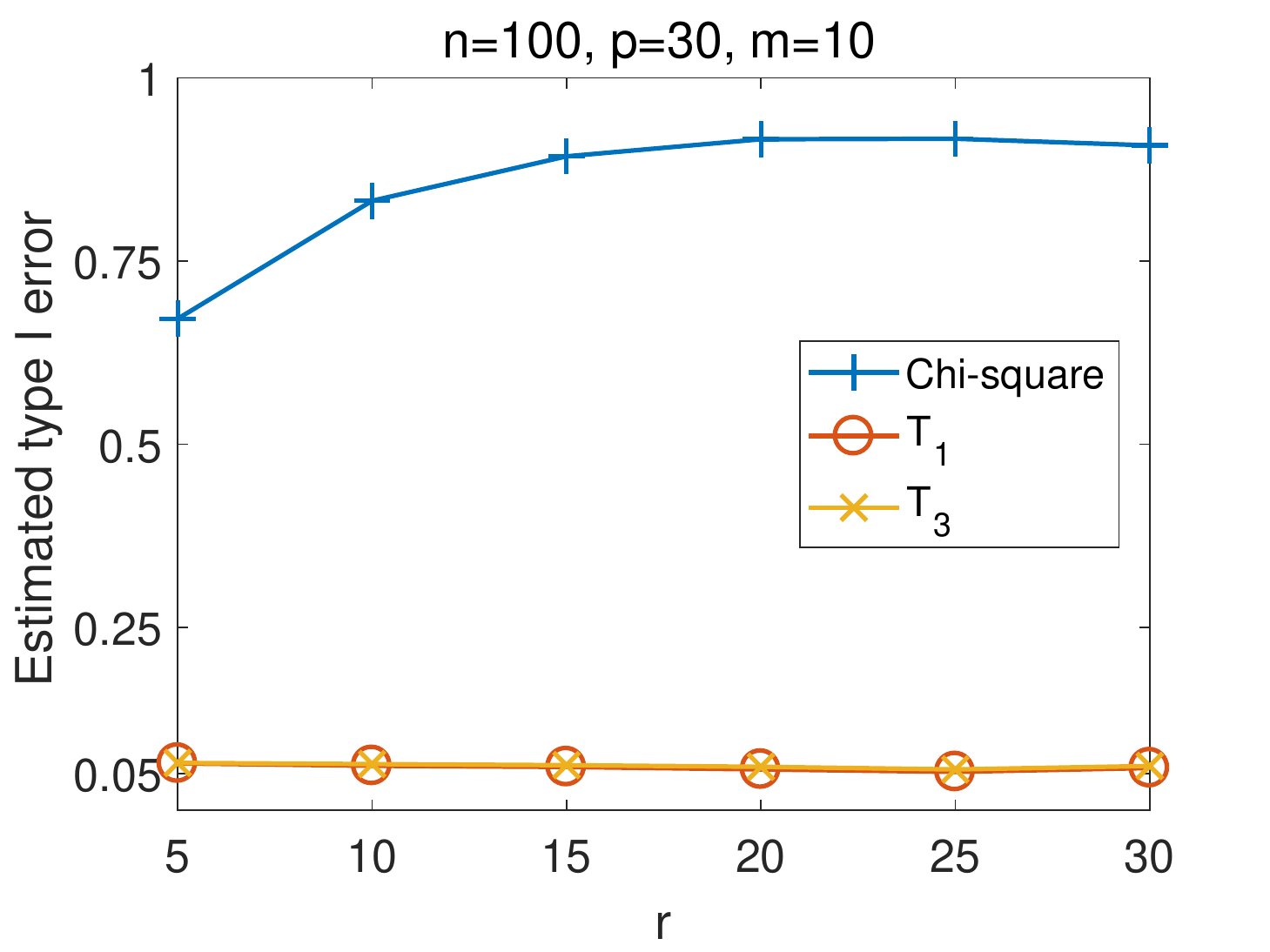}
\caption{Estimated type I error versus $r$}
\end{subfigure}
\caption{Estimated type I error}
\label{fig:typeierror}
\end{figure}

\subsubsection{Additional simulations under alternative hypotheses} \label{sec:additionaltersimu}

In this section, we generate data from the multivariate regression model $Y=XB+E$, where the rows of $X$ and $E$ are independent multivariate Gaussian with covariance matrices $\Sigma_x=(\rho^{|i-j|})_{p\times p}$ and $\Sigma=(\rho^{|i-j|})_{m\times m}$ respectively. 
 We  consider a sparse scenario   when only the $(1,1)$-entry of $B$ is nonzero with  a value $v_d$. 
We also consider a dense scenario when all the entries of $B$ are independently generated from $\mathcal{N}(0,\sigma_d^2)$. For each scenario,    we estimate the test powers for different $v_d$ or  $\sigma_d^2$ values, which are     referred to as the signal sizes in the following.  
We take $n=100, m=20, p=50, r=30$ and conduct 10,000 simulations  for two different $C$ matrices. 
In the first case, we take $C=[I_{r},\mathbf{0}_{r\times (p-r)}]$, where $I_r$ is an identity matrix of dimension $r\times r$, $\mathbf{0}_{r\times (p-r)}$ is an all zero matrix of dimension $r\times (p-r)$.   Then $H_0: CB=\mathbf{0}_{r\times m}$ examines the relationship between $Y$ and the first $r$ predictors of $X$.   
In the second case, we take $C=[I_r, \mathbf{0}_{r\times(p-r-1)}, -\mathbf{1}_r ]$, where $\mathbf{1}_r$ is an all 1 vector of length $r$, and $\mathbf{0}_{r\times(p-r-1)}$ is an all zero matrix of dimension $r\times (p-r-1)$.
 Then $H_0: CB=\mathbf{0}_{r\times m}$ tests the equivalence of effects of the first $r$ predictors and the last predictor. For two types of $B$ and two types of $C$ matrices, we plot the estimated powers of $T_1$, $T_2$, $T_3$ versus signal sizes   with $\rho=0.7$, $\rho=0.5$ and $\rho =0$ in Figures  
 \ref{fig:poweresttwo}, \ref{fig:poweresttwoprho05}
and \ref{fig:poweresttwosupp} respectively, where similar results are observed. 

 \begin{figure}[htbp]
\centering
\begin{subfigure}{\textwidth}
\centering
	\includegraphics[width=0.45\textwidth]{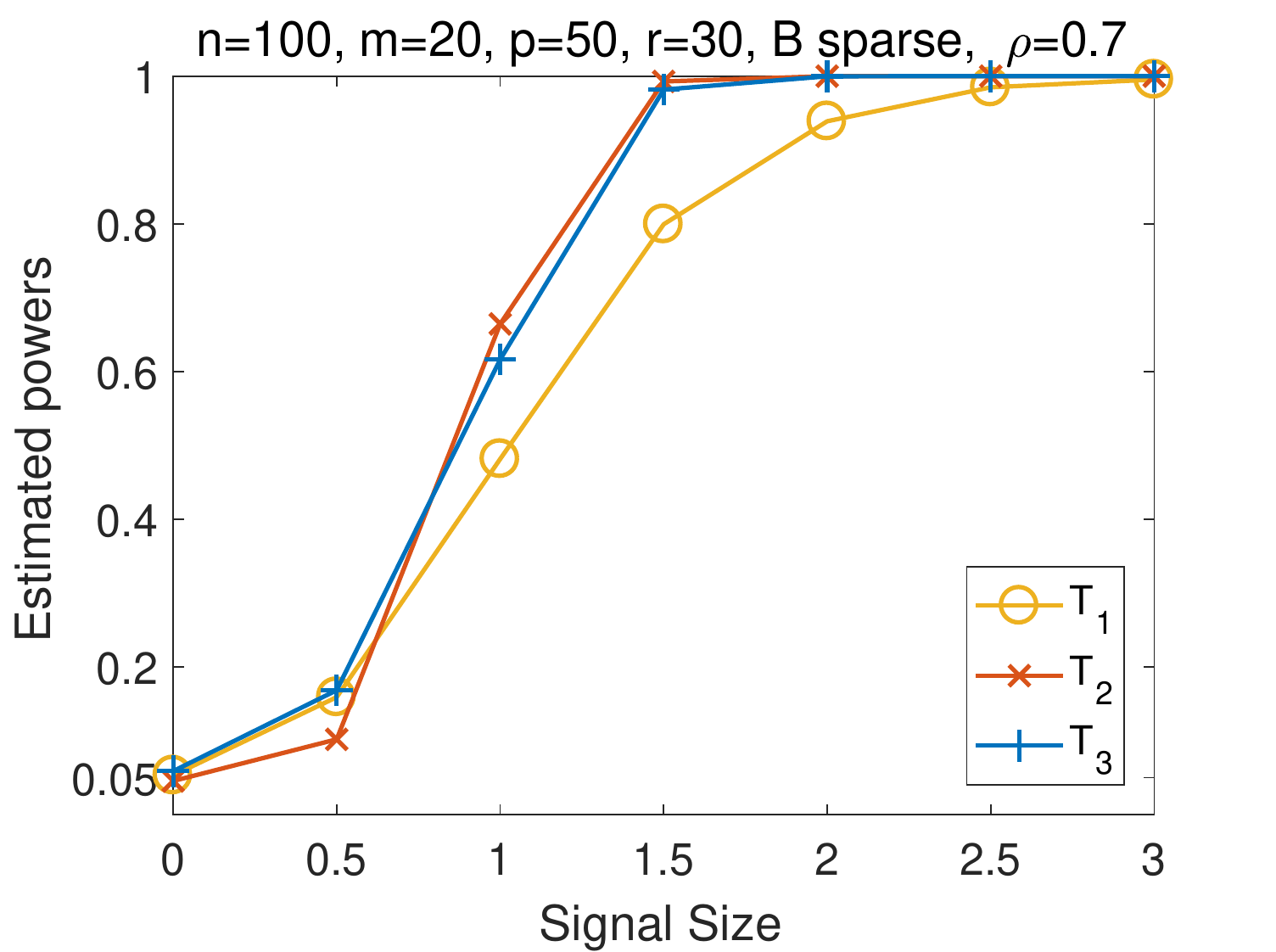}
	\includegraphics[width=0.45\textwidth]{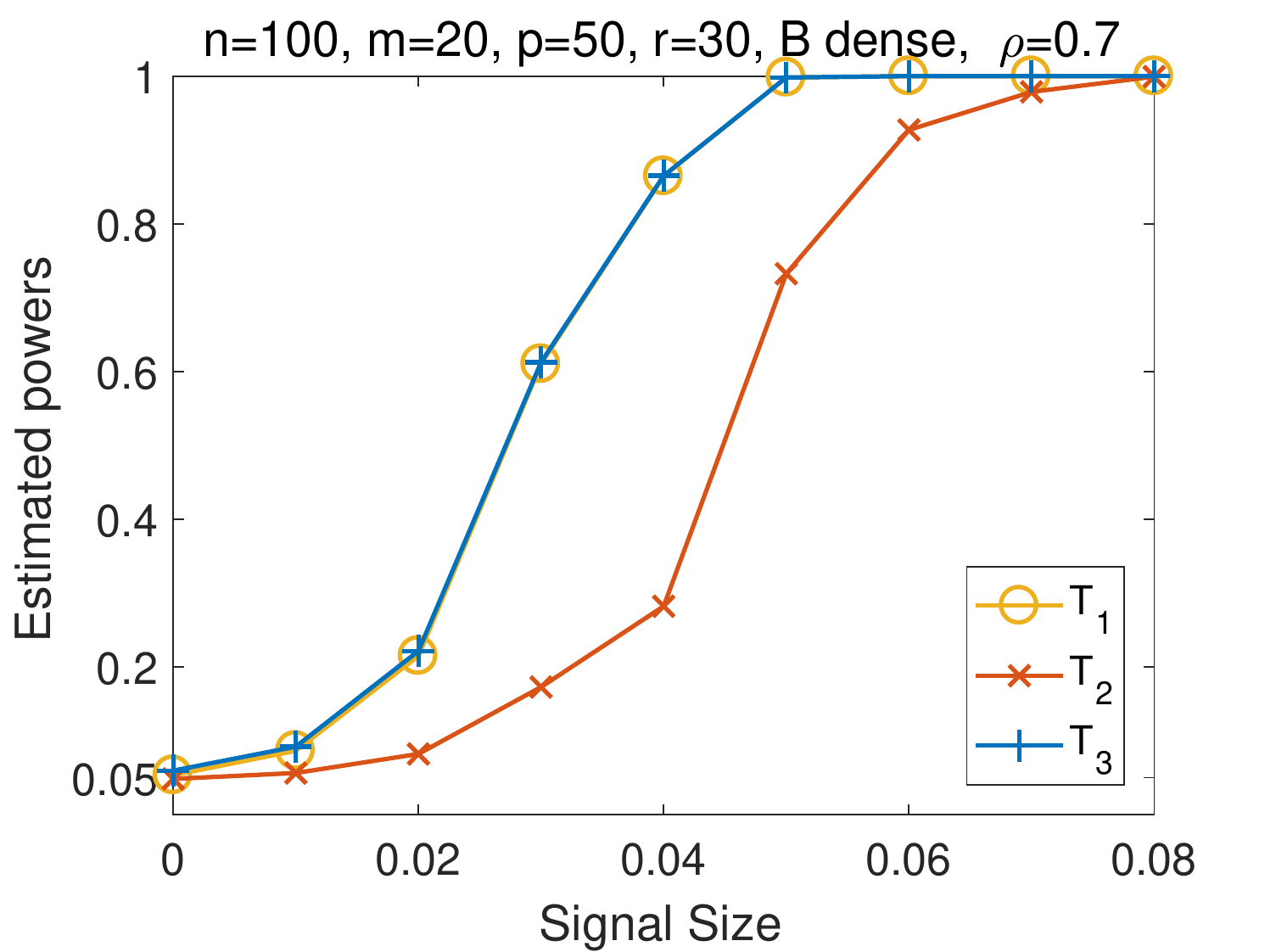}
\caption{Estimated powers versus signal sizes when $C=[I_{r},\mathbf{0}_{r\times (p-r)}]$}
\label{fig:powercident}
\end{subfigure}
\begin{subfigure}{\textwidth}
\centering
\includegraphics[width=0.45\textwidth]{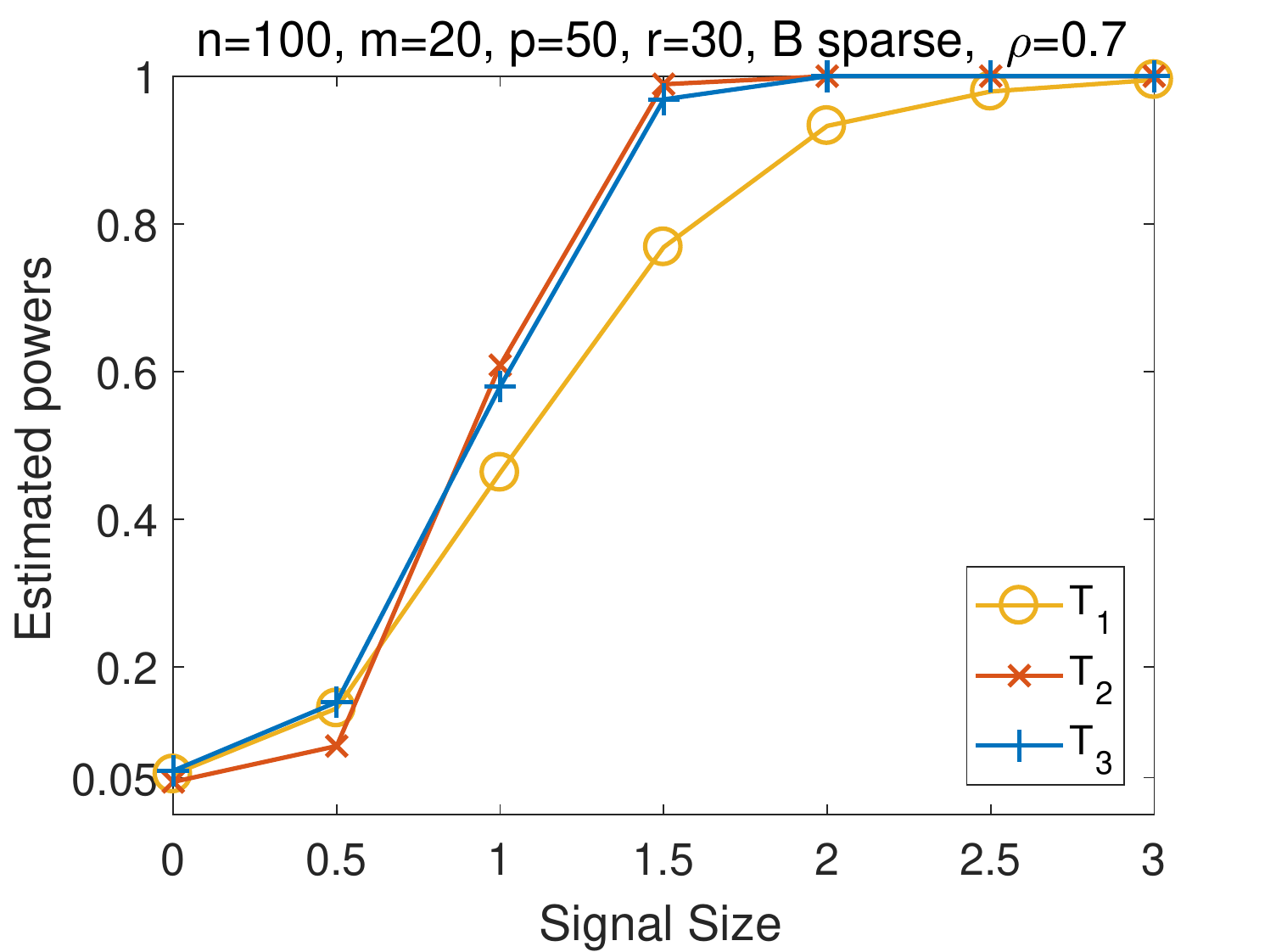}
\includegraphics[width=0.45\textwidth]{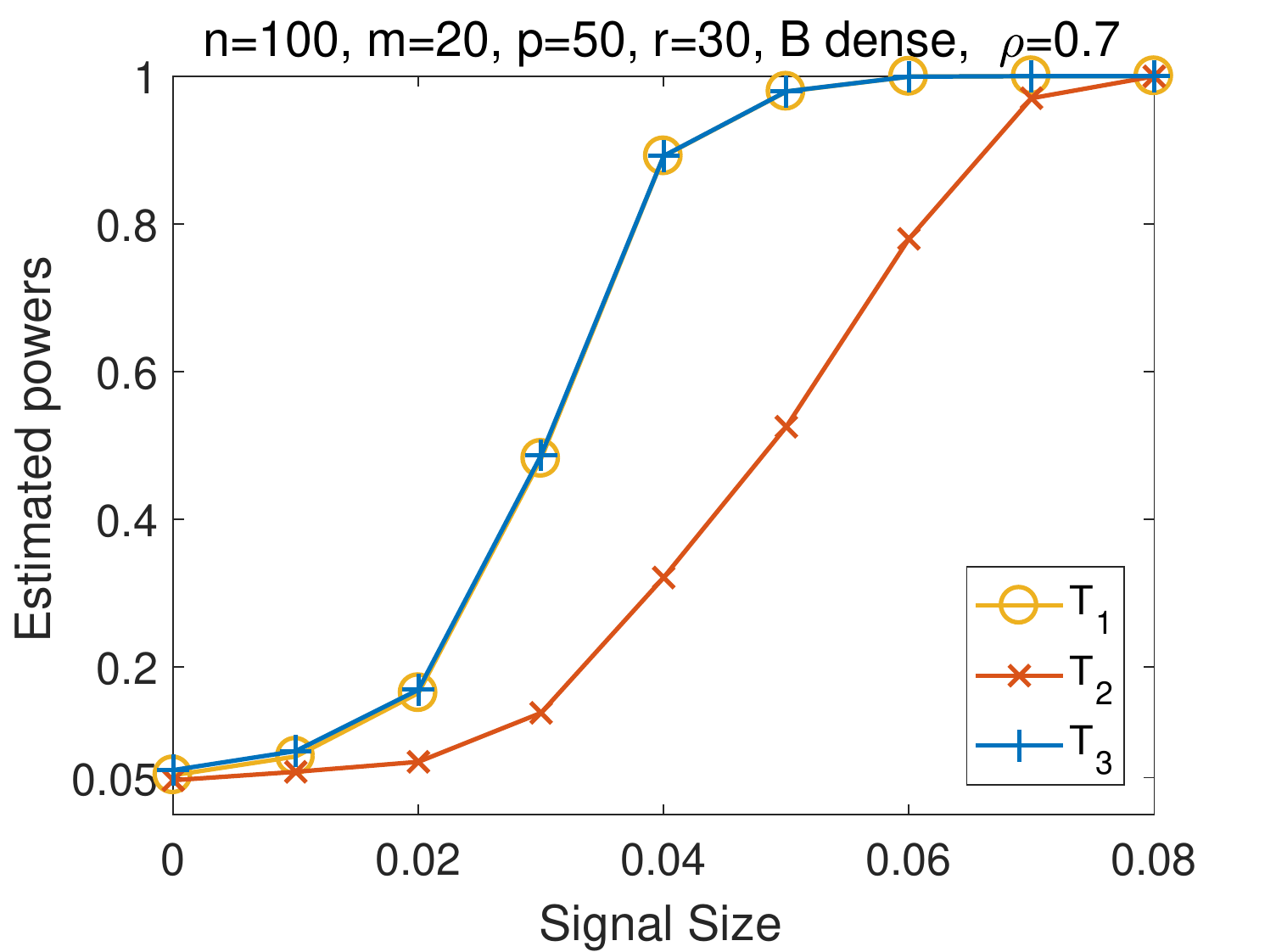}
\caption{Estimated powers versus signal sizes when $C=[I_r, \mathbf{0}_{r\times(p-r-1)}, -\mathbf{1}_r ]$}	
\label{fig:directpowercompareCiden}
\end{subfigure}
\caption{Estimated powers versus signal sizes with $\rho=0.7$}
\label{fig:poweresttwo}
\end{figure}

 \begin{figure}[htbp]
\centering
\begin{subfigure}{\textwidth}
\centering
	\includegraphics[width=0.45\textwidth]{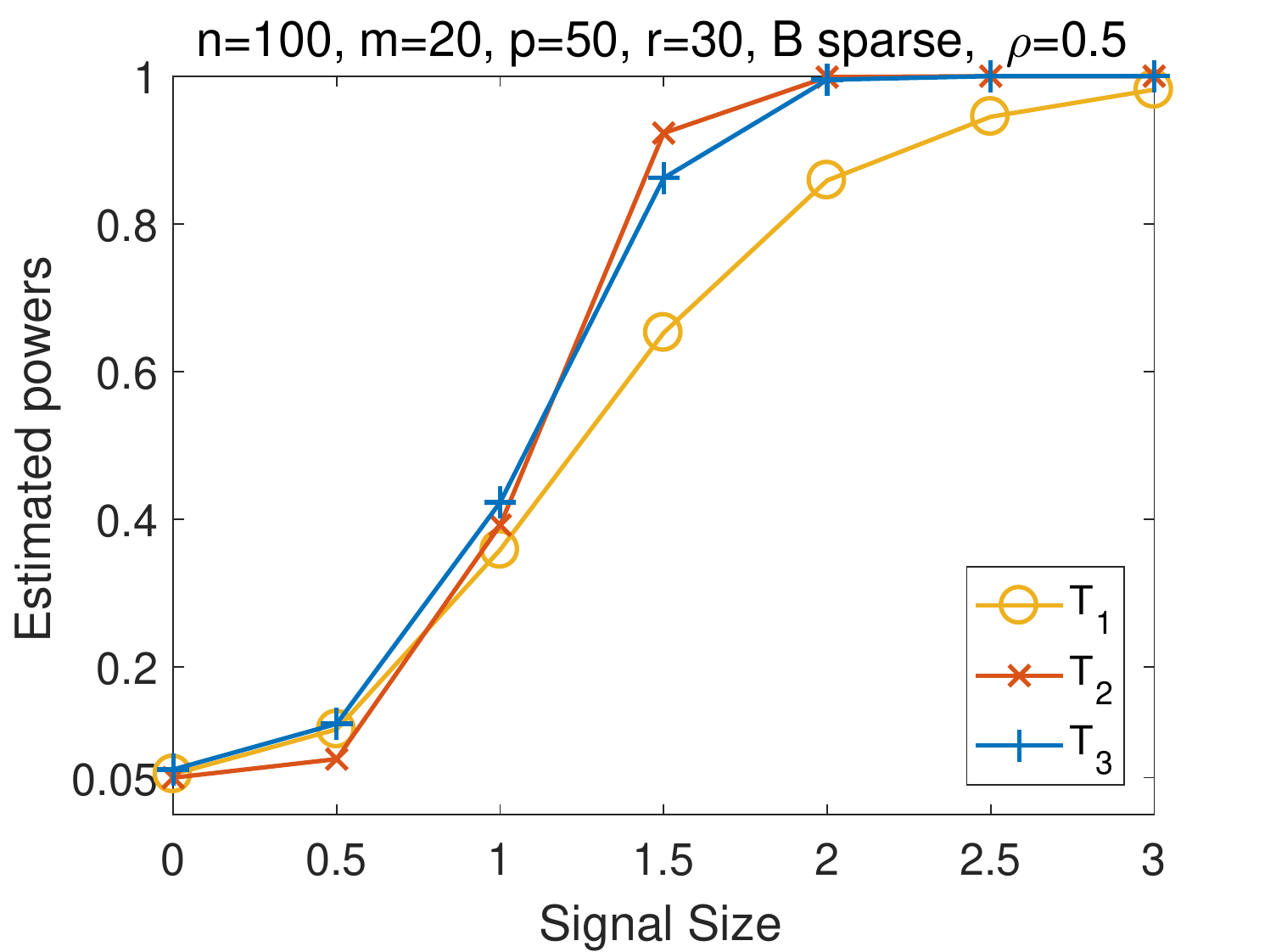}
	\includegraphics[width=0.45\textwidth]{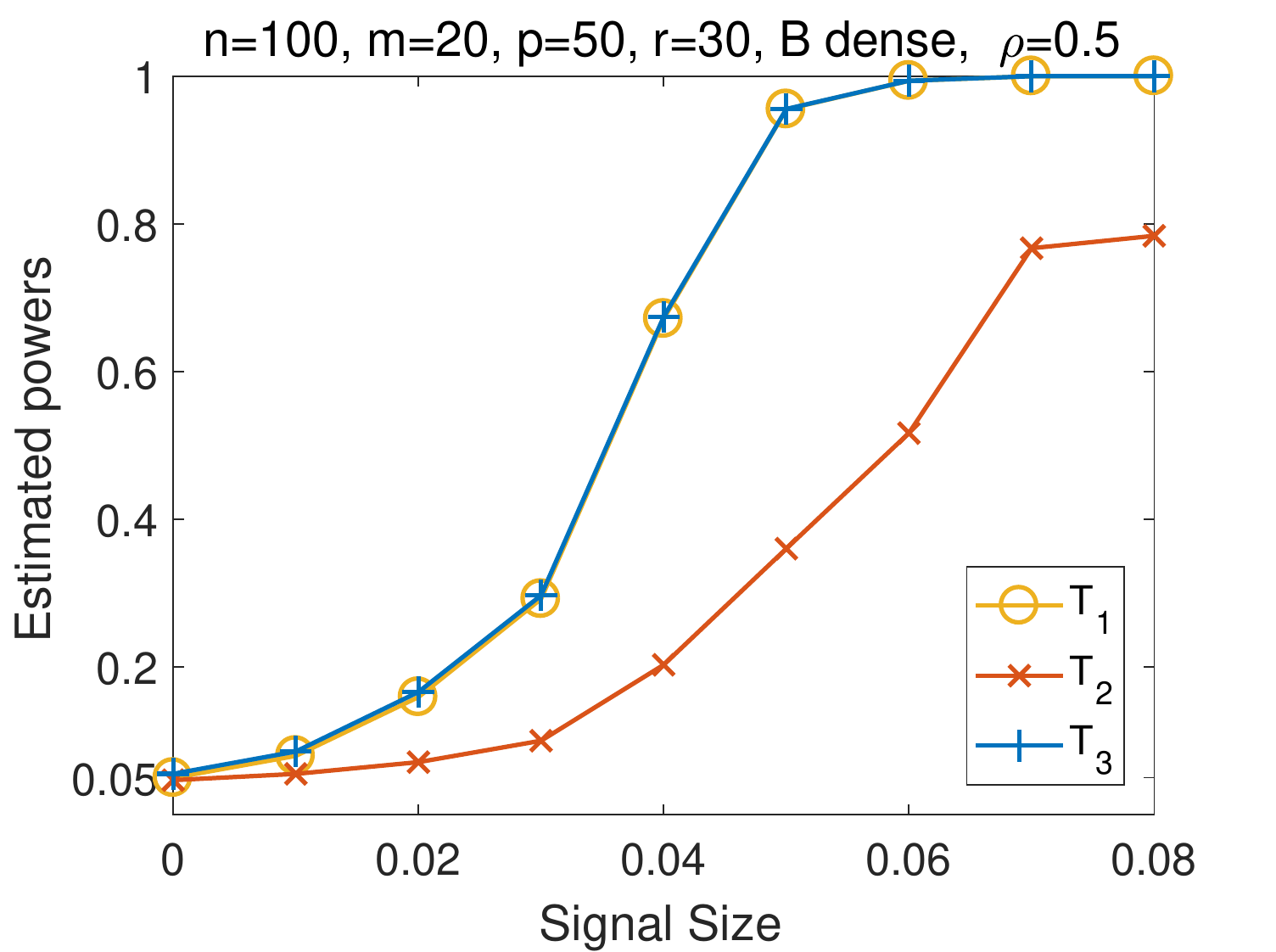}
\caption{Estimated powers versus signal sizes when $C=[I_{r},\mathbf{0}_{r\times (p-r)}]$}
\label{fig:powercidentrho}
\end{subfigure}
\begin{subfigure}{\textwidth}
\centering
\includegraphics[width=0.45\textwidth]{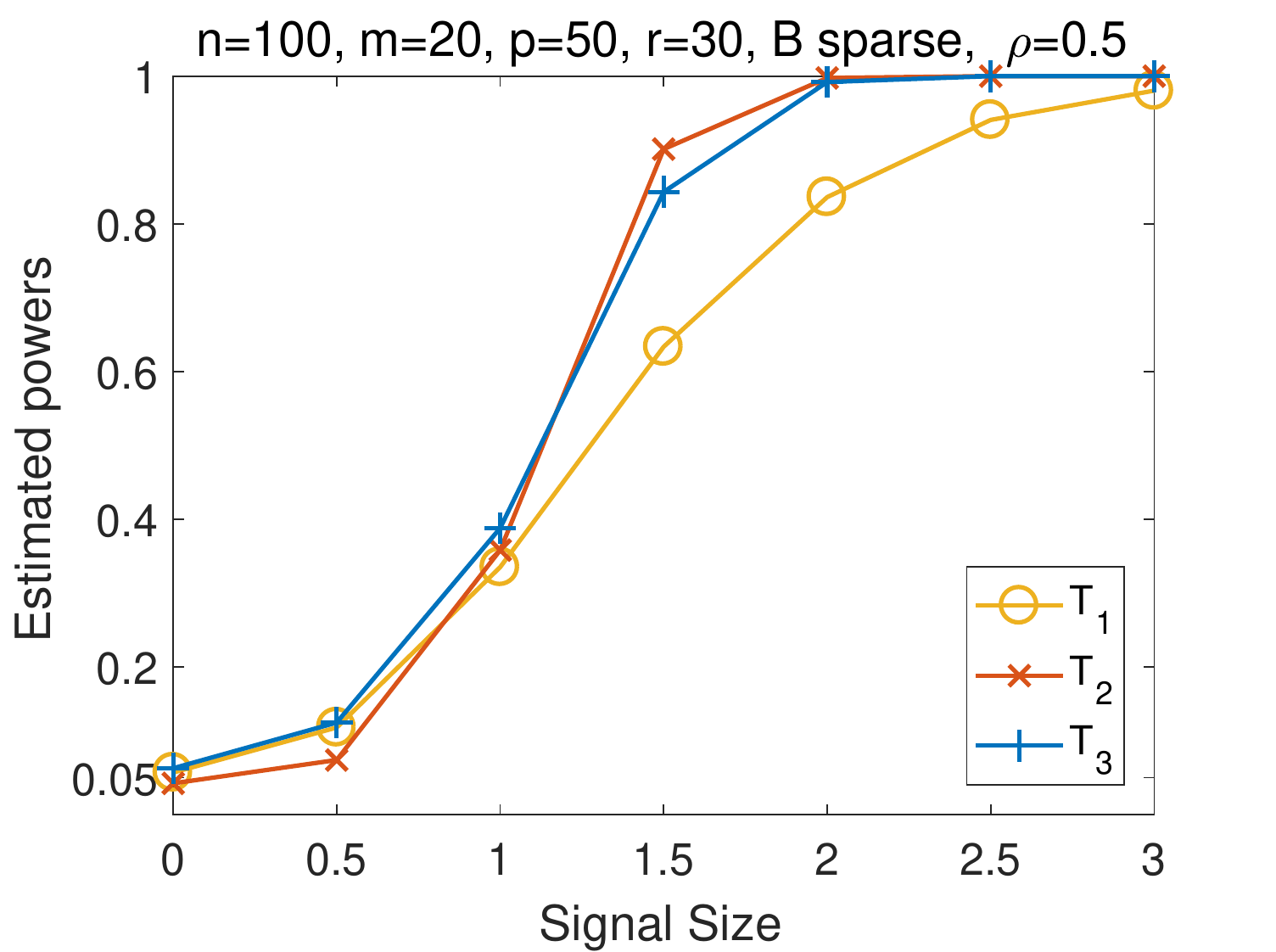}
\includegraphics[width=0.45\textwidth]{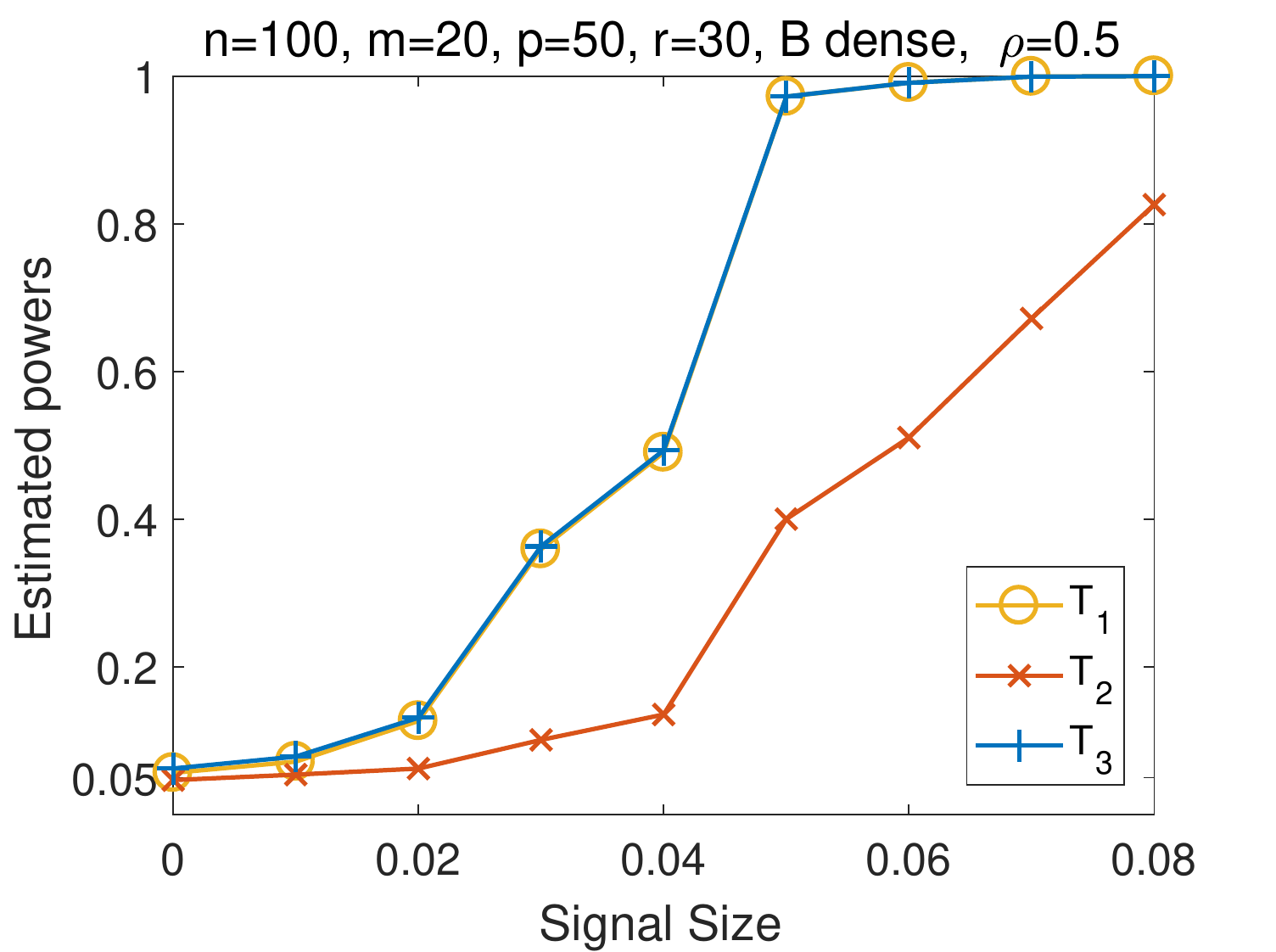}
\caption{Estimated powers versus signal sizes when $C=[I_r, \mathbf{0}_{r\times(p-r-1)}, -\mathbf{1}_r ]$}	
\label{fig:directpowercompareCidenrho}
\end{subfigure}
\caption{Estimated powers versus signal sizes with $\rho=0.5$}
\label{fig:poweresttwoprho05}
\end{figure}

\begin{figure}[htbp]
\centering
\begin{subfigure}{\textwidth}
\centering
	\includegraphics[width=0.45\textwidth]{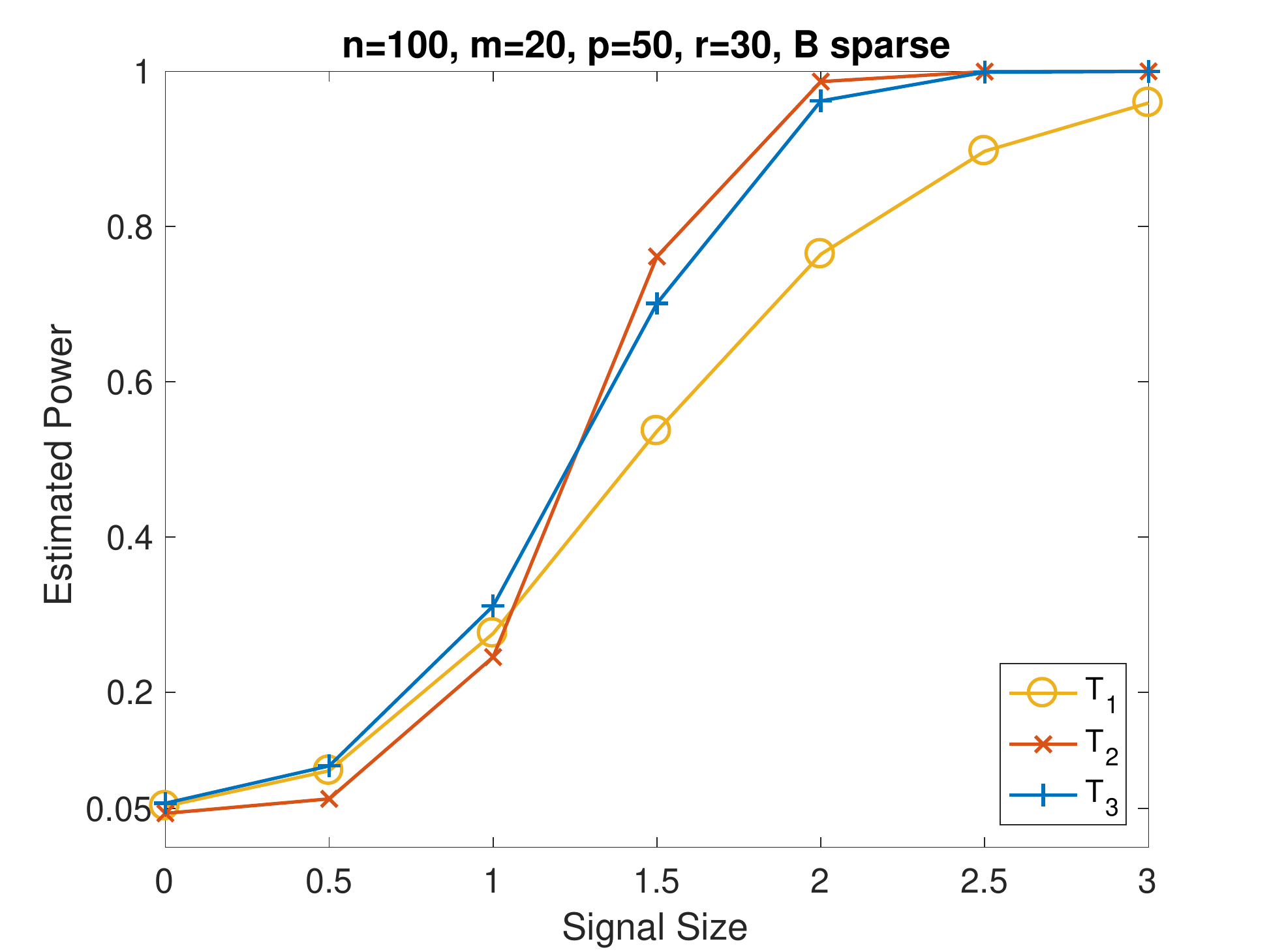}
	\includegraphics[width=0.45\textwidth]{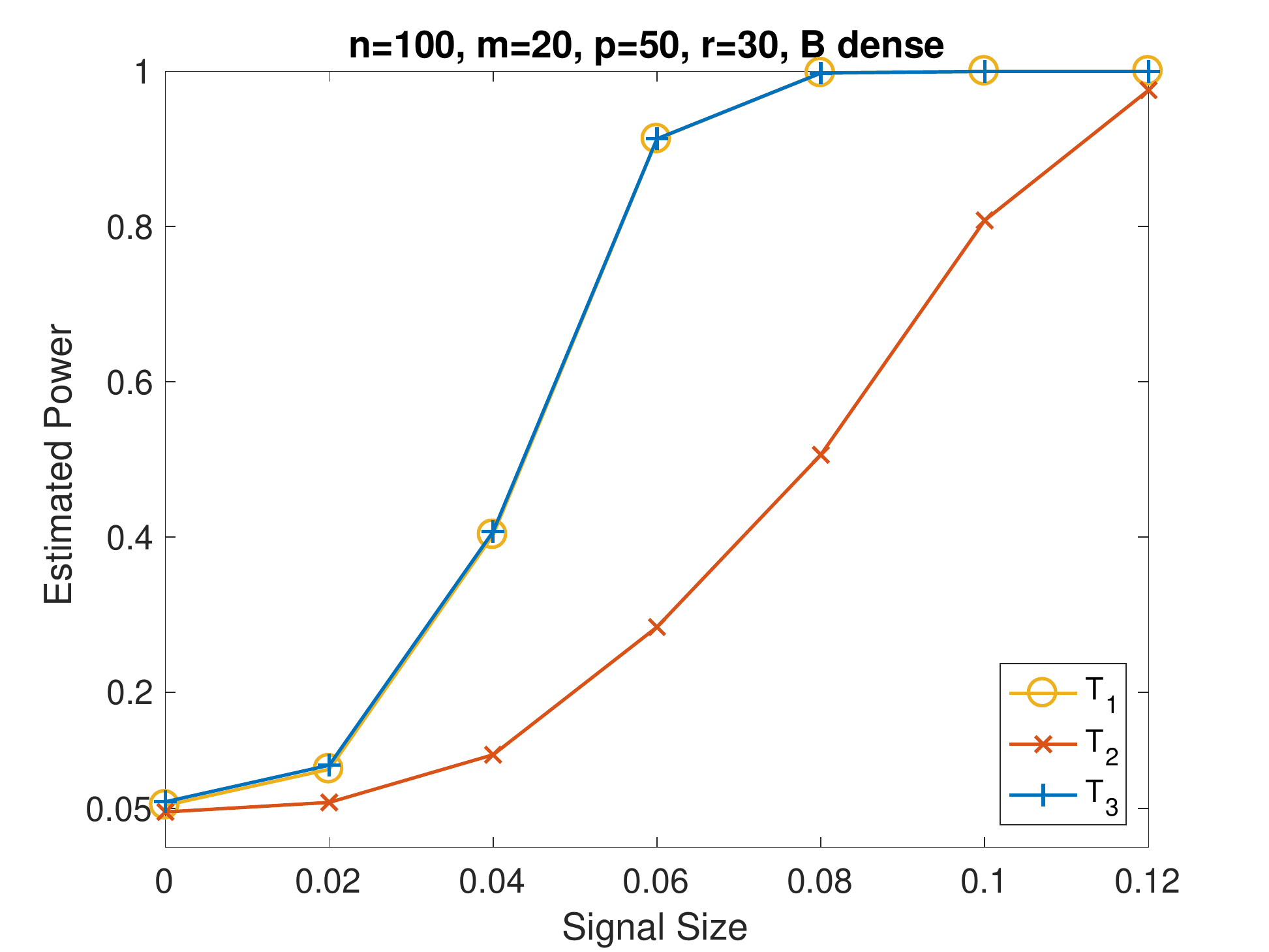}
\caption{Estimated powers versus signal sizes when $C=[I_{r},\mathbf{0}_{r\times (p-r)}]$}
\label{fig:powercident}
\end{subfigure}
\begin{subfigure}{\textwidth}
\centering
\includegraphics[width=0.45\textwidth]{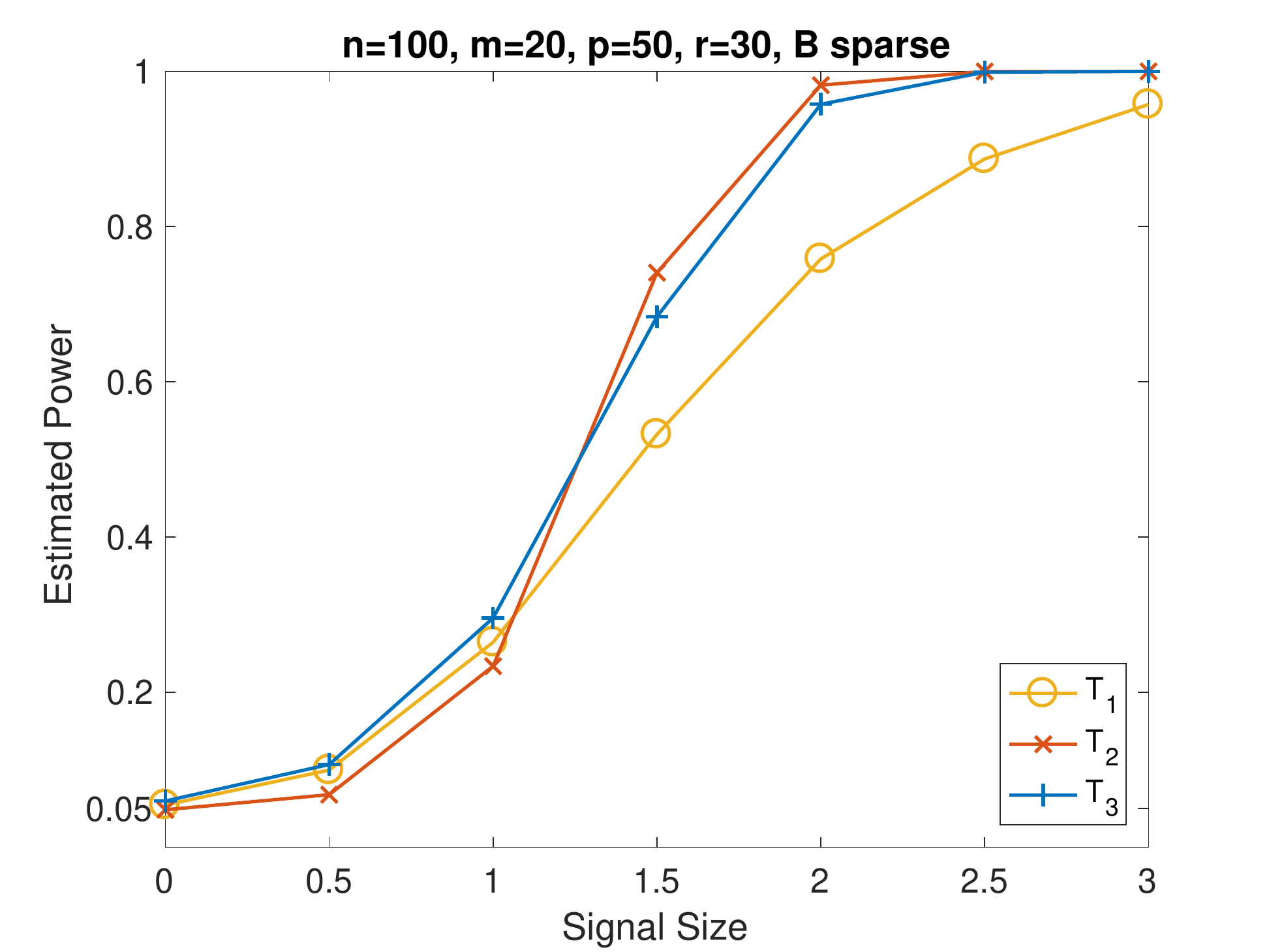}
\includegraphics[width=0.45\textwidth]{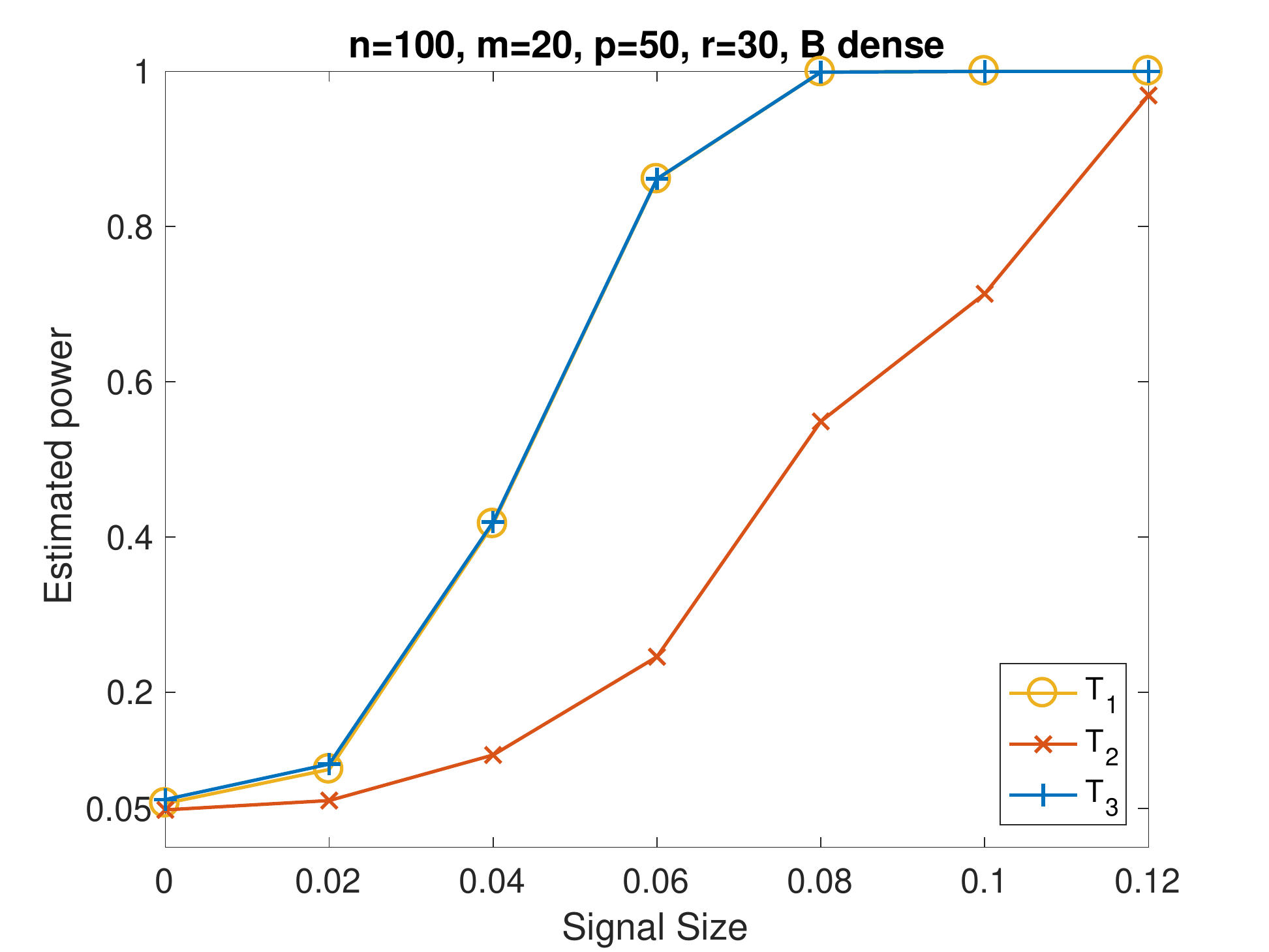}
\caption{Estimated powers versus signal sizes when $C=[I_r, \mathbf{0}_{r\times(p-r-1)}, -\mathbf{1}_r]$}	
\label{fig:directpowercompareCiden}
\end{subfigure}
\caption{Estimated powers versus signal sizes with $\rho=0$}
\label{fig:poweresttwosupp}
\end{figure}

 Figures \ref{fig:poweresttwo}--\ref{fig:poweresttwosupp} show that under the dense $B$ scenario, $T_1$ is more powerful than $T_2$; but under the sparse  $B$ scenario, $T_2$ is more powerful than   $T_1$. In addition, the combined statistic $T_3$ still maintains high power under both scenarios.  These results demonstrate the good performance of the  proposed statistic $T_3$. Note that  the patterns we observe in Figures \ref{fig:poweresttwo}--\ref{fig:poweresttwosupp}  are similar to that in Figure \ref{fig:powerthree}, which indicates that the conclusion we obtain under the canonical form can be  instructive when considering the linear form.



\subsubsection{Robustness with other distributions} \label{sec:simulotherdist}
We further conduct some simulations considering other distributions, which exhibit similar patterns as in Figure \ref{fig:poweresttwo} and imply the robustness of the proposed methods. 


\paragraph{(a) $X$ and $Y$ follow  multinomial distributions}

For $i=1,\ldots,n$ and $j=1,\ldots,p$, we generate the entry $x_{i,j}$ in $X$ independently and identically in the following way. In particular, we first generate $z_{i,j} \overset{i.i.d.}{\sim} \mathcal{N}(0,1)$, and set the value of $x_{i,j}$ as below:
\begin{eqnarray*}
x_{i,j}=\left\{\begin{array}{rl}
~ -3  \quad &  z_{i,j}<-1, \\ 
~ -2  \quad & z_{i,j}\in [-1,-0.4), \\ 
~ -1  \quad & z_{i,j}\in [-0.4,0),  \\ 
~ 1  \quad &z_{i,j}\in [0,0.4), \\ 
~ 2  \quad &z_{i,j}\in [0.4,1), \\ 
~ 3  \quad & z_{i,j}>1.
\end{array}\right.
\end{eqnarray*} 
 Given $B$ and $X$, we generate $W=XB+E$, where the entries of $E$ are i.i.d. $\mathcal{N}(0,1)$. For $i=1,\ldots,n$ and $j=1,\ldots,p$, let $w_{i,j}$ and $y_{i,j}$ denote the entries of $W$ and $Y$ respectively. We then set
\begin{eqnarray*}
y_{i,j}=\left\{\begin{array}{rl}
~ -3  \quad &  w_{i,j}<-1, \\ 
~ -2  \quad & w_{i,j}\in [-1,-0.4), \\ 
~ -1  \quad & w_{i,j}\in [-0.4,0),  \\ 
~ 1  \quad &w_{i,j}\in [0,0.4), \\ 
~ 2  \quad &w_{i,j}\in [0.4,1), \\ 
~ 3  \quad & w_{i,j}>1.
\end{array}\right.
\end{eqnarray*}  


We present the results in Figure \ref{fig:xymultnomial}, where  ``$B$ sparse" and ``$B$ dense" represent two different types of $B$ matrix, which are generated following the same method as in Section \ref{sec:additionaltersimu}. Similarly, we also take $C=[I_{r},\mathbf{0}_{r\times (p-r)}]$ and $C=[I_r, \mathbf{0}_{r\times(p-r-1)}, -\mathbf{1}_r ]$ respectively. We can observe similar patterns to that in Figure \ref{fig:poweresttwo}.

\begin{figure}[!htbp]
\centering
\begin{subfigure}{\textwidth}
\centering
\includegraphics[width=0.4\textwidth]{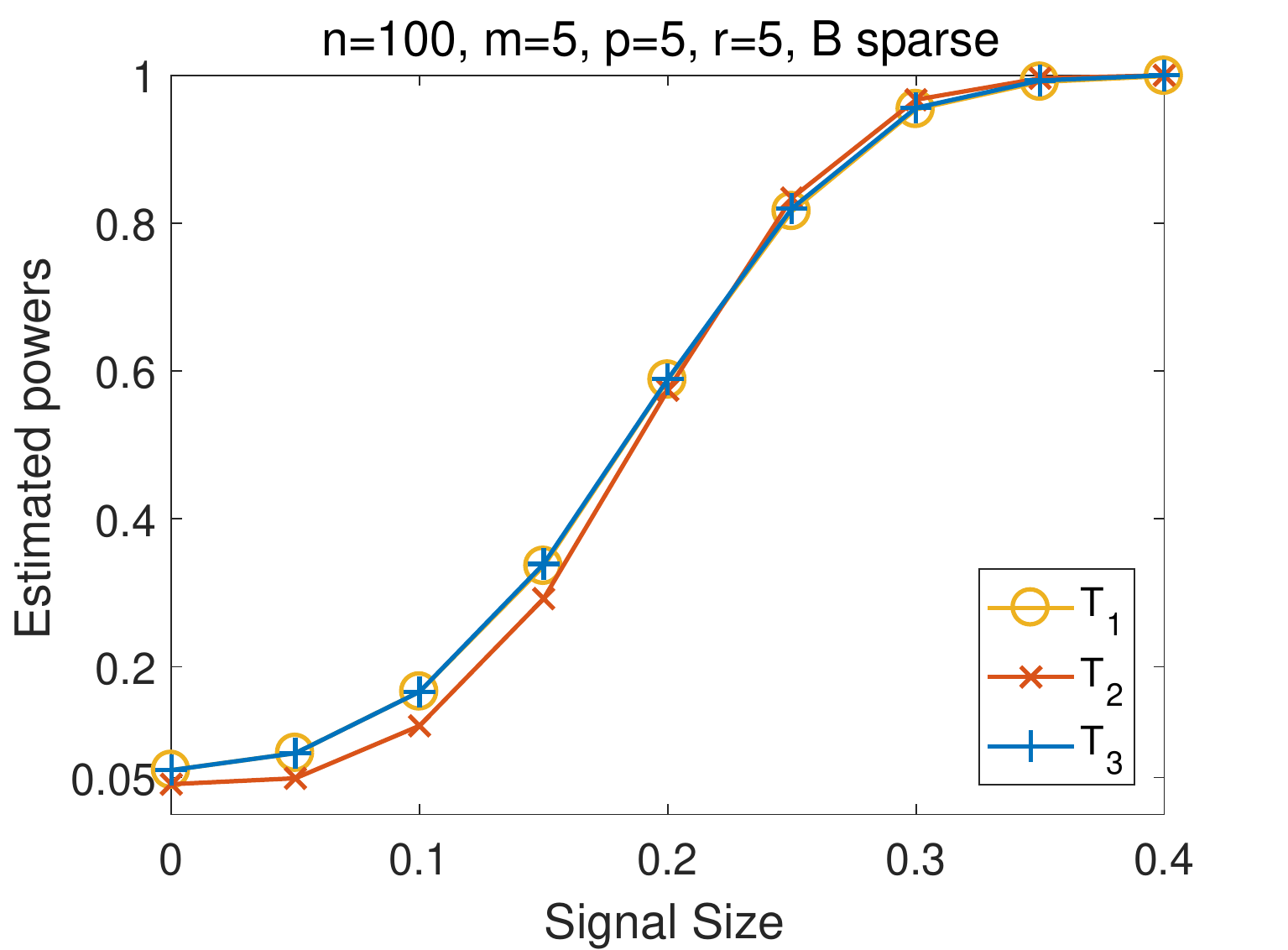}
\includegraphics[width=0.4\textwidth]{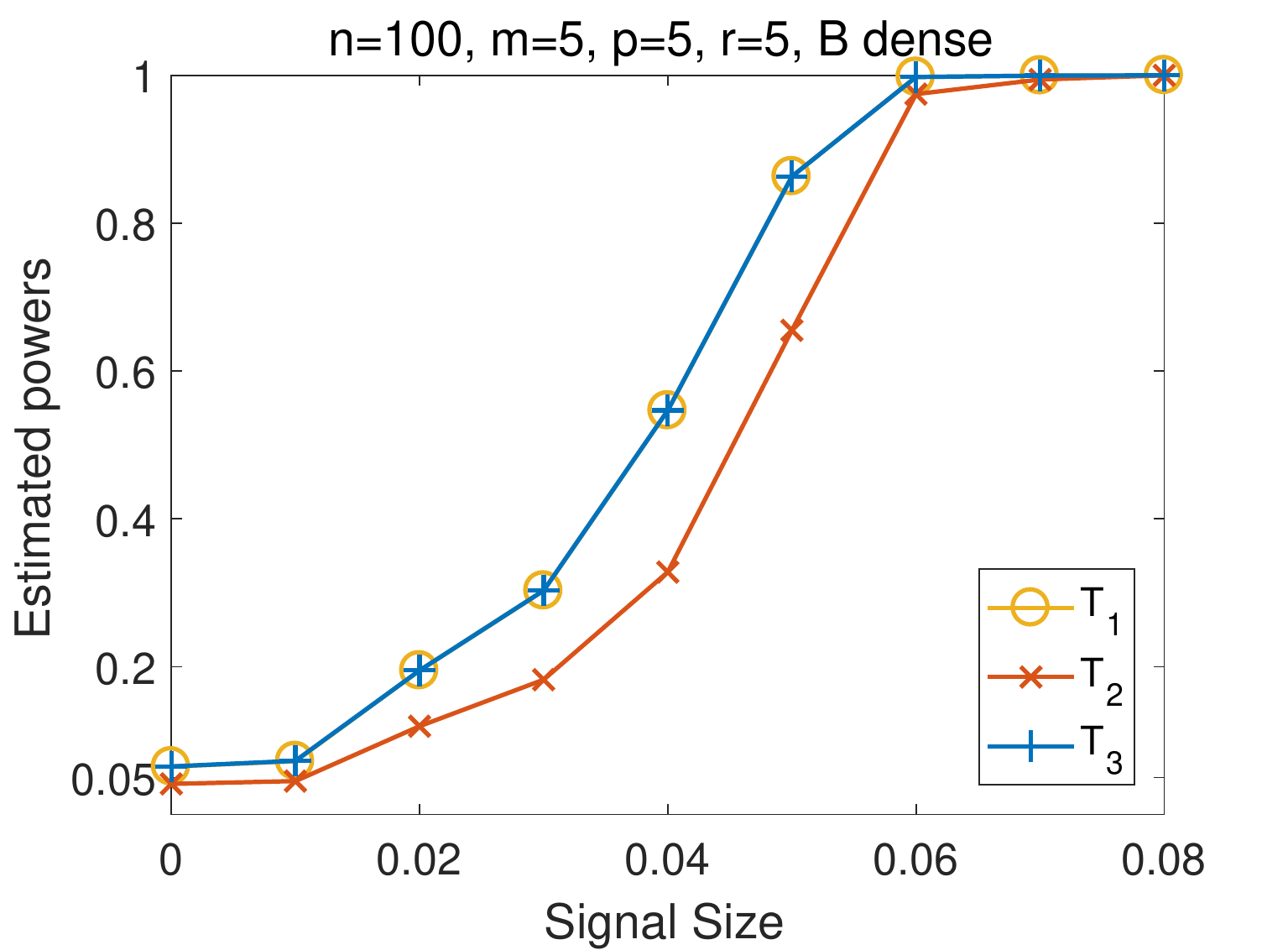}
\\
\includegraphics[width=0.4\textwidth]{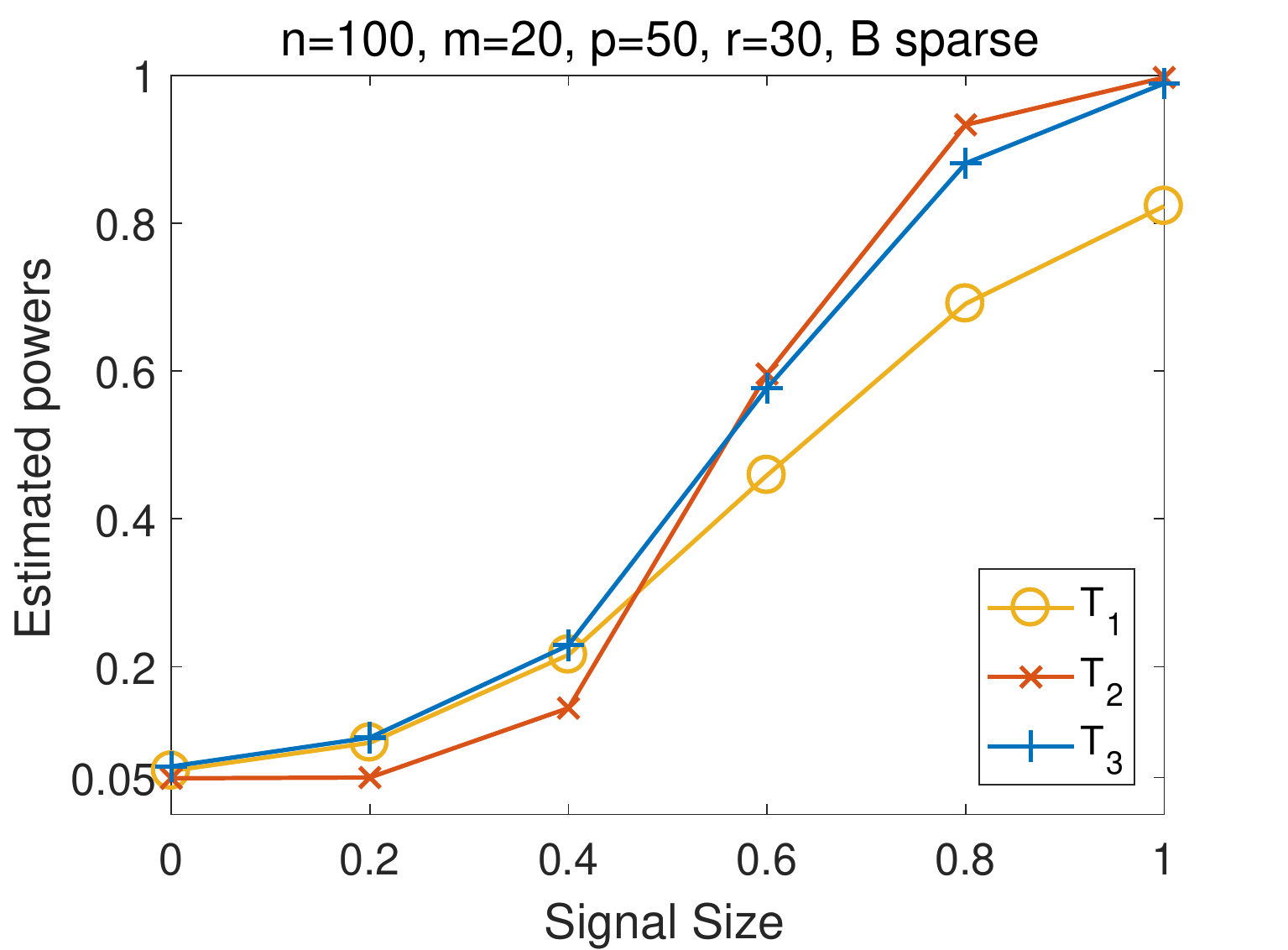}
\includegraphics[width=0.4\textwidth]{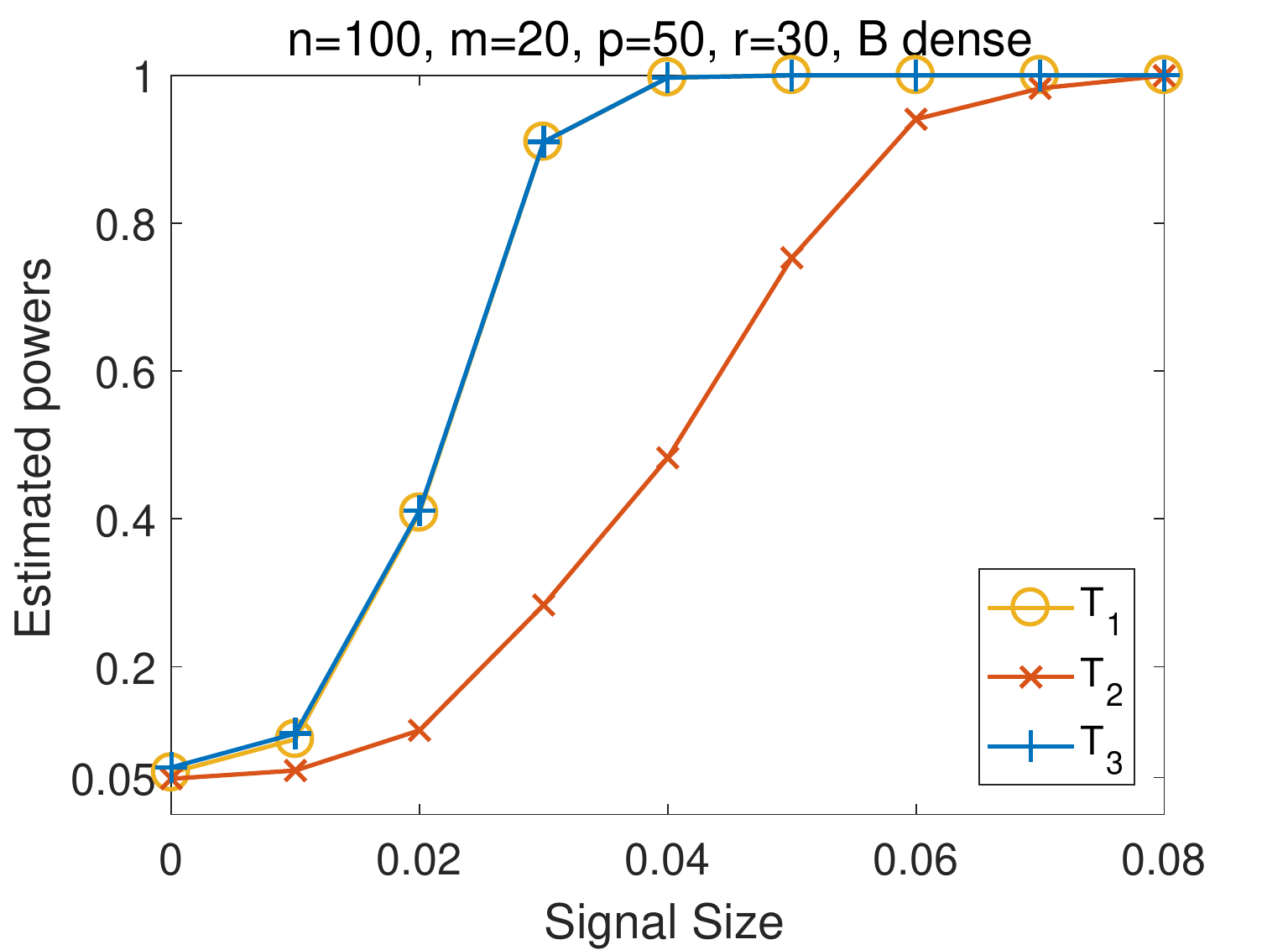}
\caption{When $C=[I_{r},\mathbf{0}_{r\times (p-r)}]$}
\end{subfigure}
\begin{subfigure}{\textwidth}
\centering
\includegraphics[width=0.4\textwidth]{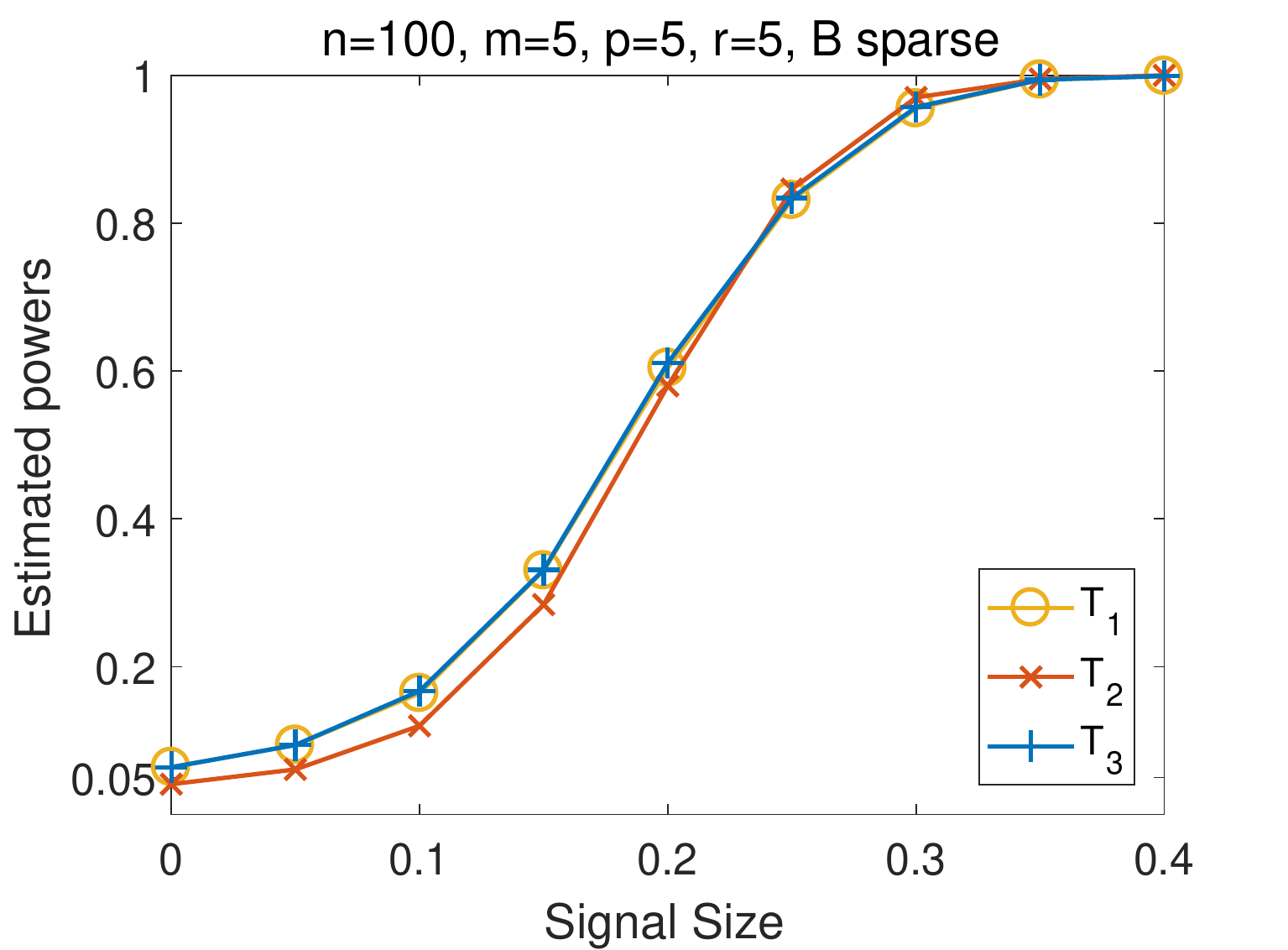}
\includegraphics[width=0.4\textwidth]{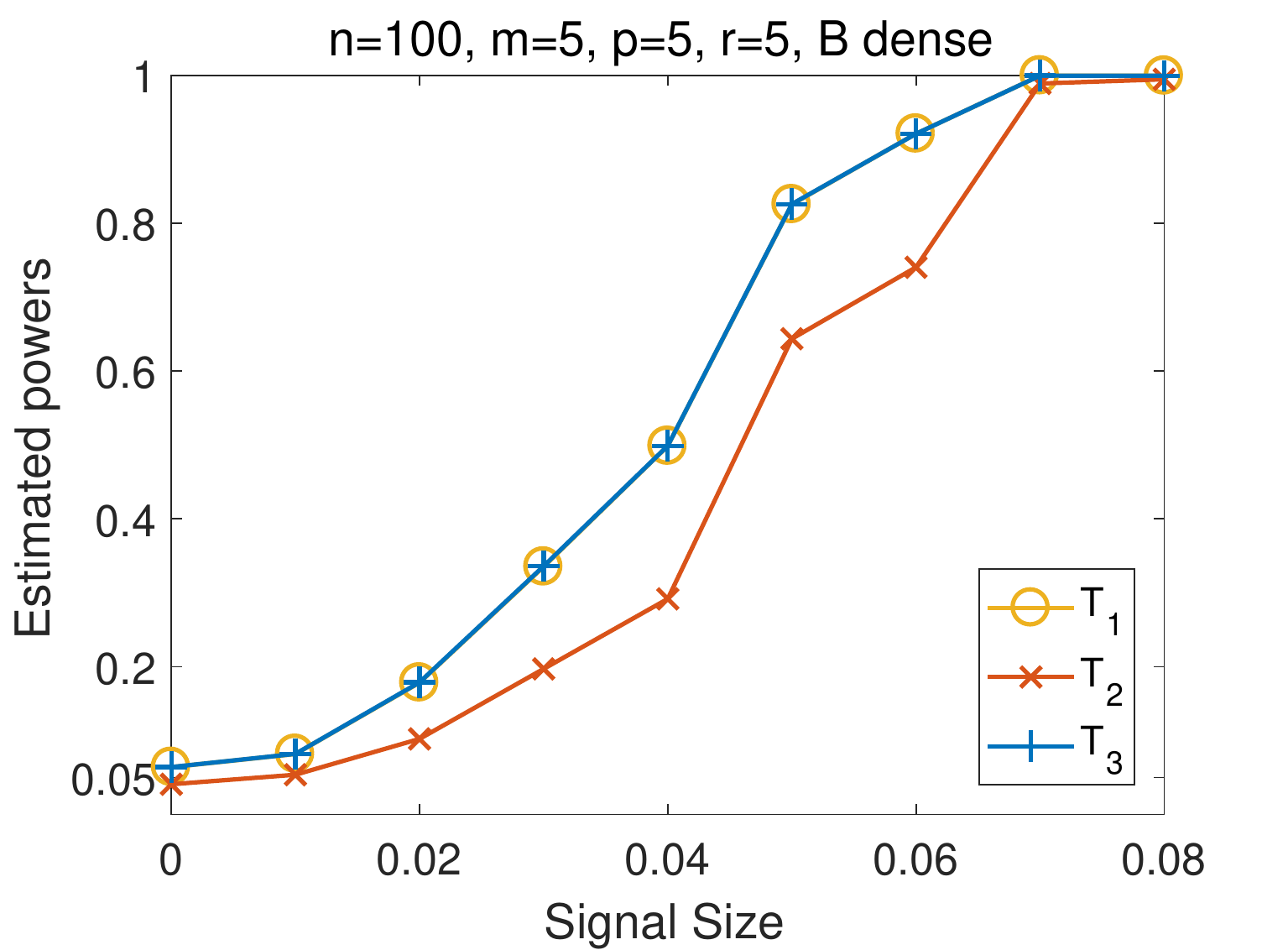}
\\
\includegraphics[width=0.4\textwidth]{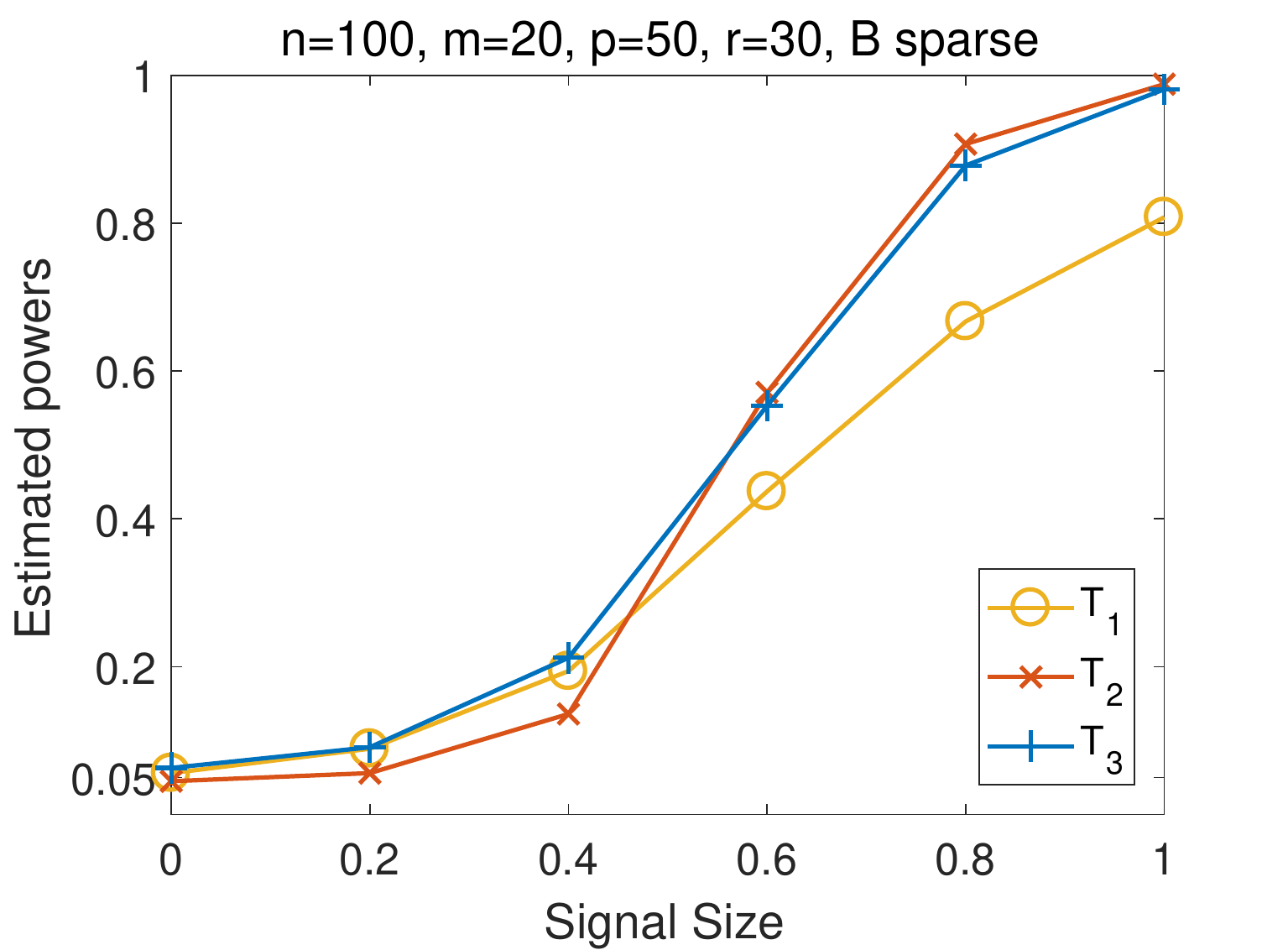}
\includegraphics[width=0.4\textwidth]{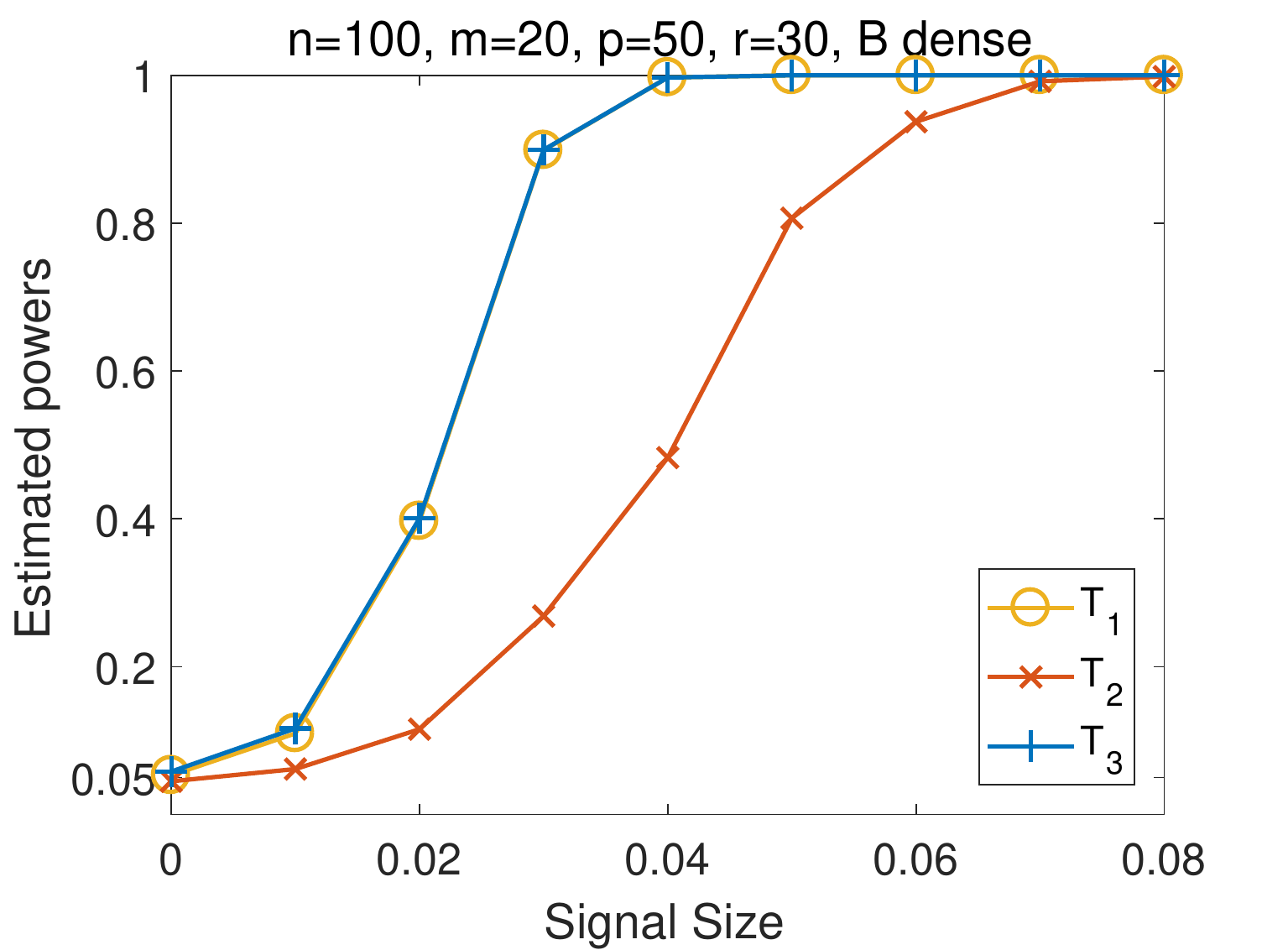}
\caption{When $C=[I_r, \mathbf{0}_{r\times(p-r-1)}, -\mathbf{1}_r ]$}
\end{subfigure}
\caption{Power comparison when $X$ and $Y$ follow multinomial distribution}
\label{fig:xymultnomial}

\end{figure}


\paragraph{(b) Errors follow  $t$ distribution}
In this part, we examine the case when the errors in matrix $E$ independently and identically follows $t$ distribution. In particular, we first generate the entries in $X$ as i.i.d. $\mathcal{N}(0,1)$. Then we generate the entries in $E$ as  i.i.d. $t_{df}$ with $df \in \{3,5\}.$     The results are summarized in Figure \ref{fig:tdisterrorperf}, where similar patterns are observed as in Figure \ref{fig:poweresttwo}.  


\begin{figure}[!htbp]
\centering
\begin{subfigure}{\textwidth}
\centering
\includegraphics[width=0.4\textwidth]{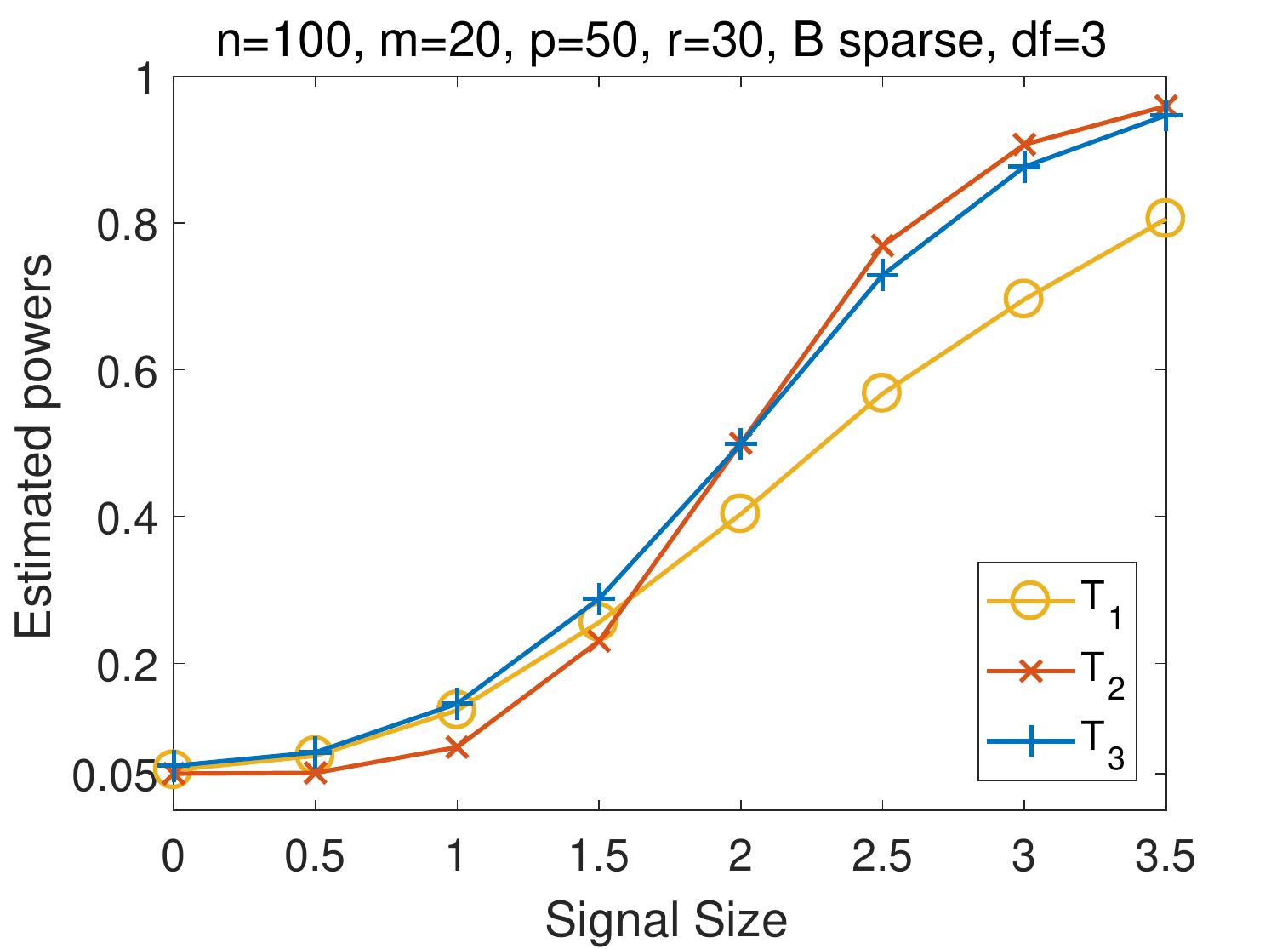}
\includegraphics[width=0.4\textwidth]{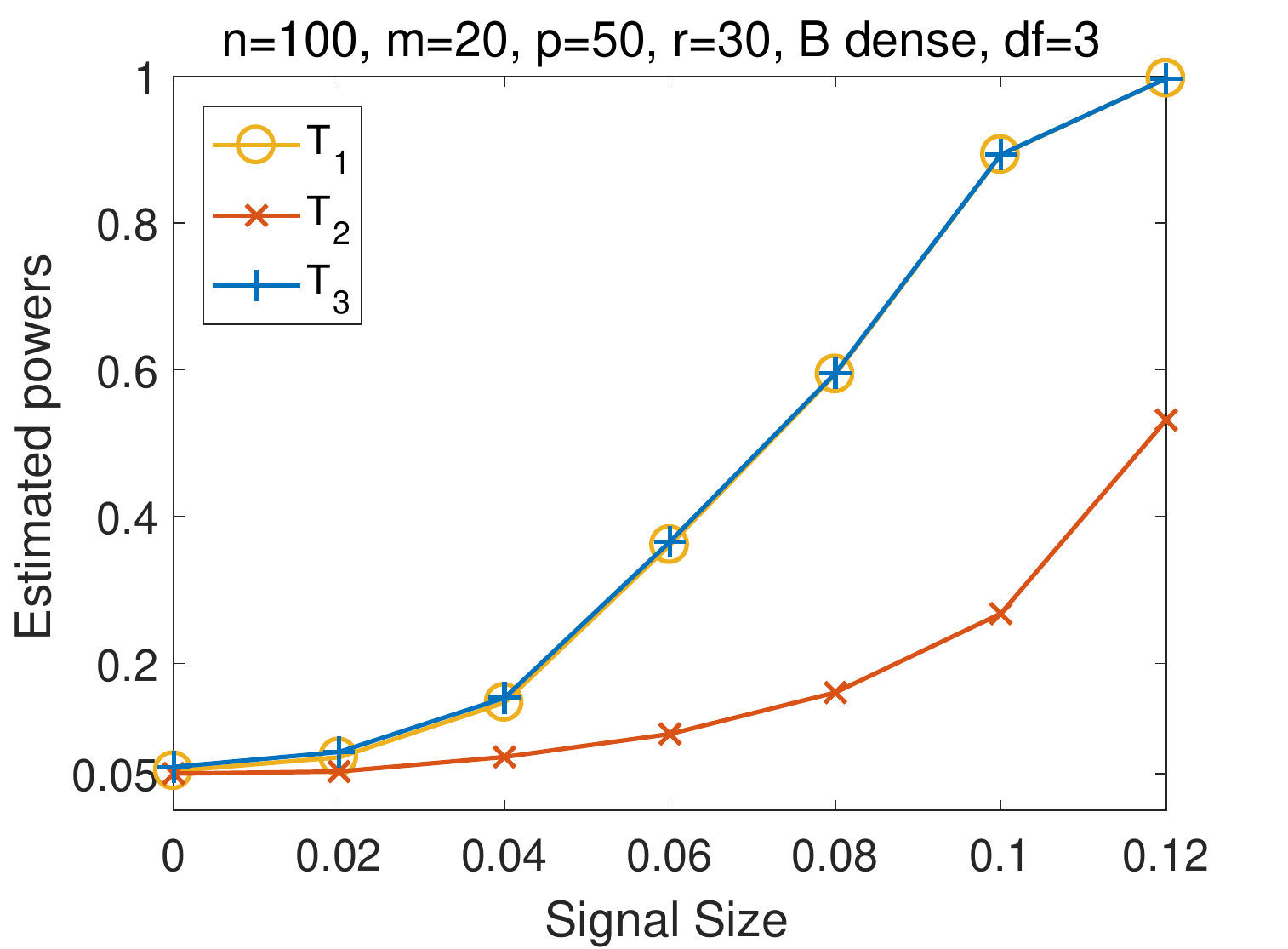}
\\
\includegraphics[width=0.4\textwidth]{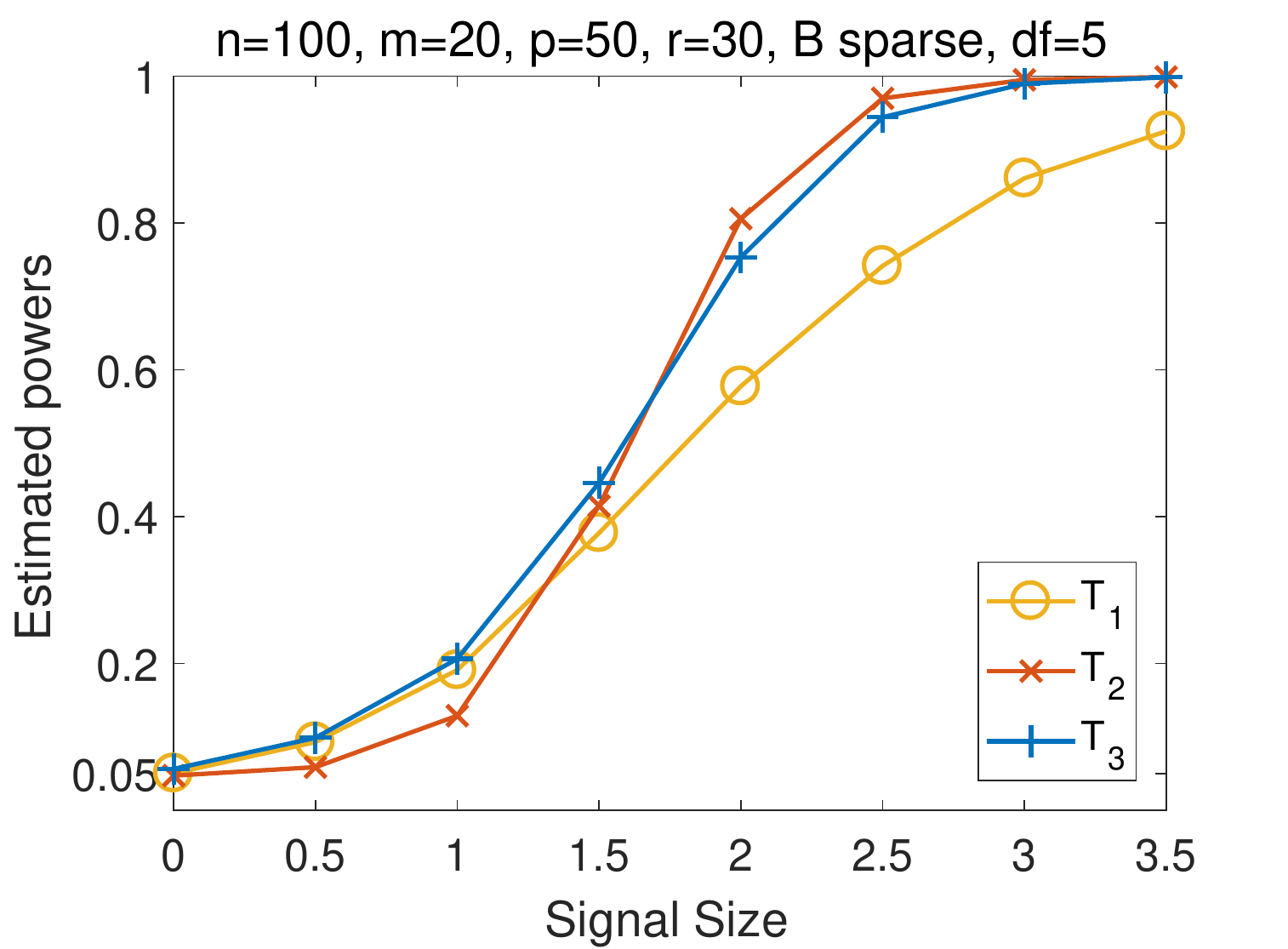}
\includegraphics[width=0.4\textwidth]{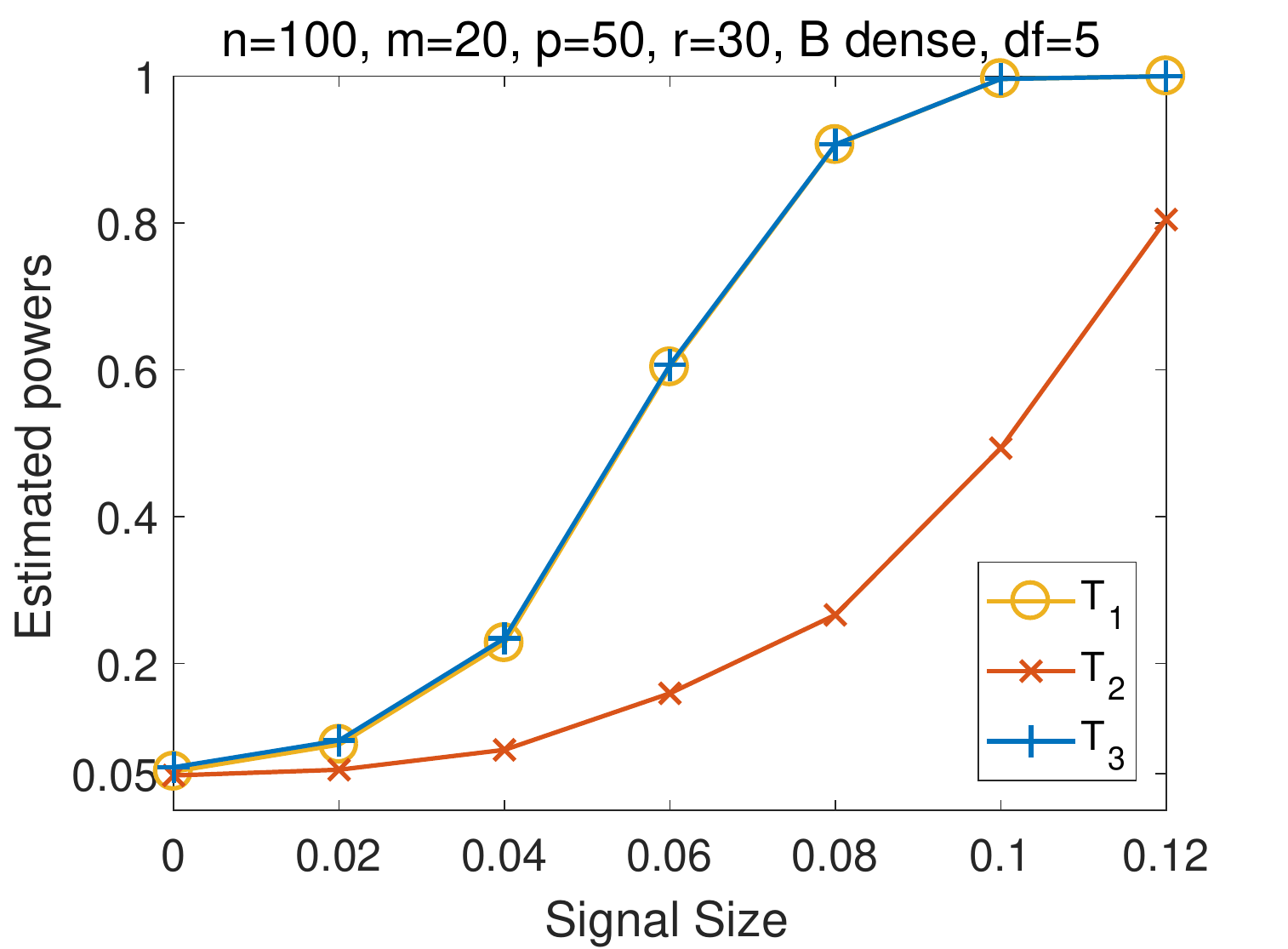}
\caption{When $C=[I_{r},\mathbf{0}_{r\times (p-r)}]$}
\end{subfigure}
\begin{subfigure}{\textwidth}
\centering
\includegraphics[width=0.4\textwidth]{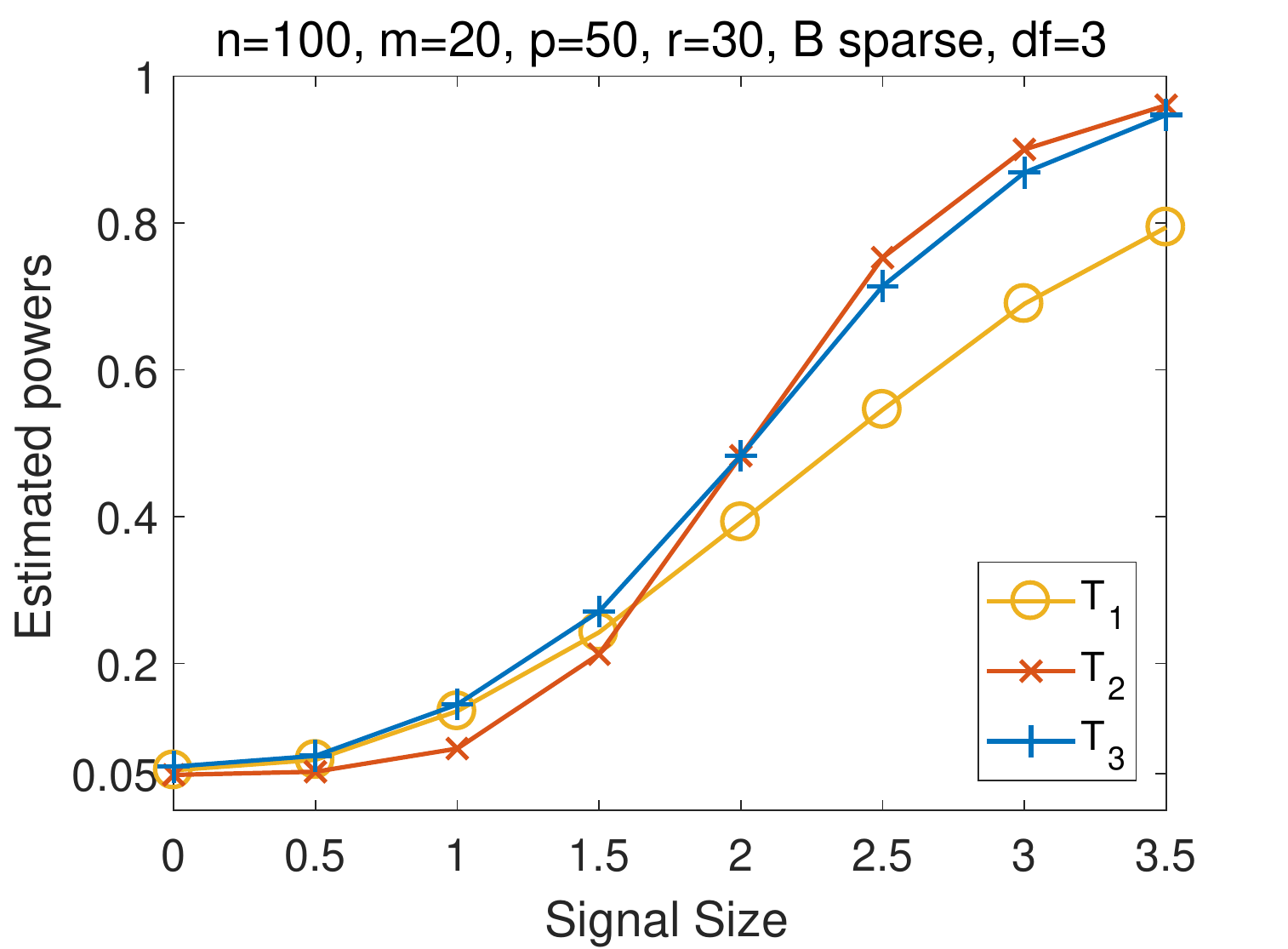}
\includegraphics[width=0.4\textwidth]{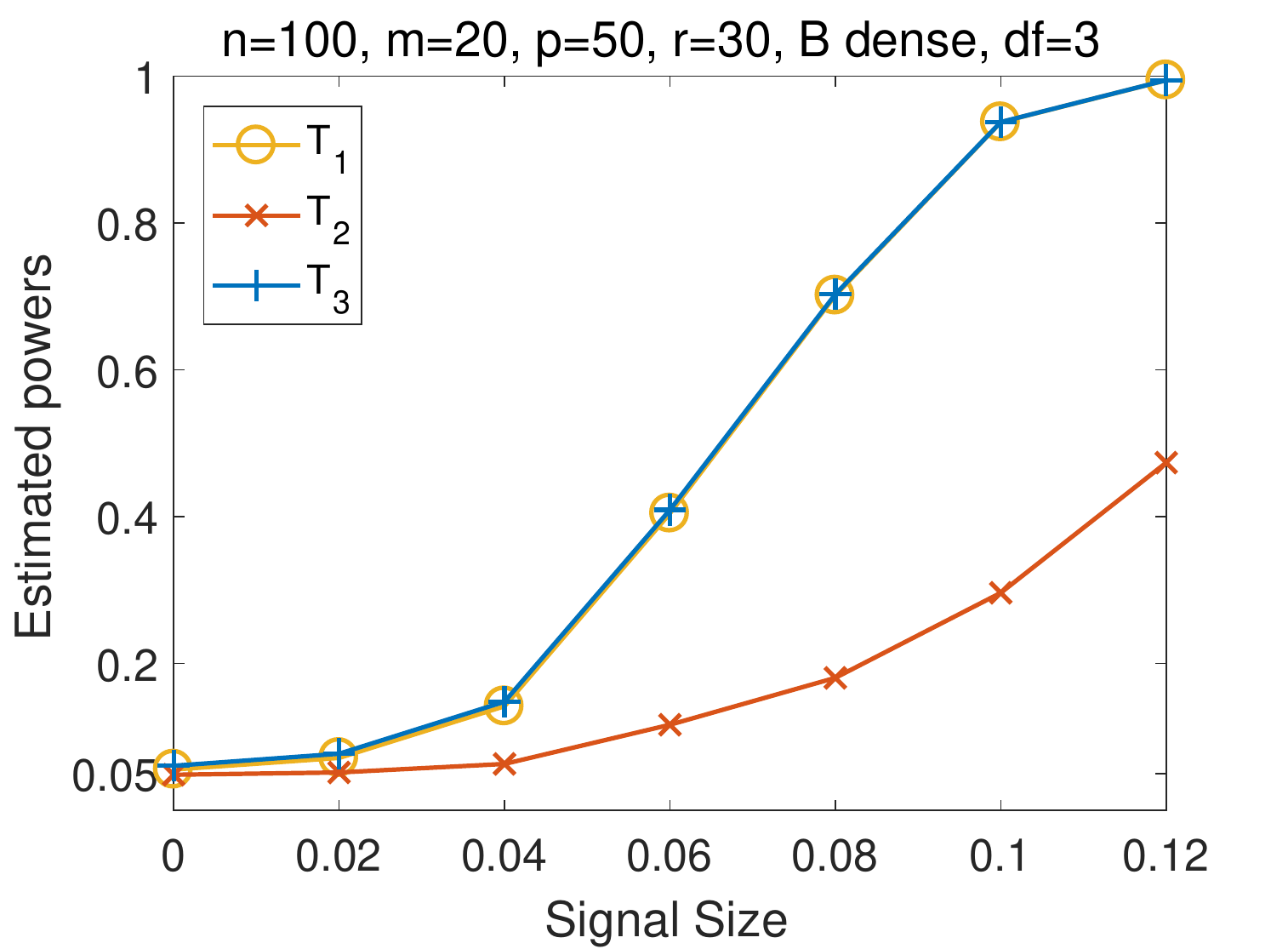}
\\
\includegraphics[width=0.4\textwidth]{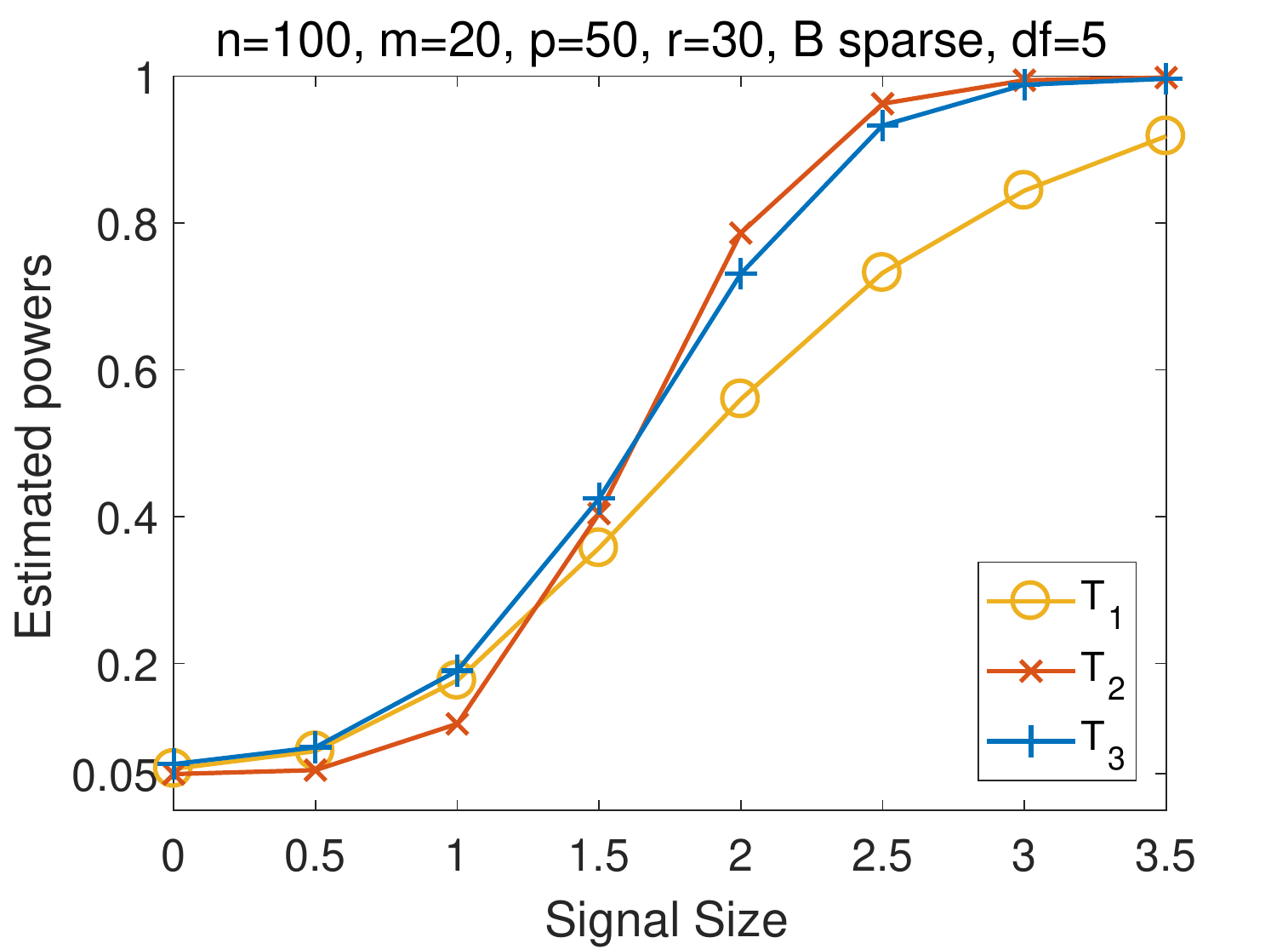}
\includegraphics[width=0.4\textwidth]{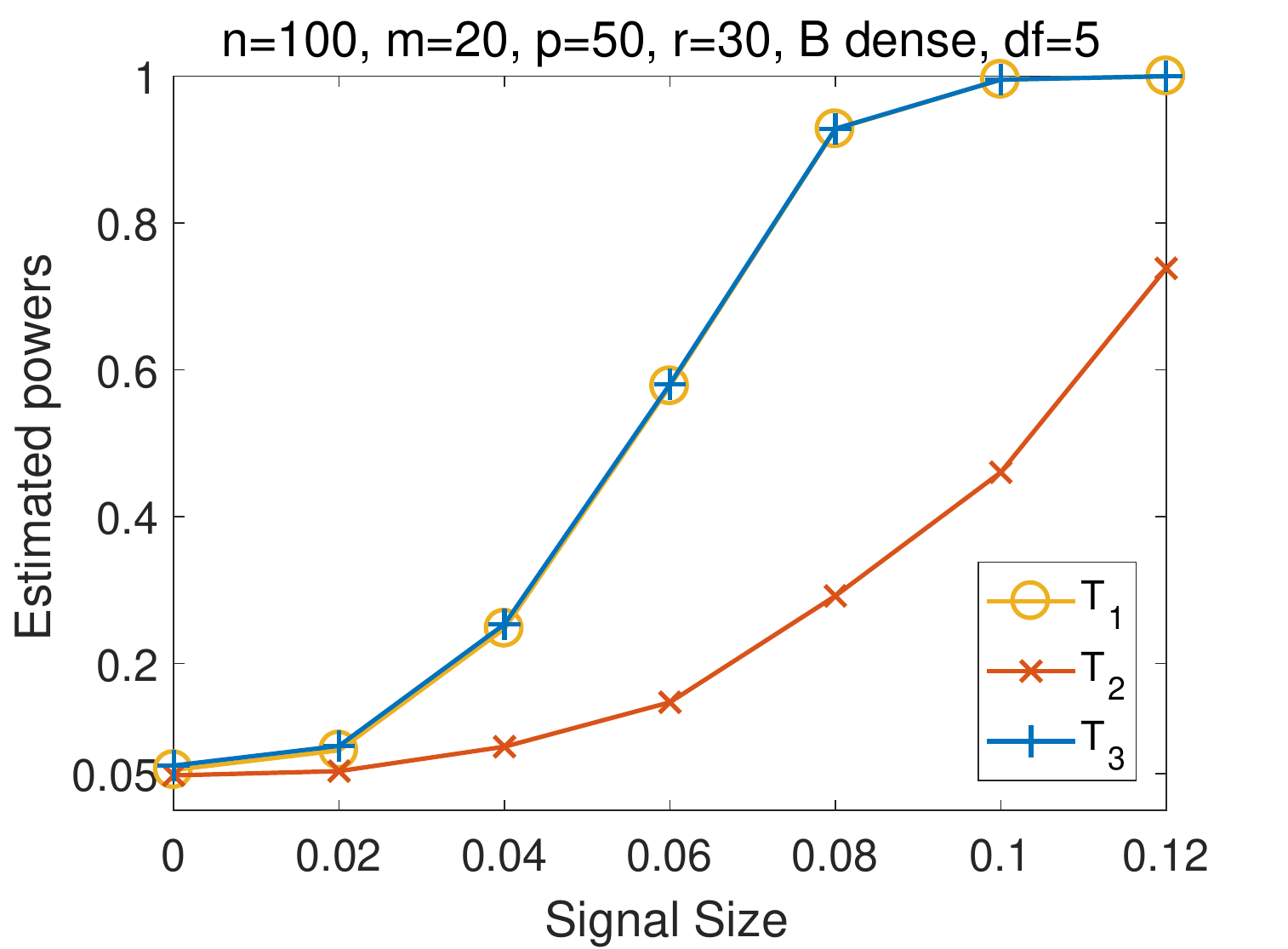}
\caption{When $C=[I_r, \mathbf{0}_{r\times(p-r-1)}, -\mathbf{1}_r ]$}
\end{subfigure}
\caption{Power comparison when entries in $E$ follow  $t$ distribution}
\label{fig:tdisterrorperf}	
\end{figure}



\subsection{Supplementary simulations when $n<p+m$} \label{sec:simulationnsmallpm}
\subsubsection{Supplementary simulations with normal distribution}\label{sec:simulationnsmallpmnormal} 

Under the similar set-up to that of Figure \ref{fig:powerestpowerscreen}, we present additional results with $r_k=5$  in  Figure \ref{fig:powerestpowerscreenrk5}, where similar patterns are observed as in  Figure \ref{fig:powerestpowerscreen}.  
\begin{figure}[htbp]
\centering
\begin{subfigure}{\textwidth}
	\includegraphics[width=0.48\textwidth]{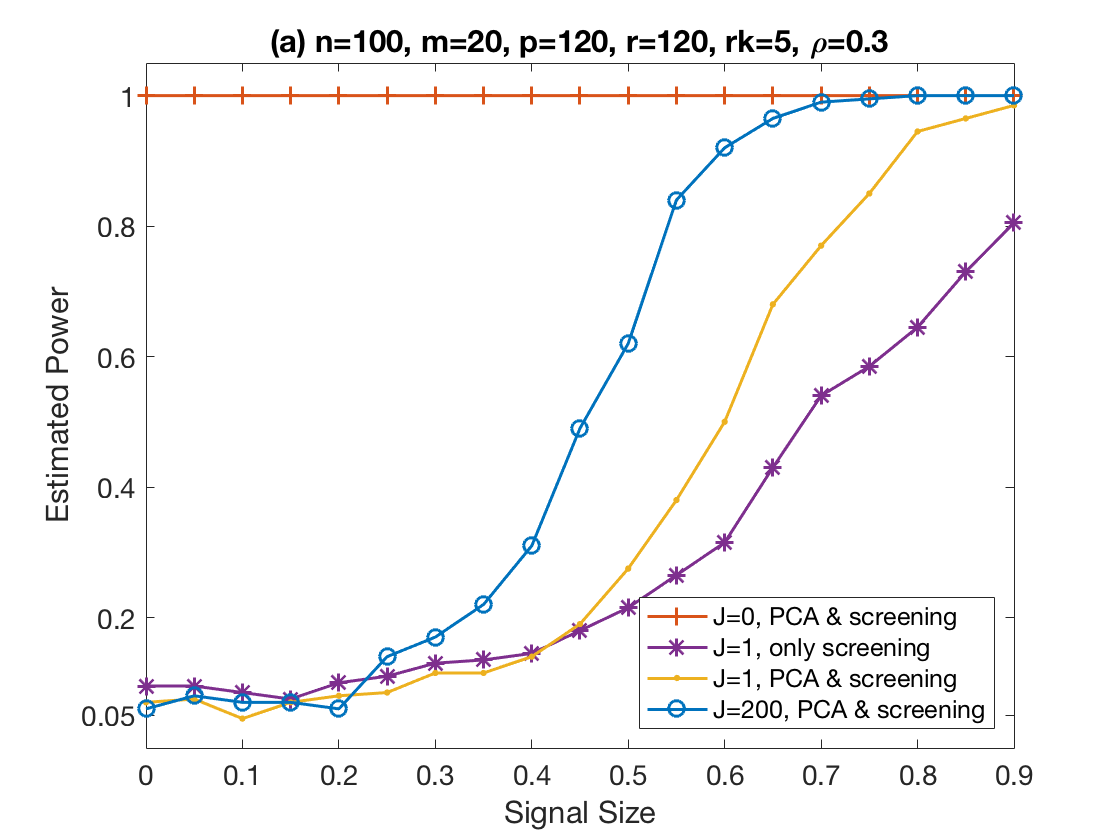}
	\includegraphics[width=0.48\textwidth]{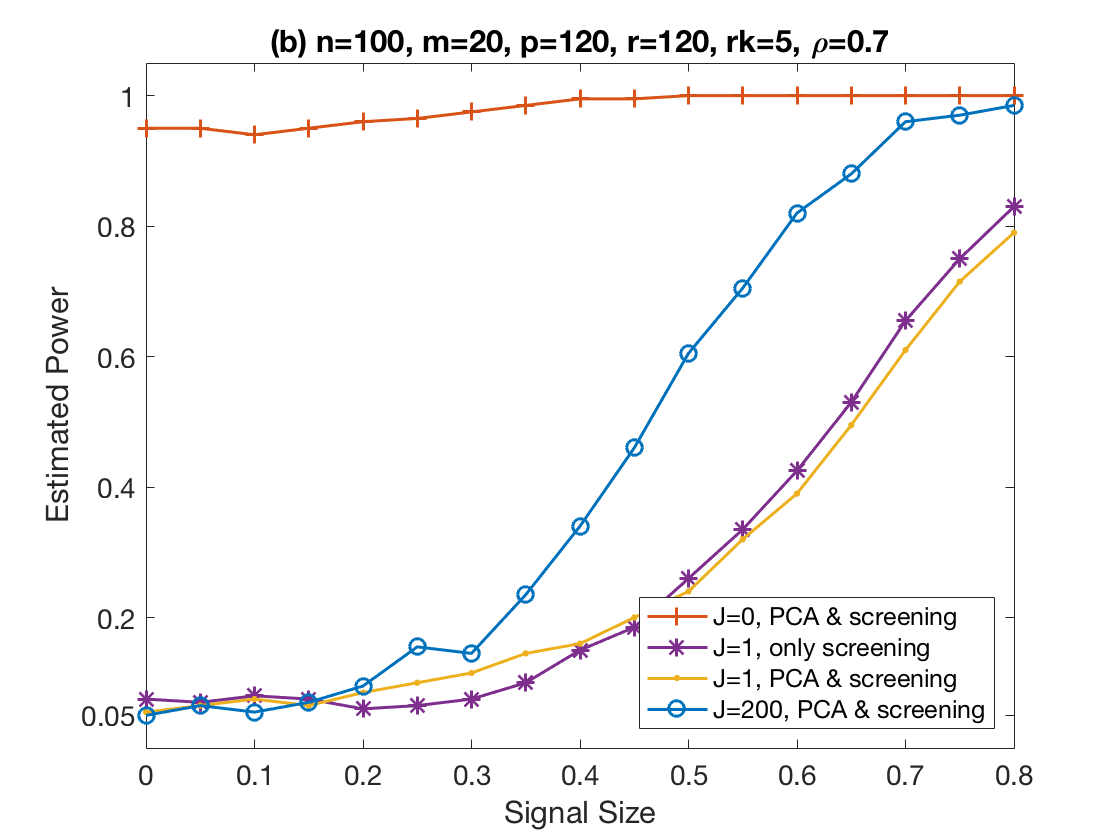}
\end{subfigure}
\begin{subfigure}{\textwidth}
	\includegraphics[width=0.48\textwidth]{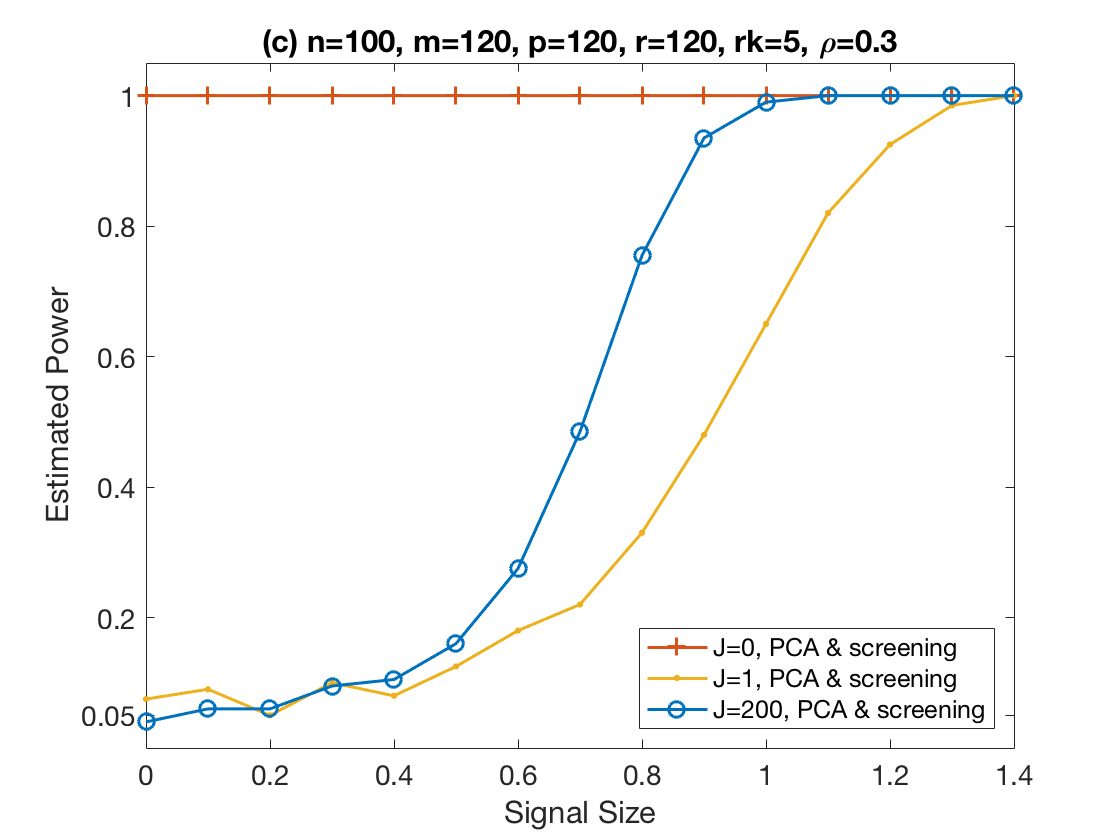}
	\includegraphics[width=0.48\textwidth]{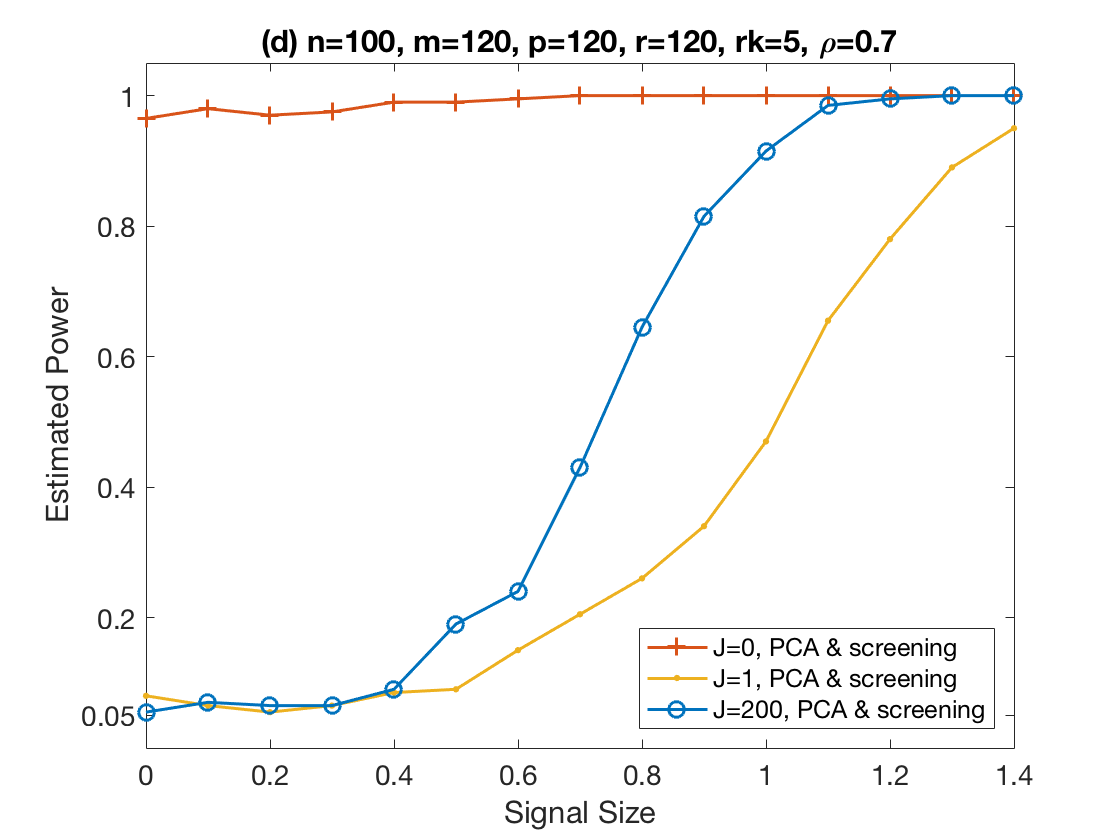}
\end{subfigure}
\caption{Estimated powers versus signal sizes when $n<m+p$}
\label{fig:powerestpowerscreenrk5}
\end{figure}

In addition, under the similar set-up to that of Figure \ref{fig:powerestpowerscreen},  we conduct simulations when $\rho=0$ and $r_k\in \{1,5\}$. The results are presented in Figure  \ref{fig:powerestpowerscreenrho0}, where similar patterns are observed as in Figure \ref{fig:powerestpowerscreen}. 

\begin{figure}[!htbp]
\centering
\begin{subfigure}{\textwidth}
	\includegraphics[width=0.48\textwidth]{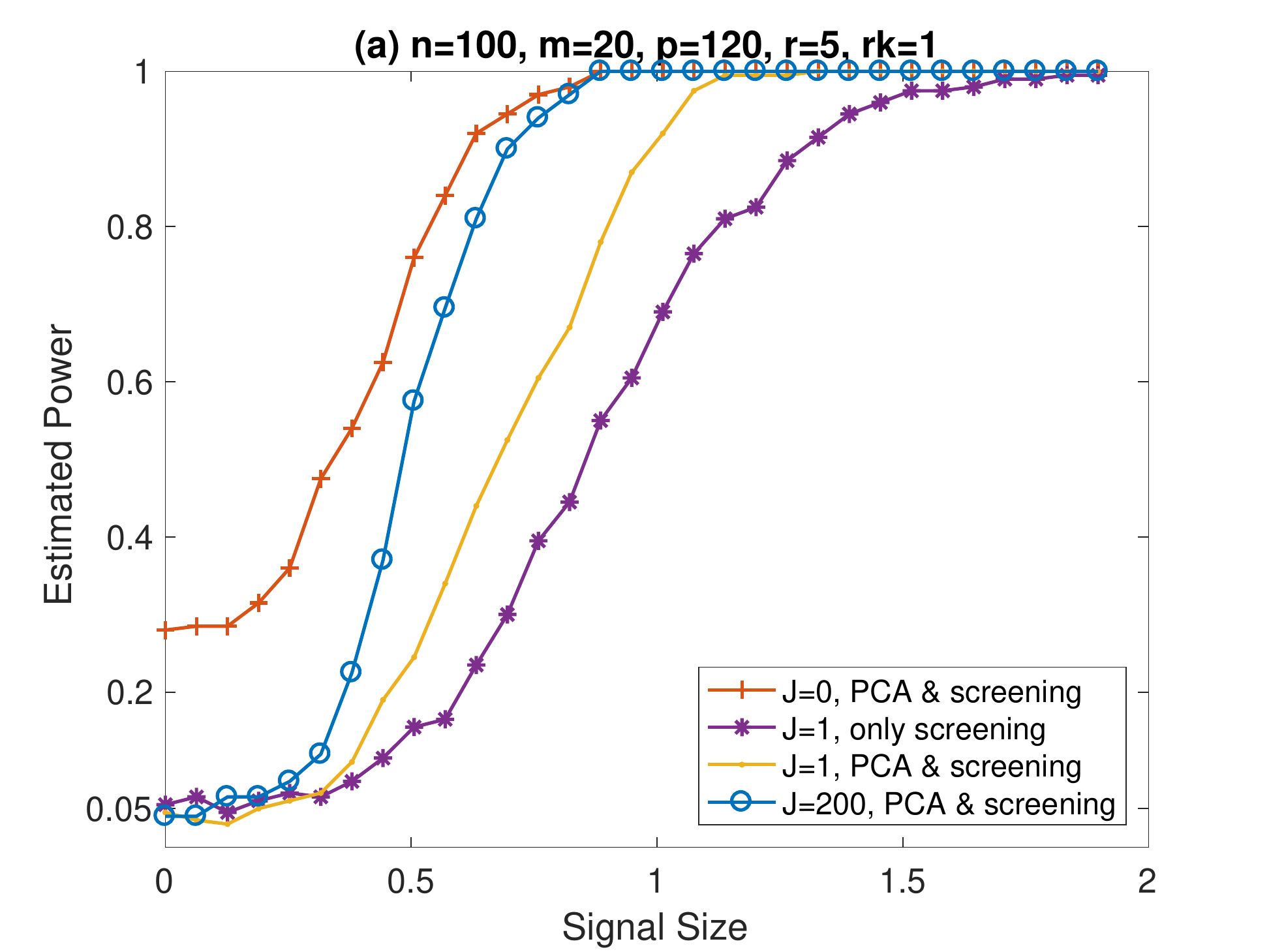}
	\includegraphics[width=0.48\textwidth]{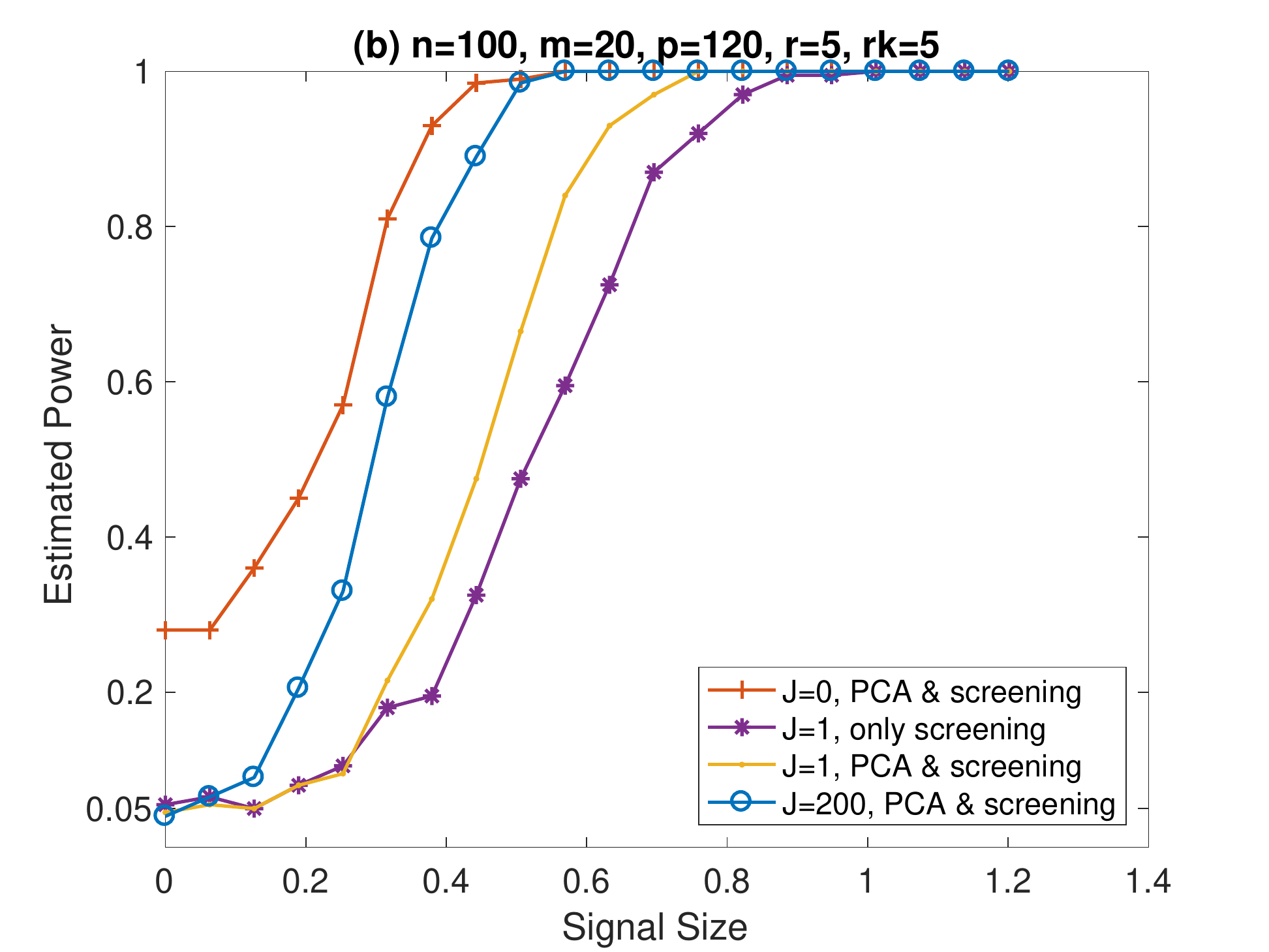}
\end{subfigure}
\begin{subfigure}{\textwidth}
	\includegraphics[width=0.48\textwidth]{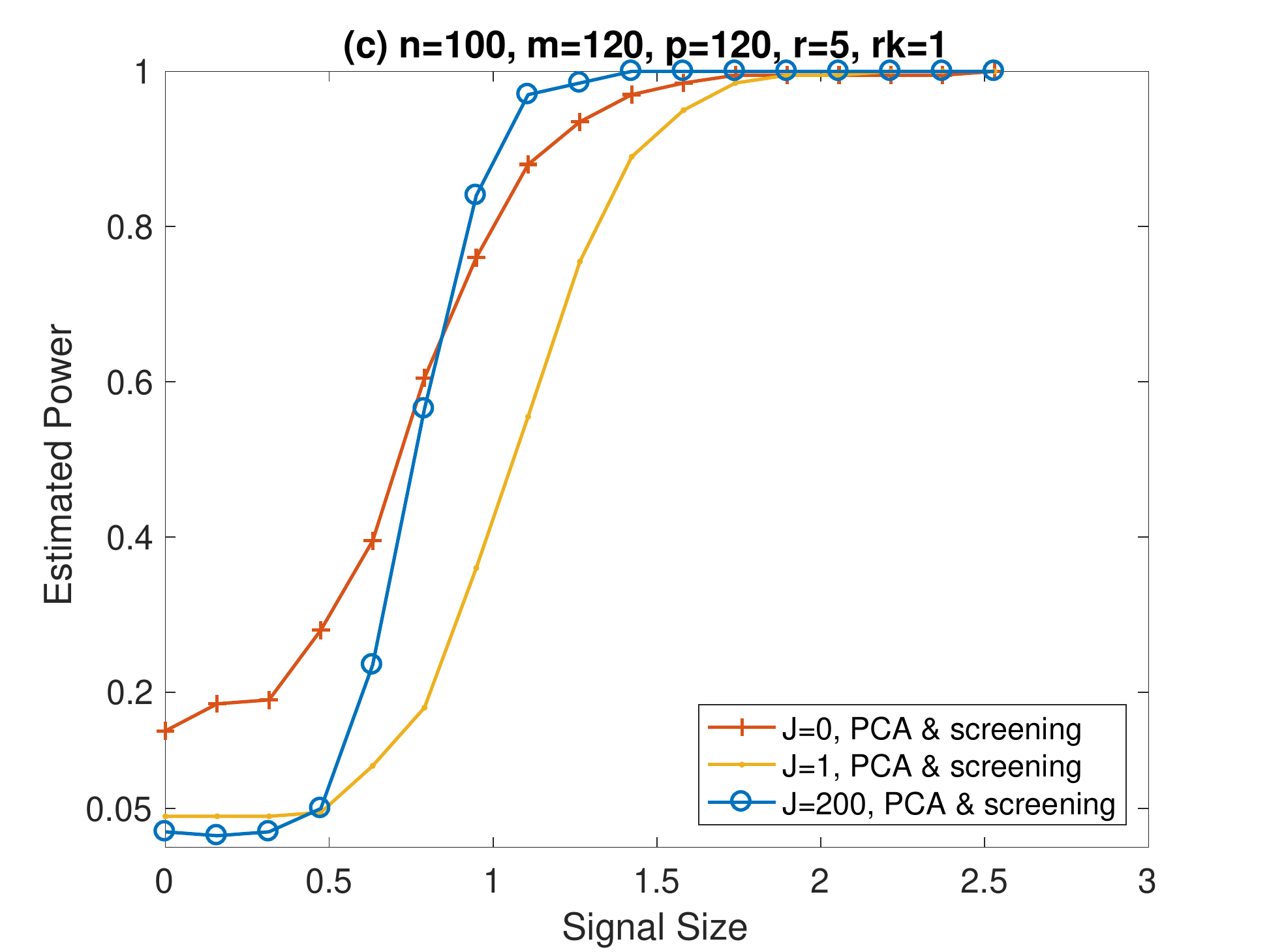}
	\includegraphics[width=0.48\textwidth]{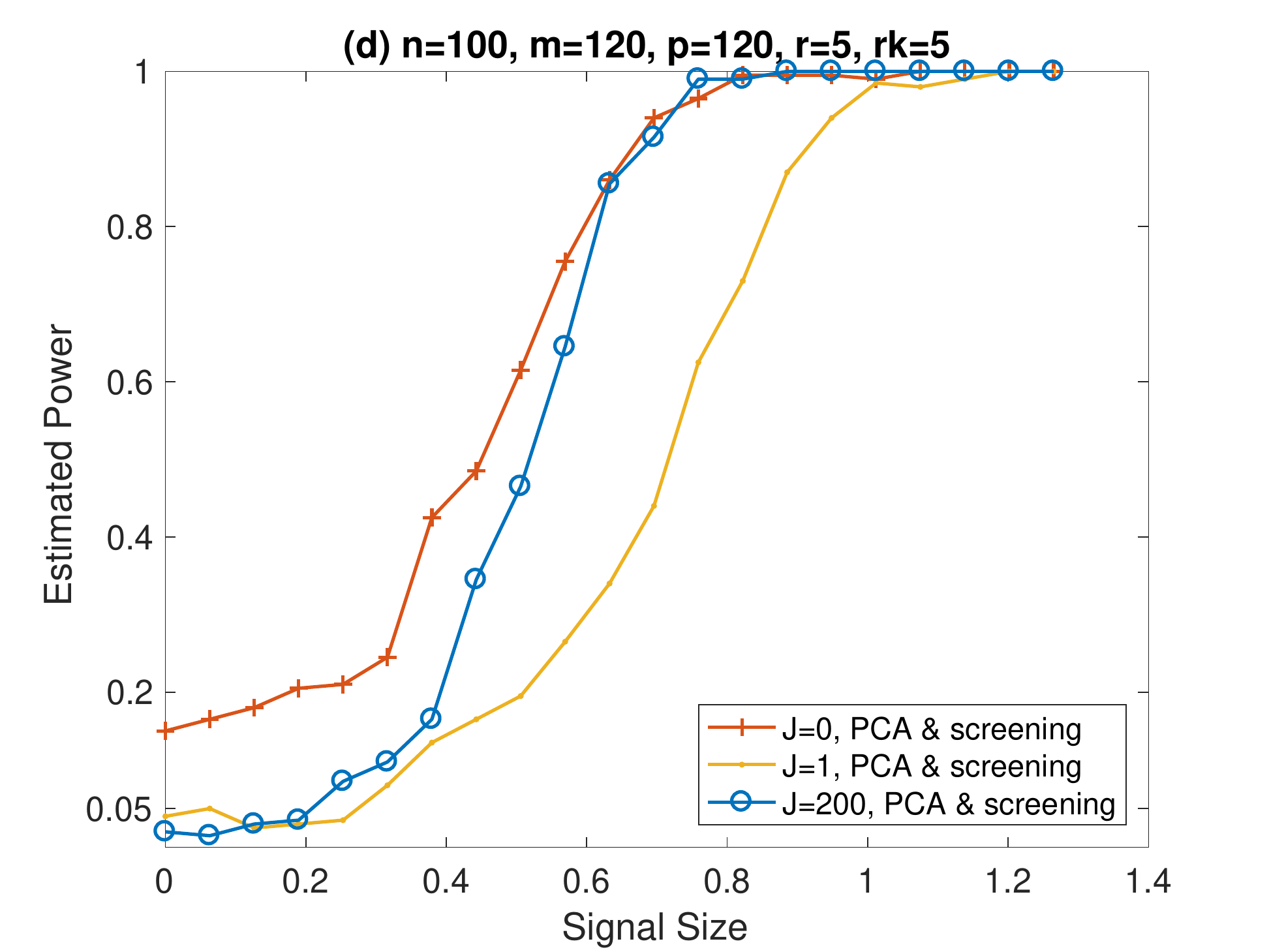}
\end{subfigure}
\caption{Estimated powers versus signal sizes when $n<m+p$}
\label{fig:powerestpowerscreenrho0}
\end{figure}

\newpage

\subsubsection{Robustness with other distributions} \label{sec:robtwostepdist}
To examine the robustness of the two-step procedure, we generate $X$ and $Y$ following Section \ref{sec:simulotherdist} with $n=100$,  $m=20$, $p=120$. We then generate $B$ and apply the testing  procedure similarly as in Section \ref{sec:simunsmallpm} with $r_k\in \{1, 5\}$. The results are presented in Figure \ref{fig:tdisttwostep}, where part (a)  gives the results 	when $X$ and $Y$ follow multinomial distribution, and parts (b) and (c) give the results when the error terms in $E$ are i.i.d. $t_3$ or $t_5$. We note that similar patterns are observed as in Figure \ref{fig:powerestpowerscreen}. This shows that the proposed two-step procedure is robust to the normal assumption. 

\begin{figure}[htbp]
\centering

\begin{subfigure}{\textwidth}
\centering
\includegraphics[width=0.45\textwidth]{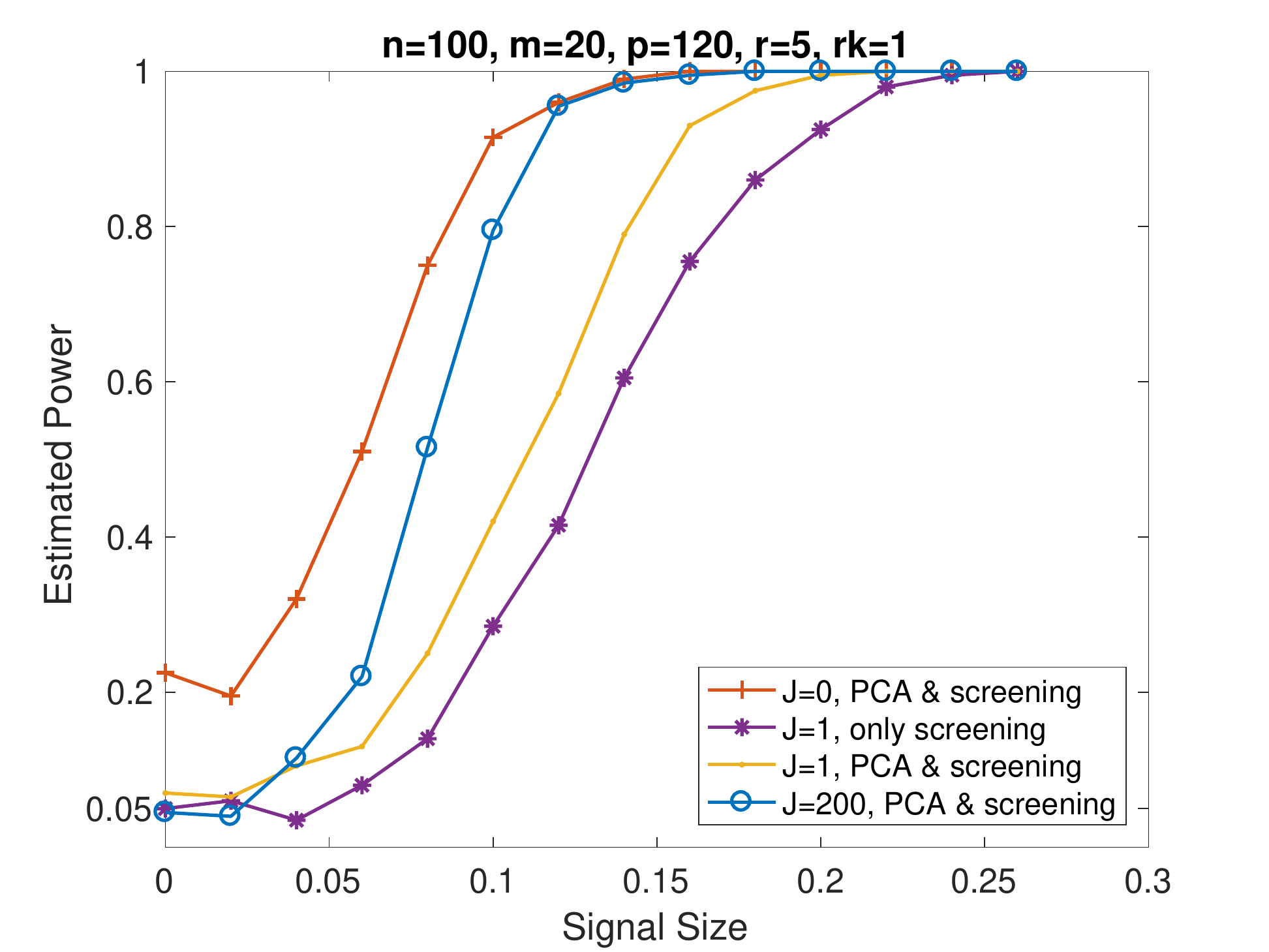}
\includegraphics[width=0.45\textwidth]{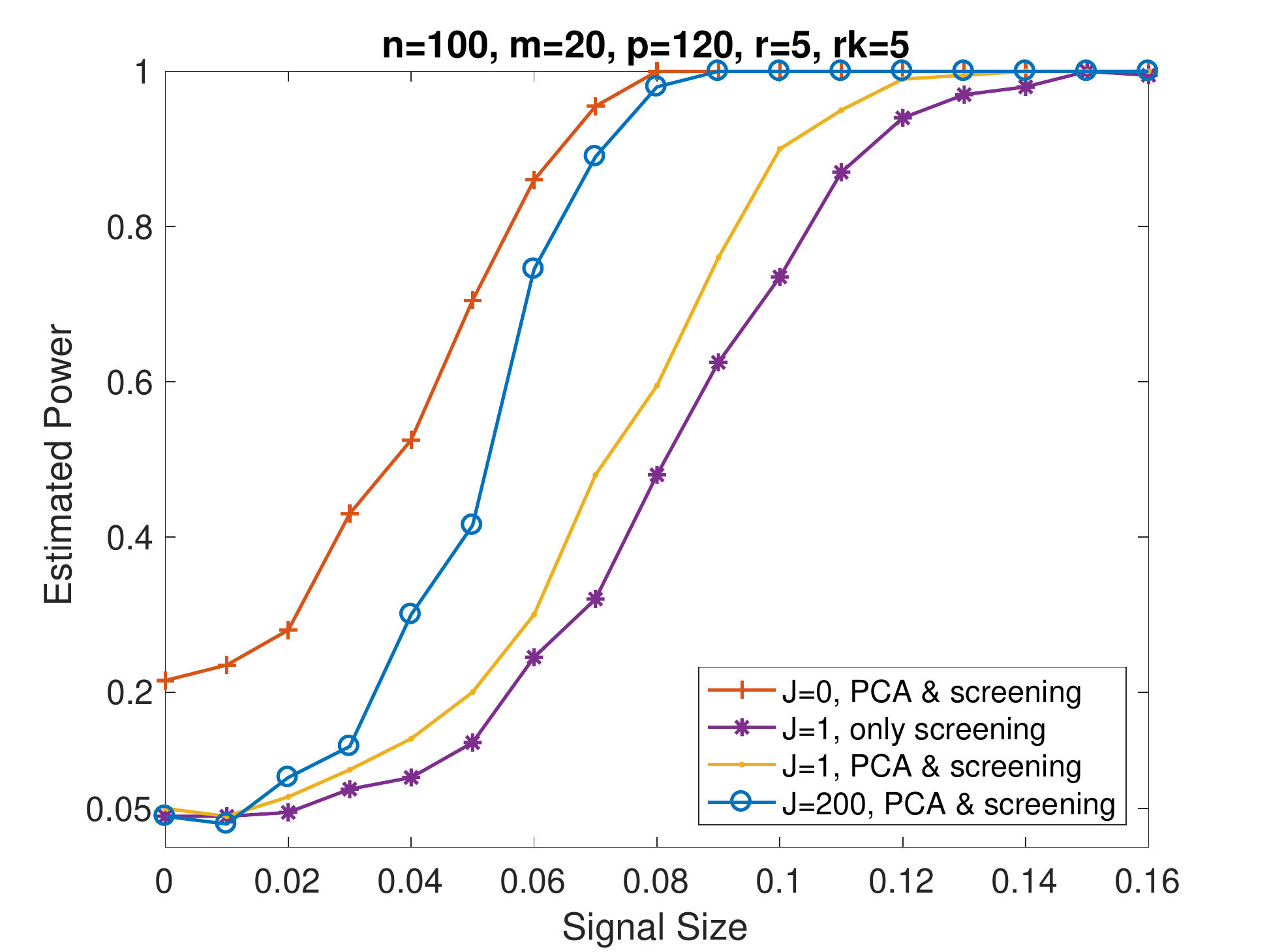}
\caption{$X$ and $Y$ follow multinomial distributions}
\label{fig:multnomitwostep}
\end{subfigure}

\begin{subfigure}{\textwidth}
\centering
\includegraphics[width=0.45\textwidth]{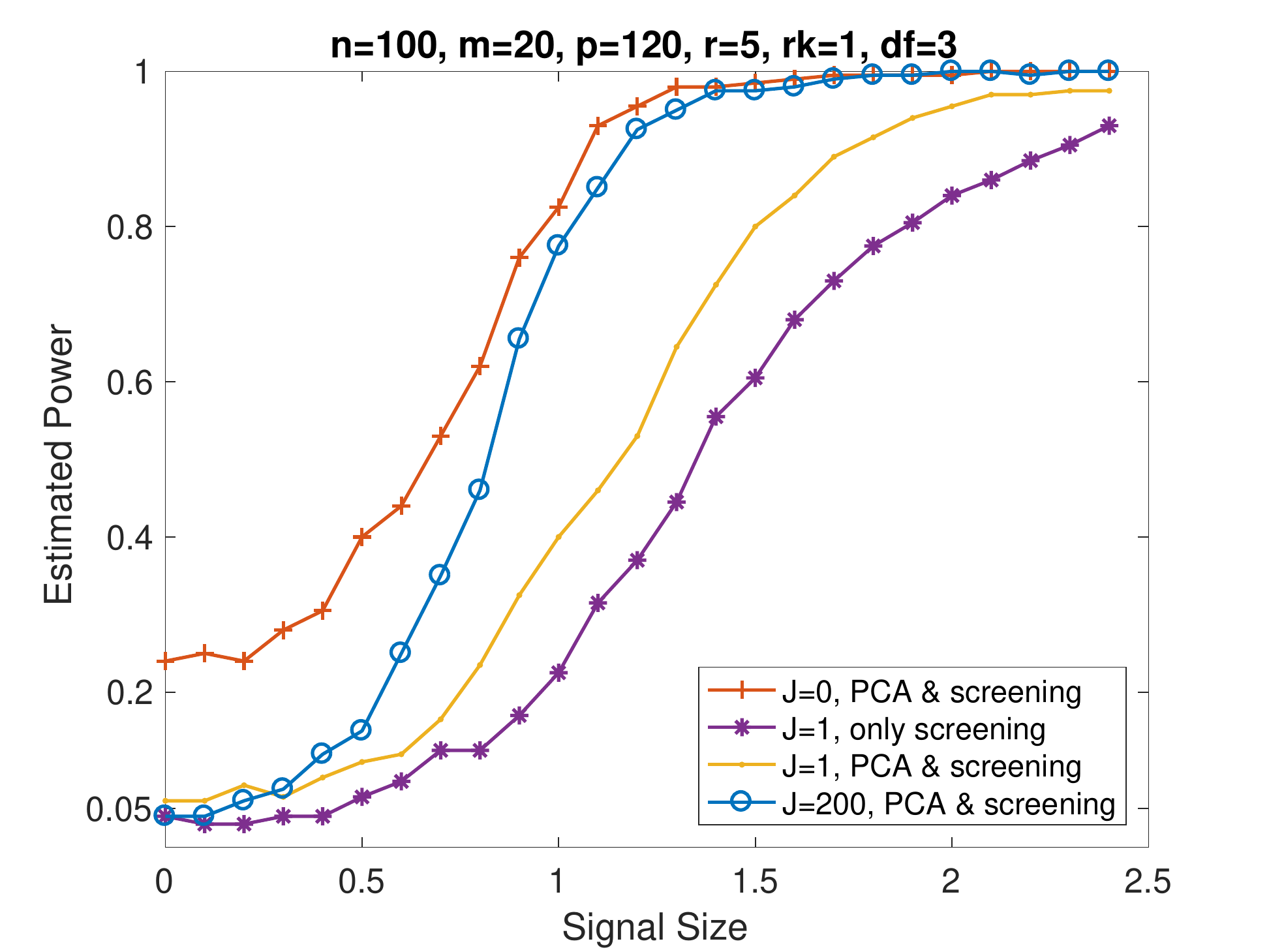}
\includegraphics[width=0.45\textwidth]{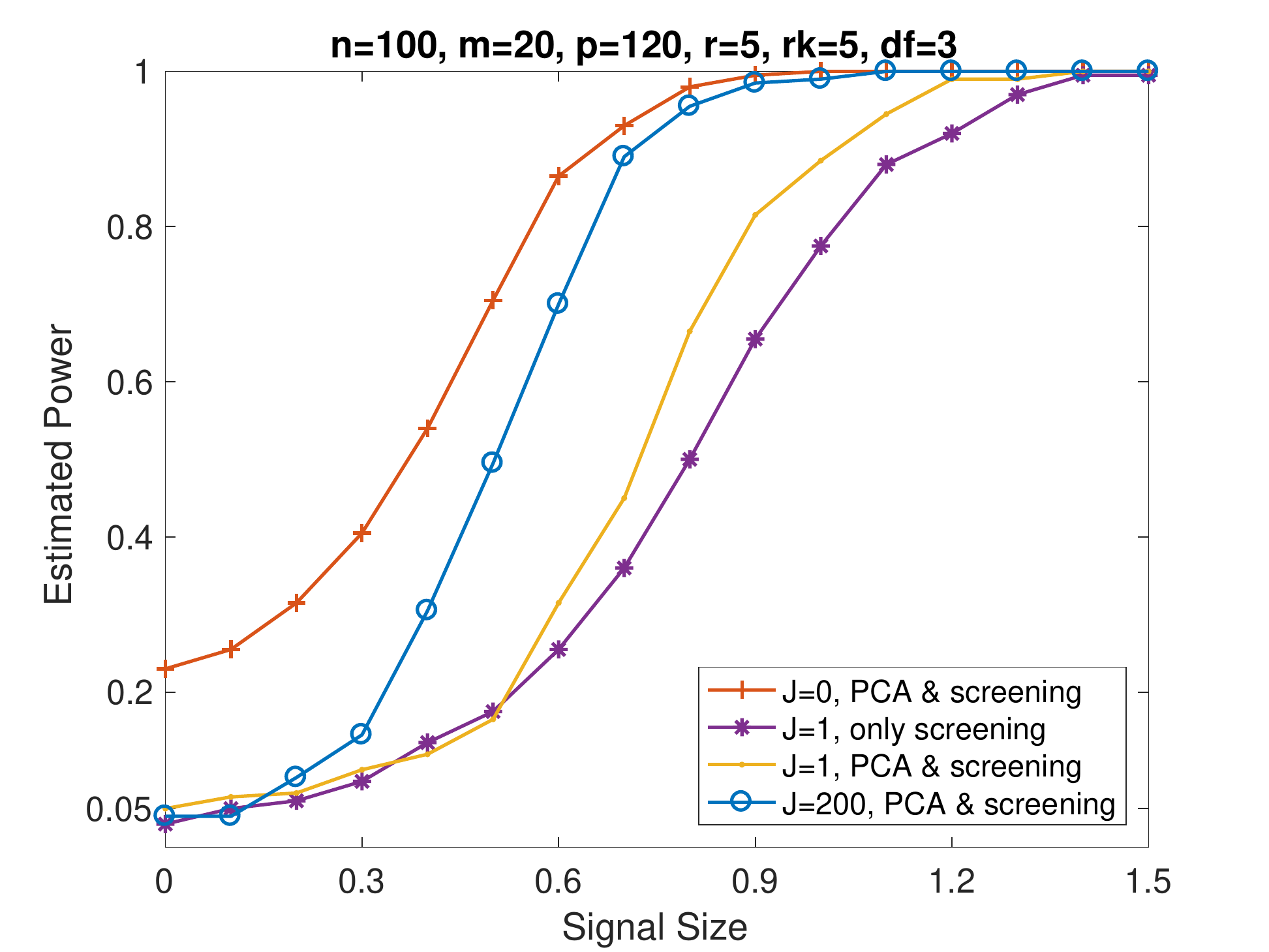}
\caption{Entries in $E$ follow  $t_3$ distribution}
\label{fig:tdisttwostep1}	
\end{subfigure}

\begin{subfigure}{\textwidth}
\centering
\includegraphics[width=0.45\textwidth]{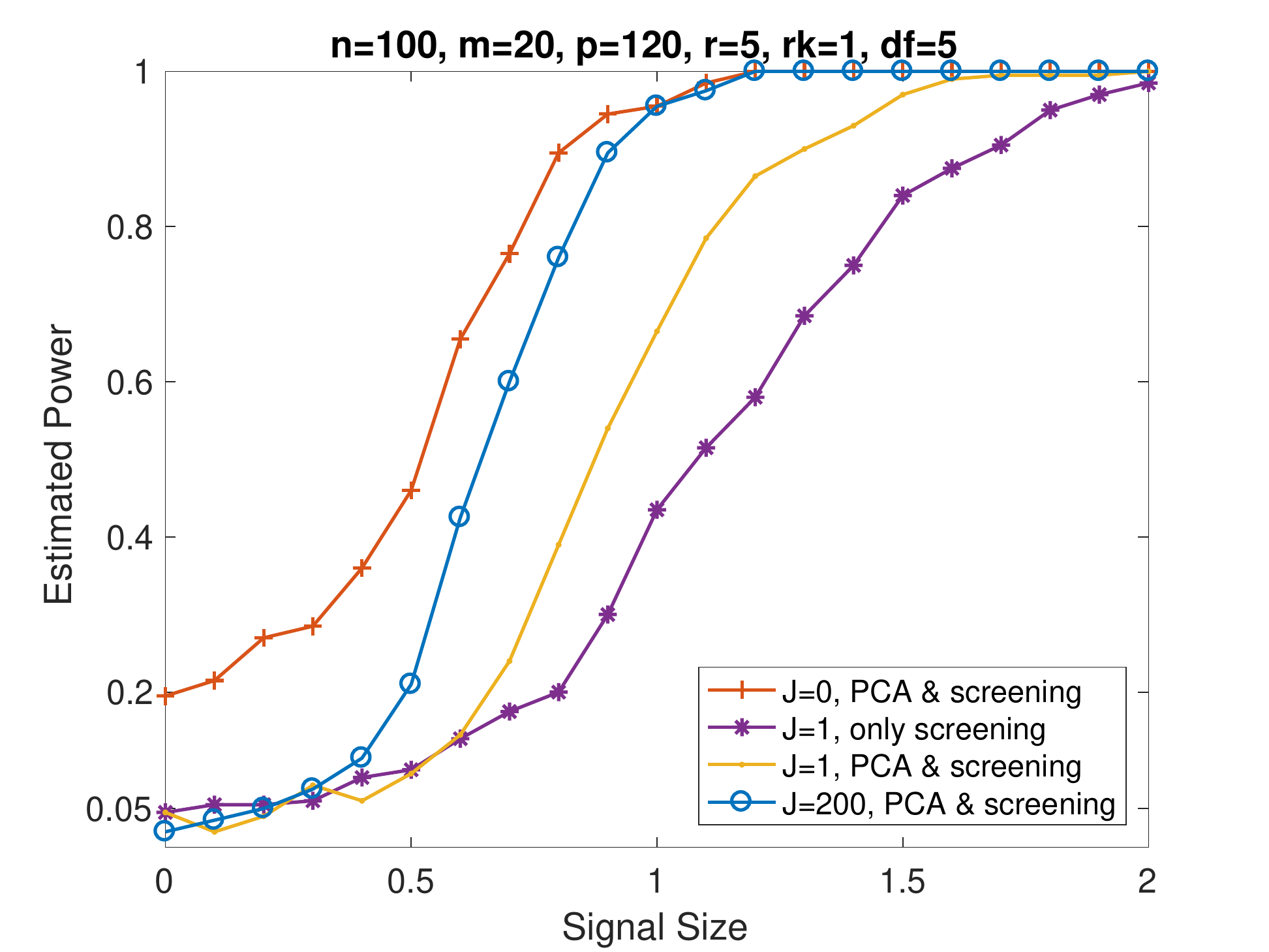}
\includegraphics[width=0.45\textwidth]{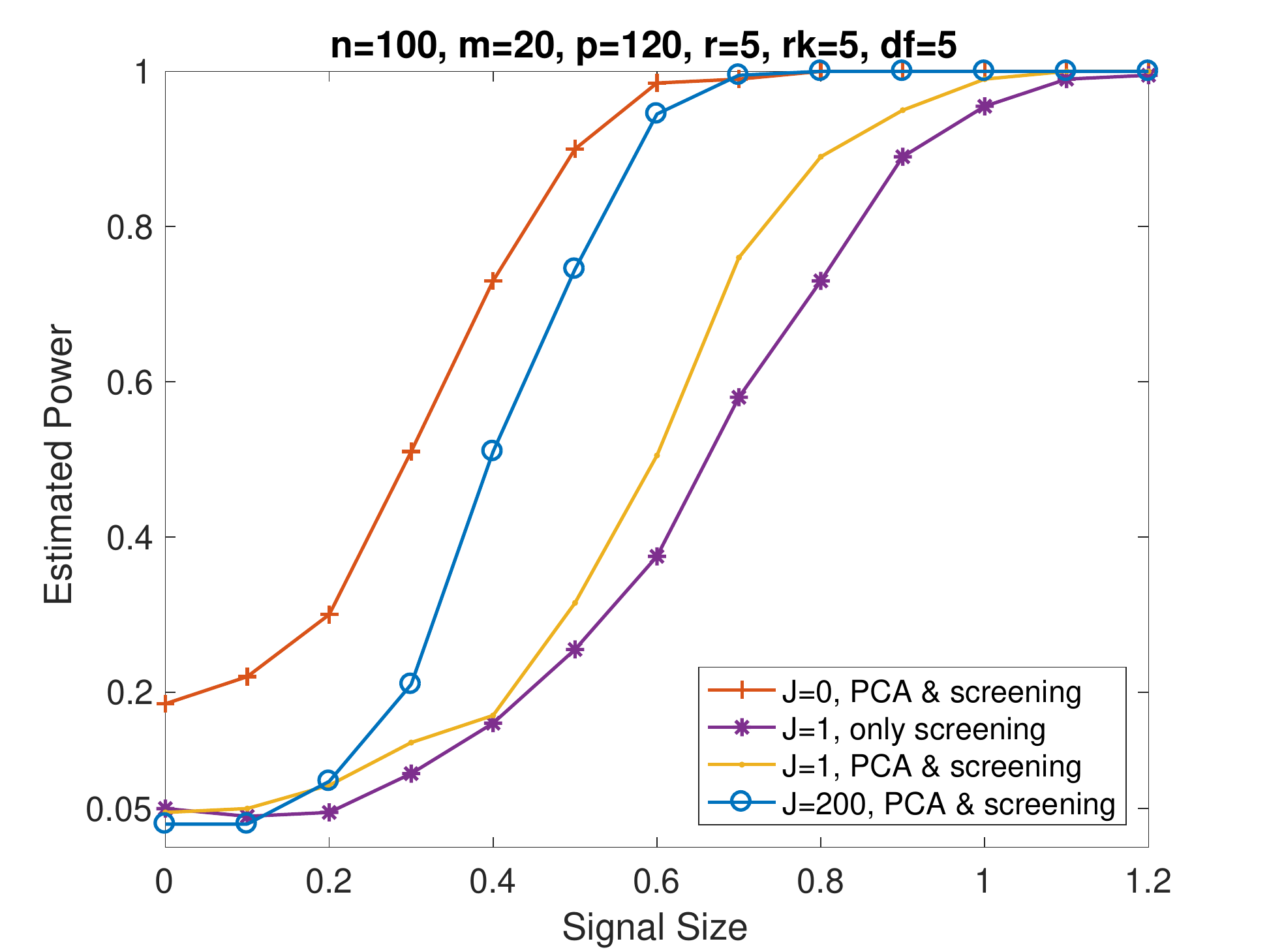}
\caption{Entries in $E$ follow $t_5$ distribution}
\label{fig:tdisttwostep2}	
\end{subfigure}

\caption{Estimated powers of two-step procedure with other distributions}
\label{fig:tdisttwostep}	
\end{figure}

\subsection{Simulations on $P\{ \psi(\alpha \gamma) \geq \gamma \}$} \label{sec:simugammavalue}
We conduct a simulation study to illustrate how the value of $P\{ \psi(\alpha \gamma) \geq \gamma \}$ depends on the correlations of the $p$-values. We consider an ``ideal" case with equal correlated  $p$-values. Specifically  we generate   $p^{(j)}=1-\Phi(V_{J,j})$ for $j=1,\ldots, J$, where   $V_J=(V_{J,1},\ldots,V_{J,J})^{\intercal} \sim \mathcal{N}(\mathbf{0},\Sigma_J)$ with $\Sigma_J=(1-\rho)I_{J,J}+\rho \mathbf{1}_J^{\intercal} \mathbf{1}_J$. Note that larger $\rho$ value implies larger correlations between $p^{(j)}$'s. 
 We take $J=200$ and  use $10^6$ Monte Carlo repetitions  
  to estimate $P\{ \psi(\alpha \gamma) \geq \gamma \}$.  
   Figure \ref{fig:estimatepcor} gives the simulation results for $\rho \in \{0,0.2,0.4,0.6,0.8,0.9,0.95,1 \}$, and   $\gamma\in(0,1)$ and $(0,0.01)$ respectively. 
 When $\rho$ is small, the largest value of  $P\{ \psi(\alpha \gamma) \geq \gamma \}$ is attained at $5\times 10^{-3}=J^{-1}$; when $\rho=1$, the largest value   is attained at $\gamma=1$. These observations are consistent with the above theoretical argument. 

\begin{figure}[htbp]
\centering
\begin{subfigure}{\textwidth}
	\includegraphics[width=0.45\textwidth]{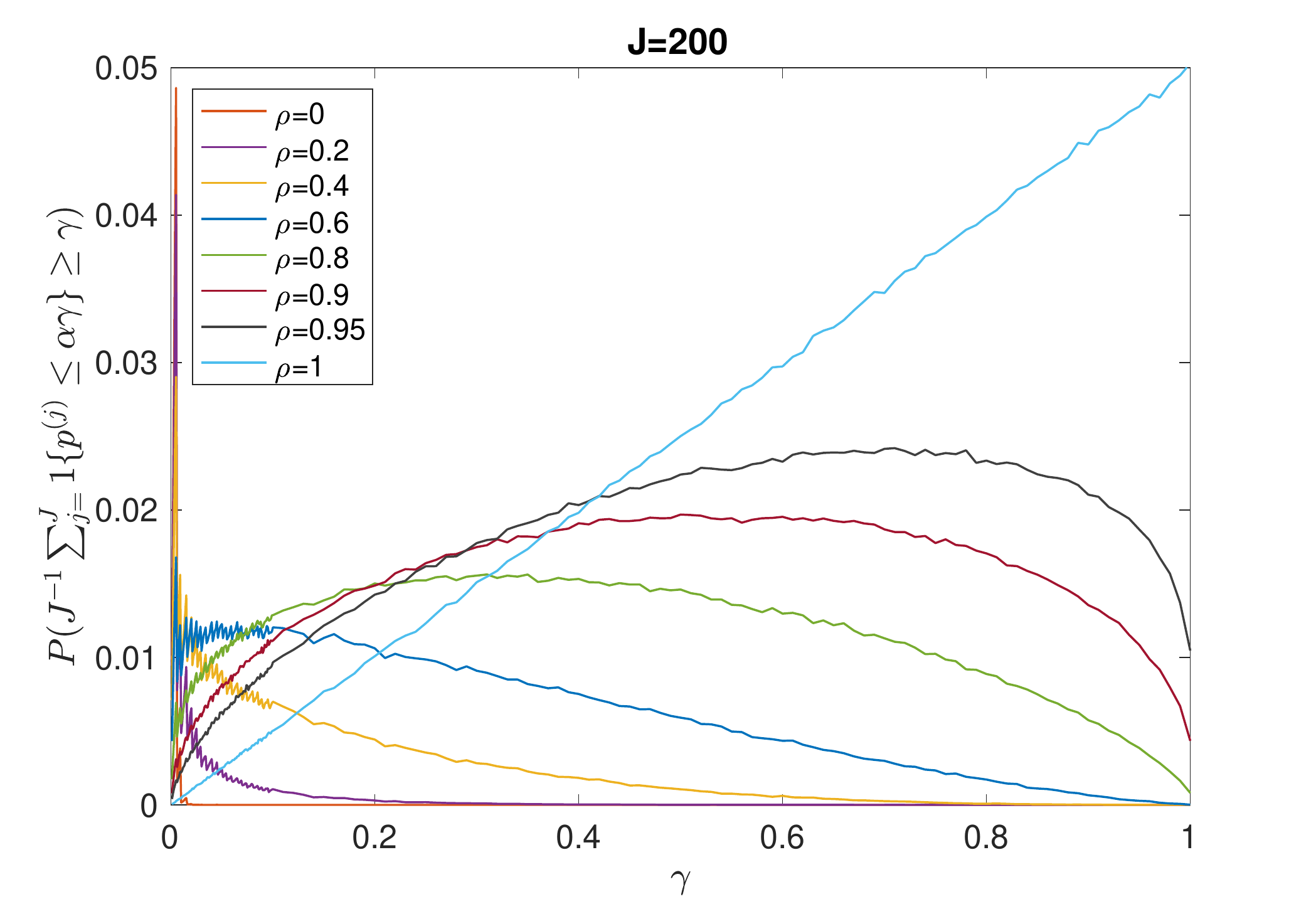}
	\includegraphics[width=0.45\textwidth]{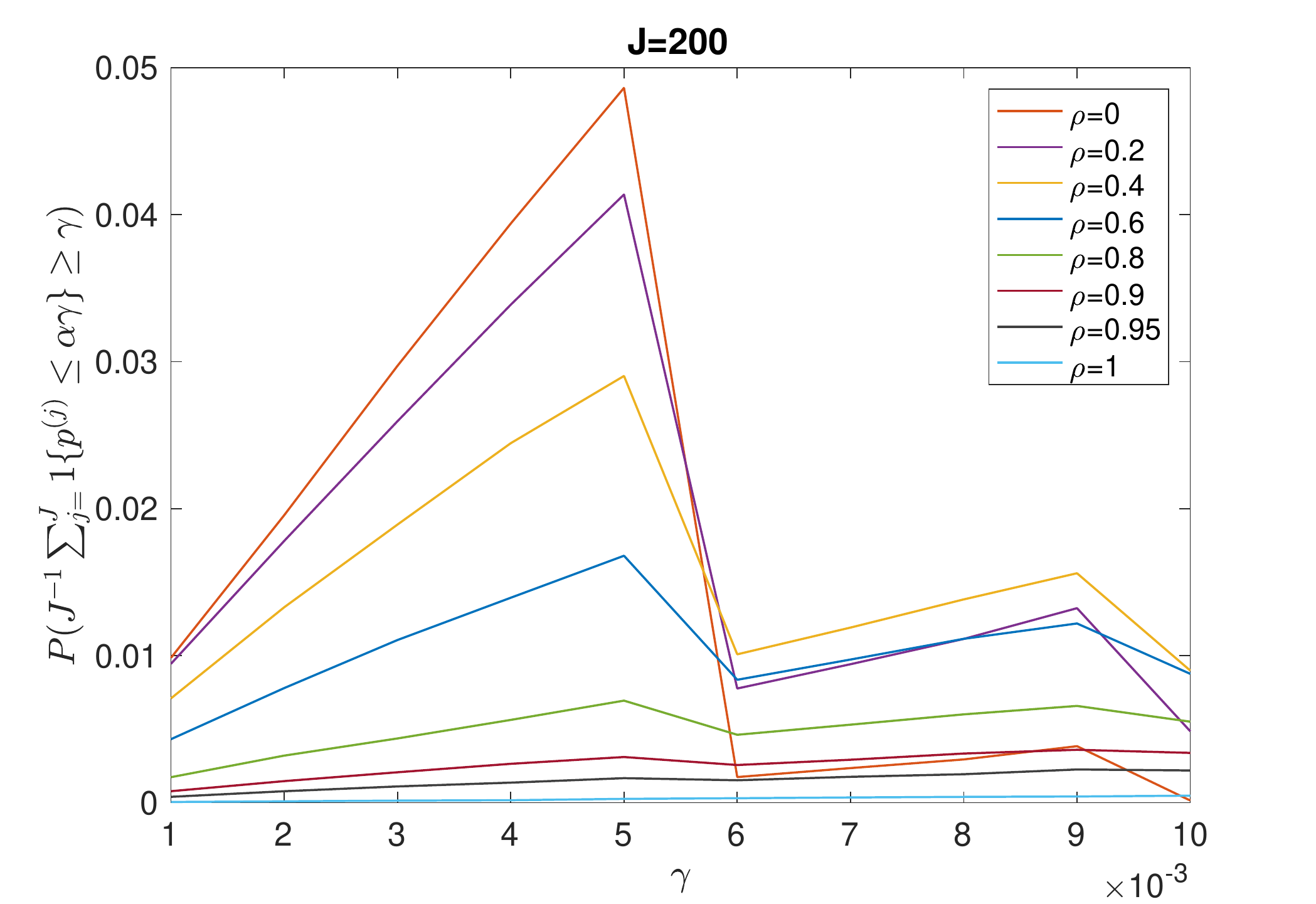}
\end{subfigure}
\caption{Estimated $P\{ \psi(\alpha \gamma) \geq \gamma \}$ versus $\gamma$ under different correlation levels}
\label{fig:estimatepcor}	
\end{figure}

\subsection{Simulations compared with screening using lasso} \label{sec:comparelasso}

In this paper, we propose the  two-stage testing procedure using the screening with canonical correlations. Note that the proposed method aggregates the joint information of the response variables, and  thus could be better than simply applying the marginal screening with respect to each response variable. 
To further study the effect of highly correlated predictors, we compare our method to  using lasso with cross-validation, which is expected to account for the dependence in the predictors while not for the dependence in the responses. 

In particular, for the screening with canonical correlations, $20\%$ predictors are selected as in Section \ref{sec:simunsmallpm}; for the screening  with lasso, we select the predictors ($\leq 20\%$ of all predictors) that minimize the MSE in 10-fold cross-validation. In the simulations, we take  $C=[I_r,\mathbf{0}_{r\times (p-r)}]$, and generate the rows of $X$ and $E$ as independent multivariate Gaussian with covariance matrices $\Sigma_x=(\rho^{|i-j|})_{p\times p}$ and $\Sigma=(\rho^{|i-j|})_{m\times m}$ respectively.  For each setting considered, we choose $\rho \in \{0.7, 0.9\}$, which are the cases when the predictors are of large correlations. 
\begin{figure}[!h]
\centering
\begin{subfigure}{\textwidth}
\centering	\includegraphics[width=0.45\textwidth]{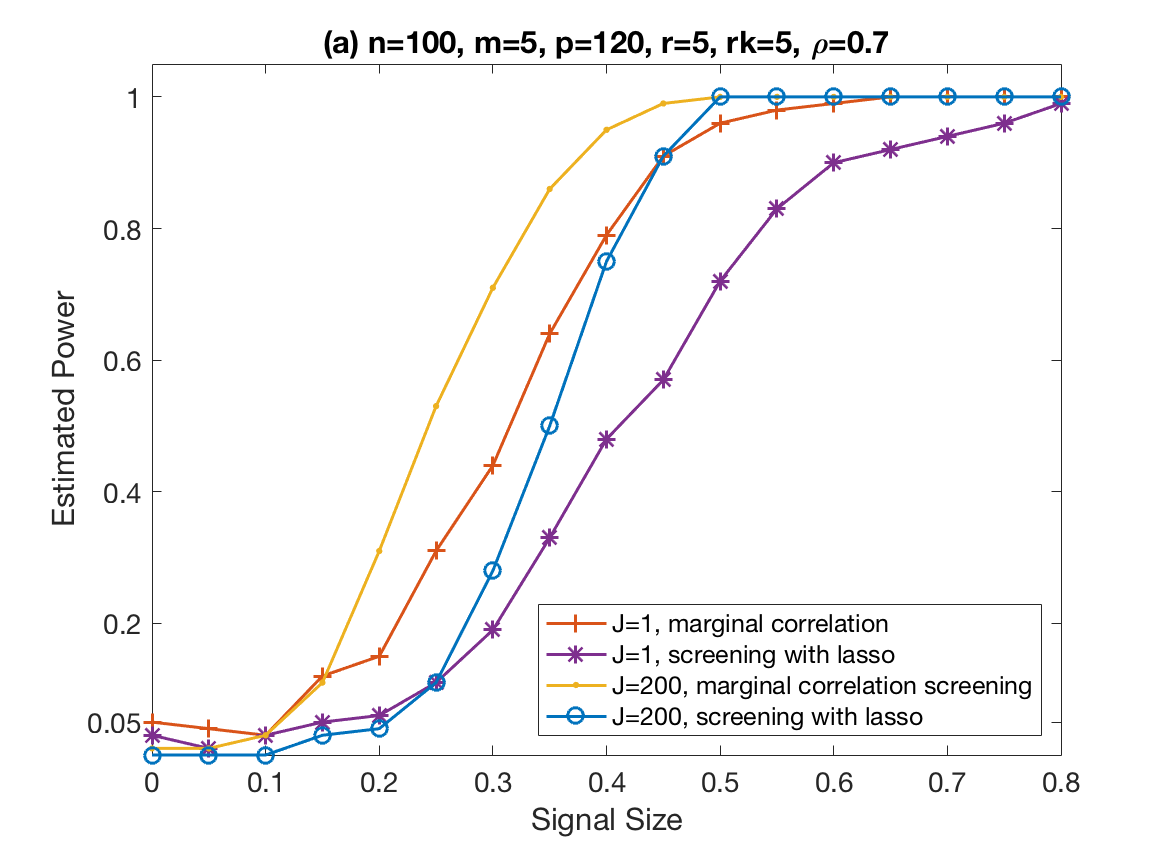}
	\includegraphics[width=0.45\textwidth]{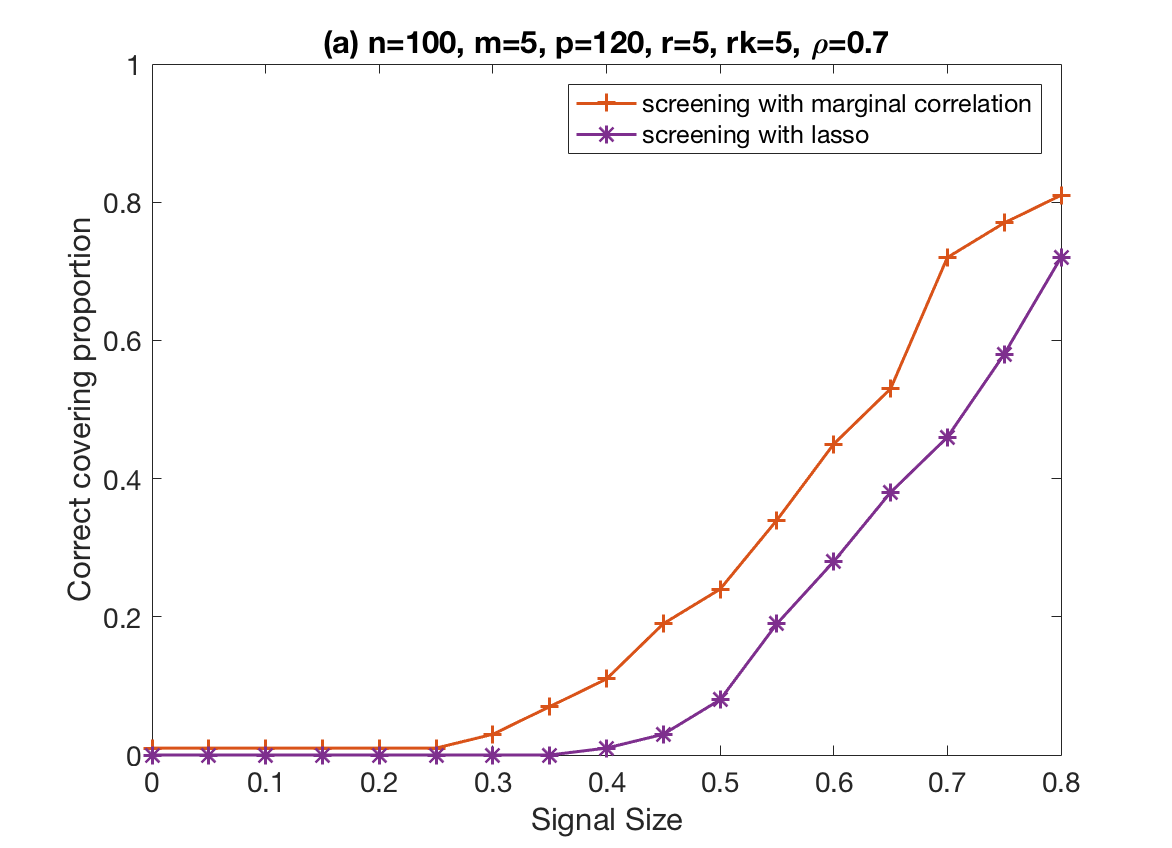}
\end{subfigure}
\begin{subfigure}{\textwidth}
\centering	\includegraphics[width=0.45\textwidth]{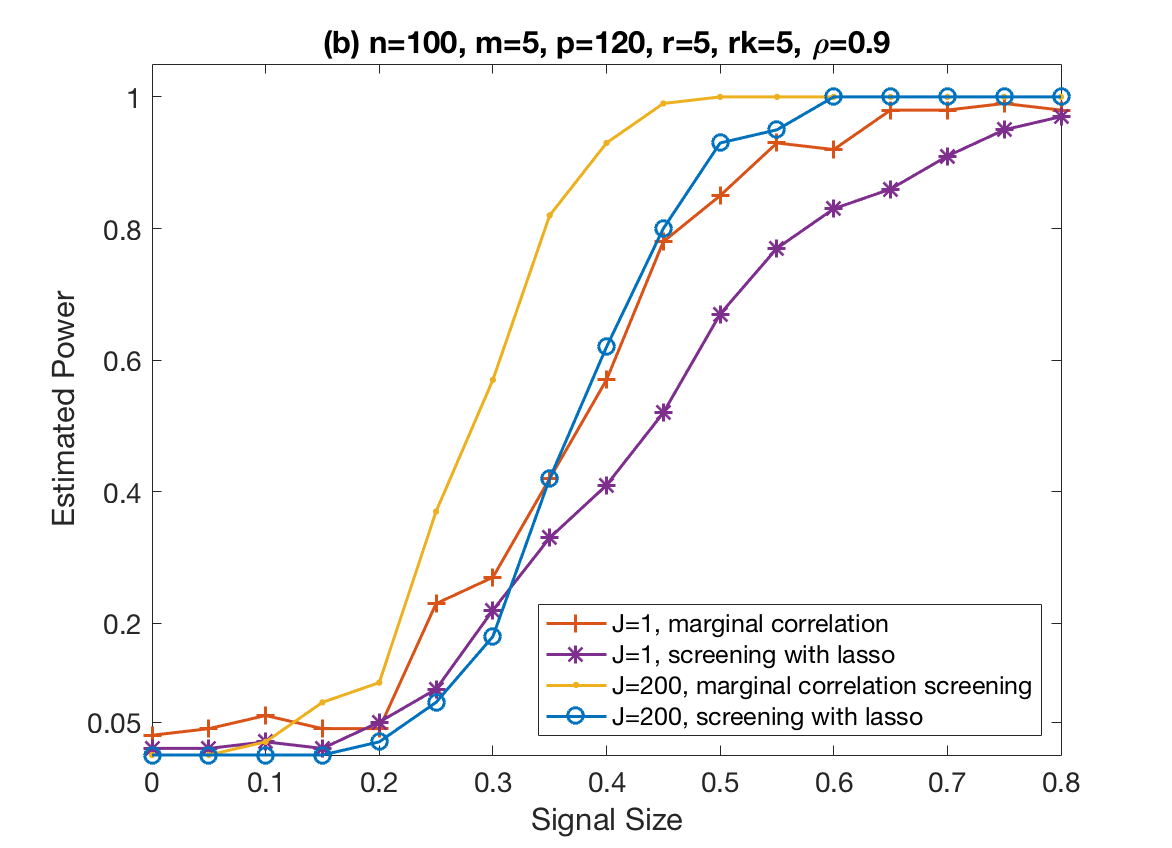}
	\includegraphics[width=0.45\textwidth]{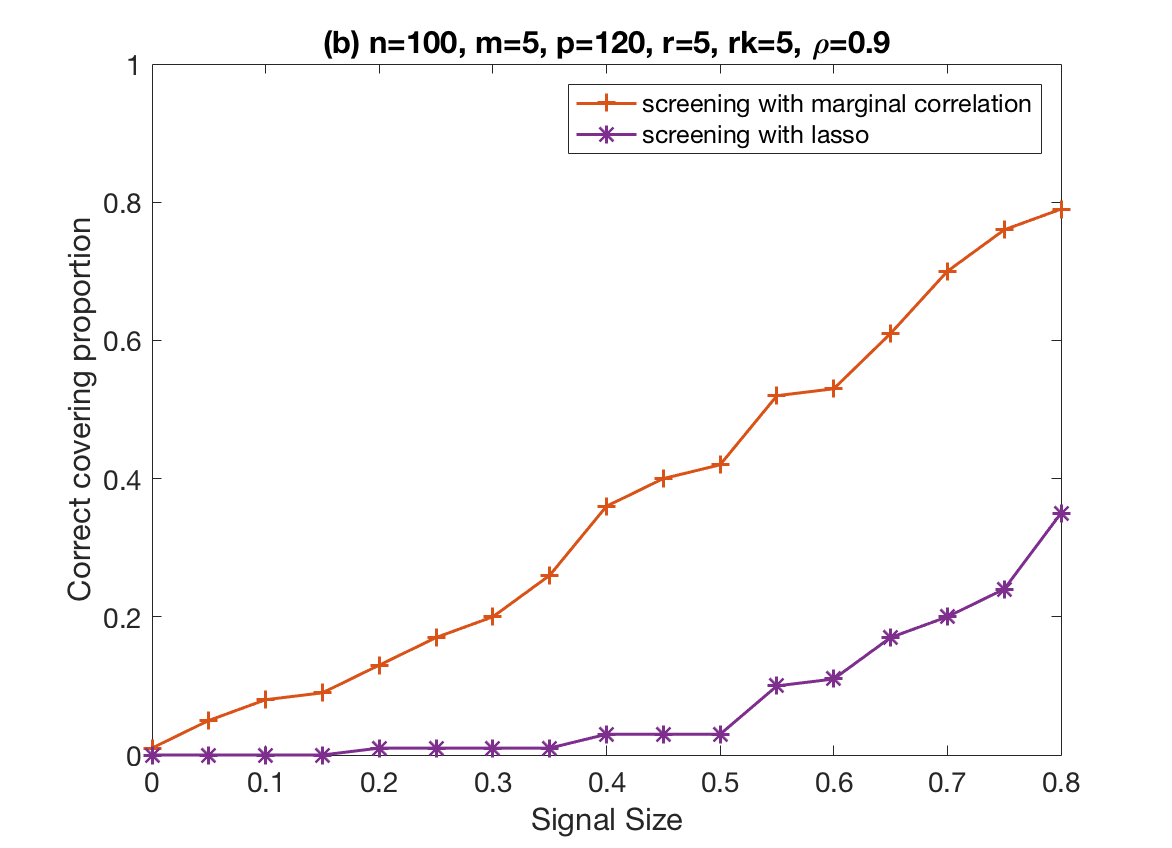}
\end{subfigure}
\caption{Screening Comparison: $B$ is diagonal}
\label{fig:powerestpowerscreencompare}
\end{figure}

We next consider two simulation settings, whose results are provided in the following Figures \ref{fig:powerestpowerscreencompare} and \ref{fig:powerestpowerscreenrandb} respectively. 
In the first setting, we choose $B$ to be a $p\times m$ diagonal matrix with  $\sigma_s$ in the first $r_k$ diagonal entries, where $\sigma_s$ represents the signal size that varies in simulations. We take $n=100, p=120, m=5, r=5$ and $r_k=5$.  In the second setting, we generate $B$ with a nonzero submatrix of size $r_k\times m$ in the upper left corner, where the entries are randomly generated from  $\mathcal{N}(0,\sigma_s^2)$. We take $n=100, p=120, m=5, r=120$ and $r_k=5$.  In both  Figures \ref{fig:powerestpowerscreencompare} and \ref{fig:powerestpowerscreenrandb}, we provide the estimated powers versus signal sizes in the left column, where $J$ represents the number of splits similarly as in Figure \ref{fig:powerestpowerscreen}. In addition, we provide the corresponding  proportion of simulations that cover the true active set (correct covering proportion) versus signal sizes in the right column.

\begin{figure}[!h]
\centering
\begin{subfigure}{\textwidth}
\centering	\includegraphics[width=0.45\textwidth]{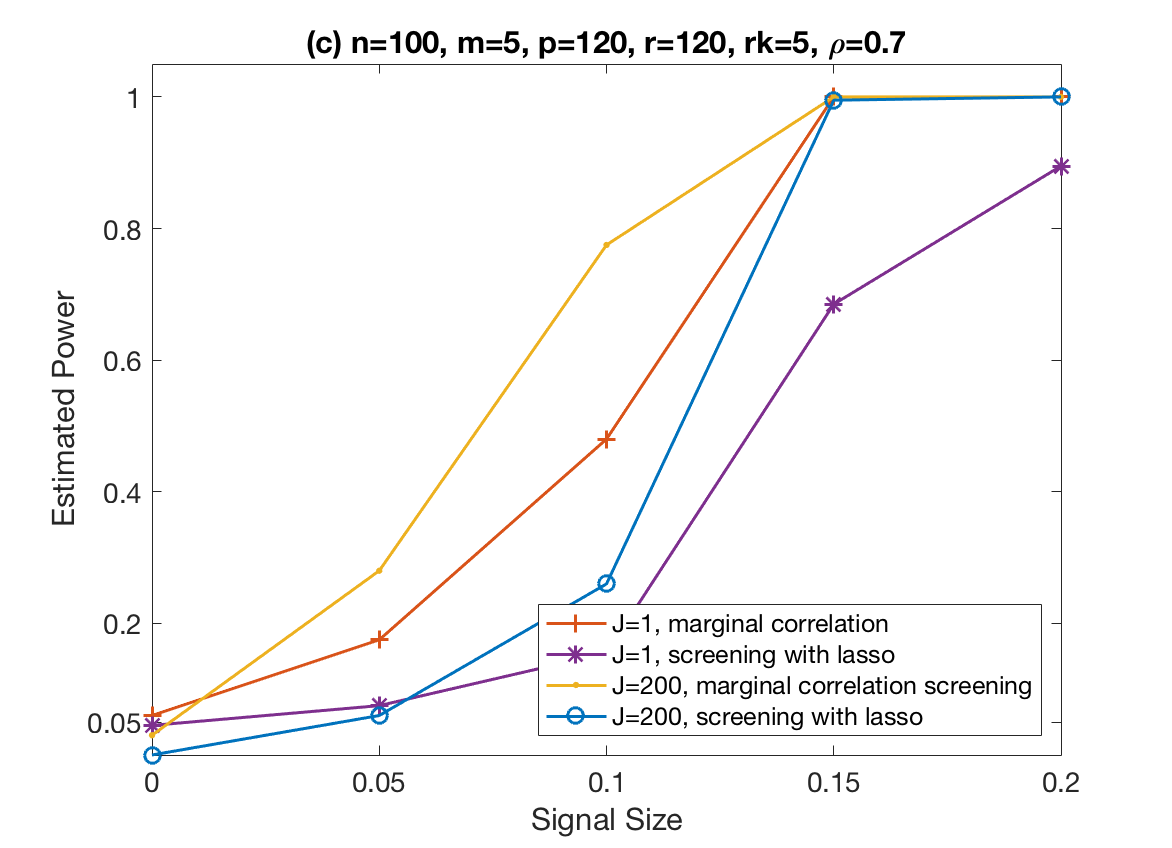}
	\includegraphics[width=0.45\textwidth]{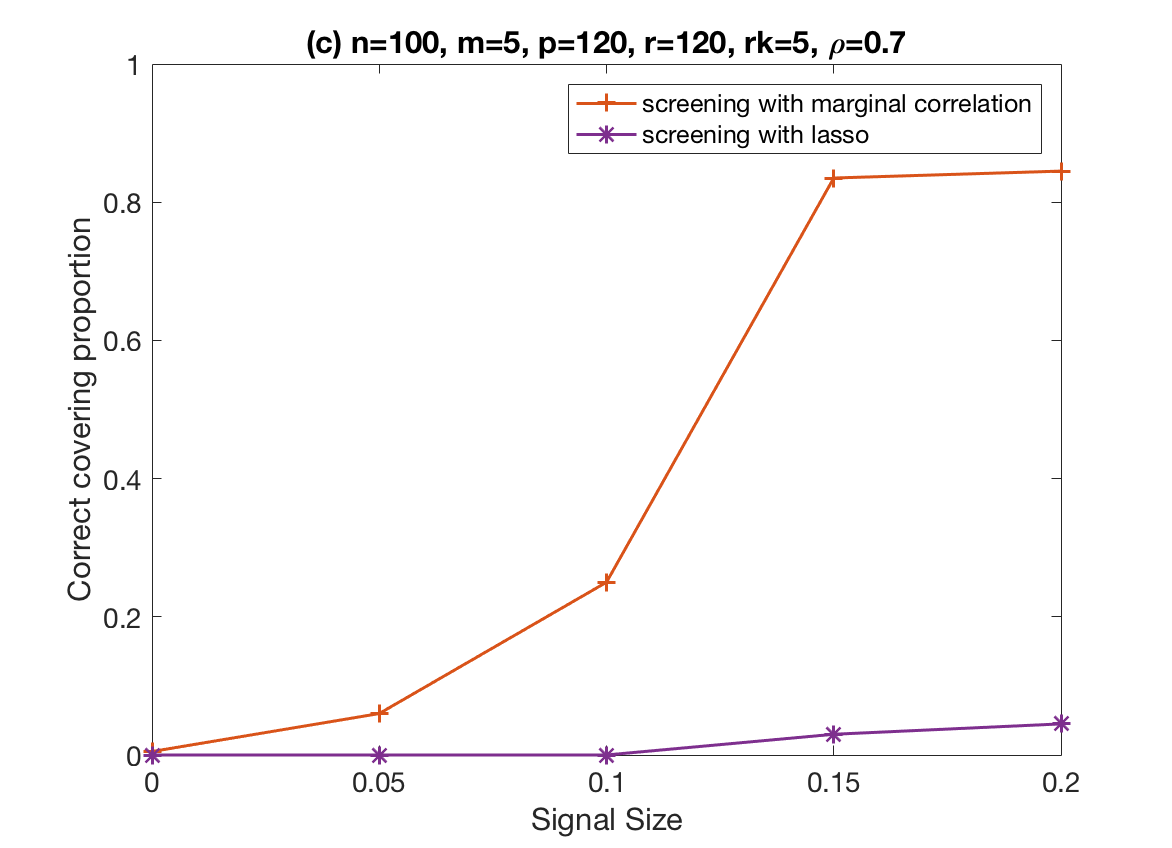}
\end{subfigure}
\begin{subfigure}{\textwidth}
\centering	\includegraphics[width=0.45\textwidth]{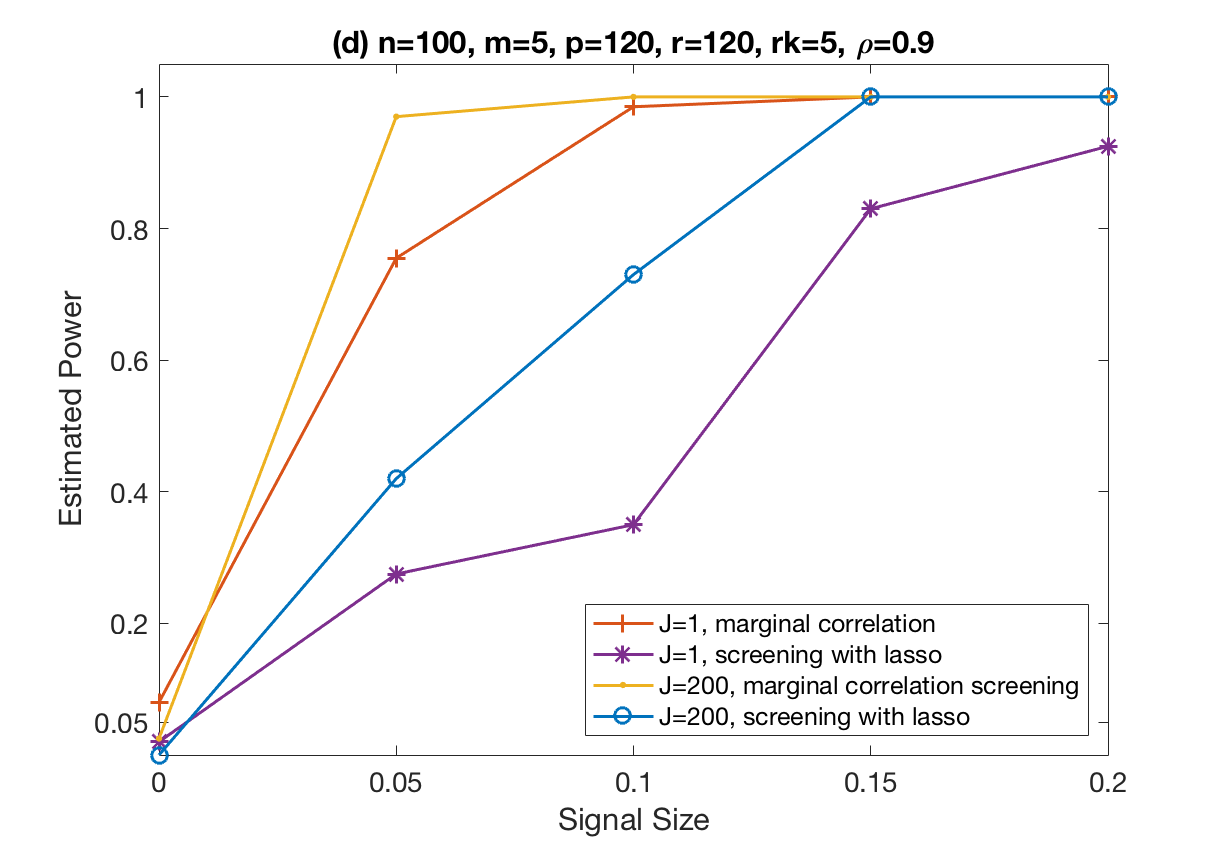}
	\includegraphics[width=0.45\textwidth]{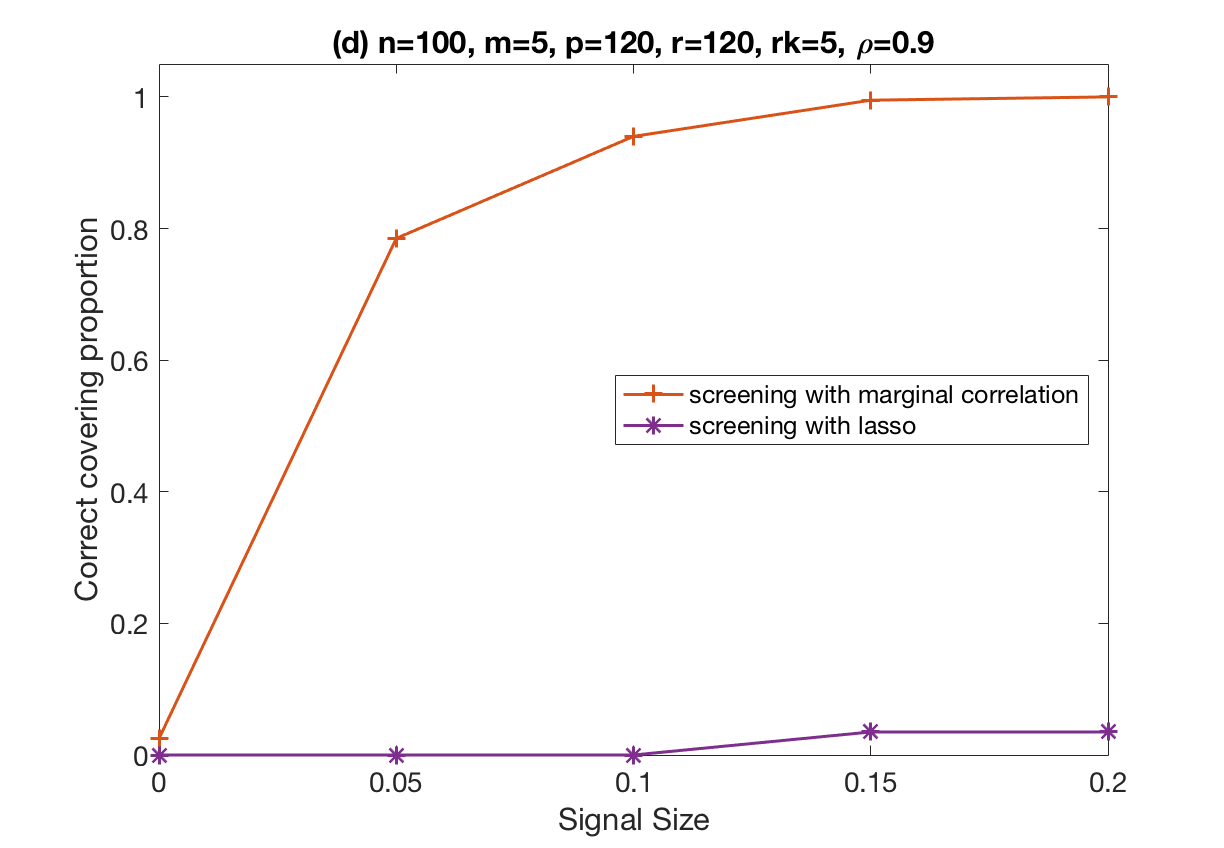}
\end{subfigure}
\caption{Screening Comparison: $B$ has a  nonzero submatrix}
\label{fig:powerestpowerscreenrandb}
\end{figure}

By the simulation results, we find that under the considered simulation settings, even though the correlations among predictors are large,  using the canonical correlation   in screening  performs better  than using lasso with cross-validation, in terms of both test power and correct covering proportion.   
The results suggest that the   correlation-based procedure can still account for the dependence among predictors reasonably under certain settings with correlated predictors. 
In addition, comparing the test power and corresponding correct covering proportion in  Figures \ref{fig:powerestpowerscreencompare} and \ref{fig:powerestpowerscreenrandb}, we find that the under selection of the true active set generally leads to loss of power in testing. To further improve the test power, it is still of interest to   develop a screening approach that could fit a wider range of scenarios and is also computationally efficient. 
Besides the two screening approaches compared here, 
we can also generalize other screening methods to the multivariate regression setting, as discussed in Remark \ref{rm:screenother} on Page \pageref{rm:screenother}. We will further study this in the follow-up research. 


\section{Supplementary Results of Real Data Analysis} \label{sec:supprealdata}

In this section, we present the analysis results of the regressions of GEPs on CNVs for the same dataset in Section \ref{sec:realdata}. 
 Then the $m$-variate response is the GEPs data and the $p$-variate predictor is the CNVs data, where now the  dimension parameters   are $(p,m)=(138,673), (87,1161), (18,516)$ for the three chromosomes  correspondingly.  
Similarly to Section  \ref{sec:realdata}, we apply the proposed procedure with $n_S=26$,  $n_T=63$ and $J=2000$.  As $m$ values are  large in this case, we choose different fixed numbers of principal components when applying PCA on the response $Y$. The chosen number of principal components and predictors are denoted as $m_0$ and $p_0$ respectively, which are generally chosen as large as possible considering the sample size given.
 We next provide the decision results in   Table \ref{testrealdtrev}, where the notations follow the same meaning as in   Table \ref{testre1}. In addition, to further illustrate the results,   we also report the boxplots of the $p$-values with respect to different chromosome pairs in Figure \ref{fig:boxpvalreverse},  where $(m_0,p_0)=(15,40)$. 
 
  \begin{table}
  \centering
\begin{tabular}{c|ccc|ccc}
\hline \hline
\multicolumn{1}{c|}{} & 
\multicolumn{6}{c}{ Chromosome pair} \\
$(m_0,p_0)$  & $8\to 8$ & $17\to 17$ & $22 \to  22$ & $8\to 17$ & $17 \to 22$ & $8\to 22$ \\
\hline
(10,45)   & \textsf{x} &\textsf{x} & \textsf{x}  & \textsf{x} &\textsf{x} &  \checkmark  \\
(15,45)   & \textsf{x} &\textsf{x} & \textsf{x}  & \textsf{x} &\textsf{x} &  \checkmark  \\
(15,40)   & \textsf{x} & \textsf{x} & \textsf{x} & \textsf{x} & \textsf{x} & \checkmark \\
(20,40)   & \textsf{x} &\textsf{x} & \textsf{x} & \textsf{x}& \textsf{x} & \checkmark \\
(20,35)   & \textsf{x} &\textsf{x} & \textsf{x} & \textsf{x}& \textsf{x} & \checkmark \\
\hline
\end{tabular}
\caption{Decision results} \label{testrealdtrev}
\end{table}
  
  \begin{figure}
\centering
 \includegraphics[width=1\textwidth]{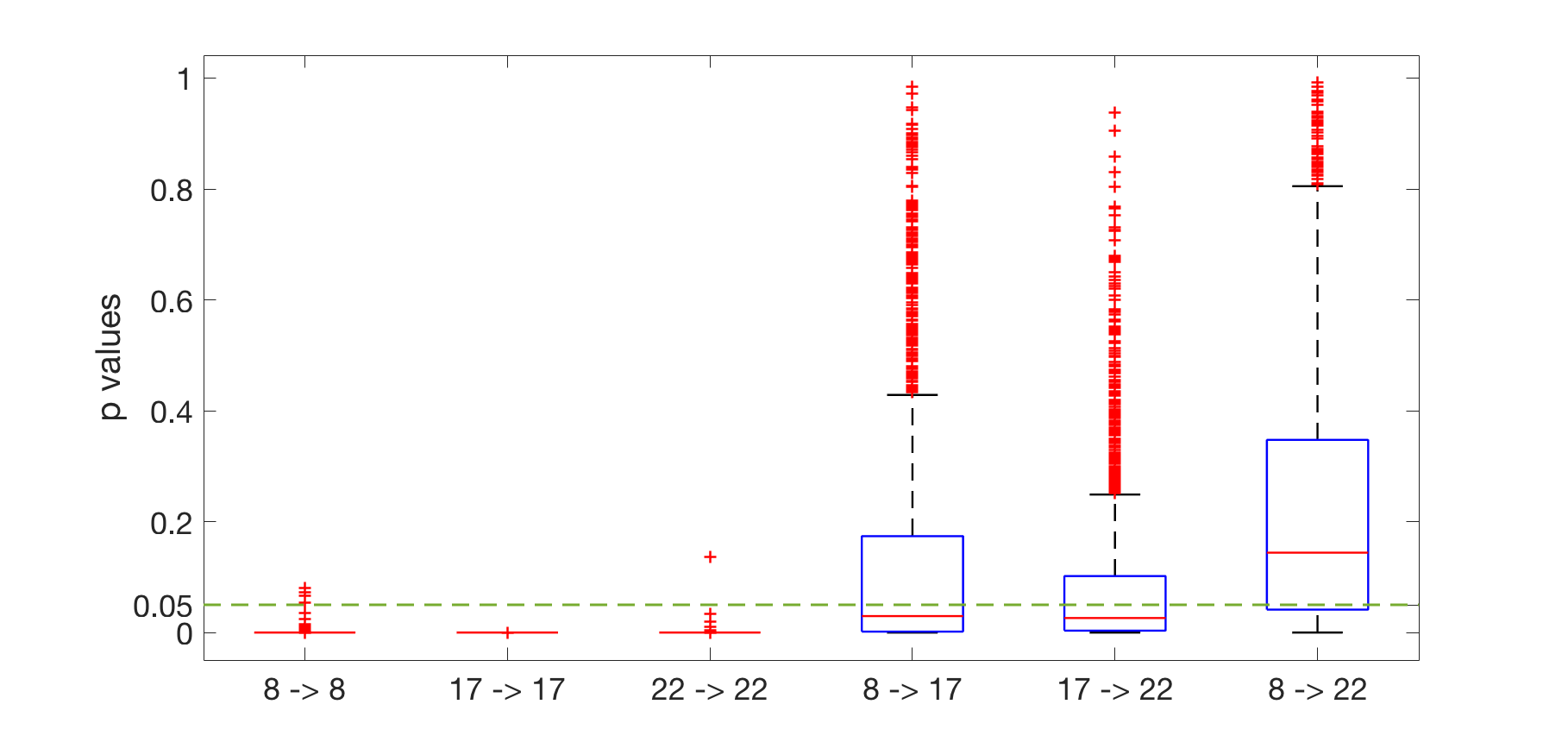}
 \caption{Boxplot of $p$ values for regressions on different chromosome pairs}
 \label{fig:boxpvalreverse}
\end{figure}
  
From the results, we can see that  the $p$-values presented in Figure \ref{fig:boxpvalreverse} support the test  results in Table \ref{testrealdtrev}. Particularly, in the boxplots of the regressions on the same chromosome pairs (the first three boxplots), the obtained $p$-values are significantly smaller than 0.05. For the regressions of the 17th   on the 8th chromosomes and   the 22nd   on the 17th chromosomes (the 4th and 5th boxplots),   the medians of the $p$-values are smaller 0.05.  These observations are consistent with  the rejections of the corresponding null hypotheses.  Moreover, for the regression of the 22nd   on the 8th chromosomes (the 6th boxplot), most of the $p$-values are greater than 0.05, which supports the decision that we accept the corresponding null hypothesis.

\bibhang=1.7pc
\bibsep=2pt
\fontsize{9}{14pt plus.8pt minus .6pt}\selectfont
\renewcommand\bibname{\large \bf References}
\expandafter\ifx\csname
natexlab\endcsname\relax\def\natexlab#1{#1}\fi
\expandafter\ifx\csname url\endcsname\relax
  \def\url#1{\texttt{#1}}\fi
\expandafter\ifx\csname urlprefix\endcsname\relax\def\urlprefix{URL}\fi

\bibliographystyle{chicago}
\bibliography{minimalBib}

\end{document}